\documentclass[reqno,10pt]{amsart}
\usepackage{amsmath,mathrsfs}
\usepackage{amssymb, amsmath}
\usepackage{graphicx}
\usepackage{dutchcal}
\usepackage{pdfsync}
\usepackage{pdfsync}
\usepackage{color}
\usepackage[colorlinks,
            linkcolor=red,
            anchorcolor=blue,
            citecolor=green
            ]{hyperref}

\def\inte#1{
\displaystyle\mathop{#1\kern0pt}^\circ }

\let\pa=\partial
\let\al=\alpha

\let\Ga=\Gamma

\let\f=\frac

\let\om=\omega
\let\G= \Gamma

\let\ka=\kappa
\def\mP{\mathbf{P}}

\def\mF{\mathbf{F}}

\def\cB{{\mathcal B}}
\def\cC{{\mathcal C}}

\def\cF{{\mathcal F}}

\def\cM{{\mathcal M}}
\def\fM{{\mathfrak M}}
\def\cN{{\mathcal N}}

\def\cP{{\mathcal P}}

\def\pa{\partial}

\def\virgp{\raise 2pt\hbox{,}}
\def\cdotpv{\raise 2pt\hbox{;}}

\def\eqdefa{\buildrel\hbox{\footnotesize def}\over =}

 \def\si{\sigma}
\def\C{\mathop{\mathbb C\kern 0pt}\nolimits}
\def\DD{\mathop{\mathbb D\kern 0pt}\nolimits}
\def\EE{\mathop{{\mathbb E \kern 0pt}}\nolimits}
\def\K{\mathop{\mathbb K\kern 0pt}\nolimits}
\def\N{\mathop{\mathbb N\kern 0pt}\nolimits}
\def\Q{\mathop{\mathbb Q\kern 0pt}\nolimits}
\def\R{\mathop{\mathbb R\kern 0pt}\nolimits}
\def\SS{\mathop{\mathbb S\kern 0pt}\nolimits}

\def \mh{\mathbf{h}}
\def \mj{\mathbf{j}}
\def \mn{\mathbf{n}}
\def \M{\mathcal M}

\def\F{{\mathfrak F }}
\def\U{{\mathfrak U }}
\def\tP{{\tilde{\mathcal{P}} }}
\def\tF{{\tilde{\mathfrak F} }}
\def\tm{{\tilde{m} }}
\def\tphi{{\tilde{\phi} }}

\def\<{\langle}
\def\>{\rangle}
\def\S{\mathbb{S}}
\def\R{\mathbb{R}}
\def\T{\mathbb{T}}
\def\Z{\mathbb{Z}}
\def\N{\mathbb{N}}

\def\si{\sigma}
\def\th{\theta}
\def\ga{\gamma}
\def\al{\alpha}
\def\be{\beta}
\def\de{\delta}
\def\vphi{\varphi}
\def\eqdefa{\buildrel\hbox{\footnotesize def}\over =}
\def\A{{\mathcal A}}
\def\gs{\gtrsim}
\def\ls{\lesssim}
\def\pa{\partial}
\def\vep{\varepsilon}
\def\ZZ{\mathop{\mathbb Z\kern 0pt}\nolimits}
\def\TT{\mathop{\mathbb T\kern 0pt}\nolimits}
\def\P{\mathop{\mathbb P\kern 0pt}\nolimits}

\def\na{\nabla}

\def\th{\theta}

\newcommand{\beq}{\begin{equation}}
\newcommand{\eeq}{\end{equation}}
\newcommand{\ben}{\begin{eqnarray}}
\newcommand{\een}{\end{eqnarray}}
\newcommand{\beno}{\begin{eqnarray*}}
\newcommand{\eeno}{\end{eqnarray*}}
\newtheorem{defi}{Definition}[section]
\newtheorem{thm}{Theorem}[section]
\newtheorem{lem}{Lemma}[section]
\newtheorem{rmk}{Remark}[section]
\newtheorem{col}{Corollary}[section]
\newtheorem{prop}{Proposition}[section]
\renewcommand{\theequation}{\thesection.\arabic{equation}}
\begin{document}

\title[Propagation of moments and sharp convergence rate for Boltzmann equation]
{Propagation of moments and sharp convergence rate for inhomogeneous non-cutoff Boltzmann equation with soft potentials
}

\author[C. Cao, L.-B. He and J. Ji]{Chuqi Cao, Ling-Bing He and Jie Ji}
\address[C.-Q. Cao]{Yau Mathematical Science Center and Beijing Institute of Mathematical Sciences and Applications, Tsinghua University\\
Beijing 100084,  P. R.  China.} \email{chuqicao@gmail.com}
\address[L.-B. He]{Department of Mathematical Sciences, Tsinghua University\\
Beijing 100084,  P. R.  China.} \email{hlb@tsinghua.edu.cn}
\address[J. Ji]{Department of Mathematical Sciences, Tsinghua University\\
Beijing 100084,  P. R.  China.} \email{jij19@mails.tsinghua.edu.cn}

%\thanks{$^*$Corresponding author: Chuqi Cao}

\begin{abstract}  We prove the well-posedness  for the  non-cutoff Boltzmann equation with soft potentials when the initial datum  is close to the {\it global Maxwellian} and has only polynomial  decay at the large velocities in   $L^2$ space. As a result, we get the {\it propagation of the exponential moments}  and the {\it sharp rates} of the convergence to the {\it global Maxwellian} which seems the first results for the original equation with soft potentials. The new ingredients of the proof lie in localized techniques, the  semigroup method as well as the propagation of the polynomial and exponential  moments in   $L^2$ space.
\end{abstract}

\maketitle

\setcounter{tocdepth}{1}
\tableofcontents

%%%%%%%%%%%%%%


\vskip 0.2cm


\renewcommand{\theequation}{\thesection.\arabic{equation}}
\setcounter{equation}{0}
%%%%%%%%%%%%%%%%%%%%%%%%%%%%%%%%%%%%%%%%%%%%%%
%%%%%%%%%%%%%%%%%%%%%%%%%%%%%%%%%%%%%%%%%%

\section{Introduction} The main purpose of the article is to investigate the well-posedness, propagation of moments and sharp convergence rates for the inhomogeneous Boltzmann equation with soft potentials. Differing from previous works, in our framework, we do not require that the linearized operator is self-adjoint and non-negative(see \eqref{pereq} for details). To be precise, we introduce the original Boltzmann equation which reads
\ben\label{1}
\partial_t F +v \cdot \nabla_x F =Q(F, F).
\een
Here $F(t, x, v) \geq 0$ is a distributional functions of colliding particles which, at time $t>0$ and position $x \in \T^3$, move with velocity $v \in \R^3$. We remark that the Boltzmann equation is one of the fundamental equations of mathematical physics and is a cornerstone of statistical physics. The Boltzmann collision operator $Q$ is a bilinear operator which acts only on the velocity variable $v$, that is
\[
 Q(G,F)(v)=\int_{\R^3}\int_{\mathbb{S}^2}B(v-v_*,\si)(G'_*F'-G_*F)d\si dv_*.
\]
%For a spatially homogeneous case, namely when $F = F(t, v)$, the equation simplifies into the spatially homogeneous Bolzmann equation
%\beno
%\partial_t F =Q(F, F).
%\eeno
Let us give some explanations on the collision operator.
\begin{enumerate}
\item  We use the standard shorthand $F=F(v),G_*=G(v_*),F'=F(v'),G'_*=G(v'_*)$, where $v',v'_*$ are given by
\ben\label{sigmare}
v'=\frac{v+v_*}{2}+\frac{|v-v_*|}{2}\sigma,~~v_*'=\frac{v+v_*}{2}-\frac{|v-v_*|}{2}\sigma,~\si\in\mathbb{S}^2.
\een
 This representation follows the parametrization of set of solutions of the physical law of elastic collision:
\beno
  v+v_*=v'+v'_*,  \quad
  |v|^2+|v_*|^2=|v'|^2+|v'_*|^2.
\eeno

\item  The nonnegative function $B(v-v_*,\si)$ in the collision operator is called the Boltzmann collision kernel. It is always assumed to depend only on $|v-v_*|$ and the deviation angle $\th$ through $\cos\th:=\frac{v-v_*}{|v-v_*|}\cdot\si$.
\item  In the present work,  our {\bf basic assumptions on the kernel $B$}  can be concluded as follows:
\begin{itemize}
  \item[$\mathbf{(A1)}$] The Boltzmann kernel $B$ takes the product form: $B(v-v_*,\si)=|v-v_*|^\ga b(\frac{v-v_*}{|v-v_*|}\cdot\si)$, where   $b$ is a nonnegative function.

  \item[$\mathbf{(A2)}$] The angular function $b(t)$ is not locally integrable and it satisfies
\[
\mathcal{K}\theta^{-1-2s}\leq \sin\theta b(\cos\theta)\leq \mathcal{K}^{-1}\theta^{-1-2s},~\mbox{with}~0<s<1,~\mathcal{K}>0.
\]

  \item[$\mathbf{(A3)}$]
  The parameter $\ga$ and $s$ satisfy the condition $-3<\gamma \le 1, s \in (0, 1)$ and $\gamma+2s>-1.$

  \item[$\mathbf{(A4)}$]  Without lose of generality, we may assume that $B(v-v_*,\si)$ is supported in the set $0\leq \th\leq \pi/2$, i.e.$\frac{v-v_*}{|v-v_*|}\cdot\si\geq0$, for otherwise $B$ can be replaced by its symmetrized form:
\beno
\overline{B}(v-v_*,\si)=|v-v_*|^\gamma\big(b(\frac{v-v_*}{|v-v_*|}\cdot\si)+b(\frac{v-v_*}{|v-v_*|}\cdot(-\si))\big) \mathrm{1}_{\frac{v-v_*}{|v-v_*|}\cdot\si\ge0},
\eeno
where $\mathrm{1}_A$ is the characteristic function of the set $A$.
\end{itemize}
  \end{enumerate}
\begin{rmk} For inverse repulsive potential, it holds that $\gamma = \frac {p-5} {p-1}$ and $s = \frac 1 {p-1}$ with $p > 2$. It is easy to check that $\gamma + 4s = 1$ which means that assumption $\mathbf{(A3)}$  is satisfied for the full range of the inverse power law model. Generally, the case $\gamma > 0$,  $\gamma = 0$, and  $\gamma < 0$ correspond to so-called hard, Maxwellian, and soft potentials respectively.
\end{rmk}
\subsection{Basic properties and the perturbation equation} We list some basic facts on the   equation.
 \smallskip

\noindent$\bullet$ {\bf Conservation Law.}  Formally if $F$ is a solution to equation \eqref{1} with the initial data $F_0$, then it enjoys the conservation of mass, momentum and the energy, that is,
\ben \f{d}{dt}\int_{\T^3\times{\R}^3} F(t,x,v)\varphi(v)dvdx= 0,\quad \varphi(v)=1,v,|v|^2\label{1.Conserv}.\een
For simplicity, we  normalize   the initial data $F_0$ in the following sense:
\ben\label{f0}
\int_{\T^3\times\R^3} F_0(x,v)\, dvdx= 1, \quad \int_{\T^3\times\R^3} F_0(x,v)\, v\, dvdx =0,
\quad \int_{\T^3\times\R^3} F_0(x,v)\, |v|^2 \, dvdx =3.
\een
This means that the  equilibrium associated to \eqref{f0} will be the standard Gaussian function, i.e.
\ben\label{DefM} \mu(v):=(2\pi)^{-3/2} e^{-|v|^2/2}, \een which has the same mass, momentum and energy as $F_0$.
\smallskip

\noindent$\bullet$ {\bf Perturbation Equation.} In our perturbation framework, we assume that
\beno
F=\mu+f,
\eeno where $f$ is the perturbed function. Then
the equation (\ref{1}) becomes
\ben\label{pereq}
\pa_tf+v\cdot\na_xf=Q(\mu,f)+Q(f,\mu)+Q(f,f):=-L(f)+Q(f,f),
\een
with the linearized operator $L=-Q(\mu,\cdot)-Q(\cdot,\mu)$. We emphasize that here $L$ is not a self-adjoint operator which is quite different from previous works. In fact, it brings the main obstruction to our problem.

\subsection{Brief review of previous results} Let us review those works which are related closely to ours.
\smallskip

\noindent$\bullet$\,{\it Existence and regularity theory for the Boltzmann equation.}  For the renormalized  solutions to the equation,  we refer to the pioneering work \cite{DL2} by DiPerna \& Lions for the angular cutoff case and  the work  \cite{AV} by Alexandre \&Villani    for the non-cutoff case.   For the conditional regularity theory and  local well-posedness result for the equation, we refer readers to \cite{AMUXY,CH,CH2,HJZ,HST,HMUY}.  Very recently,
 in \cite{IMS, IMS2, IS, IS2, IS3, S}, the authors got the global regularity with sole assumption on the uniform-in-time bounds of the macroscopic quantities, i.e.,\[
0<m_0 \le M(t, x) \le M_0, \quad E(t, x) \le E_0, \quad  H(t, x) \le H_0,
\]
for some constant $m_0, M_0, E_0, H_0$, where
\[
M(t, x) =\int_{\R^3} f(t, x, v) dv, \quad E(t, x) = \int_{\R^3} f(t, x, v) |v|^2 dv, \quad H(t, x) = \int_{\R^3} f(t, x, v) \ln f(t, x ,v) dv.
\]
Finally let us mention the recent work on the global existence of the equation with the rough initial data via De Giorgi methods(see  \cite{AMSY,SS}) for  hard  and moderate soft potentials respectively.
\smallskip

\noindent$\bullet$\,{\it Perturbing the equation in a symmetrized way.}
In this framework, we expand the solution as $F=\mu+\mu^{\f12}f$ where the perturbed part is in a specific way to make the linearized operator  self-adjoint and  non-negative. In fact, the corresponding linearized operator $L_\mu$ is defined by
\ben\label{Lmu}
L_{\mu} f  = -\frac {1} {\sqrt{\mu}}  (Q(\mu, \sqrt{\mu } f) - Q(\sqrt{\mu} f, \mu)  ),\quad \mbox{Null} (L_{\mu}  ) =\mbox{Span} \{\sqrt{\mu} ,\sqrt{\mu} v , \sqrt{\mu} |v|^2 \}.
\een
and enjoys the coercivity property(see \cite{P,MS}) as follows:
\[
 \quad (L_{\mu}f,  f)_{L^2_v} \ge C |||f|||^2, \quad  f \in \big(\hbox{Null} (L_{\mu}  )\big)^{\perp},
\]
where the triple norm represents the damping or the dissipation from the microscopic part of the function. For the global well-posedness around the global equilibrium, we refer readers to  Guo \cite{G2}  on the cutoff Boltzmann equation  and  Grassman-Strain \cite{GS} and Alexandre-Morimoto-Ukai-Xu-Yang \cite{ AMUXY2, AMUXY3, AMUXY4} on non-cutoff Boltzmann equation. See also \cite{DHWY,DLSS,DHYZ,HZ1,HZ2} and references therein for recent development. One may also check \cite{G4,GKTT} in the case of bounded domains.

\smallskip

\noindent$\bullet$\,{\it General perturbation theory via semi-group method.}  In the general perturbation framework, we refer readers to the earlier work \cite{Arkyard,Bernt}. By developing decay estimates on the resolvents and semigroups of non-symmetric operators in Banach spaces,  Gualdani-Mischler-Mouhot \cite{GMM} proved nonlinear stability for cutoff Boltzmann equation with  hard potentials in $L^1_vL^\infty_x(1+|v|^k), k>2$, with sharp rate of decay in time. It was later generalized to the non-cutoff case but still in the case of hard potentials in
\cite{HTT, AMSY}.

\bigskip

 Our current work is to investigate the equation in the similar setting but with soft potentials.
Our main goals can be summarized as follows:
\begin{enumerate}
 \item Prove the global well-posedness with the initial data that only have polynomial decay at large velocities in $L^2$ space;

 \item Prove the propagation of exponential moment in $L^2$ space;

 \item Obtain the sharp decay rate on the convergence and clarify its dependence on the initial data.
\end{enumerate}

\subsection{Basic notation and function space}  We begin with  basic notations.
\subsubsection{Notations}
  $(i)$ We write $a\ls b(a\gs b)$ indicate that there is a uniform constant $C$, which may be different on different lines, such that $a\leq Cb(a\geq Cb)$.
 We use the notation $a\sim b$ whenever $a\ls b$ and $b\ls a$.

 $(ii)$ We denote $C_{a_1,a_2,\cdots,a_n}$ by a constant depending on parameters $a_1,a_2,\cdots,a_n$. Moreover, we use parameter $\vep$ to represent different positive numbers much less than 1 and determined in different cases.

 $(iii)$ We write $a\pm$ indicate $a\pm\varepsilon$, where $\varepsilon>0$ is sufficiently small.   The notation $a^+$ means the maximum value of $a$ and $0$ and $[a]$ denotes the maximum integer which does not exceed $a$.

  $(iv)$ We use $(f, g)$ to denote the inner product of $f, g$ in the $v$ variable $(f, g)_{L^2_v}$ for short, if the integral is both in $x, v$  we will use $(f, g)_{L^2_{x, v}}$ to represent. We also use $(f, g)_{L^2_k}$ to denote $(f, g  \langle v \rangle^{2k})_{L^2_v}$.

 $(v)$ Gamma function and Beta function are defined by
\beno
 \Ga(x):=\int_0^\infty t^{x-1}e^{-t}dt,~ x>0,~B(p,q):=\int_0^1t^{p-1}(1-t)^{q-1}dt,~ p,q>0.
\eeno
We recall that Beta and Gamma functions fulfill the following properties:
\ben\label{gammafun}
B(p,q)=\frac{\G(p)\G(q)}{\G(p+q)};~\Ga(x+1)\sim_\vep\sqrt{2\pi x}\left(\frac{x}{e}\right)^x;~\Ga(x+z)\sim_{\vep,z} \Ga(x)x^z,~x\geq \vep>0,
\een
where $\sim_\vep$ or\ $\sim_{\vep,z}$ means the equivalence depends on parameter $\vep$ or $\vep$ and $z$.

\subsubsection{Function spaces} %We shall use the notation $\|\cdot\|$ or sometimes $|\cdot|$ for short to denote the norms.

  $(i)$  For real numbers $m,l$, denote $\<v\>:=(1+|v|^2)^{1/2}$, and we define the weighted Sobolev space $H_l^m$ by
  \beno
   H^m_l:=\{f(v)|\|f\|_{ H^m_l}=\|\<D\>^m\<\cdot\>^lf\|_{L^2}<+\infty\}.
  \eeno

  $(ii)$ $a(D)$ is a pseudo-differential operator with the symbol $a(\xi)$ and it is defined by
  \beno
   (a(D)f)(v):= \frac{1}{(2\pi)^3}\int_{\R^3}\int_{\R^3}e^{i(v-u)\xi}a(\xi)f(u)dud\xi.
  \eeno

$(iii)$ Moreover, we define a norm $|[\cdot]|$ as
\ben\label{3norm}
|[f]|_{L^2_{k+\ga/2}}=\left(\int_{\R^3} \int_{\R^3}  |v-v_*|^\ga\mu(v)|f(v_*)\<v_*\>^k|^2dv dv_*\right)^{1/2}.
\een
Indeed, we have $|[f]|_{L^2_{k+\ga/2}}\sim_\ga \|f\|_{L^2_{k+\ga/2}}$ since $\int_{\R^3}|v-v_*|^\ga\mu(v)dv\sim \<v_*\>^\ga$.

$(iv)$ The $L\log L$ space is defined as
\beno
L\log L:=\Big\{f(v)|\|f\|_{L\log L}=\int_{\R^3}|f|\log(1+|f|)dv\Big\}.
\eeno
%It is easy to verify that $\|f\|_{L\log L}\leq \|f\|^2_{L^2}.$

\subsection{Energy space, exponential  function and main results} To state our main results, we begin with the definitions of  energy spaces and  exponential  functions.

\begin{defi}[Energy spaces with polynomial weights]   For any function $f(x,v)$, we define
\beno
\|f\|_{H^n_xH^m_l}:=\left(\sum_{|\beta|\le n}\int_{\T^3}\|\pa_x^\beta f(x,\cdot)\|^2_{H^m_l}dx\right)^{1/2}.
\eeno
When $n=0$, we denote $H^0_xH^m_l:=L^2_xH^m_l$ for short. The energy space $X_k$    can be defined as follows:
\ben\label{X_k}
X_k:=\bigg\{ f\in L^2_{x,v}\big| \|f\|^2_{X_k}:=\sum_{|\alpha|=2}\|\pa^\al_xf\|_{ L^2_{x}L^2_{k-4|\al|}}^2+\|f\|_{L^2_x L^2_{k}}^2<\infty \bigg\},\quad \mbox{where} \quad k\ge8.
\een
The inner product in $X_k$ can be defined by
\beno
(f,g)_{X_k}:=\int_{x\in\TT^3}\bigg(\sum_{|\al|=2}\big(\<\cdot\>^{k-4|\al|}\pa^\al_xf,\<\cdot\>^{k-4|\al|}\pa^\al_xg\big)_{L^2_{v}}+\big(\<\cdot\>^{k}f,\<\cdot\>^{k}g\big)_{L^2_{v}}\bigg)dx.
\eeno
Similarly, we define
\ben\label{Y_k}
Y_k&:=&\bigg\{ f\in L^2_{x,v}\big| \|f\|^2_{Y_k}:=\sum_{|\alpha|=2}\|\pa^\al_xf\|_{ L^2_{x}H^s_{k-4|\al|+\ga/2}}^2+\|f\|_{L^2_x H^s_{k+\ga/2}}^2<\infty \bigg\},\\
\notag\bar{Y}_k&:=&\bigg\{ f\in L^2_{x,v}\big| \|f\|^2_{\bar{Y}_k}:= \sum_{|\alpha|=2}\|\pa^\al_xf\|_{ L^2_{x}H^s_{k-4|\al|+\ga/2+2s}}^2+\|f\|_{L^2_x H^s_{k+\ga/2+2s}}^2<\infty \bigg\}.
\een
\end{defi}

\begin{defi}[Energy spaces with exponential  weights]   Let $\beta \in (0,2),a>0$, $M\in\N$. The exponential function is defined as follows:
\ben\label{ex}
\mathcal{G}_\be^{a,M}(v):=\sum_{k=M}^\infty \f{a^{\f4\be ks}(\<v\>^\be)^{\f4\be ks}}{(\mathbb{G}(\f4\be k))^s},\quad\mbox{where}\quad \mathbb{G}(\f4\be k):=\Gamma(\f4\be k+1). \een
 If $\ka\in\N, s\in\R^+$ with $\ka s>6$, then the   energy space with exponential  weight can be defined by
\ben\label{exp space}
\mathcal{X}_{a,M}^\beta:=\bigg\{ f\in L^2_{x, v}\big|\|f\|_{\mathcal{X}_{a,M}^\beta}^2:=\sum_{|\al|=2}\|(\mathcal{G}_\be^{a,M})^{1/2}\<\cdot\>^{-\ka s|\al|}\pa^\al_x f\|_{L^2_{x,v}}^2+\|(\mathcal{G}_\be^{a,M})^{1/2}f\|_{L^2_{x,v}}^2<\infty\bigg\},
\een
For simplicity, we denote $\mathcal{G}_\be^{a}(v)=\mathcal{G}_\be^{a,0}(v)$ and $\mathcal{X}_{a}^\beta=\mathcal{X}_{a,0}^\beta$.
\end{defi}

\begin{rmk}
It is not difficult to see that  $\mathcal{G}_\be^{a}$ behaves like the exponential function $e^{C\<v\>^\be}$ (see Prop. \ref{expweight} below). The main advantage of $\mathcal{G}_\be^{a}$ results from the fact that it perfectly matches the structure of the collision operator which enables us to  prove the propagation of moments.
\end{rmk}

\begin{thm}[{\bf Global well-posedness and sharp decay rate}]\label{globaldecay}
  Consider the Cauchy problem of
\ben\label{14}
\partial_t f +v\cdot\na_xf= -L f +Q(f, f), \quad f|_{t=0} =f_0, \quad \Pi f_0 = 0, \quad \mu+f_0\geq0,
\een
with  the kernel $B$ verifying \big($\mathbf{(A1)}$-$\mathbf{(A4)}$\big) and
\ben\label{Pi}
\Pi g = \sum_{\varphi \in \{1,v_1,v_2,v_3,|v|^2\}} \left(\int_{\T^3\times \R^3 } g \varphi dx dv  \right)\varphi \mu.
\een
Assume that $\|f_0\|_{X_{k_1}}<\vep_0$ with $k_1:=48$ and $\varepsilon_0$ sufficiently small.
\begin{enumerate}
  \item[(i)]{\bf (When the initial data only have polynomial moments).} Let $\gamma \in (-3, 0]$.
   Then there exists a unique global solution $f$ satisfying $\mu+f(t, x, v)\ge 0$ and $\|f\|_{L^\infty([0,\infty);X_{k_1})\cap L^2([0,\infty);Y_{k_1})}\ls \vep_0$. Moreover for any $k\ge k_1$,  if $\|f_0\|_{X_{k}}<\infty$, then $f\in L^\infty([0,\infty);X_{k})\cap L^2([0,\infty);Y_{k})$.
\begin{itemize}
  \item[(1)]{\bf (Maxwellian Molecules($\gamma=0$))} There exists a  constant $  \lambda>0$ such that
\beno
\|f(t)\|_{X_k}\lesssim  e^{-\lambda t}\|f_0\|_{X_k}.
\eeno
 \item[(2)]{\bf (Soft potentials($\gamma \in (-3, 0)$))} If $k_1\leq\tilde{k}<k$, then
\ben\label{softdecaypo}
\|f(t)\|_{X_{\tilde{k}}}\lesssim  \<t\>^{-\frac{k-\tilde{k}}{|\ga|}}\|f_0\|_{X_k}.
\een
\item[(3)]{\bf (Upper bound of convergent rate for soft potentials)} Given $k>48$. Define the set of solutions to \eqref{14} by
$\mathfrak{A}:=\{f(t,x,v)~  \mbox{satisfies \eqref{14}}| \|f_0\|_{X_{k_1}}<\vep_0, \|f_0\|_{L^2_xL^1_k}\ls 1\}$. 
Then   for any $\eta>0$, we have
\ben\label{upperbound}
\sup_{f\in \mathfrak{A}}\|f(t)(1+t)^{\f k{|\ga|}+\eta}\|_{L^\infty_tL^1_{x,v}}=+\infty.
\een
\end{itemize}

 \item[(ii)] {\bf (When the initial data   have exponential moments).} Let $\gamma \in (-3, 1], s \in (0, 1), \gamma+2s>-1$. If $\|f_0\|_{\mathcal{X}_{b}^\beta}< \infty$ with $\be\in(0,2)$ and $b\in\R^+$, then there exists a constant $a=a_{f_0,\beta,b}\in(0,b)$ such that
\ben\label{softdecayexp}
\|f(t)\|_{\mathcal{X}_{a}^\beta}\leq c_1 e^{-c_2t^{\varrho}}\|f_0\|_{\mathcal{X}_{b}^\beta},
\een
  where $c_1,c_2$ are constants depending on $a,b,f_0$ and $\varrho=1$ if $\ga+\be s \ge 0$, $\varrho = \frac{\beta}{\beta-(\ga+\be s)}$ if $\gamma+\beta s<0$.
\end{enumerate}
\end{thm}
\begin{rmk} The smallness assumption is only imposed on the initial data with finite polynomial moment. Compared to the results obtained in the symmertrized perturbation framework, we do not require the smallness assumption to prove the propagation of the exponential moment.
\end{rmk}

Several comments on the results are in order:

 \subsubsection{Comment on the global well-posedness for the initial data with polynomial moment.} As we mentioned before, in general perturbation framework, the Cauchy problem of the non-cutoff Boltzmann equation with hard potential had been solved in   \cite{HTT, AMSY} if the initial data  only have finite polynomial moment. In the present work, we focus on the Maxwellian molecules and soft potentials case(i.e., $\gamma\in(-3,0]$).

 \smallskip

To prove the desired results, we rely heavily on the following observations. Some of them are new and have independent interest.
\smallskip

\noindent (i) We develop the coercivity estimates for $(-Q(F, f),f\langle \cdot\rangle^{2k})_{L^2_v}$ where $F$ is a non-negative function(see Theorem \ref{T24}). Roughly speaking, we show that
\ben\label{CR1}(-Q(F, f),f\langle \cdot\rangle^{2k})_{L^2_v}\sim C_F(\|f\|_{H^s_{k+\gamma/2}}^2+k^s\|f\|_{L^2_{k+\gamma/2}}^2)+ \mathrm{L.O.T.}
\een
Here we not only catch the explicit factor $k^s$ in front of the damping term $\|f\|_{L^2_{k+\gamma/2}}^2$ but also
 have good control of the lower order terms(L.O.T), in particular for the coefficients related to $k$. It is the key point   to define  the special exponential moment function $\mathcal{G}_\be^{a}$ and then prove its propagation.
\smallskip

 \noindent (ii) To control the lower order terms in \eqref{CR1}, we resort to the semi-group method  by Gualdani-Mischler-Mouhot \cite{GMM}. Loosely speaking,  if $\mathcal{S}_L$ denotes the semi-group generated by $-L$ defined in \eqref{pereq}, then
 we  define a scalar product by
\[
((f, g))_k := (f, g)_{L^2_k} + \eta\int_0^{+\infty} (\mathcal{S}_L(\tau )f, \mathcal{S}_L(\tau) g )_{L^2_v} d\tau.
\] Due to the fact that
\[
\int_0^{+\infty} (\mathcal{S}_L(\tau ) L f, \mathcal{S}_L(\tau) f )_{L^2_v} d\tau = -\int_0^{+\infty} \frac d {d\tau} \Vert \mathcal{S}_L(\tau) f \Vert_{L^2}^2 d\tau = \Vert f \Vert_{L^2}^2-\lim_{\tau \to \infty} \Vert \mathcal{S}_L(\tau) f \Vert_{L^2}^2 =  \Vert f \Vert_{L^2}^2,
\]
 we deduce from \eqref{CR1} that
\[
((Lf, f))_k = (Lf, f)_{L^2_k} + \eta\int_0^\infty (\mathcal{S}_L(\tau ) L f, \mathcal{S}_L(\tau) f)_{L^2_v} d\tau \sim \|f\|_{H^s_{k+\gamma/2}}^2+k^s\|f\|_{L^2_{k+\gamma/2}}^2,
\]
if we choose  proper $\eta$. This indicates that  $L$  still can be regarded as a non-negative operator in a suitable function space.
Here technically we   will use the  regularity method by F. H\'earu in \cite{H2} and the dual method(i.e., to estimate $L^*$, the dual operator of $L$) to get the short time and long time behavior of the semi-group $\mathcal{S}_L$.
\smallskip

\noindent (iii) To implement the strategy, in particular to handle the nonlinear terms,
 technically we have to estimate  the   commutator  $(\langle D\rangle^\ell Q(g, h) - Q(g,\langle D \rangle^\ell h),\langle D\rangle^\ell f)_{L^2_v}$ for all  $\ell \ge 0$. Thanks to the localized techniques  in both phase and frequency space, fortunately we get the following results which have independent  interest:

\begin{thm}\label{le1.24}
Let $\gamma \in (-3, 1], s \in (0, 1), \gamma+2s>-1,~\ell\in \R^+, ~\{\omega_i\}_{i=1}^8\subset \R$ satisfies
$\omega_i+\omega_{i+1}=\gamma+2s-1,i=1,3,5,7$. Suppose that $g, f$ and $h$ are smooth functions. For any small constant $\delta>0$ we have
 \smallskip

$\bullet$   If $2s<1/2$, then for any $N \ge 0$ we have
  \beno
&&|(\<D\>^\ell Q(g,h)-Q(g,\<D\>^\ell h),\<D\>^\ell f)_{L^2_v}|\ls C_{N,\ell}(\|g\|_{L^1_{(\ga+2s)^++(-\omega_1)^++(-\omega_2)^++\de}}\|h\|_{H^{\ell}_{\omega_1}}\|f\|_{H^{\ell}_{\omega_2}}\\
&&+\|g\|_{L^2_{2+(-\omega_3)^++(-\omega_4)^+}}\|h\|_{H^{\ell}_{\omega_3}}\|f\|_{H^{\ell}_{\omega_4}}+\|h\|_{L^2_{(-\omega_5)^++(-\omega_6)^+}}\times\|g\|_{H^{\ell}_{\omega_5}}\|f\|_{H^{\ell}_{\omega_6}}+\|h\|_{H^\ell_{(-\om_7)^++(-\om_8)^+}}\|g\|_{L^{2}_{\om_7}}\\
&&\times\|f\|_{H^{\ell}_{\om_8}}+(\|g\|_{L^1_{(\ga+2s)^++(-\omega_1)^++(-\omega_2)^++\de}}+\|g\|_{L^2_{2+(-\omega_3)^++(-\omega_4)^+}}) \|h\|_{H_{-N}^{-N}}\|f\|_{H_{-N}^{-N}}+(\|h\|_{L^2_{(-\omega_5)^++(-\omega_6)^+}}\\
&&+\|h\|_{H^\ell_{(-\omega_7)^++(-\omega_8)^+}})\|g\|_{H_{-N}^{-N}}\|f\|_{H_{-N}^{-N}}).
\eeno

 $\bullet$   If $1/2\le 2s  \le  2$,  we set
$a, b, c_i, d_i \ge  0, i=1,2,3$ and $a+b=(2s-1)1_{2s>1}+(2s-1+\de)1_{2s=1}+0_{1/2\leq2s<1}$, $c_i, d_i\in[0,2s-1/2]$ with $c_i+d_i=2s-1/2,i=1,2,3$. Then  for any $N \ge 0$,
  \beno
&&|(\<D\>^\ell Q(g,h)-Q(g,\<D\>^\ell h),\<D\>^\ell f)_{L^2_v}|\ls C_{N,\ell}(\|g\|_{L^1_{(\ga+2s)^++(-\omega_1)^++(-\omega_2)^++\de}}\|h\|_{H^{\ell+a}_{\omega_1}}\|f\|_{H^{\ell+b}_{\omega_2}}\\
&&+\|g\|_{L^2_{2+(-\omega_3)^++(-\omega_4)^+}}\|h\|_{H^{\ell+c_1}_{\omega_3}}\|f\|_{H^{\ell+d_1}_{\omega_4}}+\|h\|_{L^2_{(-\omega_5)^++(-\omega_6)^+}}\|g\|_{H^{\ell+c_2}_{\omega_5}}\|f\|_{H^{\ell+d_2}_{\omega_6}}+\|h\|_{H^\ell_{(-\om_7)^++(-\om_8)^+}}\\
&&\times\|g\|_{H^{c_3}_{\om_7}}\|f\|_{H^{\ell+d_3}_{\om_8}}+(\|g\|_{L^1_{(\ga+2s)^++(-\omega_1)^++(-\omega_2)^++\de}}+\|g\|_{L^2_{2+(-\omega_3)^++(-\omega_4)^+}})\|h\|_{H_{-N}^{-N}}\|f\|_{H_{-N}^{-N}}\\
&&+(\|h\|_{L^2_{(-\omega_5)^++(-\omega_6)^+}}+\|h\|_{H^\ell_{(-\omega_7)^++(-\omega_8)^+}})\|g\|_{H_{-N}^{-N}}\|f\|_{H_{-N}^{-N}}).
\eeno
\end{thm}

\subsubsection{Comment on propagation of the exponential moment.} The propagation of exponential moment for the inhomogeneous Boltzmann equation with cutoff for the hard potentials had been investigated in \cite{GMM}, see also \cite{F}. To our best knowledge, our result \eqref{softdecayexp} seems to the first one on the propagation of the exponential moment for the original equation with soft potentials, in particular for the inverse power law model.
\smallskip

\noindent (i)  We begin with a proposition which states that $\mathcal{G}_\be^{a}(v)$  behaves like  $e^{C_{s,\be}a\<v\>^\be}$.
\begin{prop}\label{expweight}
Suppose that $a>0$ and $\be\in (0,2)$. Then there exists constant $C$ which only depends on $a,\be$ and $s$ such that
\ben\label{pro1.1}
C_{a,s,\be}e^{C_{s,\be}\f a 2\<v\>^\be}\leq\mathcal{G}_\be^{a}(v)\leq C_{a,s,\be}e^{C_{s,\be}a\<v\>^\be}.
\een
Here and below, the constant  $C_{a,s,\be}$ depends on $a,s,\be$ and can be different in different lines. Thus for fixed $s,\be$ and any $a\in\R^+$, there exists $b<a$ such that
\ben\label{pro1.2}
\mathcal{G}_\be^{b}(v)/\mathcal{G}_\be^{a}(v)\leq C_{a,b,s,\be}e^{-C_{a,b,s,\be}\<v\>^\be}.
\een
\end{prop}
\begin{proof}
Firstly, we claim that for $a>0,\mathcal{n}\in\N^+$,
\ben\label{1.11}
  C_{a,n}e^{\f a2\<v\>^\be}\leq \sum_{k=0}^\infty \f{a^{\mathcal{n}k}(\<v\>^\be)^{\mathcal{n}k}}{(\mathcal{n}k)!}\leq C_{a,n}e^{a\<v\>^\be}\quad\mbox{and}\quad C_{a,n}e^{C_{n}\f a2\<v\>^\be}\leq\sum_{k=0}^\infty\f{a^{\mathcal{n}k}(\<v\>^\be)^{\mathcal{n}k}}{(k!)^\mathcal{n}}\leq C_{a,n}e^{C_{n}a\<v\>^\be}.
\een

\noindent\underline{ {\it (1) Proof of the first estimate in \eqref{1.11}.}} It is obvious that $\sum_{k=0}^\infty \f{a^{\mathcal{n}k}(\<v\>^\be)^{\mathcal{n}k}}{(\mathcal{n}k)!}\leq e^{a\<v\>^\be}$. To complete the estimate, we observe that
\beno
 \mathcal{n} \f{a^{\mathcal{n}k}(\<v\>^\be)^{\mathcal{n}k}}{(\mathcal{n}k)!}\geq \min\{1, a^{-\mathcal{n}}\}\<v\>^{-\mathcal{n}\be}\sum_{i=0}^{\mathcal{n}-1} \f{a^{\mathcal{n}k+i}(\<v\>^\be)^{\mathcal{n}k+i}}{(\mathcal{n}k+i)!},
\eeno
which implies that
\beno
\sum_{k=0}^\infty \f{a^{\mathcal{n}k}(\<v\>^\be)^{\mathcal{n}k}}{(\mathcal{n}k)!}\geq C_{a,\mathcal{n}}\<v\>^{-\be \mathcal{n}}e^{a\<v\>^\be}\geq C_{a,\mathcal{n}}e^{\f a2\<v\>^\be}.
\eeno

\noindent\underline{ {\it (2) Proof of the second estimate in \eqref{1.11}.}}  Since the Gamma function fulfills that
\ben\label{1.12}
\Ga(x+1)\sim_\vep \sqrt{2\pi x}(x/e)^x,x\geq\vep>0.
\een
We have  $(k!)^{\mathcal{n}}/(nk)!\sim (2\pi k)^{(\mathcal{n}-1)/2}/(\mathcal{n})^{\mathcal{n}k-1/2}$, from
which together with the first equivalence implies the second estimate.

\noindent\underline{ {\it (3) Proof of \eqref{pro1.1}}.} We only need to prove that $\mathcal{G}_\be^{a}(v)\sim\sum_{k=0}^\infty \f{C^k(\<v\>^\be)^{\mathcal{n}k}}{(k!)^\mathcal{n}}$.
Again from (\ref{1.12}), we have that for $k\geq1$,
 $\Ga^s(\f4\be k+1)\sim C_{s,\beta}k^{s/2}(\f 4\be k/ e)^{\f4\be ks}\sim C_{s,\be}^kk^{(1/2-2/\beta)s}\Ga^{\f4\be s}(k+1),
$ which implies that
\beno
\mathcal{G}_\be^{a}(v)=\sum_{k=0}^\infty \f{a^{\f4\be ks}(\<v\>^\be)^{\f4\be ks}}{(\mathbb{G}(\f4\be k))^s}\sim\sum_{k=0}^\infty\f{k^{(1/2-2/\beta)s}(C_{s,\be}a)^{\f4\be ks}(\<v\>^\be)^{\f4\be ks}}{(k!)^{\f4\be s}}.
\eeno
If $\f4 \be s\in\N^+$, we can prove (\ref{pro1.1}) directly by (\ref{1.11}).
Otherwise, suppose that $\mathcal{n}=[\f4 \be s]$, the maximum integer which does not exceed $\f4 \be s$. Then $\f4 \be s=(\mathcal{n}+1)\th$ and $\mathcal{n}=\vartheta\f4 \be s$ with $\th,\vartheta\in(0,1)$. By H$\ddot{o}$lder inequality, we have
\beno
&&\sum_{k=0}^\infty\f{k^{(1/2-2/\beta)s}(C_{s,\be}a)^{\f4\be ks}(\<v\>^\be)^{\f4\be ks}}{(k!)^{\f4\be s}}=\sum_{k=0}^\infty\f{k^{(1/2-2/\beta)s}(C_{s,\be}a)^{(\mathcal{n}+1)\th k}(\<v\>^\be)^{(\mathcal{n}+1)\th k}}{(k!)^{(\mathcal{n}+1)\th}}\\
&&\leq \Big(\sum_{k=0}^\infty\f{(2^{\f{1-\th}{\th(\mathcal{n}+1)}}C_{s,\be}a)^{(\mathcal{n}+1)k}(\<v\>^\be)^{(\mathcal{n}+1) k}}{(k!)^{(\mathcal{n}+1)}}\Big)^\th\Big(\sum_{k=0}^\infty2^{-k}k^{\f{(1/2-2/\beta)s}{1-\theta}}\Big)^{1-\th}\leq C_{s,\be}\bigg(\sum_{k=0}^\infty\f{(2C_{s,\be}a)^{(\mathcal{n}+1)k}(\<v\>^\be)^{(\mathcal{n}+1) k}}{(k!)^{(\mathcal{n}+1)}}\bigg)^\theta.
\eeno
Similarly, we   also have
\beno
\sum_{k=0}^\infty\f{(\f12C_{s,\be}a)^{\mathcal{n}k}(\<v\>^\be)^{\mathcal{n} k}}{(k!)^{\mathcal{n}}}=\sum_{k=0}^\infty\f{{(\f12C_{s,\be}a)}^{\f 4\be s\vartheta k}(\<v\>^\be)^{\f4\be k \vartheta s}}{(k!)^{\f4\be \vartheta s}}\leq C_{s,\be} \bigg(\sum_{k=0}^\infty\f{k^{(1/2-2/\beta)s}(C_{s,\be}a)^{\f4\be ks}(\<v\>^\be)^{\f4\be ks}}{(k!)^{\f4\be s}}\bigg)^{\vartheta}.
\eeno
Then we conclude the desired results by (\ref{1.11}).
\end{proof}

From the above proposition, it is clear that the parameters $a$ and $\beta$ are used to classify the exponential function $\mathcal{G}_\be^{a}$. We emphasize   that $\mathcal{G}_\be^{a}$ with $\beta\in (0,2)$ covers almost the whole classes of the exponential functions  in the present work.  The motivation of the factorial term $(\mathbb{G}(\f4\be k))^s$ in $\mathcal{G}_\be^{a}(v)$ stems from  the factor $k^s$ in $L^2$ coecivity estimate \eqref{CR1}. The series form of  $\mathcal{G}_\be^{a}(v)$ comes from the strategy that the proof of the propagation of the exponential moment can be reduced to the proof of the propagation of the $k$-th moment. The reduction enables to catch the accumulation effect by summation with respect to $k$. In particular, it allows us to  prove that the linear operator $L$ does have the spectral gap in energy space with exponential weights if $\gamma+\beta s\ge0$. Of course the price we need to pay is to keep track of the dependence of $k$ in each estimate which requires  more careful analysis.

\smallskip

(ii) In what follows, we show main difference on the propagation of the exponential moment between in $L^1$ space for the homogeneous equation and in $L^2$ space for the inhomogeneous equation.  Indeed, in $L^1$ space  the homogeneous equation behaves more or less like a linear equation in particular for the propagation of the moment. While in   $L^2$ space for the inhomogeneous equation, the strong nonlinearity will force us to face the   technical problem which arises from the upper bound of the collision operator, that is,
\[|(Q(f, g), h)_{L^2_v}| \le \Vert f \Vert_{L^2_5} \Vert g \Vert_{H^s_{\gamma/2+2s}}    \Vert h \Vert_{H^s_{\gamma/2}}.\] It is obvious that in the upper bound we   cost additional weight $\langle v\rangle^{2s}$ compared to the gain of the regularity in \eqref{CR1}. To prove the desired result,  on one hand, we consider the problem in suitable weighted Sobolev spaces due to strong nonlinearity. On the other hand,  we introduce the factor $\f4\be$ in the fractional factorial term $(\mathbb{G}(\f4\be k))^s$ in $\mathcal{G}_\be^{a}(v)$   to   balance the additional weight $\langle v\rangle^{2s}$.
\smallskip

%(iii). It is worth to emphasize that in the case of $\gamma+\beta s\ge0$ we can prove the propagation of the exponential moment without losing anything from the initial data, that is, $a=b$ in \eqref{softdecayexp}. However, to achieve it, we require the smallness condition that $\|f_0\|_{\mathcal{X}_{b}^\beta}\ll1$.

\subsubsection{Comment on the sharp decay rate for  soft potentials} Our results indeed give the classification of the sharp  convergence rate with respect to the   initial data. Some of them cannot be observed in the symmetrized perturbation theory.
\smallskip

\noindent$\bullet$ \underline{\it When the initial data only have $k$-th polynoimal moment.} 
As you can see in Theorem \ref{globaldecay}, on one hand, we obtain the convergent rate $k/|\ga|$ for the initial data with $k$-th moment in \eqref{softdecaypo}. On the other hand, for any $\eta>0
$ in \eqref{upperbound}, there is no consistent constant $M$ such that $\|f(t)\|_{L^1_{}}\leq M(1+t)^{-(\f k{|\ga|}+\eta)}\|f_0\|_{L^1_k}$.
These indicate that convergent rate $k/|\ga|$ can be  regarded as the optimal one for the solutions to the equation.
 \smallskip

\noindent$\bullet$ \underline{\it When the initial data  have exponential moment $\mathcal{G}_\be^{a}$ with $\gamma+\beta s\ge0$.} Loosely speaking, by the energy estimates, we can prove that
\[ \f{d}{dt} \|f\|_{\mathcal{X}_{a}}^2+c \|f\|_{\mathcal{X}_{a}}^2\le0, \]
thanks to the definition of $\mathcal{G}_\be^{a}$ and \eqref{CR1}.    It indicates that the linear operator $L$  with long-range interaction does have the spectral gap in space $L^2(e^{c\<v\>^\be}dv)$. In fact, one may prove that
\ben\label{Lprop}
(Lf, f e^{2 \alpha\langle  v \rangle^{\beta}} )_{L^2_v}\ge c \Vert \langle v \rangle^{\frac {\gamma} 2+\f \beta2 s } f e^{ \alpha\langle  v \rangle^{\beta}}  \Vert_{L^2}^2 -C \Vert f \Vert_{L^2}^2, \quad \alpha>0, \quad \beta \in (0, 2).
\een
This is consistence with the well-known result that the linearized and self-adjoint operator $L_\mu$ defined in \eqref{Lmu} does have the spectral gap if $\gamma+2s\ge0$.

 \smallskip

\noindent$\bullet$ \underline{\it  When the initial data  have exponential moment $\mathcal{G}_\be^{a}$ with $\gamma+\beta s<0$.}
In this situation, it is expected that stretched exponential decay rate can be obtained since the linear operator does not have spectral gap in   space $L^2(e^{c\<v\>^\be}dv)$. To see the sharpness of the decay rate in \eqref{softdecayexp}, we only need to consider the toy model which stems from \eqref{Lprop}:
\[ \pa_t (fe^{  \alpha\langle  v \rangle^{\beta}})+\langle v\rangle^{\gamma+\be s} (fe^{  \alpha\langle  v \rangle^{\beta}})=0.\]
This implies that for any $R>0$,
\[\|f(t)\|_{L^2}=\|e^{-\<v\>^{\ga+\be s}t-\alpha_1\<v\>^\be}e^{\alpha_1\<v\>^\be}f_0\|_{L^2}\le \big(e^{-\<R\>^{\ga+\be s}t}+e^{-\alpha_1\<R\>^\be}\big)\|e^{\alpha_1\<v\>^\be}f_0\|_{L^2}. \]
By choosing $\<R\>=t^{\f1{\be-(\ga+\be s)}}$, we finally obtain that
\beno
\|f(t)\|_{L^2}\leq Ce^{-Ct^{\varrho}}\|e^{\al_1\<v\>^\be}f_0\|_{L^2},
\eeno
where $ \varrho = \frac{\beta}{\beta-(\ga+\be s)}$. This shows the sharpness of \eqref{softdecayexp}.

\subsubsection{Comment on the strong connection between Boltzmann and Landau equations.} We comment that the similar results had been proved for the Landau equations(see \cite{CTW, CM}). Since these two equations are linked by so-called grazing collisions limit(for instance see \cite{DHYZ}), it seems possible to  solve both of equations simultaneously in the general perturbation setting in a unified way. In fact, the problem can be reduced to consider the  perturbation theory for the rescaled Boltzmann equation  by rescaling the collision operator $Q$ in a proper way. We leave it as a future work.

 \subsection{Organization of the paper}
In Section 2,  we give the upper bounds and coercivity estimate on collision operator $Q$. Section 3 is devoted to the global well-posedness, propagation of moments and convergence rate.  Basic properties for the Boltzmann equation, useful lemmas and commutator estimates between collision operator $Q$ and differential operator $\<D\>^\ell$ will be given in Section 4.

\setcounter{equation}{0}
\section{Analysis of the collision operator}
In this section, we shall derive several estimates for the collision operator $Q$, including the upper, lower bounds and the commutators, which will be frequently used in the later. We begin with some preliminaries.

\subsection{Two technical lemmas for the collision} We present two lemmas related to the so-called Povzer's inequality. In particular, we   introduce two decompositions for  $v'$ in terms of $v$ and  $v_*$.
\begin{lem}(see \cite{HJZ,AMSY})\label{L18}
We have the following two decompositions about $\<v'\>^2$:
\begin{itemize}
  \item  Let $\mathbf{h}:=\frac{v+v_*}{|v+v_*|},\mn:=\frac{v-v_*}{|v-v_*|}$, $\mj:=\frac{\mh-(\mh\cdot\mn)\mn}{\sqrt{1-(\mh\cdot\mn)^2}}$, $E(\theta):=\<v\>^2\cos^2(\theta/2)+\<v_*\>^2\sin^2(\theta/2)$ and  $\si=\cos\th\mn+\sin\th\hat{\omega}$ with $\hat{\omega}\in\S^1(\mn)\eqdefa\{\hat{\om}\in\S^2|\hat{\om}\perp \mn\}$. Then\,  $\mh\cdot\si=(\mh\cdot\mn)\cos\th+\sqrt{1-(\mh\cdot \mn)^2}\sin\th(\mj\cdot\hat{\omega})$ and
      \beno
      \<v'\>^2=E(\th)+\sin\theta(\mj \cdot\hat{\omega})\tilde{h}
      \eeno
        with
      $\tilde{h}:=\f12\sqrt{|v+v_*|^2|v-v_*|^2-(v+v_*)^2\cdot(v-v_*)^2}=\sqrt{|v|^2|v_*|^2-(v\cdot v_*)^2}.$
  \item If $\omega = \frac {\sigma - (\sigma \cdot \mn)\mn } {|\sigma - (\sigma \cdot \mn)\mn |}$(which implies $\omega \perp (v-v_*)$), then
\ben\label{Eth}
\langle v' \rangle^2 = E(\theta)+  \sin  \theta |v-v_*| v \cdot \omega, \quad\mbox{which implies}\quad \sin  \theta |v-v_*| v \cdot \omega = \sin\theta(\mj \cdot\hat{\omega})\tilde{h}.
\een
We also have $\omega = \tilde{\omega} \cos \frac \theta 2  + \frac {v'-v_*} {|v'-v_*|} \sin \frac \theta 2$, where $\tilde{\omega}=(v'-v)/|v'-v|$.
\end{itemize}
 Moreover, if $k\geq2$, $l_k:=[(k+1)/2]$, then
\ben\label{gamma}
&&\sum_{p=1}^{l_{k/2}-1}\frac{\Ga(k/2+1)}{\Ga(p+1)\Ga(k/2+1-p)}[(\<v\>^2\cos^2(\th/2))^p(\<v_*\>^2\sin^2(\th/2))^{k/2-p}\\
\notag&&+(\<v\>^2\cos^2(\th/2))^{k/2-p}(\<v_*\>^2\sin^2(\th/2))^{p}]+(\cos^2(\th/2))^{k/2}\<v\>^k+(\sin^2(\th/2))^{k/2}\<v_*\>^k\\
\notag&\leq&(E(\th))^{k/2}\leq\sum_{p=1}^{l_{k/2}}\frac{\Ga(k/2+1)}{\Ga(p+1)\Ga(k/2+1-p)}[(\<v\>^2\cos^2(\th/2))^p(\<v_*\>^2\sin^2(\th/2))^{k/2-p}\\
\notag&&+(\<v\>^2\cos^2(\th/2))^{k/2-p}(\<v_*\>^2\sin^2(\th/2))^{p}]+(\cos^2(\th/2))^{k/2}\<v\>^k+(\sin^2(\th/2))^{k/2}\<v_*\>^k,
\een
where $\Ga$ denotes the Gamma function.
\end{lem}

\begin{lem}\label{lemma2.2}
(i) If $k  \ge   4$, then
\beno
 \int_{\S^2}b(\cos\th)(1-\cos^k\f{\th}2) d\si\sim k^s;\,
 \int_{\S^2}b(\cos\th)\sin^{k-\f32-\f\ga2}\f{\th}2d\si\ls \f1{k};\,
 \int_{\S^2}b(\cos\th)\sin^2\f{\th}2\cos^{k}\f{\th}2d\th\ls k^{s-1}.
\eeno

(ii) Recall that $l_k:=[(k+1)/2]$, then
\ben\label{lk}
\sum_{p=1}^{l_{k/2}}\frac{\Ga(k/2+1)}{\Ga(p+1)\Ga(k/2+1-p)}\frac{\Ga(p-s)\Ga(k/2-p-s)}{\Ga(k/2-2s)}\ls k^s.
\een
\end{lem}
\begin{proof} We first give the proof to $(i)$. By change of variables, the latter two inequalities  can be proved by the definition of Gamma function.
 For the first one, by   change of variable $u=\sin\f\th 2$, we get
\beno
\int_{\S^2}b(\cos\th)(1-(1-\cos^2(\th/2))^{k/2})d\si\sim \int_0^{1/2} u^{-1-2s}(1-(1-u^2)^{k/2})du:=\nu(k).
\eeno
Taking derivative with respect to $k$, we have
\beno
\nu'(k)=-\f12\int_0^{1/2} u^{-1-2s}(1-u^2)^{k/2}\ln(1-u^2)du\sim \int_0^1u^{1-2s}(1-u^2)^{k/2}du=\f{\Ga(1-s)\Ga(k/2+1)}{\Ga(\ka/2+2-s)}\sim k^{s-1},
\eeno
where we use   (\ref{gammafun}). Then we have $\nu(k)=\nu(4)+\int_4^k \nu'(p)dp\sim k^s,$ which implies  the desired result.

As for (\ref{lk}), again by (\ref{gammafun}), we have
$\Ga(k/2+1)\sim\Ga(k/2-2s)k^{1+2s},~\Ga(p+1)\sim\Ga(p-s)p^{1+s},~\Ga(k/2+1-p)\sim\Ga(k/2-p-s)(k/2-p)^{1+s},$
which imply that
\beno
\sum_{p=1}^{l_{k/2}}\frac{\Ga(k/2+1)}{\Ga(p+1)\Ga(k/2+1-p)}\frac{\Ga(p-s)\Ga(k/2-p-s)}{\Ga(k/2-2s)}\leq k^s\sum_{p=1}^{l_{k/2}}\Big(\f{k}{p(k/2-p)}\Big)^{1+s}.
\eeno
Notice that $k/2-p\sim k$  since $p\leq l_{k/2}$. The above can be bounded by $k^s\sum_{p=1}^\infty p^{-(1+s)}\ls k^s$. We complete the proof of this lemma.
\end{proof}

\subsection{Upper bounds of the collision operator $Q$}
We first give the upper bound of   $Q(h,\mu)$.
\begin{lem}\label{L21}
Suppose that $\ga\in (-3,1], \gamma+2s>-1$ ,$k \ge 22$ and  $\mu(v)= (2\pi)^{-3/2} e^{-|v|^2/2}$. Then for smooth functions $g$ and $h$, we have
\ben\label{poly1}
&&|(Q( h , \mu), g \langle v \rangle^{2k})|\le   \Vert b(\cos \theta) \sin^{k - 3/ 2-\gamma/2}( \theta/ 2) \Vert_{L^1_\theta}|[ h ]|_{L^2_{k+\f\gamma2}}|[ g]|_{L^2_{k+\f\gamma2}} + C_{k} \Vert h \Vert_{L^2_{k+\f\gamma2-\f12}}\Vert g \Vert_{L^2_{k+\f\gamma2-\f12}}
\een
%\ben\label{poly1}
%&&|(Q( h , \mu), g \langle v \rangle^{2k})|\le  \gamma_3 \Vert b(\cos \theta) \sin^{k - 3/ 2-\gamma/2}( \theta/ 2) \Vert_{L^1_\theta}\Vert h \Vert_{L^2_{k+\f\gamma2}}\Vert g \Vert_{L^2_{k+\f\gamma2}} + C_{k} \Vert h \Vert_{L^2_{k+\f\gamma2-\f12}}\Vert g \Vert_{L^2_{k+\f\gamma2-\f12}}
%\een
 where $C_k>0$ is a constant depending on $k$ and $|[\cdot]|$ is defined in (\ref{3norm}).
Moreover, if $2js\geq20$, we have
\ben\label{exp1}
&&|(Q(h,\mu),g\<v\>^{4js})|-\|b(\cos\th)\sin^{2js-3/2-\ga/2}(\th/2)\|_{L^1_\th}|[h]|_{L^2_{2js+\ga/2}}|[g]|_{L^2_{2js+\ga/2}}\\
\notag&\ls& j\|h\|_{L^2_{2js+\ga/2-1}}\|g\|_{L^2_{2js+\ga/2-1}}+j^{4}\|h\|_{L^2_{2js+\ga/2-4}}\|g\|_{L^2_{2js+\ga/2-4}}+\|h\|_{L^2_{2js+\ga/2-1/2}}\|g\|_{L^2_{2js+\ga/2-1/2}}\\
\notag&&+j^7\|\mu\<v\>^{2js+12}\|_{L^2}\|h\|_{L^2_{17}}\|g\|_{L^2_{2js+\ga/2-1}}+\mathcal{H}_j(h,\mu,g),
\een
where $\mathcal{H}_j$ is defined as follows
\ben\label{Hgfh}
\mathcal{H}_j(g,f,h)&:=&j^2\int_{\R^6\times\S^2}\int_0^1b(\cos\th)\sin^2\th|v-v_*|^{\ga}(1-t)(E(\th)+t\tilde{h}\sin\th(\mj\cdot\hat{\omega}))^{js-2}\tilde{h}^2(\mj\cdot\hat{\om})^2\notag\\
&&\qquad\qquad\times|g(v_*)|f(v)|h(v')\<v'\>^{2js}|dvdv_*d\si dt.
\een
\end{lem}
\begin{proof}
We first have
$(Q( h , \mu), g \langle v \rangle^{2k}) =(Q(h,\mu\<v\>^k),g\<v\>^k)+(\<v\>^kQ(h,\mu)-Q(h,\mu\<v\>^k),g\<v\>^k)$.
The first term in the right hand side is easily estimated by Lemma \ref{L12} as follows:
\beno
| (Q(h, \mu \langle \cdot \rangle^{k} ), g \langle \cdot \rangle^{k}) |&\lesssim&
  \|h\|_{L^2_8}\|\mu\<v\>^k\|_{H^{2s}_{\ga/2+2s+1}}\Vert g \Vert_{L^2_{k+\gamma/2-1}}.
\eeno
For the second term, we split it into two parts:
\beno
&&(\<v\>^kQ(h,\mu)-Q(h,\mu\<v\>^k),g\<v\>^k)
=\int_{\R^6\times\S^2}b(\cos\th)|v-v_*|^\ga h(v_*)\mu(v) g(v')\<v'\>^k(E(\th)^{k/2}-\<v\>^k)dvdv_*d\si\\
&&+\int_{\R^6\times\S^2}b(\cos\th)|v-v_*|^\ga h(v_*)\mu(v) g(v')\<v'\>^k(\<v'\>^k-E(\th)^{k/2})dvdv_*d\si:=\mathscr{A}+\mathscr{B}.
\eeno

\underline{\it Step ~1: Estimate of $\mathscr{A}$.} By Lemma \ref{L18} (\ref{gamma}), we have
\ben\label{Eexpan}
&&|E(\th)^{k/2}-\<v\>^k|\le \sum_{p=1}^{l_{k/2}}\frac{\Ga(k/2+1)}{\Ga(p+1)\Ga(k/2+1-p)}[(\<v\>^2\cos^2(\th/2))^p(\<v_*\>^2\sin^2(\th/2))^{k/2-p}\\
\notag&&+(\<v\>^2\cos^2(\th/2))^{k/2-p}(\<v_*\>^2\sin^2(\th/2))^{p}]+(1-(\cos^2(\th/2))^{k/2})\<v\>^k+(\sin^2(\th/2))^{k/2}\<v_*\>^k.
\een
\noindent$\bullet$ For the term containing $(1-(\cos^2(\th/2))^{k/2})\<v\>^k$, due to the fact that $\<v'\>\leq\<v\>\<v_*\>$, we have
\beno
&&\int_{\R^6\times\S^2}b(\cos\th)|v-v_*|^\ga |h(v_*)|\mu(v) |g(v')\<v'\>^k|(1-(\cos^2(\th/2))^{k/2})\<v\>^kdvdv_*d\si\\
&\leq&\int_{\R^6\times\S^2}b(\cos\th)(1-(\cos^2(\th/2))^{k/2})|v-v_*|^\ga|h(v_*)\<v_*\>^7|\mu(v)\<v\>^{k+7} |g(v')\<v'\>^{k-7}|dvdv_*d\si.
\eeno
Thanks to Lemma \ref{chv},  Lemma \ref{lemma2.2} and Lemma \ref{L28}, it can be bounded by
\beno
&&\|\mu\<v\>^{k+7}\|_{L^\infty} \int_{\S^2}b(\cos\th)(1-(\cos^2(\th/2))^{k/2})\cos^{-3-\ga}(\th/2)d\si\int_{\R^6}|v-v_*|^\ga|h(v_*)|\<v_*\>^7|g(v)|\<v\>^{k-7}dvdv_*\\
&\lesssim&  k^s \|\mu\<v\>^{k+7}\|_{L^\infty}\|h\|_{L^2_{10}}\|g\|_{L^2_{k+\ga/2-1}}.
\eeno

\noindent$\bullet$ For the term containing $(\sin^2(\theta/2))^{k/2}\<v_*\>^k$, by Lemma \ref{chv} and Lemma \ref{L110}, we have
\beno
&&\int_{\R^6\times\S^2}b(\cos\theta )  |v-v_*|^\ga |h(v_*)|\mu(v) |g(v')\<v'\>^k(\sin^2(\theta/2))^{k/2}    \<v_*\>^k|dvdv_*d\si\\
&\leq&\int_{\S^2}b(\cos\theta)\sin^{k-3/2-\ga/2}   (\theta/2)  d\si\left(\int_{\R^6}|v-v_*|^\ga |h(v_*)\<v_*\>^k|^2\mu(v)dv_*dv\right)^{\f12}  \left(\int_{\R^6}|v-v'|^\ga |g(v')\<v'\>^k|^2|\mu(v)dv  'dv\right)^{\f12}\\
&\leq& \|b(\cos\th)\sin^{k-3/2-\ga/2}(\theta/2)\|_{L^1_\theta}|[h]|_{L^2_{k+\ga/2}}|[g]|_{L^2_{k+\ga/2}}.
\eeno

\noindent$\bullet$ Similarly, for the term containing  $p=1,2$ in (\ref{Eexpan}), we have
\beno
&&k\int_{\R^6\times\S^2}b(\cos\th)|v-v_*|^\ga |h(v_*)|\mu(v) |g(v')\<v'\>^k|\big[\<v\>^2\cos^2(\th/2)(\<v_*\>^2\sin^2(\th/2))^{k/2-1}
 +(\<v\>^2\cos^2(\th/2))^{k/2-1}\<v_*\>^2
 \\ && \times \sin^2(\th/2)\big]dvdv_*d\si +k^2\int_{\R^6\times\S^2}b(\cos\th)|v-v_*|^\ga |h(v_*)|\mu(v) |g(v')\<v'\>^k|\big[(\<v\>^2\cos^2(\th/2))^2(\<v_*\>^2\sin^2(\th/2))^{k/2-2}\\
\notag&&+(\<v\>^2\cos^2(\th/2))^{k/2-2}(\<v_*\>^2\sin^2(\th/2))^2\big]dvdv_*d\si  \lesssim k\|\mu\<v\>^{k+7}\|_{L^\infty}\|h\|_{L^2_{11}}\|g\|_{L^2_{k+\ga/2-1}}+k\|h\|_{L^2_{k+\ga/2-1}}\|g\|_{L^2_{k+\ga/2-1}}\\
&& +k^2\|\mu\<v\>^{k+5}\|_{L^\infty}\|h\|_{L^2_{13}}\|g\|_{L^2_{k+\f\ga2-1}}+k^2\|h\|_{L^2_{k+\f\ga2-2}}\|g\|_{L^2_{k+\f\ga2-2}}.
\eeno

\noindent$\bullet$ By singular change of variables in Lemma \ref{chv}, we   obtain that
\beno
&&\mathcal{I}:=\sum_{p=3}^{l_{k/2}}\frac{\Ga(k/2+1)}{\Ga(p+1)\Ga(k/2+1-p)}\int_{\R^6\times\S^2}b(\cos\th)|v-v_*|^\ga|h(v_*)|\mu(v) |g(v')\<v'\>^k|\big[(\<v\>^2\cos^2(\th/2))^p(\<v_*\>^2\\
\notag&&\times\sin^2(\th/2))^{k/2-p}+(\<v\>^2\cos^2(\th/2))^{k/2-p}(\<v_*\>^2\sin^2(\th/2))^{p}\big]dvdv_*d\si\\
&\leq&\sum_{p=3}^{l_{k/2}}\frac{\Ga(k/2+1)}{\Ga(p+1)\Ga(k/2+1-p)}\int_{\R^6\times\S^2}b(\cos\th)|v-v_*|^\ga|h(v_*)|\mu(v) |g(v')\<v'\>^k|\big[(\cos^2(\th/2))^p(\sin^2(\th/2))^{k/2-p}\\
\notag&&+(\cos^2(\th/2))^{k/2-p}(\sin^2(\th/2))^{p}\big]\big[(\<v\>^2)^3(\<v_*\>^2)^{k/2-3}+(\<v_*\>^2)^3(\<v\>^2)^{k/2-3}\big]dvdv_*d\si\\
&\leq&\underbrace{\sum_{p=3}^{l_{k/2}}\frac{\Ga(k/2+1)}{\Ga(p+1)\Ga(k/2+1-p)}\int_{\S^2}b(\cos\th)\big[(\cos^2(\th/2))^p(\sin^2(\th/2))^{\f k 2-p-\f32-\f\ga 2}+(\cos^2(\th/2))^{\f k2-p}(\sin^2(\th/2))^{p-\f32-\f\ga 2}\big]d\si}_{\mathcal{C}(k)}\\
&&\times(\|h\|_{L^2_{15}}\|\mu\<v\>^{k+3}\|_{L^\infty}\|g\|_{L^2_{k+\ga/2-1}}+\|h\|_{L^2_{k+\ga/2-3}}\|g\|_{L^2_{k+\ga/2-3}}).
\eeno
%\beno
%&&\int_{\R^6\times\S^2}b(\cos\th)|v-v_*|^\ga |h(v_*)|\mu(v) |g(v')\<v'\>^k|(\<v\>^2\cos^2(\th/2))^p(\<v_*\>^2\sin^2(\th/2))^{k/2-p}
%dvdv_*d\si\\
%&\leq&\int_{\R^6\times\S^2}b(\cos\th)|v-v_*|^\ga |h(v_*)|\mu(v) |g(v')\<v'\>^k|\<v\>^{2p}\<v_*\>^{k-2p}\sin^{k/2}(\th/2)dvdv_*d\si\\
%&\leq& \int_{\S^2}b(\cos\th)\sin^{k/2-3/2-\ga}(\th/2)d\si\left(\int_{\R^6}|v-v_*|^\ga |h(v_*)\<v_*\>^{k-2p+1}|^2)|\mu(v)\<v\>^{2p+1}dv_*dv|\right)^{1/2}\\
%&&\times\left(\int_{\R^6}|v-v'|^\ga |g(v')\<v'\>^{k-1}|^2|\mu(v)\<v\>^{2p+1}dv'dv|\right)^{1/2}\\
%&\leq& C_k\|h\|_{L^2_{k+\ga/2-1}}\|g\|_{L^2_{k+\ga/2-1}}.
%\eeno
%Similarly, we also have that
%\beno
%&&\int_{\R^6\times\S^2}b(\cos\th)|v-v_*|^\ga |h(v_*)|\mu(v) |g(v')\<v'\>^k|(\<v\>^2\cos^2(\th/2))^{k/2-p}(\<v_*\>^2\sin^2(\th/2))^{p}\\
%&\leq& C_k\|h\|_{L^2_{k+\ga/2-1/2}}\|g\|_{L^2_{k+\ga/2-1/2}}.
%\eeno
Let us compute the coefficient $\mathcal{C}(k)$ in the above. By change of variable, we have
\ben\label{sum}
\mathcal{C}(k)&\ls&\sum_{p=3}^{l_{k/2}}\frac{\Ga(k/2+1)}{\Ga(p+1)\Ga(k/2+1-p)}\int_0^{\f12}\big[(1-t)^{p-s-1}t^{\f k2-p-\f32-\f\ga 2-s-1}+t^{p-\f32-\f\ga 2-s-1}(1-t)^{\f k2-p-s-1}\big]dt\\
\notag&\ls& \sum_{p=3}^{l_{k/2}}\frac{\Ga(k/2+1)}{\Ga(p+1)\Ga(k/2+1-p)}\left(\frac{\Ga(p-s)\Ga(k/2-p-s-\f32-\f\ga 2)}{\Ga(k/2-2s-\f32-\f\ga 2)}+\frac{\Ga(p-s-\f32-\f\ga 2)\Ga(k/2-p-s)}{\Ga(k/2-2s-\f32-\f\ga 2)}\right).
\een
Notice that
$\frac{\Ga(p-s)\Ga(k/2-p-s-\f32-\f\ga 2)}{\Ga(k/2-2s-\f32-\f\ga 2)}+\frac{\Ga(p-s-\f32-\f\ga 2)\Ga(k/2-p-s)}{\Ga(k/2-2s-\f32-\f\ga 2)}\ls k^{\f32+\f\ga 2}\frac{\Ga(p-s)\Ga(k/2-p-s)}{\Ga(k/2-2s)}$. By Lemma \ref{lemma2.2}, it holds that $\mathcal{C}(k)\ls k^{\f32+\f\ga 2+s}\le k^3$ which yields that $\mathcal{I}\ls k^3(\|h\|_{L^2_{k+\f\ga2-3}}\|g\|_{L^2_{k+\f\ga2-3}}+\|h\|_{L^2_{15}}\|\mu\<v\>^{k+3}\|_{L^\infty}\|g\|_{L^2_{k+\ga/2-1}})$.

Patching together all the estimates and using the interpolation that
\ben\label{interp}
k^2\|h\|_{L^2_{k+\ga/2-2}}\|g\|_{L^2_{k+\ga/2-2}}\leq k\|h\|_{L^2_{k+\ga/2-1}}\|g\|_{L^2_{k+\ga/2-1}}+k^3\|h\|_{L^2_{k+\ga/2-3}}\|g\|_{L^2_{k+\ga/2-3}},
\een
we derive that
\beno
&&\mathscr{A}-\|b(\cos\th)\sin^{k-3/2-\ga/2}(\theta/2)\|_{L^1_\theta}|[h]|_{L^2_{k+\ga/2}}|[g]|_{L^2_{k+\ga/2}}\ls k\|h\|_{L^2_{k+\ga/2-1}}\|g\|_{L^2_{k+\ga/2-1}}\\
&&+k^3\|h\|_{L^2_{k+\ga/2-3}}\|g\|_{L^2_{k+\ga/2-3}}+k^3\|\mu\<v\>^{k+9}\|_{L^\infty}\|h\|_{L^2_{15}}\|g\|_{L^2_{k+\ga/2-1}}.
\eeno

\underline{\it Step ~2: Estimate of $\mathscr{B}$.} Thanks to Lemma \ref{L18} and Taylor expansion, we have
\ben\label{v'}
\notag\<v'\>^k-(E(\th))^{k/2}&=&\frac{k}{2}\big[(\<v\>^2\cos^2(\th/2))^{k/2-1}+((E(\th))^{k/2-1}-(\<v\>^2\cos^2(\th/2))^{k/2-1})\big]|v-v_*|\sin\th(v\cdot\om)\\
&&+\frac{k(k-2)}{4}\int_0^1(1-t)(E(\th)+t\tilde{h}\sin\th(\mj\cdot\hat{\om}))^{k/2-2}dt \tilde{h}^2\sin^2\th(\mj\cdot\hat{\om})^2,
\een
 which yields that
\beno
&&|\mathscr{B}|\leq\left|k\int_{\R^6\times\S^2}b(\cos\th)\cos^{k-1}(\th/2)\sin(\th/2)(v\cdot\omega)|v-v_*|^{1+\ga}h(v_*)\mu(v)\<v\>^{k-2}g(v')\<v'\>^kdvdv_*d\si\right|\\
&&+\frac{k}{2}\int_{\R^6\times\S^2}b(\cos\th)\sin\th|v-v_*|^{1+\ga}|v\cdot\om||(E(\th))^{\f{k}2-1}-(\<v\>^2\cos^2(\th/2))^{k/2-1}||h(v_*)|\mu(v)|g(v')\<v'\>^k|dvdv_*d\si\\
&&+k^2\int_{\R^6\times\S^2}\int_0^1b(\cos\th)|v-v_*|^{\ga}(1-t)(E(\th)+t\tilde{h}\sin\th(\mj\cdot\tilde{\omega}))^{\f{k}2-2}\tilde{h}^2\sin^2\th(\mj\cdot\hat{\om})^2|h(v_*)|\mu(v)\\
&&\times|g(v')\<v'\>^k|dvdv_*d\si:=\mathscr{B}_1+\mathscr{B}_2+\mathscr{B}_3.
\eeno

\noindent$\bullet$\underline{\it Estimate of $\mathscr{B}_1$.} By Lemma \ref{L18},   $v\cdot\om=v_*\cdot\om$  and
$\omega = \tilde{\omega} \cos \frac \theta 2  + \frac {v'-v_*} {|v'-v_*|} \sin \frac \theta 2$ with $ \tilde{\omega} = \frac {v'-v} {|v'-v|}$. We get that
\beno
\mathscr{B}_1&\leq& \left|k\int_{\R^6\times\S^2} b(\cos \theta) |v-v_*|^{1+\ga} (v_*\cdot \tilde{\omega}) \cos^{k} \frac \theta 2 \sin \frac \theta 2h(v_*) \mu(v)\<v\>^{k-2} g'\langle v' \rangle^{k}  dvdv_* d\sigma\right|
\\
&&+\left|k\int_{\R^6\times\S^2} b(\cos \theta) |v-v_*|^{1+\gamma} (v_*\cdot \frac {v'-v_*} {|v'-v_*|}) \cos^{k-1} \frac \theta 2 \sin^2 \frac \theta 2 h(v_*) \mu(v)\<v\>^{k-2}g(v')\langle v' \rangle^{k}   dvdv_* d\sigma\right|
\\
&:=& \mathscr{B}_{1,1} +\mathscr{B}_{1,2}.
\eeno

For $\mathscr{B}_{1, 2}$, by the regular change of variable(see Lemma \ref{chv}) and Lemma \ref{L28}, we have
\beno
\mathscr{B}_{1,2}
&\lesssim &k\int_{\R^6\times\S^2} b(\cos \theta) \cos^{k-1} \frac \theta 2 \sin^2 \frac \theta 2|v-v_*|^{1+\gamma}   \langle v_* \rangle^{7 }  |h_*| \mu  \langle v \rangle^{k+4}  |g(v')| \langle v' \rangle^{k-6}  dvdv_* d\sigma
\\
&\lesssim &k\|\mu\<v\>^{k+4}\|_{L^\infty}\int_{\mathbb{S}^2} b(\cos \theta) \sin^2   \frac \theta 2 d\sigma \int_{\R^6} |v'-v_*|^{1+\gamma}  \langle v_* \rangle^{7}  |h_*|  |g(v')|\langle v' \rangle^{k-6}dv' dv_*
\\
&\lesssim & k\|\mu\<v\>^{k+4}\|_{L^\infty}\Vert h \Vert_{L^{2}_{11}}\Vert g \Vert_{L^{2}_{k+\gamma/2-1}}.
\eeno

For $\mathscr{B}_{1, 1}$, our key observation results from the fact that
\beno
 k\int_{\R^6\times\S^2} b(\cos \theta) |v-v_*|^{1+\gamma} (v_*\cdot \tilde{\omega}) \cos^{k} \frac \theta 2 \sin \frac \theta 2h(v_*) \mu(v')\langle v' \rangle^{k -2} g'\langle v' \rangle^{k} dvdv_* d\sigma=0,
\eeno which implies that
\beno
\mathscr{B}_{1, 1} &=& \left|k \int_{\R^6\times\S^2}b(\cos \theta) |v-v_*|^{1+\gamma} (v_* \cdot \tilde{\omega}) \cos^{k} \frac \theta 2 \sin \frac \theta 2 h_* (\mu \langle v \rangle^{k-2} -\mu' \langle v' \rangle^{k -2}) g'\langle v' \rangle^{k} dv dv_* d\sigma\right|.\\ &\leq& \left|k \int_{\R^6\times\S^2} b(\cos \theta) |v-v_*|^{1+\gamma} (v_* \cdot \tilde{\omega}) \cos^{k} \frac \theta 2 \sin \frac \theta 2 h_* g'\langle v' \rangle^{ k-8} (\mu \langle v \rangle^{ k + 6} -\mu' \langle v' \rangle^{k + 6}) dv dv_* d\sigma\right|
\\
&&+ \left|k  \int_{\R^6\times\S^2} b(\cos \theta) |v-v_*|^{1+\gamma} (v_* \cdot \tilde{\omega}) \cos^{k} \frac \theta 2 \sin \frac \theta 2 h_* g'\langle v' \rangle^{k-8} \mu \langle v \rangle^{ k  -2}(\langle v' \rangle^{8} - \langle v \rangle^{8 }) dv dv_* d\sigma\right|
\\
&:=&\mathscr{B}_{1,1,1} + \mathscr{B}_{1,1,2}.
\eeno
Thanks to the mean value theorem, it holds that
$|\mu \langle v \rangle^{ k + 6} -\mu' \langle v' \rangle^{  k + 6 } | \lesssim |v'-v| \int_0^1 \na(\mu\<\cdot\>^{k+6}) (v+t(v-v')) dt  \leq \|\na(\mu\<v\>^{k+6} )\|_{L^\infty}|v-v_*| \sin \frac \theta 2.$
Then by regular change of variable in Lemma \ref{chv} and Lemma \ref{L28}, we  are led to
\beno
\mathscr{B}_{1,1,1}
&\lesssim & k\|\na(\mu\<v\>^{k + 6})\|_{L^\infty}\int_{\mathbb{S}^2} b(\cos \theta)  \cos^{k-3-\gamma} \frac \theta 2 \sin^2 \frac \theta 2  d\sigma\int_{\R^6} |v'-v_*|^{2+\gamma}   \langle v_* \rangle |h_*| |g'|\langle v' \rangle^{k-8} dv' dv_*
\\
&\lesssim &k\|\na(\mu\<v\>^{k+  6 })\|_{L^\infty}\Vert h \Vert_{L^{2}_{10}}\Vert g \Vert_{L^{2}_{k+\gamma/2-1}}.
\eeno

For $\mathscr{B}_{1,1,2}$, similarly, using
$| \langle v' \rangle^{8} - \langle v \rangle^{8} | \lesssim |v'-v| (\langle v' \rangle^{7} + \langle v \rangle^{7})  \lesssim |v-v_*| \sin \frac \theta 2 \langle v\rangle^{7} \langle v_*\rangle^{7},$
then from the regular change of variable and Lemma \ref{L28}, we have
\beno
\mathscr{B}_{1,1,2}
&\lesssim & k\|\mu\<v\>^{k+5} \|_{L^\infty}  \int_{\R^6\times\mathbb{S}^2} b(\cos \theta) \cos^k \frac \theta 2 \sin^2 \frac \theta 2 |v-v_*|^{2+\gamma}  \langle v_* \rangle^{8}  |h_*| |g'| \langle v' \rangle^{k - 8} \  dv dv_* d\sigma
\\
&\lesssim & k\|\mu\<v\>^{k+5} \|_{L^\infty}\int_{\R^3} \int_{\R^3}    |v'-v_*|^{2+\gamma}  \langle v_* \rangle^{8} | h_*| |g'|\langle v' \rangle^{k-8}  dv' dv_*.
\lesssim k\|\mu\<v\>^{k+ 5}\|_{L^\infty}    \|h\|_{L^2_{12}}   \|g\|_{L^2_{k+\ga/2-1}}.
\eeno

Gathering together all the above, we derive that
\beno
\mathscr{B}_{1}\ls k\|\mu\<v\>^{k+6}\|_{H^3}\|h\|_{L^2_{12}}\|g\|_{L^2_{k+\ga/2-1}}.
\eeno

\noindent$\bullet$\underline{\it Estimate of $\mathscr{B}_2$.} We observe that $(E(\th))^{\f k2-1}-(\<v\>^2\cos^2(\f{\th}2))^{\f k2-1}$ enjoys the similar structure as \eqref{Eexpan}, i.e.
 \ben\label{Ek}
&&\notag|(E(\th))^{k/2-1}-(\<v\>^2\cos^2(\th/2))^{k/2-1}|\leq \sum_{p=1}^{l_{k/2-1}}\frac{\Ga(k/2)}{\Ga(p+1)\Ga(k/2-p)}[(\<v\>^2\cos^2(\th/2))^p(\<v_*\>^2\sin^2(\th/2))^{k/2-1-p}\\
&&\quad\quad\quad\quad+(\<v\>^2\cos^2(\th/2))^{k/2-1-p}(\<v_*\>^2\sin^2(\th/2))^{p}]+(\sin^2(\th/2))^{k/2-1}\<v_*\>^{k-2}.
\een
Thus following the similar argument, one may conclude that
\beno
|\mathscr{B}_2|&\leq& \|h\|_{L^2_{k+\ga/2-1/2}}\|g\|_{L^2_{k+\ga/2-1/2}}+k^{4}\|h\|_{L^2_{k+\ga/2-4}}\|g\|_{L^2_{k+\ga/2-4}}+k^4\|\mu\<v\>^{k+9}\|_{L^\infty}\|h\|_{L^2_{17}}\|g\|_{L^2_{k+\ga/2-1}}.
\eeno

\noindent\underline{\it Estimate of $\mathscr{B}_3$.} We first give the proof for polynomial case (\ref{poly1}). Recall that
\beno
\mathscr{B}_3=k^2\int_{\R^6\times\S^2}\int_0^1b(\cos\th)|v-v_*|^{\ga}(1-t)(E(\th)+t\tilde{h}\sin\th(\mj\cdot\hat{\omega}))^{\f k2-2}\tilde{h}^2\sin^2\th(\mj\cdot\hat{\om})^2|h(v_*)|\mu(v)|g(v')\<v'\>^k|dvdv_*d\si dt.
\eeno
Since $|\tilde{h}\sin\th(\mj\cdot\hat{\om})|=|v-v_*|\sin\th|(v\cdot\om)|=|v-v_*|\sin\th|(v_*\cdot\om)|\ls E(\th)$, one has
\beno
|\mathscr{B}_3|&\leq&C_k\int_{\R^6\times\S^2}b(\cos\th)\sin^2\th|v-v_*|^{2+\ga}\min\{\<v\>^2,\<v_*\>^2\}\max\{\<v\>^2\cos^2(\th/2),\<v_*\>^2\sin^2(\th/2)\}^{k/2-2}\\
&&\times|h(v_*)|\mu(v)|g(v')\<v'\>^k|dvdv_*d\si,
\eeno
which yields that $\mathscr{B}_3\leq C_k\|h\|_{L^2_{k+\ga/2-1}}\|g\|_{L^2_{k+\ga/2-1}}$.

Finally we conclude that
\beno
|\mathscr{B}|&\ls&  k^4\|\mu\<v\>^{k+9}\|_{H^3}\|h\|_{L^2_{17}}\|g\|_{L^2_{k+\ga/2-1}}+\|h\|_{L^2_{k+\ga/2-1/2}}\|g\|_{L^2_{k+\ga/2-1/2}}+k^{4}\|h\|_{L^2_{k+\ga/2-4}}\|g\|_{L^2_{k+\ga/2-4}}\\
&&+C_k\|h\|_{L^2_{k+\ga/2-1}}\|g\|_{L^2_{k+\ga/2-1}}.
\eeno
Noting that $k+\ga/2-1/2>17$ if $k\geq22$,   we complete the proof of (\ref{poly1}) by combining all the estimates.
\smallskip

Due to the fact that $k^4\|\mu\<v\>^{k+9}\|_{H^3}\ls k^7\|\mu\<v\>^{k+12}\|_{L^2}$, if $k=2js$, we can easily obtain (\ref{exp1}) since $\mathscr{B}_3=\mathcal{H}_j(h,\mu,g)$ will be revisited in Lemma \ref{B_3}. This ends the proof of this lemma.
\end{proof}
\smallskip

The following estimates is about a commutator on the collision operator $Q$ with weight $\<v\>^k$.
\begin{lem}\label{L23}
Suppose $\ga\in(-3,1],\ga+2s>-1$ and $k\geq22$. Then for any $s\leq \mathbf{a}\leq1$,
\beno
&&|(\<v\>^kQ(g,f)-Q(g,\<v\>^kf),h\<v\>^k)|\leq k^s\|g\|_{L^2_{14}}\|f\|_{L^2_{k+\ga/2}}\|h\|_{L^2_{k+\ga/2}}+ \f1k\|f\|_{L^2_{14}}\|g\|_{L^2_{k+\ga/2}}\|h\|_{L^2_{k+\ga/2}}\\
&&+C_k\|f\|_{L^2_{14}}\|g\|_{H^s_{k-1+\ga/2}}\|h\|_{H^s_{k+\ga/2}}+C_k\|g\|_{L^2_{14}}\|f\|_{H^s_{k+s-\mathbf{a}+\ga/2}}\|h\|_{H^s_{k+\mathbf{a}-1+\ga/2}}+C_k\|g\|_{L^2_{14}}\|f\|_{H^s_{k-1+\ga/2}}\|h\|_{H^s_{k+\ga/2}}.
\eeno
Moreover, if $2js\geq22$, we have
\beno
&&|(\<v\>^{2js}Q(g,f)-Q(g,\<v\>^{2js}f),h\<v\>^{2js})|\ls j^s\|g\|_{L^2_{14}}\|f\|_{L^2_{2js+\ga/2}}\|h\|_{L^2_{2js+\ga/2}}+\f1j\|f\|_{L^2_{14}}\|g\|_{L^2_{2js+\ga/2}}\|h\|_{L^2_{2js+\ga/2}}\\
&&+\|f\|_{L^2_{14}}\|g\|_{H^s_{2js-1+\ga/2}}\|h\|_{H^s_{2js+\ga/2}}+j^4\|f\|_{L^2_{14}}\|g\|_{H^s_{2js-8+\ga/2}}\|h\|_{H^s_{2js+\ga/2}}+j^{\f12(1+s)} \|g\|_{L^2_{14}}\|f\|_{H^s_{2js+s-\mathbf{a}+\ga/2}}\\
&&\times\|h\|_{H^s_{2js-2+\mathbf{a}+\ga/2}}+j^{\f12(1+s)} \|g\|_{L^2_{14}}\|f\|_{H^s_{2js+s-\mathbf{a}+\ga/2}}\|h\|_{L^2_{2js-1+\mathbf{a}+\ga/2}}+j\|g\|_{L^2_{11}}\|f\|_{H^s_{2js-2+\ga/2}}\|h\|_{H^s_{2js+\ga/2}}\\
&&+j^4\|g\|_{L^2_{14}}\|f\|_{H^s_{2js-8+\ga/2}}\|h\|_{H^s_{2js+\ga/2}}+j^{3+s}\|f\|_{L^2_{14}}\|g\|_{L^2_{2js-7+\ga/2}}\|h\|_{L^2_{2js+\ga/2}}+\mathcal{H}_j(g,f,h),
\eeno
where $\mathcal{H}_j$ is defined in (\ref{Hgfh}). We address that $\|f\|_{H^s_{2js}}:=\|\<D\>^s\<\cdot\>^{2js}f\|_{L^2}$.
\end{lem}
\begin{proof}
It is not difficult to see that
\beno
&&(\<v\>^kQ(g,f)-Q(g,\<v\>^kf),h\<v\>^k)
=\int_{\R^6\times\S^2}b(\cos\th)|v-v_*|^\ga g(v_*)f(v) h(v')\<v'\>^k(E(\th)^{k/2}-\<v\>^k)dvdv_*d\si\\
&&+\int_{\R^6\times\S^2}b(\cos\th)|v-v_*|^\ga g(v_*)f(v) h(v')\<v'\>^k(\<v'\>^k-E(\th)^{k/2})dvdv_*d\si:=\mathscr{D}+\mathscr{E}.
\eeno

\underline{\it Step ~1: Estimate of $\mathscr{D}$.} Thanks to \eqref{Eexpan}, one may easily apply  the strategy used for $\mathcal{A}$ in Lemma \ref{L21} to $\mathscr{D}$. In what follows, we only point out the main difference.

\noindent$\bullet$ For the term containing $(1-(\cos^2(\th/2))^{k/2})\<v\>^k$, we consider it by two cases: $|v-v_*|\leq1$ and $|v-v_*|>1$. If $|v-v_*|\leq1$, then $\<v\>\sim \<v_*\>$. By Lemma \ref{L110} and Lemma \ref{lemma2.2}, we get that
\beno
\mathbf{C}_1&:=&\int_{|v-v_*|\leq1}b(\cos\th)|v-v_*|^\ga |g(v_*)f(v) h(v')|\<v'\>^k(1-(\cos^2(\th/2))^{k/2})\<v\>^kdvdv_*d\si\\
&\ls& k^s\|g\|_{L^2_{|\ga|+4}}\|f\|_{H^s_{k-2+\ga/2}}\|h\|_{H^s_{k+\ga/2}}.
\eeno
While if $|v-v_*|>1$, if $\gamma \le 0$ then $|v-v_*|^\ga\sim\<v-v_*\>^\ga\ls\<v_*\>^{|\ga|}\<v\>^\ga$ and $\<v'\>^{|\ga|/2}\leq\<v\>^{|\ga|/2}\<v_*\>^{|\ga|/2}$, we get that
\beno
\mathbf{C}_2&:=&\int_{|v-v_*|>1}b(\cos\th)|v-v_*|^\ga |g(v_*)f(v) h(v')|\<v'\>^k(1-(\cos^2(\th/2))^{k/2})\<v_*\>^kdvdv_*d\si\\
&\ls&k^s\|g\|_{L^2_{|\ga|+4}}\|f\|_{L^2_{k+\ga/2}}\|h\|_{L^2_{k+\ga/2}}.
\eeno
If $\gamma >0$ we can use $|v-v_*| \sim  |v_* - v'| $, and $|v-v_*|^\gamma  \sim |v-v_*|^{\gamma/2}  |v_* -v'|^{\gamma/2} \le \langle v_* \rangle^{\gamma}   \langle v \rangle^{\gamma/2} \langle v' \rangle^{\gamma/2}  $ to get the same result. 
For the term containing $(\sin^2(\theta/2))^{k/2}\<v_*\>^k$ and $p=1,2$, they can be handled similarly. Indeed, they are bounded from above by
\beno &&k\|f\|_{L^2_{|\ga|+6}}\|g\|_{H^s_{k-2+\ga/2}}\|h\|_{H^s_{k+\ga/2}}+k^2\|f\|_{L^2_{|\ga|+6}}\|g\|_{H^s_{k-4+\ga/2}}\|h\|_{H^s_{k+\ga/2}}+k\|g\|_{L^2_{|\ga|+6}}\|f\|_{H^s_{k-2+\ga/2}}\|h\|_{H^s_{k+\ga/2}}\\
&&+k^2\|g\|_{L^2_{|\ga|+6}}\|f\|_{H^s_{k-4+\ga/2}}\|h\|_{H^s_{k+\ga/2}}+\f1k\|f\|_{L^2_{|\ga|+6}}\|g\|_{L^2_{k+\ga/2}}\|h\|_{L^2_{k+\ga/2}}.
\eeno

\noindent$\bullet$ For the term $p\geq3$, considering that the weights for $g$ and $f$ have to be balanced, we split the integration domain into two parts: $\<v\>\geq\<v_*\>$ and $\<v\><\<v_*\>$. Let us give a detailed estimate for the typical term in (\ref{Eexpan}): \beno&& \mathcal{I}:=\sum_{p=3}^{l_{k/2}}\frac{\Ga(k/2+1)}{\Ga(p+1)\Ga(k/2+1-p)}\int_{\R^6}\int_{\S^2}b(\cos\th)|v-v_*|^\ga|g(v_*)||f(v)| |h(v')\<v'\>^k|\\&&\quad\times[(\<v\>^2\cos^2(\th/2))^p(\<v_*\>^2\sin^2(\th/2))^{k/2-p}
+(\<v\>^2\cos^2(\th/2))^{k/2-p}(\<v_*\>^2\sin^2(\th/2))^{p}]dvdv_*d\si.\eeno
We decompose it into
  $\mathbf{A}$ and $\mathbf{B}$ which denote the regions $\<v\>\geq\<v_*\>$ and $\<v\><\<v_*\>$ respectively. Then
\beno
\mathbf{A}
&\leq&\sum_{p=3}^{l_{k/2}}\frac{\Ga(k/2+1)}{\Ga(p+1)\Ga(k/2+1-p)}\int_{\<v\>\geq\<v_*\>}\int_{\S^2}b(\cos\th)|v-v_*|^\ga|g(v_*)||f(v)| |h(v')\<v'\>^k|\Big[(\cos^2(\th/2))^p\\
\notag&&\times(\sin^2(\th/2))^{k/2-p}+(\cos^2(\th/2))^{k/2-p}(\sin^2(\th/2))^{p}\Big](\<v_*\>^2)^3(\<v\>^2)^{k/2-3}dvdv_*d\si\\
&\leq&\sum_{p=3}^{l_{k/2}}\frac{\Ga(k/2+1)}{\Ga(p+1)\Ga(k/2+1-p)}\int_{\S^2}b(\cos\th)[(\cos^2(\th/2))^p(\sin^2(\th/2))^{k/2-p}+(\cos^2(\th/2))^{k/2-p}(\sin^2(\th/2))^{p}]d\si\\
&&\times\|g\|_{L^2_{|\ga|+8}}\|f\|_{H^s_{k-6+\ga/2}}\|h\|_{H^s_{k+\ga/2}}\ls k^3\|g\|_{L^2_{|\ga|+8}}\|f\|_{H^s_{k-6+\ga/2}}\|h\|_{H^s_{k+\ga/2}}.
\eeno
Similarly, we have
$\mathbf{B}\ls k^3\|f\|_{L^2_{|\ga|+8}}\|g\|_{H^s_{k-6+\ga/2}}\|h\|_{H^s_{k+\ga/2}}$. Then  we conclude that
\beno
|\mathscr{D}|&\ls& \f1k\|f\|_{L^2_{11}}\|g\|_{L^2_{k+\ga/2}}\|h\|_{L^2_{k+\ga/2}}+k\|f\|_{L^2_{|\ga|+6}}\|g\|_{H^s_{k-2+\ga/2}}\|h\|_{H^s_{k+\ga/2}}+k^3\|f\|_{L^2_{11}}\|g\|_{H^s_{k-6+\ga/2}}\|h\|_{H^s_{k+\ga/2}}\\
&&+k \|g\|_{L^2_{11}}\|f\|_{H^s_{k-2+\ga/2}}\|h\|_{H^s_{k+\ga/2}}+k^3\|g\|_{L^2_{11}}\|f\|_{H^s_{k-6+\ga/2}}\|h\|_{H^s_{k+\ga/2}}+k^s \|g\|_{L^2_{11}}\|f\|_{L^2_{k+\ga/2}}\|h\|_{L^2_{k+\ga/2}}.
\eeno

\underline{\it Step ~2: Estimate of $\mathscr{E}$.} Similar to the estimate of $\mathscr{B}$ in Lemma \ref{L21}, we have
\beno
\mathscr{E}&=&\left|k\int_{\R^6\times\S^2}b(\cos\th)\cos^{k-1}(\th/2)\sin(\th/2)(v\cdot\omega)|v-v_*|^{1+\ga}g(v_*)f(v)\<v\>^{k-2}h(v')\<v'\>^kdvdv_*d\si\right|\\
&&+\frac{k}{2}\int_{\R^6\times\S^2}b(\cos\th)\sin\th|v-v_*|^{1+\ga}|v\cdot\om||(E(\th))^{k/2-1}-(\<v\>^2\cos^2(\th/2))^{k/2-1}||g(v_*)|f(v)|\\
&&\times|h(v')\<v'\>^k|dvdv_*d\si+\mathcal{H}_{k/(2s)}(g,f,h):=\mathscr{E}_1+\mathscr{E}_2+\mathcal{H}_{k/(2s)}(g,f,h).
\eeno

\noindent$\bullet\,$\underline{\it Estimate of $\mathscr{E}_2$.} Due to (\ref{Ek}), we shall give the estimates term by term. For the term $(\sin^2(\th/2))^{k/2-1}\<v_*\>^{k-2}$, the fact that $|v-v_*||v\cdot\om|=|v-v_*||v_*\cdot\om|\ls \<v\>\<v_*\>$(see Lemma \ref{L18}), Lemma \ref{L110} and Lemma \ref{lemma2.2} imply that
\beno
&&k\int_{\R^6\times\S^2}b(\cos\th)\sin\th|v-v_*|^{1+\ga}|v\cdot\om|(\sin^2(\th/2))^{k/2-1}\<v_*\>^{k-2}|g(v_*)||f(v)||h(v')\<v'\>^k|dvdv_*d\si\\
&\ls&\|f\|_{L^2_{|\ga|+3}}\|g\|_{H^s_{k-1+\ga/2}}\|h\|_{H^s_{k+\ga/2}}.
\eeno

Let us choose $p=1$ in the summation form of (\ref{Ek}) as the typical term.  By Lemma \ref{lemma2.2}, it suffices to consider the term as follows
\beno
&&\mathcal{I}_1:=k^2\int_{\R^6\times\S^2}b(\cos\th)\sin\th|v-v_*|^{\ga}\<v\>\<v_*\>|g(v_*)||f(v)||h(v')\<v'\>^k|[\<v\>^2\cos^2(\th/2)(\<v_*\>^2\sin^2(\th/2))^{k/2-2}\\
&& +(\<v\>^2\cos^2(\th/2))^{k/2-2}\<v_*\>^2\sin^2(\th/2)]dvdv_*d\si:=\mathcal{I}_{1,1}+\mathcal{I}_{1,2}.
\eeno
It is easy to see that
$\mathcal{I}_{1,1}\ls k\|f\|_{L^2_{|\ga|+5}}\|g\|_{H^s_{k-3+\ga/2}}\|h\|_{H^s_{k+\ga/2}}$. We split the integration domain of $\mathcal{I}_{1,2}$ into two parts: $|v-v_*|\leq1$ and $|v-v_*|>1$, denote them by $\mathbf{D}_1$ and $\mathbf{D}_2$. As the estimate of $\mathbf{C}_1$ and $\mathbf{C}_2$, we get that $\mathbf{D}_1\ls k^2\|g\|_{L^2_{|\ga|+6}}\|f\|_{H^s_{k-4+\ga/2}}\|h\|_{H^s_{k+\ga/2}}$ and
%\beno
%\mathbf{D}_1&=&k^2\int_{\<v\>\sim\<v_*\>}b(\cos\th)\sin\th|v-v_*|^{\ga}\<v\>\<v_*\>|g(v_*)||f(v)||h(v')\<v'\>^k|(\<v\>^2\cos^2(\th/2))^{k/2-2}\<v_*\>^2\sin^2(\th/2)]dvdv_*d\si\\
%&\ls&k^2\|g\|_{L^2_{|\ga|+6}}\|f\|_{H^s_{k-4+\ga/2}}\|h\|_{H^s_{k+\ga/2}}.
%\eeno
%If $|v-v_*|>1$, then $|v-v_*|^\ga\sim\<v-v_*\>^\ga\ls\<v_*\>^{|\ga|}\<v\>^\ga$ and $\<v'\>^{|\ga|/2}\leq\<v\>^{|\ga|/2}\<v_*\>^{|\ga|/2}$. Therefore we have
%\beno
$\mathbf{D}_2\ls  k^{1+s}\|g\|_{L^2_{|\ga|+5}}\|f\|_{L^2_{k-3+\ga/2}}\|h\|_{L^2_{k+\ga/2}}$,
%\eeno
where we use Lemma \ref{lemma2.2}. Similarly, for $p=2$ and $p=3$, the associated terms can be bounded by
$k^2\|f\|_{L^2_{|\ga|+7}}\|g\|_{H^s_{k-5+\ga/2}}\|h\|_{H^s_{k+\ga/2}}+k^3\|g\|_{L^2_{|\ga|+8}}\|f\|_{H^s_{k-6+\ga/2}}\|h\|_{H^s_{k+\ga/2}}+k^{2+s}\|g\|_{L^2_{|\ga|+7}}\|f\|_{L^2_{k-5+\ga/2}}\|h\|_{L^2_{k+\ga/2}}$ and
$k^3\|f\|_{L^2_{|\ga|+9}}\|g\|_{H^s_{k-7+\ga/2}}\|h\|_{H^s_{k+\ga/2}}+k^4\|g\|_{L^2_{|\ga|+10}}\|f\|_{H^s_{k-8+\ga/2}}\|h\|_{H^s_{k+\ga/2}}+k^{3+s}\|g\|_{L^2_{|\ga|+9}}\|f\|_{L^2_{k-7+\ga/2}}\|h\|_{L^2_{k+\ga/2}}$.

For $p\geq4$, we need to split the associated term into two parts: $\<v\>\geq\<v_*\>$ and $\<v\><\<v_*\>$.  Similar to the estimate of $\mathbf{A}$ and $\mathbf{B}$, one has
\beno
k^{3+s}\|g\|_{L^2_{|\ga|+11}}\|f\|_{H^s_{k-9+\ga/2}}\|h\|_{H^s_{k+\ga/2}}+k^{3+s}\|f\|_{L^2_{|\ga|+11}}\|g\|_{H^s_{k-9+\ga/2}}\|h\|_{H^s_{k+\ga/2}}.
\eeno
Then by interpolation \eqref{interp}, we conclude that
\beno
&&|\mathscr{E}_2|\ls\|f\|_{L^2_{14}}\|g\|_{H^s_{k-1+\ga/2}}\|h\|_{H^s_{k+\ga/2}}+k^4\|f\|_{L^2_{14}}\|g\|_{H^s_{k-8+\ga/2}}\|h\|_{H^s_{k+\ga/2}}+k\|g\|_{L^2_{14}}\|f\|_{H^s_{k-2+\ga/2}}\|h\|_{H^s_{k+\ga/2}}\\
&&\quad+k^4\|g\|_{L^2_{14}}\|f\|_{H^s_{k-8+\ga/2}}\|h\|_{H^s_{k+\ga/2}}+k^{1+s}\|g\|_{L^2_{14}}\|f\|_{L^2_{k-3+\ga/2}}\|h\|_{L^2_{k+\ga/2}}+k^{3+s}\|g\|_{L^2_{14}}\|f\|_{L^2_{k-7+\ga/2}}\|h\|_{L^2_{k+\ga/2}}.
\eeno

\noindent$\bullet$\underline{\it Estimate of $\mathscr{E}_1$.} Still by Lemma \ref{L18}, we have
\beno
|\mathscr{E}_1|&\leq& k\left|\int_{\R^6\times\mathbb{S}^2} b(\cos \theta) |v-v_*|^{1+\ga} ( v_*\cdot \tilde{\omega}) \cos^k(\th/2) \sin(\th/2) g(v_*) f(v)\<v\>^{k-2}  h(v')\<v'\>^k dvdv_* d\sigma\right|
\\
&& +k\left|\int_{\R^6\times\mathbb{S}^2} b(\cos \theta) |v-v_*|^{1+\ga}  (v_*\cdot \frac {v'-v_*} {|v'-v_*|}) \cos^{k-1}(\th/2)\sin^2(\th/2) g(v_*)f(v)\<v\>^{k-2} h(v')\<v'\>^k dvdv_* d\sigma\right|
\\
&:=& \mathscr{E}_{1, 1} +\mathscr{E}_{1, 2}.
\eeno

For $\mathscr{E}_{1, 2}$, we split the integration domain into two parts: $|v-v_*|\leq1$ and $|v-v_*|>1$. Similar to the argument for $\mathbf{C}_1$ and $\mathbf{C}_2$ or $\mathbf{D}_1$ and $\mathbf{D}_2$, we can derive that
\beno
 \mathscr{E}_{1, 2} &\lesssim&   k\int_{\R^6\times\mathbb{S}^2} b(\cos \theta) \sin^{2} \frac \theta 2 \cos^k\f\th 2|v-v_*|^{1+\gamma} \langle v_* \rangle  |g_*| |f| \langle v \rangle^{k-2}   |h'|  \langle v' \rangle^k dvdv_* d\sigma\\
 &\ls&k^s\|g\|_{L^2_{|\ga|+4}}\|f\|_{H^s_{k+\ga/2-2}}\|h\|_{H^s_{k+\ga/2}}+k^s\|g\|_{L^2_{|\ga|+7}}\|f\|_{L^2_{k+\ga/2-1}}\|h\|_{L^2_{k+\ga/2}}.
\eeno

For $\mathscr{E}_{1, 1}$, similar to $\mathscr{B}_{1, 1}$, for $s\leq\mathbf{a}\leq1$, we have
\beno
\mathscr{E}_{1, 1} &=&k \left|\int_{\R^6\times\mathbb{S}^2} b(\cos \theta) |v-v_*|^{1+\gamma} (v_* \cdot \tilde{\omega}) \cos^k \frac \theta 2 \sin \frac \theta 2 g_* h' \langle v' \rangle^{k} \frac 1 {\langle v \rangle^{2-\mathbf{a}} }(f \langle v \rangle^{k-\mathbf{a}} -f' \langle v' \rangle^{k-\mathbf{a}})dv dv_* d\sigma\right|\\
&&+k \left|\int_{\R^6\times\mathbb{S}^2} b(\cos \theta) |v-v_*|^{1+\gamma} (v_* \cdot \tilde{\omega}) \cos^k \frac \theta 2 \sin \frac \theta 2 g_*   h'f' \langle v' \rangle^{2k-\mathbf{a}} (\frac 1 {\langle v \rangle^{2-\mathbf{a}}}- \frac 1 {\langle v' \rangle^{2-\mathbf{a}}}) dv dv_* d\sigma\right|
:=\mathscr{E}_{1,1,1} + \mathscr{E}_{1,1,2}.
\eeno
We first give the estimate to $\mathscr{E}_{1,1,2}$. By combining  the fact $|\langle v \rangle^{-(2-\mathbf{a})}-\langle v' \rangle^{-(2-\mathbf{a})}|\ls |v-v_*| \sin \frac \theta 2\langle v' \rangle^{-3+\mathbf{a}}\langle v_*\rangle^{2}$, the   change of variable and splitting of the integration domain into $|v-v_*|>1$ and $|v-v_*|\leq1$,
 we   get that
\beno
\mathscr{E}_{1,1,2}
&\lesssim&k^s\|g\|_{L^2_8}\|f\|_{H^s_{k-4+\ga/2}}\|h\|_{H^s_{k+\ga/2}}+k^s\|g\|_{L^2_{10}}\|f\|_{L^2_{k-1+\ga/2}}\|h\|_{L^2_{k+\ga/2}}.
\eeno

For  $\mathscr{E}_{1, 1, 1}$, by Cauchy-Schwarz inequality, we have
\beno
\mathscr{E}_{1,1,1} &\lesssim & k \left( \int_{\R^6\times\mathbb{S}^2} b(\cos \theta) |v-v_*|^{\gamma}   |g_*| (f \langle v \rangle^{k-\mathbf{a}} -f' \langle v' \rangle^{k-\mathbf{a}})^2  dv dv_* d\sigma \right)^{1/2}
\\
&&\times\left( \int_{\R^6\times\mathbb{S}^2} b(\cos \theta) \frac{ |v-v_*|^{\gamma+2} \langle v_*\rangle^2} {\langle v \rangle^{4-2\mathbf{a}}} \cos^{2k} \frac \theta 2\sin^2 \frac \theta 2 |g_*| |h'|^2\langle v' \rangle^{2k}   dv dv_* d\sigma \right)^{1/2}.
\eeno
Since $|v-v_*|^2 \langle v_*\rangle^2\langle v \rangle^{-(4-2\mathbf{a})}\<v'\>^{2-2\mathbf{a}}\lesssim \langle v_* \rangle^8$,
  by regular change of variable, we   derive that
\beno
&&\int_{\R^6\times\mathbb{S}^2} b(\cos \theta) \frac{ |v-v_*|^{\gamma+2} \langle v_*\rangle^2} {\langle v \rangle^{4-2\mathbf{a}}} \cos^{2k} \frac \theta 2\sin^2 \frac \theta 2 |g_*| |h'|^2\langle v' \rangle^{2k}   dv dv_* d\sigma
\\
&\lesssim&\int_{\mathbb{S}^2} b(\cos \theta) \cos^{2k-3-\gamma} \frac \theta 2\sin^2 \frac \theta 2 d\sigma \int_{\R^3} \int_{\R^3}      |v'-v_*|^{\gamma} \langle v_*\rangle^8  |g_*| |h'|^2\langle v' \rangle^{2k-2+2\mathbf{a}}   dv' dv_*
\\
&\lesssim& k^{s-1}\Vert g \Vert_{L^2_{|\gamma| +8}} \Vert h \Vert_{H^s_{k-2+\mathbf{a}+\gamma/2}}^2+k^{s-1}\Vert g \Vert_{L^2_{|\gamma| +8}} \Vert h \Vert_{L^2_{k-1+\mathbf{a}+\gamma/2}}^2.
\eeno
Observing that $(a-b)^2 =-2a(b-a)+ (b^2 -a^2)$, we have
\begin{equation*}
\begin{aligned}
 &\mathcal{I}_2:=\int_{\R^3} \int_{\R^3} \int_{\mathbb{S}^2} b(\cos \theta) |v-v_*|^{\gamma}  |g_*| (f \langle v \rangle^{k-\mathbf{a}} -f' \langle v' \rangle^{k-\mathbf{a}})^2  dv dv_* d\sigma
\\
=&-2(Q(|g|, f \langle \cdot \rangle^{k-\mathbf{a}}) , f \langle \cdot \rangle^{k-\mathbf{a}} )+\int_{\R^3} \int_{\R^3} \int_{\mathbb{S}^2} b(\cos \theta) |v-v_*|^{\gamma}  |g_*| (|f'|^2 \langle v' \rangle^{2k-2\mathbf{a}}  - |f|^2 \langle v \rangle^{2k-2\mathbf{a}} )  dv dv_* d\sigma.
\end{aligned}
\end{equation*}
By cancellation lemma  and Lemma \ref{L12}, we have $\mathcal{I}_2  \lesssim \Vert g \Vert_{L^2_{|\gamma| +7 }} \Vert f \Vert_{H^s_{k+s-\mathbf{a}+\gamma/2}}^2$.
Thus we obtain that
\beno
\mathscr{E}_{1,1,1}\ls k^{(1/2)(1+s)}\|g\|_{L^2_{11}}\|f\|_{H^s_{k+s-\mathbf{a}+\ga/2}}\|h\|_{H^s_{k-2+\mathbf{a}+\ga/2}}+k^{(1/2)(1+s)}\|g\|_{L^2_{11}}\|f\|_{H^s_{k+s-\mathbf{a}+\ga/2}}\|h\|_{L^2_{k-1+\mathbf{a}+\ga/2}}.
\eeno
Together with the estimate of $\mathscr{E}_{1,1,2}$, we get
\beno
\mathscr{E}_{1,1}&\ls& k^s\|g\|_{L^2_8}\|f\|_{H^s_{k-4+\ga/2}}\|h\|_{H^s_{k+\ga/2}}+k^s\|g\|_{L^2_{10}}\|f\|_{L^2_{k-1+\ga/2}}\|h\|_{L^2_{k+\ga/2}}\\
&&+k^{(1/2)(1+s)}\|g\|_{L^2_{11}}\|f\|_{H^s_{k+s-\mathbf{a}+\ga/2}}\|h\|_{H^s_{k-2+\mathbf{a}+\ga/2}}+k^{(1/2)(1+s)}\|g\|_{L^2_{11}}\|f\|_{H^s_{k+s-\mathbf{a}+\ga/2}}\|h\|_{L^2_{k-1+\mathbf{a}+\ga/2}}.
\eeno

Thus we arrive at
\beno
|\mathscr{E}_1|&\ls& k^s\|g\|_{L^2_{11}}\|f\|_{L^2_{k-1+\ga/2}}\|h\|_{L^2_{k+\ga/2}}+k^s\|g\|_{L^2_{11}}\|f\|_{H^s_{k+\ga/2-2}}\|h\|_{H^s_{k+\ga/2}}\\
&&+k^{(1/2)(1+s)}\|g\|_{L^2_{11}}\|f\|_{H^s_{k+s-\mathbf{a}+\ga/2}}\|h\|_{H^s_{k-2+\mathbf{a}+\ga/2}}+k^{(1/2)(1+s)}\|g\|_{L^2_{11}}\|f\|_{H^s_{k+s-\mathbf{a}+\ga/2}}\|h\|_{L^2_{k-1+\mathbf{a}+\ga/2}}.
\eeno

\noindent$\bullet$ \underline{Estimate of $\mathcal{H}_{k/(2s)}(g,f,h)$}. For polynomial case, similar to the estimate of $\mathscr{B}_3$ in Lemma \ref{L21}, we have
\beno
&&\mathcal{H}_{\f{k}{2s}}(g,f,h)\ls C_k\int_{\R^6\times\S^2}b(\cos\th)\sin^2\th|v-v_*|^{2+\ga}\min\{\<v\>^2,\<v_*\>^2\}\max\{\<v\>^2\cos^2(\th/2),\<v_*\>^2\sin^2(\th/2)\}^{\f{k}2-2}\\
&&\times|g(v_*)||f(v)||h(v')\<v'\>^k|dvdv_*d\si\ls C_k\|f\|_{L^2_{14}}\|g\|_{H^s_{k-2+\ga/2}}\|h\|_{H^s_{k+\ga/2}}+C_k\|g\|_{L^2_{14}}\|f\|_{H^s_{k-2+\ga/2}}\|h\|_{H^s_{k+\ga/2}}.
\eeno

At the end, by gathering together the estimates of $\mathscr{D}$ and $\mathscr{E}$, we conclude the desired result.
\end{proof}

In the next, we derive the upper bounds for the commutators of $Q$ with the spatial derivatives.
\begin{lem}\label{Qspatial}
Let $k\geq22$ and $\al$ be any multi-index such that $|\al|=2$ and $F=\mu+f$. Then
\beno
&&\left|\int_{\T^3}\int_{\R^3}(\pa^\al_xQ(F,g)-Q(F,\pa^\al_xg))\<v\>^{2k-8|\al|}\pa^\al_x hdxdv\right|
\le C_k(\| f\|_{H^2_{x}L^2_{14}}\|g\|_{Y_{k+s-1}}+\|g\|_{H^2_{x}L^2_{14}}\|f\|_{Y_{k+s-1}})\|h\|_{Y_k}\\
&&+(k^s\|f\|_{H^2_{x}L^2_{14}}\|g\|_{X_{k+\ga/2}}+\f1k\|g\|_{H^2_{x}L^2_{14}}\|f\|_{X_{k+\ga/2}})\|h\|_{X_{k+\ga/2}}.\eeno
\end{lem}
\begin{proof}
By the Leibniz rule for the bilinear operator and the fact $\partial^\alpha \mu =0, |\alpha| \ge 1$, it holds that
\[
\pa^{\al}_xQ(F, g)-Q(F,\pa^\al_x g)=\sum_{|\al_1|\neq0}C_{\al_1,\al_2}Q(\pa^{\al_1}_x  F ,\pa^{\al_2}_x g)=\sum_{|\al_1|\neq0}C_{\al_1,\al_2}Q(\pa^{\al_1}_x  f ,\pa^{\al_2}_x g),
\]
where $\alpha_2 =\alpha-\alpha_1$. By the fact $\int_{\TT^3} g(x)h(x)f(x)dx=\sum_{p\in\Z^3}\sum_{q\in\Z^3} \hat{g}(p)\hat{h}(q-p)\hat{f}(q)$ where $\hat{f}$ denotes the Fourier transform w.r.t. $x$ variable. Therefore
\beno \mathcal{I}:=\int_{\TT^3}\int_{\R^3} Q(\pa^{\al_1}_x  f ,\pa^{\al_2}_x g)\<v\>^{2k-8|\al|}\pa^\al_x hdxdv=\sum_{p\in\Z^3}\sum_{q\in\Z^3}p^{\alpha_1}(q-p)^{\alpha_2}q^{\alpha} \int_{\R^3}Q(\hat{f}(p),\hat{g}(q-p))\<v\>^{2k-8|\al|}\hat{h}(q)dv.\eeno
By Lemma  \ref{L12} and  Lemma \ref{L23} with $\mathbf{a}=1$, one has
\beno &&|\mathcal{I}|\ls \sum_{p\in\Z^3}\sum_{q\in\Z^3} |p|^{\alpha_1}|q-p|^{\alpha_2}|q|^{\alpha}\big(C_k\|\hat{f}(p)\|_{L^2_5}\|\<\cdot\>^{k-4|\al|}\hat{g}(q-p)\|_{H^s_{\ga/2+2s}}\|\<\cdot\>^{k-4|\al|}\hat{h}(q)\|_{H^s_{\ga/2}}+ C_k\|\hat{f}(p)\|_{L^2_{14}}\\
&&\times \|\<\cdot\>^{k-4|\al|}\hat{g}(q-p)\|_{H^s_{s-1+\ga/2}} \|\<\cdot\>^{k-4|\al|}\hat{h}(q) \|_{H^s_{\ga/2}} +C_k\| \hat{g}(q-p)\|_{L^2_{14}}\|\<\cdot\>^{k-4|\al|}\hat{f}(p)\|_{H^s_{-1+\ga/2}}\|\<\cdot\>^{k-4|\al|}\hat{h}(q)\|_{H^s_{\ga/2}}\\
&&+\f1k\| \hat{g}(q-p)\|_{L^2_{14}}\|\<\cdot\>^{k-4|\al|}\hat{f}(p)\|_{L^2_{\ga/2}}\|\<\cdot\>^{k-4|\al|}\hat{h}(q)\|_{L^2_{\ga/2}}+k^s\| \hat{f}(p)\|_{L^2_{14}}\|\<\cdot\>^{k-4|\al|}\hat{g}(q-p)\|_{L^2_{\ga/2}}\|\<\cdot\>^{k-4|\al|}\hat{h}(q)\|_{L^2_{\ga/2}}\big).   \eeno
Noticing that \ben\label{estiprodx}\sum_{p\in\Z^3}\sum_{q\in\Z^3}  | A_pB_{q-p}C_q|\ls(\sum_{p\in\Z^3} \<p\>^{2a} |A_p|^2)^{\f12}(\sum_{p\in\Z^3} \<p\>^{2b} |B_p|^2)^{\f12} (\sum_{p\in\Z^3}  |C_p|^2)^{\f12},\een where $a+b>\f32$ with $a,b\ge0$. We give the detailed estimates for the first term and others can be handled similarly. For $|\al_1|=|\al_2|=1$, choose $a=1,b=\f{3-s}4$ we have
\beno
&&\sum_{p\in\Z^3}\sum_{q\in\Z^3} |p||q-p||q|^{2}\big(C_k\|\hat{f}(p)\|_{L^2_5}\|\<\cdot\>^{k-4|\al|}\hat{g}(q-p)\|_{H^s_{\ga/2+2s}}\|\<\cdot\>^{k-4|\al|}\hat{h}(q)\|_{H^s_{\ga/2}}\\
&\leq&C_k(\sum_{p\in\Z^3} \<p\>^{2a+2}\|\hat{f}(p)\|_{L^2_5}^2)^{\f12}(\sum_{p\in\Z^3} \<p\>^{2b+2} \|\<\cdot\>^{k-8+2s}\hat{g}(p)\|_{H^s_{\ga/2}}^2)^{\f12} (\sum_{q\in\Z^3} \<q\>^4 \|\<\cdot\>^{k-4|\al|}\hat{h}(q)\|_{H^s_{\ga/2}}^2)^{\f12}.
\eeno
Since
$\sum_{p\in\Z^3} \<p\>^{2b+2} \|\<\cdot\>^{k-8+2s}\hat{g}(p)\|_{H^s_{\ga/2}}^2=\sum_{p\in\Z^3} \<p\>^{2b+2} \|\<\cdot\>^{k-4(b+1)+s-1}\hat{g}(p)\|_{H^s_{\ga/2}}^2\leq \|g\|^2_{Y_{k+s-1}}$, then the above can be bounded by $C_k \|f\|_{H^2_x L^2_{14}}\|g\|_{Y_{k+s-1}}\|h\|_{Y_k}$ thanks to the definition of the energy space $Y_k$(see \eqref{Y_k}).
For $|\al_1|=2,|\al_2|=0$, choose $a=0,b=\f{7-s}4$, we can also obtain the same bound by the fact that $\sum_{p\in\Z^3} \<p\>^{2b} \|\<\cdot\>^{k-8+2s}\hat{g}(p)\|_{H^s_{\ga/2}}^2=\sum_{p\in\Z^3} \<p\>^{2b} \|\<\cdot\>^{k-4b+s-1}\hat{g}(p)\|_{H^s_{\ga/2}}^2\leq \|g\|^2_{Y_{k+s-1}}$. Then we complete the proof of this lemma.
\end{proof}

\subsection{Lower bounds of  the  collision operator $Q$} We have:
\begin{thm}\label{T24}
Suppose that $\ga\in(-3,1], s \in (0, 1), \ga+2s>-1$. $F= \mu +g$ satisfies
\ben\label{Fcondition}
F \ge 0,\quad  \Vert F \Vert_{L^1} \ge 1/2, \quad \Vert F \Vert_{L^1_2} +\Vert F \Vert_{L \log L} \le 4.
\een
If $k\geq22$, then there exists a constant $\ga_1>0$ such that
\ben\label{211}
&&\notag(Q(F, f), f \langle v \rangle^{2k} )+ \ga_1 \Vert f \Vert_{H^s_{k+\f\gamma2}}^2  +\frac {1} {8} \Vert  b(\cos \theta) (1- \cos^{2k-3-\gamma} \frac \theta 2  )\Vert_{L^1_\theta}|[ f ]|_{L^2_{k+\f\gamma2}}^2 \le  C_{k}  \Vert f \Vert_{L^2_{k-1+\f\gamma2}}^2+C_k\Vert f \Vert_{L^2_{14 }} \Vert g \Vert_{H^s_{ k-1+\f\gamma2 }}\\
&&\times\Vert f \Vert_{H^s_{ k + \f\gamma2}}+C_k\|f\|_{L^2_{14}}\|g\|_{L^2_{k+\f \ga 2}}\|f\|_{L^2_{k+\f\ga 2}} +C_k\Vert g \Vert_{L^2_{14 }} \Vert f \Vert_{H^s_{ k-1 + \f\gamma2}}\|f\|_{H^2_{k+\f\ga 2}}+   C_k (\|g\|_{L^2_{14}}+\|g\|^4_{L^2_{14}})\|f\|^2_{L^2_{k+\f\ga 2}}.
\een
In particular when $g=0$, we have
\ben\label{212}
&&(Q(\mu, f), f \langle v \rangle^{2k} )+ \ga_1\Vert f \Vert_{H^s_{k+\f{\gamma}2}}^2+\frac {1} {8} \Vert  b(\cos \theta) ( 1- \cos^{2k-3-\gamma} \frac \theta 2  )\Vert_{L^1_\theta}|[ f ]|_{L^2_{k+\f{\gamma}2}}^2 \le  C_{k}  \Vert f \Vert_{L^2_{k+\f{\gamma}2-1}}^2.
\een
Moreover if $2js\geq22$, we have
\ben\label{exp2}
&&\hspace{0.5cm}(Q(F,f),f\<v\>^{4js})+\ga_1 \|f\|^2_{H^s_{2js+\ga/2}}+\frac{\ga_2}{8}\|b(\cos\th)(1-\cos^{4js-3-\ga}(\th/2))\|_{L^1_\th}\|f\|^2_{L^2_{2js+\ga/2}}\\
\notag&&\hspace{0.2cm}\ls j^{-1} \|f\|_{L^2_{14}}\|F\|_{L^2_{2js+\ga/2}}\|f\|_{L^2_{2js+\ga/2}}+j\|f\|_{L^2_{14}}\|F\|_{H^s_{2js-2+\ga/2}}\|f\|_{H^s_{2js+\ga/2}}\\
\notag&&\hspace{0.3cm}+j^{4}\|f\|_{L^2_{14}}\|F\|_{H^s_{2js-8+\ga/2}}\|f\|_{H^s_{2js+\ga/2}}+\|F\|^4_{L^2_{14}}\|f\|^2_{L^2_{2js+\ga/2}}+j\|F\|_{L^2_{14}}\|f\|_{H^s_{2js-2+\ga/2}}\|f\|_{H^s_{2js+\ga/2}}\\
\notag&&\hspace{0.3cm}+j^4\|F\|_{L^2_{14}}\|f\|_{H^s_{2js-8+\ga/2}}\|f\|_{H^s_{2js+\ga/2}}+j^{1+s}\|F\|_{L^2_{14}}\|f\|^2_{H^s_{2js-2+\ga/2}}+j^s\|F\|_{L^2_{14}}\|f\|_{L^2_{2js-1+\ga/2}}\|f\|_{L^2_{2js+\ga/2}}\\
\notag&&\hspace{0.3cm}+j^{3+s}\|F\|_{L^2_{14}}\|f\|_{L^2_{2js- 7  +\ga/2}}\|f\|_{L^2_{2js+\ga/2}}+j^{1+s}\|F\|_{L^2_{14}}\|f\|^2_{L^2_{2js-1+\ga/2}}+\mathcal{H}_j(F,f,f),
\een
where $\gamma_2$ is a constant verifying
\ben\label{ga2}
\int_{\R^3} |v-v_*|^\gamma F(v_*) dv_*\geq \gamma_2 \langle v \rangle^\gamma.
\een

\end{thm}
\begin{proof} We begin with the proof of the existence of $\ga_2$ in (\ref{ga2}).
On one hand, it is easy to see that $\int_{\R^3} |v-v_*|^{\ga}F(v_*)dv_*\ge \int_{\R^3}1_{|v|\ge R}1_{|v_*|\le R/2} |v-v_*|^{\ga}F(v_*)dv_*\gtrsim   |v|^\gamma\mathrm{1}_{|v|\ge R}(1/2-16/R^2)$. On the other hand, if $\gamma<0$, it holds that  $\int_{\R^3} |v-v_*|^{\ga}F(v_*)dv_*\ge 1_{|v|\le 2R}\int_{\R^3} |v-v_*|^{\gamma}1_{|v_*|\le R}F(v_*)dv_*\ge(3R)^{\gamma} (1/2-16/R^2)$. While if $\gamma>0$, then $\int_{\R^3} |v-v_*|^{\ga}F(v_*)dv_*\ge 1_{|v|\le 2R}\int_{\R^3} |v-v_*|^{\gamma}1_{|v-v_*|\ge r}   1_{ F(v_*) \le R }   F(v_*)dv_*\ge r^{\gamma} (1/2-16/R^2-4/(\log R)- 4 \pi Rr^3)$. Then the result (\ref{ga2}) follows by choosing proper $R$ and $r$.

To prove the main results, we introduce the following two decompositions of $(Q(F,f),f\<v\>^{2k})$:
\beno
&&(Q(F,f),f\<v\>^{2k})=(Q(F,f\<v\>^k),f\<v\>^k)+\int_{\R^6\times\S^2}b(\cos\th)|v-v_*|^\ga F_*ff'\<v'\>^k(E(\th)^{k/2}-\<v\>^k)dvdv_*d\si\\
&&\qquad\qquad\qquad\qquad+\int_{\R^6\times\S^2}b(\cos\th)|v-v_*|^\ga F_*ff'\<v'\>^k(\<v'\>^k-E(\th)^{k/2})dvdv_*d\si\\
&&\mbox{\,\,and}\quad(Q(F,f),f\<v\>^{2k})
=\int_{\R^6\times\S^2}b(\cos\th)|v-v_*|^\ga F_*(ff'\<v'\>^{k}E(\th)^{k/2}-f^2\<v\>^{2k})dvdv_*d\si\\
&&\qquad\qquad\qquad\qquad\qquad\qquad+\int_{\R^6\times\S^2}b(\cos\th)|v-v_*|^\ga F_*ff'\<v'\>^{k}(\<v'\>^k-E(\th)^{k/2})dvdv_*d\si.
\eeno
We remark that the first one is used to gain the regularity while the other one is to gain the weight. To make full use of them, we introduce parameter  $\eta<1$ to derive that
\beno
&&(1+\eta)(Q(F,f),f\<v\>^{2k})\leq \int_{\R^6\times\S^2}b(\cos\th)|v-v_*|^\ga F_*(|f||f'|\<v'\>^{k}E(\th)^{\f k2}-f^2\<v\>^{2k})dvdv_*d\si\\
&&+\eta(Q(F,f\<v\>^k),f\<v\>^k)+(1+\eta)\int_{\R^6\times\S^2}b(\cos\th)|v-v_*|^\ga F_*ff'\<v'\>^k(\<v'\>^k-E(\th)^{k/2})dvdv_*d\si\\
&&+\eta\int_{\R^6\times\S^2}b(\cos\th)|v-v_*|^\ga F_*ff'\<v'\>^k(E(\th)^{ k/2}-\<v\>^k)dvdv_*d\si:=\mathscr{J}_1+\mathscr{J}_2+\mathscr{J}_3+\mathscr{J}_4.
\eeno

\underline{\it Step ~1: Estimate of $\mathscr{J}_1$.} Using(\ref{gamma}), we have
\beno
&&\mathscr{J}_1\leq\int_{\R^6\times\S^2}b(\cos\theta)|v-v_*|^\ga F_*(|f||f'|\<v'\>^{k}\<v\>^k\cos^k(\th/2)-f^2\<v\>^{2k})dvdv_*d\si\\
&&+\sum_{p=1}^{l_{k/2}}\frac{\Ga(k/2+1)}{\Ga(p+1)\Ga(k/2+1-p)}\int_{\R^6\times\S^2}   b(\cos\theta)  |v-v_*|^\gamma F_*   |f|  |f'|  \langle v' \rangle^k       [(\<v\>^2\cos^2(\th/2))^p(\<v_*\>^2\sin^2(\theta/2))^{k/2-p}\eeno\beno
\notag&&+(\<v\>^2\cos^2(\theta/2))^{k/2-p}(\<v_*\>^2\sin^2(\theta/2))^{p}] dvdv_*d\si+\int_{\R^6\times \S^2}b(\cos\theta)|v-v_*|^\ga F_*  |f|  |f'|  \<v'\>^k(\sin^2(\theta/2))^{k/2}\\
&&\times\<v_*\>^k dv dv_* d\si
:=\mathscr{J}_{1,1}+\mathscr{J}_{1,2}+\mathscr{J}_{1,3}.
\eeno

  Due to the cancellation lemma in Lemma  \ref{canlemma} and the fact that
$|f| |f'| \langle v' \rangle^{k} \langle v \rangle^k \cos^k \frac \theta 2 -|f|^2 \langle v \rangle^{2k} \le \frac 1 2 (|f'|^2\langle v' \rangle^{2k} \cos^{2k} \frac \theta 2-|f|^2 \langle v \rangle^{2k} )$, we have
\beno
\mathscr{J}_{1,1}&\leq& \frac 1 2\int_{\R^6\times\mathbb{S}^2} b(\cos \theta) |v-v_*|^\ga  F_* |f|^2  \langle v \rangle^{2k} ( \cos^{2k-3-\gamma} \frac \theta 2 -1    ) dv dv_* d\sigma\\
& \leq& - \frac {1} {2} \Vert  b(\cos \theta) (1- \cos^{2k-3-\gamma} \frac \theta 2  )\Vert_{L^1_\theta}\int_{\R^6}|v-v_*|^\ga F_*|f(v)\<v\>^k|^2dvdv_*.
\eeno
  Thanks to Lemma \ref{L110} and Lemma \ref{lemma2.2} and the estimates of $\mathscr{D}$ in Lemma \ref{L23}, we derive that
\beno
&&\mathscr{J}_1+\frac{\ga_2}{2}\|b(\cos\th)(1-\cos^{2k-3-\ga}(\th/2))\|_{L^1_\th}\|f\|^2_{L^2_{k+\ga/2}}\ls \f1k \|f\|_{L^2_{14}}\|F\|_{L^2_{k+\ga/2}}\|f\|_{L^2_{k+\ga/2}}\\
&&+k\|f\|_{L^2_{14}}\|F\|_{H^s_{k-2+\ga/2}}\|f\|_{H^s_{k+\ga/2}}+k\|F\|_{L^2_{14}}\|f\|_{H^s_{k-2+\ga/2}}\|f\|_{H^s_{k+\ga/2}}\\
&&+k^3\|f\|_{L^2_{14}}\|F\|_{H^s_{k-6+\ga/2}}\|f\|_{H^s_{k+\ga/2}}+k^3\|F\|_{L^2_{14}}\|f\|_{H^s_{k-6+\ga/2}}\|f\|_{H^s_{k+\ga/2}}.
\eeno

\underline{\it Step ~2: Estimate of $\mathscr{J}_2$.} By coercivity estimate in Lemma \ref{L13}, if $-1-2s<\ga<-2s$, we have
\beno
\mathscr{J}_2&\le& -\ga_1\eta  \|f\|^2_{H^s_{k+\ga/2}}+C\eta\|f\|^2_{L^2_{k+\ga/2}}+C\eta\|F\|^4_{L^2_{|\ga|+2}}\|f\|^2_{L^2_{k+\ga/2}},
\eeno
where $\ga_1$ and $C$ depend on $\|F\|_{L^1_1}$ and $\|F\|_{L\log L}$. By the fact that $F=\mu+g$, we also have that
\ben\label{J22}
\mathscr{J}_2
&\leq& -\ga_1\eta  \|f\|^2_{H^s_{k+\ga/2}}+C\eta\|f\|^2_{L^2_{k+\ga/2}}+C\eta\|g\|^4_{L^2_{|\ga|+2}}\|f\|^2_{L^2_{k+\ga/2}}.
\een
While $\ga+2s\geq0$, we have
$\mathscr{J}_2\le -\ga_1\eta \|f\|^2_{H^s_{k+\ga/2}}+C\eta\|f\|^2_{L^2_{k+\ga/2}}$.

\underline{\it Step ~3: Estimates of $\mathscr{J}_3$.} 
Due to (\ref{v'}), we have
\beno
|\mathscr{J}_3|&=&(1+\eta)\left|\int_{\R^6\times\S^2}b(\cos\th)|v-v_*|^\ga F_*ff'\<v'\>^k(\<v'\>^k-E(\th)^{k/2})dvdv_*d\si\right|\\
&\ls& k\left|\int_{\R^6\times\S^2}b(\cos\th)|v-v_*|^{1+\ga} F_*f\<v\>^{k-2}f'\<v'\>^k\cos^{k-1}(\th/2)\sin(\th/2)(v\cdot\om)dvdv_*d\si\right|\\
&&+k\left|\int_{\R^6\times\S^2}b(\cos\th)|v-v_*|^{1+\ga} F_*ff'\<v'\>^k[(E(\th))^{k/2-1}-(\<v\>^2\cos^2(\th/2))^{k/2-1}]\sin\th(v\cdot\om)dvdv_*d\si\right|\\
&&+\mathcal{H}_{k/(2s)}(F,f,f):=\mathscr{J}_{3,1}+\mathscr{J}_{3,2}+\mathcal{H}_{k/(2s)}(F,f,f).
\eeno

\noindent\underline{\it Estimate of $\mathscr{J}_{3,1}$.} We remark that the estimate of $\mathscr{J}_{3,1}$ is similar to the estimate of $\mathscr{E}_{1}$ in Lemma \ref{L23}. Recall that in Lemma \ref{L18},
\beno
\omega = \tilde{\omega} \cos \frac \theta 2  + \frac {v'-v_*} {|v'-v_*|} \sin \frac \theta 2,\quad \tilde{\omega} = \frac {v'-v} {|v'-v|}.
\eeno
Then we have
\beno
\mathscr{J}_{3,1}&\leq&k\left|\int_{\R^6\times\S^2} b(\cos \theta) |v-v_*|^{1+\ga} (v_*\cdot \tilde{\omega}) \cos^{k} \frac \theta 2 \sin \frac \theta 2F_* f\<v\>^{k-2} f'\langle v' \rangle^{k}  dvdv_* d\sigma\right|
\\
&&+k\left|\int_{\R^6\times\S^2} b(\cos \theta) |v-v_*|^{1+\gamma} (v_*\cdot \frac {v'-v_*} {|v'-v_*|}) \cos^{k-1} \frac \theta 2 \sin^2 \frac \theta 2 F_* f\<v\>^{k-2}f'\langle v' \rangle^{k}   dvdv_* d\sigma\right|
\\
&:=& \mathscr{J}_{3,1,1} +\mathscr{J}_{3,1,2}.
\eeno
We first give the estimate of $\mathscr{J}_{3,1,2}$, similar to the estimate of $\mathscr{E}_{1,2}$,
\beno
\mathscr{J}_{3,1,2}&\leq&k\int_{\R^6\times \S^2}b(\cos\th)\sin^2(\th/2)\cos^k\f\th2|v-v_*|^{1+\ga}F_*|f||f'|\<v'\>^k\<v\>^{k-2}\<v_*\>dvdv_*d\si\\
&\leq&k^s\|F\|_{L^2_{14}}\|f\|_{H^s_{k+\ga/2-2}}\|f\|_{H^s_{k+\ga/2}}+k^s\|F\|_{L^2_{14}}\|f\|_{L^2_{k+\ga/2-1}}\|f\|_{L^2_{k+\ga/2}}.
\eeno
Then we give the estimate of $\mathscr{J}_{3,1,1}$, by symmetry as before
\[
\mathscr{J}_{3,1,1}=k\left|\int_{\R^6\times\S^2} b(\cos \theta) |v-v_*|^{1+\ga}  (v_*\cdot \tilde{\omega}) \cos^{k} \frac \theta 2 \sin \frac \theta 2F_*  f'\langle v' \rangle^{k}(f\<v\>^{k-2}-f'\<v'\>^{k-2}) dvdv_* d\sigma\right|.
\]
Moreover,
\beno
\mathscr{J}_{3,1,1}&\leq&k\left|\int_{\R^6\times\S^2} b(\cos \theta) |v-v_*|^{1+\ga} (v_*\cdot \tilde{\omega}) \cos^{k} \frac \theta 2 \sin \frac \theta 2F_*  f'\langle v' \rangle^{k}\frac{1}{\<v\>^2}(f\<v\>^k-f'\<v'\>^k) dvdv_* d\sigma\right|\\
&&+k\left|\int_{\R^6\times\S^2} b(\cos \theta) |v-v_*|^{1+\ga} (v_*\cdot \tilde{\omega}) \cos^{k} \frac \theta 2 \sin \frac \theta 2F_*  (f'\<v'\>^k)^2(\frac{1}{\<v\>^2}-\frac{1}{\<v'\>^2}) dvdv_* d\sigma\right|\\
&:=&\Lambda_1+\Lambda_2.
\eeno
Since
\[
\left|\frac{1}{\<v\>^2}-\frac{1}{\<v'\>^2}\right|\leq\frac{|v-v_*|\sin(\th/2)\max\{\<v_*\>,\<v\>\}}{\<v'\>^2\<v\>^2}   \lesssim \frac{|v-v_*|\sin(\theta/2)  \langle v_* \rangle}{\<v'\>^2\<v\>}  ,
\]
we obtain that
\beno
\Lambda_2&\ls&k\int_{\R^6\times\S^2}b(\cos\th)|v-v_*|^{2+\ga}\sin^2(\th/2)\cos^k\f\th2F_*\<v_*\>^3|f'|^2\<v'\>^{2k-3}dvdv_*d\si\\
&\ls& k^s\|F\|_{L^2_{14}}\|f\|_{H^s_{k-2+\ga/2}}\|f\|_{H^s_{k+\ga/2}}+k^s\|F\|_{L^2_{14}}\|f\|_{L^2_{k-1+\ga/2}}\|f\|_{L^2_{k+\ga/2}}.
\eeno
%while $\ga+2s\geq0$, we have
% \beno
%  \Lambda_2&\ls&k^s\|F\|_{L^2_5}\|f\|^2_{L^2_{k+\ga/2-1}}.
%\eeno
Then we give the estimate of $\Lambda_1$, by cancellation lemma we have
\beno
\Lambda_1&\ls&\vep \int_{\R^6\times\S^2}b(\cos\th)|v-v_*|^\ga F_*(f\<v\>^k-f'\<v'\>^k)^2dvdv_*d\si\\
&&+C_{\vep}k^2\int_{\R^6\times\S^2}b(\cos\th)\frac{|v-v_*|^{2+\ga}\<v_*\>^2}{\<v\>^4}\sin^2(\th/2)\cos^{2k}\f\th2F_*|f'|^2\<v'\>^{2k}dvdv_*d\si:= \Lambda_{1,1}+\Lambda_{1,2}.
\eeno
Since $(a-b)^2 =-2a(b-a)+ (b^2 -a^2)$, we have
\beno
\Lambda_{1,1}&\leq&-2\vep\int_{\R^6\times\S^2} b(\cos\th)|v-v_*|^\ga F_*f\<v\>^k(f'\<v'\>^k-f\<v\>^k)dvdv_*d\si\\
&&+\vep\int_{\R^6\times\S^2} b(\cos\th)|v-v_*|^\ga F_*((f'\<v'\>^k)^2-(f\<v\>^k)^2)dvdv_*d\si
\ls -2\vep (Q(F,f\<v\>^k),f\<v\>^k)+\vep\|f\|^2_{L^2_{k+\ga/2}}.
\eeno
Due to the fact that $\frac{|v-v_*|^2\<v_*\>^2}{\<v\>^4}\ls \frac{\max\{\<v_*\>^2,\<v\>^2\}\<v_*\>^2}{\<v\>^4} \lesssim \frac{ \langle v_* \rangle^4}  {\langle v \rangle^2} $, we have
\beno
\Lambda_{1,2}&\ls& C_{\vep}k^2\int_{\R^6\times\S^2}b(\cos\th)|v-v_*|^\ga\sin^2(\th/2)\cos^{2k}\frac \theta 2 F_*\<v_*\>^6|f'|^2\<v'\>^{2k-2}dvdv_*d\si\\
&\ls& C_{\vep}k^{1+s}(\|F\|_{L^2_{14}}\|f\|^2_{H^s_{k-2+\ga/2}}+\|F\|_{L^2_{14}}\|f\|^2_{L^2_{k-1+\ga/2}}).
\eeno
%while for $\ga+2s\geq0$, we have
%\beno
%\Lambda_{1,2}&\ls& k^{1+s}C_\vep (1+\eta)\|F\|_{L^2_5}\|f\|^2_{L^2_{k-1+\ga/2}}.
%\eeno
Patching together above estimates, we conclude that
\beno
\mathscr{J}_{3,1}&\ls&k^s\|F\|_{L^2_{14}}\|f\|_{H^s_{k-2+\ga/2}}\|f\|_{H^s_{k+\ga/2}}+k^s\|F\|_{L^2_{14}}\|f\|_{L^2_{k-1+\ga/2}}\|f\|_{L^2_{k+\ga/2}}\\
&-&2\vep (Q(F,f\<v\>^k),f\<v\>^k)+\vep\|f\|^2_{L^2_{k+\ga/2}}+C_{\vep}k^{1+s}(\|F\|_{L^2_{14}}\|f\|^2_{H^s_{k-2+\ga/2}}+\|F\|_{L^2_{14}}\|f\|^2_{L^2_{k-1+\ga/2}}).
\eeno

%While $\ga+2s\geq0$, we have
%\beno
%\mathscr{J}_{3,1}&\ls&-2\vep (1+\eta)(Q(F,f\<v\>^k),f\<v\>^k)+\vep(1+\eta)\|f\|^2_{L^2_{k+\ga/2}}\\
%&&+k^{1+s}C_\vep (1+\eta)\|F\|_{L^2_5}\|f\|^2_{L^2_{k-1+\ga/2}}.
%\eeno
\noindent\underline{\it Estimate of $\mathscr{J}_{3,2}$.} One may use the same argument in the estimate of $\mathscr{E}_{2}$ in Lemma \ref{L23} to obtain that
\beno
\mathscr{J}_{3,2}&\ls&\|f\|_{L^2_{14}}\|F\|_{H^s_{k-1+\ga/2}}\|f\|_{H^s_{k+\ga/2}}+k^4\|f\|_{L^2_{14}}\|F\|_{H^s_{k-8+\ga/2}}\|f\|_{H^s_{k+\ga/2}}+\|F\|_{L^2_{14}}\|f\|_{H^s_{k+\ga/2}}\|f\|_{H^s_{k+\ga/2}}\\
&&+k^4\|F\|_{L^2_{14}}\|f\|_{H^s_{k-8+\ga/2}}\|f\|_{H^s_{k+\ga/2}}+k^{1+s}\|F\|_{L^2_{14}}\|f\|_{L^2_{k-3+\ga/2}}\|f\|_{L^2_{k+\ga/2}}+k^{2+s}\|F\|_{L^2_{14}}\|f\|_{L^2_{k-5+\ga/2}}\|f\|_{L^2_{k+\ga/2}}.
\eeno

Thus, we conclude that for any   $\epsilon\ll\eta$,
\beno
&&|\mathscr{J}_3|\ls \vep\|f\|^2_{L^2_{k+\f{\ga}2}}-\vep (Q(F,f\<v\>^k),f\<v\>^k)+k\|F\|_{L^2_{14}}\|f\|_{H^s_{k-2+\f{\ga}2}}\|f\|_{H^s_{k+\f{\ga}2}}+k^4\|F\|_{L^2_{14}}\|f\|_{H^s_{k-8+\f{\ga}2}}\|f\|_{H^s_{k+\f{\ga}2}}\\
&&+k^s\|F\|_{L^2_{14}}\|f\|_{L^2_{k-1+\f{\ga}2}}\|f\|_{L^2_{k+\f{\ga}2}}+C_{\vep}k^{1+s}(\|F\|_{L^2_{14}}\|f\|^2_{H^s_{k-2+\f{\ga}2}}+\|F\|_{L^2_{14}}\|f\|^2_{L^2_{k-1+\f{\ga}2}})+k^{3+s}\|F\|_{L^2_{14}}\|f\|_{L^2_{k-7+\f{\ga}2}}\\
&&\times\|f\|_{L^2_{k+\ga/2}}+\|f\|_{L^2_{14}}\|F\|_{H^s_{k-1+\ga/2}}\|f\|_{H^s_{k+\ga/2}}+k^4\|f\|_{L^2_{14}}\|F\|_{H^s_{k-8+\ga/2}}\|f\|_{H^s_{k+\ga/2}}+\mathcal{H}_{k/2s}(F,f,f).\eeno Similar to the estimate of $ \mathcal{H}_{k/2s}(g,f,h)$ as before, we can also obtain
\beno \mathcal{H}_{k/2s}(F,f,f)\ls C_k\|f\|_{L^2_{14}}\|F\|_{H^s_{k-2+\ga/2}}\|f\|_{H^s_{k+\ga/2}}+C_k\|F\|_{L^2_{14}}\|f\|_{H^s_{k-2+\ga/2}}\|f\|_{H^s_{k+\ga/2}}.
\eeno

\underline{\it Step ~3: Estimates of $\mathscr{J}_4$.} 
 Observe that  structures of  $\mathscr{J}_4$ are similar to  $\mathscr{D}$ in Lemma \ref{L23}. Copying the argument used there, we may arrive at
\beno
&&\mathscr{J}_4-\eta\|b(\cos\th)(1-\cos^{k}(\th/2))\|_{L^1_\th}\|f\|^2_{L^2_{k+\f{\ga}2}}\ls \f 1 k \|f\|_{L^2_{14}}\|F\|_{L^2_{k+\f{\ga}2}}\|f\|_{L^2_{k+\f{\ga}2}}+k\|f\|_{L^2_{14}}\|F\|_{H^s_{k-2+\f{\ga}2}}\|f\|_{H^s_{k+\f{\ga}2}}\\
&&+k\|F\|_{L^2_{14}}\|f\|_{H^s_{k-2+\ga/2}}\|f\|_{H^s_{k+\ga/2}}+k^3\|f\|_{L^2_{14}}\|F\|_{H^s_{k-6+\ga/2}}\|f\|_{H^s_{k+\ga/2}}+k^3\|F\|_{L^2_{14}}\|f\|_{H^s_{k-6+\ga/2}}\|f\|_{H^s_{k+\ga/2}}.
\eeno

Patching together all the estimates of $\mathscr{J}_i,i=1,2,3,4$, we are led to
\beno
&&(1+\eta)(Q(F,f),\<v\>^{2k})+\left(\frac{1}{2}\|b(\cos\th)(1-\cos^{2k-3-\ga}(\th/2))\|_{L^1_\th}-\eta\|b(\cos\th)(1-\cos^{k}(\th/2))\cos^{-3/2-\ga/2}(\th/2)\|_{L^1_\th}\right)\\
&&\times\int_{\R^6}|v-v_*|^\ga F_*|f(v)\<v\>^k|^2dvdv_*-(C\eta+\vep)\|f\|^2_{L^2_{k+\ga/2}}+\f{\ga_1}2(\eta-2\vep) \|f\|^2_{H^s_{k+\ga/2}}\ls C_{\eta,\vep}\times\mathrm{L.O.T},
\eeno
where
\beno
&&\mathrm{L.O.T}=\f1k \|f\|_{L^2_{14}}\|F\|_{L^2_{k+\ga/2}}\|f\|_{L^2_{k+\ga/2}}+k\|f\|_{L^2_{14}}\|F\|_{H^s_{k-2+\ga/2}}\|f\|_{H^s_{k+\ga/2}}\\
\notag&& +k^{4}\|f\|_{L^2_{14}}\|F\|_{H^s_{k-8+\ga/2}}\|f\|_{H^s_{k+\ga/2}}+\|F\|^4_{L^2_{14}}\|f\|^2_{L^2_{k+\ga/2}}+k\|F\|_{L^2_{14}}\|f\|_{H^s_{k-2+\ga/2}}\|f\|_{H^s_{k+\ga/2}}\\
\notag&& +k^4\|F\|_{L^2_{14}}\|f\|_{H^s_{k-8+\ga/2}}\|f\|_{H^s_{k+\ga/2}}+k^{1+s}\|F\|_{L^2_{14}}\|f\|^2_{H^s_{k-2+\ga/2}}+k^s\|F\|_{L^2_{14}}\|f\|_{L^2_{k-1+\ga/2}}\|f\|_{L^2_{k+\ga/2}}\\
\notag&& +k^{3+s}\|F\|_{L^2_{14}}\|f\|_{L^2_{k-7+\ga/2}}\|f\|_{L^2_{k+\ga/2}}+k^{1+s}\|F\|_{L^2_{14}}\|f\|^2_{L^2_{k-1+\ga/2}}+\mathcal{H}_{k/2s}(F,f,f).
\eeno

To complete the proof of the lemma, it suffices to explain how to derive the gain of the weight in (\ref{211}) and  (\ref{exp2}). By Lemma \ref{lemma2.2}, it holds that for $\eta\ll1$,
\beno &&\frac{1}{2}\|b(\cos\th)(1-\cos^{2k-3-\ga}(\th/2))\|_{L^1_\th}-\eta\|b(\cos\th)(1-\cos^{k}(\th/2))\cos^{-3/2-\ga/2}(\th/2)\|_{L^1_\th} \\ &&\quad\quad\quad\quad\quad\ge \frac{1}{8}\|b(\cos\th)(1-\cos^{2k-3-\ga}(\th/2))\|_{L^1_\th}.\eeno
From this together with \eqref{ga2} and the fact \beno
\int_{\R^6}|v-v_*|^\ga F_*|f(v)\<v\>^k|^2dvdv_*\geq |[f]|^2_{L^2_{k+\ga/2}}-\|g\|_{L^2_{5}}\|f\|^2_{H^s_{k-1+\ga/2}}-\|g\|_{L^2_{5}}\|f\|^2_{L^2_{k+\ga/2}},\eeno  where we  use $F=\mu+g$ and Lemma \ref{L116}, we get the gain of the weight. This ends the proof.
\end{proof}

\subsection{Analysis of $\mathcal{H}_j$} To get the bounds of collision operator $Q$ with exponential weight, we need a more detailed calculation on $\mathcal{H}_j$. We first give a auxiliary lemma:
\begin{lem}\label{abc}
Let $1/2\leq s<1$ and $\cN_0:=[2s^{-1}]+1$, $j>\cN_0$. Then we have
\ben\label{mk1} \notag&&j^2\sum_{\substack{1\leq k\leq j-\cN_0 \\ 0\leq m\leq j-\cN_0-k\\1\leq m+k/2\leq  [(j-\cN_0)/2]} }(C_{j-\cN_0}^k)^s(C_{j-\cN_0-k}^m)^s(k+1)^{-2}2^{sk}\frac{\Ga(s(j-k/2-m-1)-1)\Ga(s(m+k/2-1)+1)}{\Ga(s(j-2))}\\
&&\hspace{3cm}\ls\sum_{m=1}^{[\frac{j-\cN_{0}}{2}]}(C_{j-1}^{j-m-1})^s m^{-(3/2)s}j^{s +1}+\sum_{m=1}^{[\frac{j-\cN_{0}+1}{2}]}(C_{j-1}^{j-m-1})^s m^{-s}j^{(1/2)s+1}
\een
and
\ben\label{mk2}
 &&j^2\sum_{\substack{1\leq k\leq j-\cN_0 \\ 0\leq m\leq j-\cN_0-k\\ [(j-\cN_0)/2]\leq m+k/2\leq j-\cN_0 } }(C_{j-\cN_0}^k)^s(C_{j-\cN_0-k}^m)^s\frac{\Ga(s(j-k/2-m-1)-1)\Ga(s(m+k/2-1)+1)}{\Ga(s(j-2))}\notag\\
&&\times (k+1)^{-2}2^{sk} \ls\sum_{m=[\frac{j-\cN_{0}}{2}]}^{j-\cN_{0}}(C_{j-1}^{j-m-1})^s (j-m)^{-2-\f52 s+2\cN_0s }j^{2s+3-2\cN_{0}s }+\sum_{m=[\frac{j-\cN_{0}+1}{2}]}^{j-\cN_{0}+1}(C_{j-1}^{j-m-1})^s\notag\\
&&\hspace{3cm} \times(j-m)^{-2s-2+2\cN_{0}s}j^{\f32s+3-2\cN_{0}s}.
\een
\end{lem}
\begin{proof} We only give a detailed proof to  $(\ref{mk1})$ since \eqref{mk2} can be derived similarly. We denote the left-hand side of $(\ref{mk1})$ by $\cC$. Thanks to $(\ref{gammafun})$, we have
$\frac{\Ga(s (j-k/2-m-1)-1)\Ga(s(m+k/2-1)+1)}{\Ga(s(j-2)))}\ls\frac{\Ga^s (j-m-k/2)\Ga^s (m+k/2)}{\Ga^s (j-2)}(j-m-\f k2)^{-1-s}(\f k2)^{-s+1}$, which gives that
\beno
 \cC&\ls& j^2\sum_{\substack{1\leq k\leq j-\cN_0 \\ 0\leq m\leq j-\cN_0-k\\1\leq m+k/2\leq  [(j-\cN_0)/2]} }2^{s k}\frac{\Ga^s (m+k/2)\Ga^s (j-m-k/2)\G^s(j-\cN_{0}+1)}{\Ga^s (k+1)\Ga^s (m+1)\Ga^s (j-\cN_{0}-k-m+1)\G^s(j-2)}\\
&&\times(j-m-k/2)^{-1-s}(m+k/2)^{-s+1}(k+1)^{-2}.
\eeno
Since $\G^s(j-\cN_{0}+1)\sim \G^s(j-2)(j-2)^{(-\cN_{0}+3)s}$ and  $\Ga(k+1)\sim 2^{k+1}\Ga(k/2+1/2)\Ga(k/2+1)$, we have
\beno
 \cC
&\ls& j^2\sum_{\substack{1\leq k\leq j-\cN_0 \\ 0\leq m\leq j-\cN_0-k\\1\leq m+k/2\leq  [(j-\cN_0)/2]} }\frac{\Ga^s (m+k/2)\Ga^s (j-m-k/2)}{\Ga^s (k/2+1/2)\Ga^s (k/2+1)\Ga^s (m+1)\Ga^s (j-\cN_{0}-k-m+1)}\\
&&\times(j-m-k/2)^{-1-s}(m+k/2)^{-s+1}(k+1)^{-2}(j-2)^{(-\cN_{0}+3)s }:=\cC_1+\cC_2,
\eeno
where $\cC_1$ and $\cC_2$ denote the above summation  when $k$ is even and   when $k$ is odd  respectively. When $k$ is even, by the change of variable from $k$ to $2k$ and then from $m+k$ to $m$, we have
\beno
\cC_{1}&\ls&j^2\sum_{k=1}^{[\frac{j-\cN_{0}}{2}]}\sum_{m=k,m+k\geq1}^{j-\cN_{0}-k}\frac{\Ga^s (j-m)\Ga^s (m)}{\Ga^s (k+1/2)\Ga^s (k+1)\Ga^s (m-k+1)\Ga^s (j-\cN_{0}-m-k+1)}\\
&&\times (k+1)^{-2}(j-m)^{-s-1}m^{-s+1}(j-2)^{(-\cN_{0}+3)s}.
\eeno
Thanks to $\G^s(j-m)\sim\G^s(j-m-\cN_{0}+1)(j-m-\cN_{0}+1)^{(\cN_{0}-1)s}$ and $\G^s(m)\sim\G^s(m+\cN_{0})(m+\cN_{0})^{-s \cN_{0}}$, we deduce that
\beno
\cC_{1}
&\ls&j^2\sum_{k=0}^{[\frac{j-\cN_{0}}{2}]}\sum_{m=k,m+k\geq1}^{j-\cN_{0}-k}\frac{\Ga^s (j-m-\cN_{0}+1)\Ga^s (m+\cN_{0})}{\Ga^s (k+1/2)\Ga^s (k+1)\Ga^s (m-k+1)\Ga^s (j-\cN_{0}-m-k+1)}\\
&&\times (k+1)^{-2}(j-m)^{-1+(\cN_{0}-2)s }m^{-s+1-\cN_{0}s }(j-2)^{(-\cN_{0}+3)s }.
\eeno
 Observing that
$\frac{\Ga^s (j-m-\cN_{0}+1)\Ga^s (m+\cN_{0})}{\Ga^s (k+1/2)\Ga^s (k+1)\Ga^s (m-k+1)\Ga^s (j-\cN_{0}-m-k+1)}
=\frac{\G^s(j-m-\cN_{0}+1)}{\G^s(j-\cN_{0}-m-k+1)\G^s(k+1)}\frac{\G^s(m+\cN_{0})}{\G^s(k+\cN_{0})\G^s(m-k+1)}\frac{\G^s(k+\cN_{0})}{\G^s(k+1/2)}\\\sim(C_{j-m-\cN_{0}}^{j-m-\cN_{0}-k})^s (C_{m+\cN_{0}-1}^{\cN_{0}+k-1})^s (k+1)^{(\cN_{0}-1/2)s}\leq (C_{j-1}^{j-m-1})^s(k+1)^{(\cN_{0}-1/2)s },$
 by exchanging the order of the summation, we derive that
$\cC_1\ls \sum\limits_{m=1}^{[\frac{j-\cN_{0}}{2}]}(C_{j-1}^{j-m-1})^s m^{-(3/2)s}j^{s +1}$.
When $k$ is odd, we change variables, from $k$ to $2k-1$ and then from $m+k$ to $m$, to get $\cC_2\ls \sum\limits_{m=1}^{[\frac{j-\cN_{0}+1}{2}]}(C_{j-1}^{j-m-1})^s m^{-s}j^{(1/2)s+1}.$ This ends the proof.
\end{proof}

\begin{lem}\label{B_3}
Suppose $\ga\in(-3,1],\ga+2s>-1$, $1/2\leq s<1$ and $\cN_0:=[2s^{-1}]+1$.  If
 $2js\geq22$, then $\mathcal{H}_j(g,f,h)\ls\bar{\mathcal{H}}_j(g,f,h)+\hat{\mathcal{H}}_j(g,f,h)$ where $\mathcal{H}_j$ is defined in  \eqref{Hgfh} and
\beno
&&\bar{\mathcal{H}}_j(g,f,h):= j\|f\|_{L^2_{14}}\|g\|_{H^s_{2js-2+\ga/2}}\|h\|_{H^s_{2js+\ga/2}}+j^4\|f\|_{L^2_{14}}\|g\|_{H^s_{2js-8+\ga/2}}\|h\|_{H^s_{2js+\ga/2}}\\
&&\qquad\qquad\qquad+j\|g\|_{L^2_{14}}\|f\|_{H^s_{2js-2+\ga/2}}\|h\|_{H^s_{2js+\ga/2}}+j^4\|g\|_{L^2_{14}}\|f\|_{H^s_{2js-8+\ga/2}}\|h\|_{H^s_{2js+\ga/2}}\\
&&\qquad\qquad\qquad+j^{1+s}\|g\|_{L^2_{14}}\|f\|_{L^2_{2js-2+\ga/2}}\|h\|_{L^2_{2js+\ga/2}}+j^{2+s}\|g\|_{L^2_{14}}\|f\|_{L^2_{2js-4+\ga/2}}\|h\|_{L^2_{2js+\ga/2}},\\
&&\hat{\mathcal{H}}_j(g,f,h):=\sum_{m=1}^{[\frac{j-\cN_0}{2}]}(C_{j-1}^{j-m-1})^sm^{-(3/2)s}j^{s+1}\|g\<\cdot\>^{2(ms+2)}\|_{L^1}\|f(\<\cdot\>^2)^{(j-m)s+\f\ga4-1}\|_{L^2}\|h\<\cdot\>^{2js+\f\ga2}\|_{L^2}
\\&&+\sum_{m=1}^{[\frac{j-\cN_0+1}{2}]}(C_{j-1}^{j-m-1})^sm^{-s}j^{\f12s+1}\|g(\<\cdot\>^2)^{(m-1/2)s+2}\|_{L^1}\|f(\<\cdot\>^2)^{(j-m+1/2)s+\ga/4-1}\|_{L^2}\|h\<\cdot\>^{2js+\ga/2}\|_{L^2}\\
&&+\sum_{m=[\frac{j-\cN_0}{2}]}^{j-\cN_0+1}(C_{j-1}^{j-m-1})^s(j-m)^{-2-(5/2)s+2\cN_0s}j^{2s+3-2\cN_0s}\|g(\<\cdot\>^2)^{sm+\ga/4-4}\|_{L^2}\|f(\<\cdot\>^2)^{s(j-m)+6}\|_{L^1}\\
&&\times\|h\<\cdot\>^{2js+\ga/2}\|_{L^2}+\sum_{m=1}^{[\frac{j-\cN_0+1}{2}]}(C^{j-m-1}_{j-1})^s m^{-\f32s}j^{1+s}\|g(\<\cdot\>^2)^{(m+\cN_0)s+2}\|_{L^1}\|f(\<\cdot\>^2)^{(j-m-\cN_0)s+\f\ga4-1}\|_{L^2}\\
&&\times\|h\<\cdot\>^{2js+\f\ga2}\|_{L^2}+\sum_{m=[\frac{j-\cN_0}{2}]}^{j-\cN_0}(C_{j-1}^{j-m-1})^s(j-m)^{-\f52s-2+2\cN_0s}j^{2s+3-2\cN_0s}\|g(\<\cdot\>^2)^{(m+\cN_0)s+\f\ga4-1}\|_{L^2}\\
&&\times\|f(\<\cdot\>^2)^{(j-m-\cN_0)s+1}\|_{L^1}\|h\<\cdot\>^{2js+\ga/2}\|_{L^2}+\sum_{m=[\frac{j-\cN_0+1}{2}]}^{j-\cN_0+1}(C_{j-1}^{j-m-1})^s(j-m)^{-2s-2+2\cN_0s}j^{\f32s+3-2\cN_0s}\\
&&\times\|g(\<\cdot\>^2)^{(m+\cN_0-\f12)s+\f\ga4-1}\|_{L^2}\|f(\<\cdot\>^2)^{(j-m-\cN_0+\f12)s+1}\|_{L^1}\|h\<\cdot\>^{2js+\f\ga2}\|_{L^2}.
\eeno

\end{lem}
\begin{proof}
 We only give a detailed proof for the case $\ga\leq0$. We recall that $\mathcal{H}_j(g,f,h)=j^2\int_{\R^6\times\S^2}\int_0^1b(\cos\th)\sin^2\th|v-v_*|^{\ga}(1-t)(E(\th)+t\tilde{h}\sin\th(\mj\cdot\hat{\omega}))^{js-2}\tilde{h}^2(\mj\cdot\hat{\om})^2 |g(v_*)|f(v)|h(v')\<v'\>^{2js}|dvdv_*d\si dt.$ Since $(\cN_{0}-1)s \leq 2<\cN_{0}s $, then $s j-2=s (j-\cN_{0})+\cN_{0}s -2$. We have
 \beno
 &&(E(\th)+t\tilde{h}\sin\th(\mj\cdot\hat{\omega}))^{js-2}\le \big[E(\th)^{(j-\cN_0)s}+\sum_{k=1}^{j-\cN_{0}}\sum_{m=0}^{j-\cN_{0}-k}(C_{j-\cN_{0}}^k)^s (C_{j-\cN_{0}-k}^m)^s (\<v\>^2\cos^2(\th/2))^{s (j-\cN_{0}-k-m)}\\&&\qquad\qquad\qquad\times
(\<v_*\>^2\sin^2(\th/2))^{s m}(t\tilde{h}\sin\th|\mj\cdot\hat{\omega}|)^{s k}\big] (E(\th)+t\tilde{h}\sin\th(\mj\cdot\hat{\omega}))^{\cN_0s-2},
 \eeno
 where we use Newton binomial expansion
$(a+b)^{(j-\cN_{0})s}=(\sum\limits_{k=0}^{j-\cN_{0}}C^k_{j-\cN_{0}} a^{(j-\cN_{0}-k)}b^k)^{s}$ and the fact $(a+b)^s\leq a^s+b^s$  with $a,b>0,s<1$.  Plugging the above into the definition of $\mathcal{H}_j$,
we may split it by $\mathcal{H}^{(1)}_j$ and $\mathcal{H}^{(2)}_j$ which correspond to the case $k=0$ and the case $k>0$ respectively.

\underline{\it Step 1: Estimate of $\mathcal{H}^{(1)}_j$.} Observing that $(E(\th)+t\tilde{h}\sin\th(\mj\cdot\hat{\omega}))^{\cN_0s-2}\ls E(\th)^{\cN_0s-2}$ and $\tilde{h} \leq \<v\>\<v_*\>$ , we first have
$\mathcal{H}^{(1)}_j\ls j^2\int_{\R^6\times\S^2}b(\cos\th)\sin^2\th|v-v_*|^{\ga}E(\th)^{js-2}|g(v_*)\<v_*\>^2|f(v)\<v\>^2|h(v')\<v'\>^{2js}|dvdv_*d\si$. By the expansion \eqref{Eth} that
$(E(\th))^{js-2}\le \sum_{p=0}^{l_{js-2}}\frac{\Ga(js-1)}{\Ga(p+1)\Ga(js-1-p)}[(\<v\>^2\cos^2(\th/2))^p(\<v_*\>^2\sin^2(\th/2))^{js-2-p}+(\<v\>^2\cos^2(\th/2))^{js-2-p}(\<v_*\>^2\sin^2(\th/2))^{p}]$, we may copy the argument used for $\mathscr{E}_1$ and $\mathscr{E}_2$  in Lemma \ref{L23} to get that
\beno
&&\mathcal{H}_j^{(1)}\ls j\|f\|_{L^2_{14}}\|g\|_{H^s_{2js-2+\f{\ga}2}}\|h\|_{H^s_{2js+\f{\ga}2}}+j^4\|f\|_{L^2_{14}}\|g\|_{H^s_{2js-8+\f{\ga}2}}\|h\|_{H^s_{2js+\f{\ga}2}}+j\|g\|_{L^2_{14}}\|f\|_{H^s_{2js-2+\f{\ga}2}}\|h\|_{H^s_{2js+\f{\ga}2}}\\
&&+j^4\|g\|_{L^2_{14}}\|f\|_{H^s_{2js-8+\f{\ga}2}}\|h\|_{H^s_{2js+\f{\ga}2}}+j^{1+s}\|g\|_{L^2_{14}}\|f\|_{L^2_{2js-2+\f{\ga}2}}\|h\|_{L^2_{2js+\f{\ga}2}}+j^{2+s}\|g\|_{L^2_{14}}\|f\|_{L^2_{2js-4+\f{\ga}2}}\|h\|_{L^2_{2js+\f{\ga}2}}.
\eeno

\underline{\it Step 2: Estimate of $\mathcal{H}^{(2)}_j$.} We first
recall that
\beno
&&\mathcal{H}^{(2)}_j=\sum_{k=1}^{j-\cN_{0}}\sum_{m=0}^{j-\cN_{0}-k}(C_{j-\cN_{0}}^k)^s (C_{j-\cN_{0}-k}^m)^sj^2\int_{\R^6\times\S^2}\int_0^1b(\cos\th)\sin^2\th|v-v_*|^{\ga}(1-t)(\<v\>^2\cos^2(\th/2))^{s (j-\cN_{0}-k-m)}\\
&&\times(\<v_*\>^2\sin^2(\th/2))^{s m}(t\tilde{h}\sin\th|\mj\cdot\hat{\omega}|)^{s k}(E(\th)+t\tilde{h}\sin\th(\mj\cdot\hat{\omega}))^{\cN_0s-2}\tilde{h}^2(\mj\cdot\hat{\om})^2|g(v_*)|f(v)|h(v')\<v'\>^{2js}|dvdv_*d\si dt.
\eeno
To get the desired result, we decompose  $\mathcal{H}^{(2)}_j$ into two parts: $\mathcal{H}^{(2)}_{j,1}$ and $\mathcal{H}^{(2)}_{j,2}$ according to the integration domain: $\<v\>\geq\<v_*\>$ and $\<v\>\leq\<v_*\>$ respectively. We first focus on  $\mathcal{H}^{(2)}_{j,1}$.
Noticing that $(E(\th)+t\tilde{h}\sin\th(\mj\cdot\hat{\omega}))^{\cN_0s-2}\ls(\max\{\<v\>^2,\<v_*\>^2\}\times\cos^2(\th/2))^{\cN_0s-2}\ls(\<v\>^2\cos^2(\th/2))^{\cN_0s-2}$, we have
\beno
&&\mathcal{H}^{(2)}_{j,1}\leq \sum_{k=1}^{j-\cN_{0}}\sum_{m=0}^{j-\cN_{0}-k}(C_{j-\cN_{0}}^k)^s (C_{j-\cN_{0}-k}^m)^sj^2\int_{\R^6\times\S^2}\int_0^1 1_{\<v\>\geq\<v_*\>}b(\cos\th)\sin^{2+sk}\th |v-v_*|^\ga(1-t)t^{sk}\eeno \beno
\times(\<v\>^2\cos^2(\th/2))^{s (j-\cN_{0}-k-m)}(\<v_*\>^2\sin^2(\th/2))^{s m}\tilde{h}^{sk+2}(\<v\>^2\cos^2(\th/2))^{\cN_0s-2}|g(v_*)|f(v)|h(v')\<v'\>^{2js}|dvdv_*d\si dt.
\eeno
Since $\ga+2s+1>0$ and  $\tilde{h}\leq\min\{\<v\>,\<v_*\>\}|v-v_*|$(see Lemma \ref{L18}), we have
\beno
\tilde{h}^{sk+2}|v-v_*|^\ga1_{\<v\>\geq\<v_*\>} \leq|v-v_*|^{2+\ga+s}\<v_*\>^{sk+2}\<v\>^{s(k-1)}1_{\<v\>\geq\<v_*\>}\leq\<v_*\>^{sk+2}\<v\>^{sk+2+\ga}.
\eeno
Moreover, since $b(\cos\th)\sim \sin^{-2-2s}\th$ and $\<v'\>^{-\ga/2}\leq\<v\>^{-\ga/2}\<v_*\>^{-\ga/2}$
for $\ga\leq0$, we   derive that
\beno
&&\mathcal{H}^{(2)}_{j,1}\leq j^2\int_{\R^6\times\S^2}1_{\<v\>\geq\<v_*\>}|g(v_*)||f(v)||h'\<v'\>^{2js+\ga/2}|\Big(\sum_{k=1}^{j-\cN_0}\sum_{m=0}^{j-\cN_0-k}(1+k)^{-2}2^{ks}(C_{j-\cN_0}^k)^s(C_{j-\cN_0-k}^m)^s\\
&&\times(\cos^2\frac{\th}{2})^{(j-\f k2-m)s-2-s}(\sin^2\f\th2)^{s(m+\f k2)-s}(\<v\>^2)^{s(j-m-\f k2)+\ga/4-1}(\<v_*\>^2)^{s(m+\f k2)-\ga/4+1}dvdv_*d\si,
\eeno
where we use the fact that $\int_0^1(1-t)t^{sk}dt\sim(1+k)^{-2}$. We remark that the same result holds if $\ga>0$. To get more precise estimate, we divide $\mathcal{H}^{(2)}_{j,1}$ into two cases: $m+k/2\in[1,[(j-\cN_0)/2]]$ and $m+k/2\in[[(j-\cN_0)/2],j-\cN_0]$, and denote them by $\mathcal{H}^{(2)}_{j,1,1}$ and $\mathcal{H}^{(2)}_{j,1,2}$.

\underline{\it Estimate of $\mathcal{H}^{(2)}_{j,1,1}$.} For this case, by the regular change of variable,  we derive that
\beno
&&\mathcal{H}^{(2)}_{j,1,1}\le j^2\sum_{\substack{1\leq k\leq j-\cN_0 \\ 0\leq m\leq j-\cN_0-k\\1\leq m+k/2\leq  [(j-\cN_0)/2]} }(C_{j-\cN_0}^k)^s(C_{j-\cN_0-k}^m)^s(k+1)^{-2}2^{sk}\int_0^{\f\pi2}(\cos^2\th/2)^{(j-\f k2-m-1)s-\f32}\\
&&\times(\sin^2\th/2)^{(m+\f k2)s+\f12-s}d\th\|g(\<\cdot\>^2)^{s(m+k/2)-\ga/4+1}\|_{L^1}\|f(\<\cdot\>^2)^{s(j-m-k/2)+\ga/4-1}\|_{L^2}\|h\<\cdot\>^{2js+\ga/2}\|_{L^2}\\
&&\leq j^2\sum_{\substack{1\leq k\leq j-\cN_0 \\ 0\leq m\leq j-\cN_0-k\\1\leq m+k/2\leq  [(j-\cN_0)/2]} }(C_{j-\cN_0}^k)^s(C_{j-\cN_0-k}^m)^s(k+1)^{-2}2^{sk}\frac{\Ga(s(j-k/2-m-1)-1)\Ga(s(m+k/2-1)+1)}{\Ga(s(j-2))}\\
&&\times\|g(\<\cdot\>^2)^{s(m+k/2)-\ga/4+1}\|_{L^1}\|f(\<\cdot\>^2)^{s(j-m-k/2)+\ga/4-1}\|_{L^2}\|h\<\cdot\>^{2js+\ga/2}\|_{L^2}.
\eeno
Thanks to Lemma \ref{abc}(\ref{mk1}), we have
\beno
&&\mathcal{H}^{(2)}_{j,1,1}\ls\sum_{m=1}^{[\frac{j-\cN_0}{2}]}(C_{j-1}^{j-m-1})^sm^{-(3/2)s}j^{s+1}\|g(\<\cdot\>^2)^{ms-\ga/4+1}\|_{L^1}\|f(\<\cdot\>^2)^{(j-m)s+\ga/4-1}\|_{L^2}\|h\<\cdot\>^{2js+\ga/2}\|_{L^2}\\
&&+\sum_{m=1}^{[\frac{j-\cN_0+1}{2}]}(C_{j-1}^{j-m-1})^sm^{-s}j^{(1/2)s+1}\|g(\<\cdot\>^2)^{(m-1/2)s-\ga/4+1}\|_{L^1}\|f(\<\cdot\>^2)^{(j-m+1/2)s+\ga/4-1}\|_{L^2}\|h\<\cdot\>^{2js+\ga/2}\|_{L^2}.
\eeno
Here for the first term in the above, we change the variables from $m+k/2$ to  $m+k$ and then from   $m+k$ to $ m$ for even $k$. For the second term, we change the variables: $m+k/2\rightarrow m+k-1/2\rightarrow m-1/2$.

%When $k$ is even, change the variable from $k$ to $2k$ and consider $m+k\in[1,[(j-\cN_0)/2]]$ and $m+k\in[[(j-\cN_0)/2],j-\cN_0]$ respectively and denote them by $\mathcal{H}^2_{1,1}$ and $\mathcal{H}^2_{1,2}$.

\underline{\it Estimate of $\mathcal{H}^{(2)}_{j,1,2}$.} Since $m+k/2\in[(j-\cN_0)/2,j-\cN_0]$ and $\<v\>\ge \<v_*\>$, by singular change of variable we get that
\beno
&&\mathcal{H}^{(2)}_{j,1,2}\ls j^2\sum_{\substack{1\leq k\leq j-\cN_0 \\ 0\leq m\leq j-\cN_0-k\\ [(j-\cN_0)/2]\leq m+k/2\leq j-\cN_0 } }(C_{j-\cN_0}^k)^s(C_{j-\cN_0-k}^m)^s(k+1)^{-2}2^{sk}\int_0^{\pi/2}(\cos^2\th/2)^{(j-\f k2-m-1)s-3/2}\\
&&\times(\sin^2\f\th2)^{(m+\f k2)s+\f12-s-\f34-\f\ga4}d\th\|g (\<\cdot\>^2)^{s(m+\f k2)+\f\ga4-4}\|_{L^2}\|f(\<\cdot\>^2)^{s(j-m-\f k2)-\f\ga4+4}\|_{L^1}\|h\<\cdot\>^{2js+\ga/2}\|_{L^2}\\
&&\le j^2\sum_{\substack{1\leq k\leq j-\cN_0 \\ 0\leq m\leq j-\cN_0-k\\ [(j-\cN_0)/2]\leq m+k/2\leq j-\cN_0 } }(C_{j-\cN_0}^k)^s(C_{j-\cN_0-k}^m)^s(k+1)^{-2}2^{sk}\frac{\Ga(s(j-k/2-m-1)-1)\Ga(s(m+k/2-1)+1)}{\Ga(s(j-2))}\\
&&\times\|g(\<\cdot\>^2)^{s(m+\f k2)+\f\ga4-4}\|_{L^2}\|f(\<\cdot\>^2)^{s(j-m-\f k2)-\f\ga4+4}\|_{L^1}\|h\<\cdot\>^{2js+\ga/2}\|_{L^2}.
\eeno
Here we observe that $\frac{\Ga(s(m+\f k2-1)-\f34-\f\ga4)}{\Ga(s(j-2)-\f34-\f\ga4)}\sim\frac{\Ga(s(m+\f k2-1))}{\Ga(s(j-2))}$ if $m+\f k2\sim j$. By Lemma \ref{abc}(\ref{mk2}), we deduce that
\beno
&&\mathcal{H}^{(2)}_{j,1,2}\ls\sum_{m=[\frac{j-\cN_0}{2}]}^{j-\cN_0}(C_{j-1}^{j-m-1})^s(j-m)^{-2-(5/2)s+2\cN_0s}j^{2s+3-2\cN_0s}\|g(\<\cdot\>^2)^{sm+\ga/4-4}\|_{L^2}\|f(\<\cdot\>^2)^{s(j-m)-\ga/4+4}\|_{L^1}\\
&&\times\|h\<\cdot\>^{2js+\ga/2}\|_{L^2}+\sum_{m=[\frac{j-\cN_0+1}{2}]}^{j-\cN_0+1}(C_{j-1}^{j-m-1})^s(j-m)^{-2-2s+2\cN_0s}j^{(3/2)s+3-2\cN_0s}\|g(\<\cdot\>^2)^{s(m-\f12)+\ga/4-4}\|_{L^2}\\
&&\times\|f(\<\cdot\>^2)^{s(j-m+\f12)-\ga/4+4}\|_{L^1}\|h\<\cdot\>^{2js+\ga/2}\|_{L^2}\leq\sum_{m=[\frac{j-\cN_0}{2}]}^{j-\cN_0+1}(C_{j-1}^{j-m-1})^s(j-m)^{-2-\f52s+2\cN_0s}j^{2s+3-2\cN_0s}\\
&&\times\|g(\<\cdot\>^2)^{sm+\f\ga4-4}\|_{L^2}\|f(\<\cdot\>^2)^{s(j-m)-\f\ga4+5}\|_{L^1}\|h\<\cdot\>^{2js+\f\ga2}\|_{L^2}.
\eeno
Finally we conclude that
\beno
&&\mathcal{H}^{(2)}_{j,1}\ls\sum_{m=1}^{[\frac{j-\cN_0}{2}]}(C_{j-1}^{j-m-1})^sm^{-\f32s}j^{s+1}\|g(\<\cdot\>^2)^{ms+2}\|_{L^1}\|f(\<\cdot\>^2)^{(j-m)s+\f\ga4-1}\|_{L^2}\|h\<\cdot\>^{2js+\f\ga2}\|_{L^2}\\
&+&\sum_{m=1}^{[\frac{j-\cN_0+1}{2}]}(C_{j-1}^{j-m-1})^sm^{-s}j^{\f12s+1}\|g(\<\cdot\>^2)^{(m-\f12)s+2}\|_{L^1}\|f(\<\cdot\>^2)^{(j-m+\f12)s+\f\ga4-1}\|_{L^2}\|h\<\cdot\>^{2js+\f\ga2}\|_{L^2}\\
&+&\sum_{m=[\frac{j-\cN_0}{2}]}^{j-\cN_0+1}(C_{j-1}^{j-m-1})^s(j-m)^{-2-\f52s+2\cN_0s}j^{2s+3-2\cN_0s}\|g(\<\cdot\>^2)^{sm+\f\ga4-4}\|_{L^2}\|f(\<\cdot\>^2)^{s(j-m)+6}\|_{L^1}\|h\<\cdot\>^{2js+\f\ga2}\|_{L^2}.
\eeno

For $\mathcal{H}^{(2)}_{j,2}$, since  $\<v\><\<v_*\>$, we get that   $(E(\th)+t\tilde{h}\sin\th(\mj\cdot\omega))^{\cN_{0}\ka-2}\ls(\max\{\<v\>^2,\<v_*\>^2\}\cos^2(\th/2))^{\cN_{0}\ka-2}=(\<v_*\>^2\cos^2(\th/2))^{\cN_{0}\ka-2}$ and $\tilde{h}\ls |v-v_*||v|$. Using these, one may derive that the coefficient of $\mathcal{H}^{(2)}_{j,2}$ is exactly as same as that of $\mathcal{H}^{(2)}_{j,1}$. Thus  by taking care of the exponents of $\<v\>$ and $\<v_*\>$, we get that
\beno
&&\mathcal{H}^{(2)}_{j,2}\ls\sum_{m=1}^{[\frac{j-\cN_0+1}{2}]}(C^{j-m-1}_{j-1})^s m^{-(3/2)s}j^{1+s}\|g(\<\cdot\>^2)^{(m+\cN_0)s+2}\|_{L^1}\|f(\<\cdot\>^2)^{(j-m-\cN_0)s+\ga/4-1}\|_{L^2}\\
&&\times\|h\<\cdot\>^{2js+\ga/2}\|_{L^2}+\sum_{m=[\frac{j-\cN_0}{2}]}^{j-\cN_0}(C_{j-1}^{j-m-1})^s(j-m)^{-(5/2)s-2+2\cN_0s}j^{2s+3-2\cN_0s}\|g(\<\cdot\>^2)^{(m+\cN_0)s+\ga/4-1}\|_{L^2}\\
&&\times\|f(\<\cdot\>^2)^{(j-m-\cN_0)s+1}\|_{L^1}\|h\<\cdot\>^{2js+\ga/2}\|_{L^2}+\sum_{m=[\frac{j-\cN_0+1}{2}]}^{j-\cN_0+1}(C_{j-1}^{j-m-1})^s(j-m)^{-2s-2+2\cN_0s}j^{(3/2)s+3-2\cN_0s}\\
&&\times\|g(\<\cdot\>^2)^{(m+\cN_0-1/2)s+\ga/4-1}\|_{L^2}\|f(\<\cdot\>^2)^{(j-m-\cN_0+1/2)s+1}\|_{L^1}\|h\<\cdot\>^{2js+\ga/2}\|_{L^2}.
\eeno
We complete the proof of this lemma by patching together all the estimates.
\end{proof}

A similar result holds for the case $0<s<1/2$. We have
\begin{lem}
Suppose $\ga\in(-3,1],\ga+2s>-1$, $0<s<1/2$ and $\cN_1:=[s^{-1}]+1$. If
\ben\label{Hgfh2}
\mathcal{H}_j(g,f,h)&:=&j\int_{\R^6\times\S^2}\int_0^1b(\cos\th)\sin\th|v-v_*|^{\ga}(1-t)(E(\th)+t\tilde{h}\sin\th(\mj\cdot\hat{\omega}))^{js-1}\tilde{h}(\mj\cdot\hat{\om})\\
\notag&&\times|g(v_*)|f(v)|h(v')\<v'\>^{2js}|dvdv_*d\si dt.
\een
Then we have
\beno
&&\mathcal{H}_j(g,f,h)\ls j^{1/2+s}\|g\<\cdot\>^{2}\|_{L^1}\|f(\<\cdot\>^2)^{js+\ga/4-1/2}\|_{L^2}\|h|\<\cdot\>^{2js+\ga/2}\|_{L^2}+\sum_{m=1}^{[\frac{j-\cN_1}{2}]}(C_{j-1}^{j-m-1})^sm^{1-s-\cN_1s}j^{s+1/2}\\
&&\times\|g(\<\cdot\>^2)^{ms+2}\|_{L^1}\|f(\<\cdot\>^2)^{(j-m)s+\ga/4-1/2}\|_{L^2}\|h\<\cdot\>^{2js+\ga/2}\|_{L^2}+\sum_{m=1}^{[\frac{j-\cN_1+1}{2}]}(C_{j-1}^{j-m-1})^sm^{-(1/2)s}j^{(1/2)s+1/2}\\
&&\times\|g(\<\cdot\>^2)^{(m-1/2)s+2}\|_{L^1}\|f(\<\cdot\>^2)^{(j-m+1/2)s+\ga/4-1/2}\|_{L^2}\|h\<\cdot\>^{2js+\ga/2}\|_{L^2}+\sum_{m=[\frac{j-\cN_1}{2}]}^{j-\cN_1}(C_{j-1}^{j-m-1})^s(j-m)^{-2s+\cN_1s}\\&&\times j^{2s+1/2-2\cN_1s}\|g(\<\cdot\>^2)^{sm+\ga/4-4}\|_{L^2}|f(\<\cdot\>^2)^{s(j-m)+6}\|_{L^1}\|h\<\cdot\>^{2js+\f\ga2}\|_{L^2}+\sum_{m=1}^{[\frac{j-\cN_1+1}{2}]}(C^{j-m-1}_{j-1})^s m^{1-s-\cN_1s}j^{s+\f12}\eeno\beno
&&\times
\|g(\<\cdot\>^2)^{(m+\cN_1)s+2}\|_{L^1} \|f(\<\cdot\>^2)^{(j-m-\cN_1)s+\ga/4-1/2}\|_{L^2}\|h\<\cdot\>^{2js+\ga/2}\|_{L^2}+\sum_{m=[\frac{j-\cN_1}{2}]}^{j-\cN_1}(C_{j-1}^{j-m-1})^s(j-m)^{-2s+\cN_1s}\\
&&\times j^{2s+3/2-2\cN_1s}\|g(\<\cdot\>^2)^{(m+\cN_1)s+\ga/4-1/2}\|_{L^2}\|f(\<\cdot\>^2)^{(j-m-\cN_1)s+1/2}\|_{L^1}\|h\<\cdot\>^{2js+\ga/2}\|_{L^2}+\sum_{m=[\frac{j-\cN_1+1}{2}]}^{j-\cN_1+1}(C_{j-1}^{j-m-1})^s\\
&&\times(j-m)^{-\f32s+\cN_1s}j^{\f32s+\f32-2\cN_1s}\|g(\<\cdot\>^2)^{(m+\cN_1-\f12)s+\f\ga4-\f12}\|_{L^2}\|f(\<\cdot\>^2)^{(j-m-\cN_1+\f12)s+\f12}\|_{L^1}\|h\<\cdot\>^{2js+\f\ga2}\|_{L^2}.
\eeno
\end{lem}

\subsection{Corollaries on linearized Boltzmann operator $L$}
With the bounds  and coercivity estimates of collision operator $Q$ in hand, we obtain some corollaries on linearized Boltzmann operator.
\begin{col}\label{L}
Recall that
$L=-Q(\mu, f)-Q(f, \mu)$.
 Let $k\geq22$ and $f \in H^s_{k+\gamma/2}$ with $s \in(0, 1),\ga\in(-3,1]$ and $\gamma+2s >-1 $. There exist constants $c_0, C_k>0$ such that
\beno
( L f, f \langle v \rangle^{2k} )_{L^2_v}\ge c_0 \Vert f \Vert^2_{H^s_{k+\gamma/2}}- C_{k}  \Vert f \Vert_{L^2_v}^2.
\eeno
For inhomogeneous case since $(-v \cdot \nabla_x f, f \langle v \rangle^{2k})_{L^2_{x,v}}=0$, we have
\beno
(Lf,f\<v\>^{2k})_{L^2_{x,v}}\ge c_0\|f\|^2_{L^2_xH^s_{k+\ga/2}}-C_k\|f\|^2_{L^2_{x,v}}.
\eeno
\end{col}
\begin{proof}
The desired results are easily derived from  Lemma \ref{L21}, Theorem \ref{T24} and the fact that for any $\th\in[0,\pi/2],k\geq22$,
$(1-\cos^{2k-3-\ga}\f\th2)-8\sin^{k-\f32-\f\ga2}\f\th2\geq \sin^2{\f\th2}(1-\f{8}{2^{\f{2k-7-\ga}{4}}})>\f12\sin^2{\f\th2}$.
\end{proof}
\smallskip

\begin{col}\label{c2.2}
Let $\ga\in(-3,0]$, $k\geq k_0\geq22$ and $-L=Q(\mu, f)+Q(f, \mu)=A+B$ with $A=M\chi_R$ and $B=-L-M\chi_R$ for some $M, R>0$ large, where $\chi_R$ is the truncation function in a ball with center zero and radius $R>0$. $\mathcal{S}_L$ and $\mathcal{S}_B$ are semi-groups generated by $-L$ and $B$ respectively. Then we have
\beno
\|\mathcal{S}_B(t)\|_{H^2_xL^2_k\rightarrow H^2_xL^2_{k_0}}\leq \vartheta(t;k_0,k):=\left\{
\begin{aligned}
 C\<t\>^{-\frac{k-k_*}{|\gamma|}}, &  &\mbox{for~~ any}~k_*\in(k_0,k) ~\mbox{with}~ \ga\in(-3,0),  \\
Ce^{-Ct}, &  & k=k_0~\mbox{with}~\ga=0,
\end{aligned}
\right.
\eeno
and $\|A\|_{H^2_xL^2_{\mu^{-(1/2+\epsilon)}}\rightarrow H^2_xL^2_k}\leq C$.
Here $\|f\|_{L^2_{\mu^{-(\f12+\epsilon)}}}:=\left(\int_{\R^3}(|f|\mu^{-(1/2+\epsilon)})^2dv\right)^{1/2}$ with $\epsilon>0$ sufficiently small. Moreover, it holds that
$\|\mathcal{S}_L(t)\Pi^\perp\|_{H^2_xL^2_k\rightarrow H^2_xL^2_{k_0}}\leq \vartheta(t;k_0,k),$
where $\Pi^\perp$ is a projector onto the orthogonal of $\mathrm{Ker}(L)$.
\end{col}
\begin{proof}
 Thanks to Corollary \ref{L}, we have
\beno
(Bf,f)_{L^2_k}=-(Lf,f\<v\>^{2k})_{L^2_v}-(M\chi_{R}f,f\<v\>^{2k})_{L^2_v}\leq -c_0\|f\|^2_{H^s_{k+\gamma/2}}+C_k\|f\|_{L^2}^2-M\int_{|v|\leq R}|f|^2dv.
\eeno
Noticing that
$\int_{\R^3}|f|^2dv\leq\int_{|v|\leq R}|f|^2dv+\frac{1}{R^{2}}\int_{\R^3}|f|^2\<v\>^{2}dv\leq\int_{|v|\leq R}|f|^2dv+\frac{1}{R^{2}}\|f\|^2_{H^s_{2}}$, we get that there exists  suitably large $M$ and $R$  which depend on $k$ such that
$(Bf,f)_{L^2_k}\lesssim -\|f\|^2_{H^s_{k+\gamma/2}}.$
 This implies the desired results for the case $\ga=0$.

  When $\ga\in(-3,0)$,
on one hand, we have $\|\mathcal{S}_B(t)f\|_{L^2_{k}}\leq \|f\|_{L^2_k}.$ On the other hand, we have
\beno
\frac{d}{dt}\|\mathcal{S}_B(t)f\|^2_{L^{2}_{k_0}}\leq-c\|\mathcal{S}_B(t)f\|^2_{L^2_{k_0+\gamma/2}}\leq-c\<R\>^{\gamma}\|\mathcal{S}_B(t)f\|^2_{L^2_{k_0}}+C\<R\>^{2(k-k_0)+\gamma}\|\mathcal{S}_B(t)f\|^2_{L^2_{k}},
\eeno
where we use the following interpolation
$\<R\>^{\gamma}\|f\|^2_{L^2_{k_0}}\leq \|f\|^2_{L^2_{k_0+\gamma/2}}+\<R\>^{2(k-k_0)+\gamma}\|f\|^2_{L^2_{k}}.$
Integrating the differential inequality, we obtain that
\beno
\|\mathcal{S}_B(t)f\|^2_{L^2_{k_0}}\lesssim e^{-c\<R\>^\gamma t}\|f\|^2_{L^2_{k_0}}+\<R\>^{2(k-k_0)}\|f\|^2_{L^2_k}\leq\inf_{R>0}(e^{-c\<R\>^\gamma t}+\<R\>^{2(k-k_0)})\|f\|^2_{L^2_k},
\eeno
which implies the desired results by choosing $\<R\>=(\<t\>[\log(1+t)]^{-2(k-k_0)/c})^{-1/\ga}$.

To get the bound of operator $A$, by Duhamel's formula,
$\mathcal{S}_L\Pi^\perp=\Pi^\perp \mathcal{S}_B+\mathcal{S}_L\Pi^\perp\ast A\mathcal{S}_B,$
we have
\beno
&&\vartheta(t)^{-1}\|\mathcal{S}_L\Pi^\perp\|_{H^2_x L^2_k\rightarrow H^2_x L^2_{k_0}}\lesssim  \vartheta(t)^{-1}\|\mathcal{S}_B\Pi^\perp\|_{H^2_x L^2_k\rightarrow H^2_x L^2_{k_0}}\\&&\qquad
 +(\vartheta^{-1}\|\Pi^\perp \mathcal{S}_L\|_{H^2_x L^2(\mu^{-(1/2+\epsilon)})\rightarrow H^2_x L^2_{k_0}}\ast\vartheta^{-1}\|A\mathcal{S}_B\|_{H^2_x L^2_{k}\rightarrow H^2_x L^2(\mu^{-(1/2+\epsilon})}).
\eeno
Since   $t\mapsto \vartheta^{-1}(t)\|\mathcal{S}_B(t)\Pi^\perp\|_{H^2_x L^2_k\rightarrow H^2_x L^2_{k_0}}\in L^\infty(\R_+),$ $t\mapsto \vartheta^{-1}(t)\|\Pi^\perp \mathcal{S}_L(t)\|_{H^2_x L^2(\mu^{-(1/2+\epsilon)})\rightarrow H^2_x L^2({\mu^{-1/2}})\rightarrow H^2_x L^2_{k_0}}\in L^1(\R_+)$ and $t\mapsto \vartheta^{-1}(t)\|A\mathcal{S}_B(t)\|_{H^2_x L^2_k\rightarrow H^2_x L^2_{k_0}\rightarrow H^2_x L^2(\mu^{-(1/2+\epsilon)})}\in L^\infty(\R_+)$,
where the second result comes from Lemma \ref{SL},  we conclude the desired results.
\end{proof}

\setcounter{equation}{0}
\section{Global well-posedness, propagation of moments and sharp convergence rate} In this section, we shall give the proof to Theorem \ref{globaldecay}. We divide it into several steps.   Since the local well-posedness and the non-negativity of the original solution to \eqref{1} has been well-established in \cite{HJZ,HST}, we only   provide  {\it a priori  estimates} for the equations.

\subsection{Preliminaries} We start with some key lemmas.

\begin{lem}\label{T26}
Suppose $k\geq22$ and $F(x,v)= \mu +f(x,v)$ satisfies
\beno
F \ge 0,\quad  \Vert F(x,\cdot) \Vert_{L^1} \ge 1/2, \quad \Vert F(x,\cdot) \Vert_{L^1_2} +\Vert F(x,\cdot) \Vert_{L \log L} \le 4,\quad \forall x\in\T^3
\eeno
Then there exist constants $c_0, C_k>0$ such that
\beno
( Q(\mu+f, \mu+f) , f )_{X_k}
&\le& -c_0 \Vert f \Vert_{Y_k}^2-k^s(1-\|f\|_{H^2_xL^2_{14}}-\|f\|^4_{H^2_xL^2_{14}})\|f\|^2_{X_{k+\ga/2}} \\
 &&+ C_k\|f\|_{H^2_xL^2_{14}} \Vert f \Vert_{Y_{k+s-1}}\|f\|_{Y_k}+ C_{k}  \Vert f \Vert_{H^2_xL^2_v}^2.
\eeno
\end{lem}
\begin{proof}
By definition, we have
\beno
( Q(\mu+f, \mu+f) , f )_{X_k}&=&\sum_{|\al|=0,2}(Q(\mu+f,\pa^\al_x f),\pa^\al_x f\<v\>^{2k-8|\al|})+( Q(\pa^\al_x f,\mu),\pa^\al_x f\<v\>^{2k-8|\al|})
\\
&&+\sum_{|\al_1|\neq0}C_{\al_1,\al_2}(Q(\pa^{\al_1}_xf,\pa^{\alpha -\alpha_1}_x f),\pa^\al_xf\<v\>^{2k-8|\al|}):=\mathcal{Q}_1+\mathcal{Q}_2.
\eeno
Thanks to Lemma \ref{L21}, Theorem \ref{T24} and $k\geq22$, we have
\beno
\mathcal{Q}_1&\leq& -c_0\|f\|^2_{Y_k}-(k^s-\|f\|_{H^2_xL^2_{14}}-\|f\|^4_{H^2_xL^2_{14}})\|f\|^2_{X_{k+\ga/2}}+C_k\|f\|^2_{H^2_xL^2_v}+C_k\sum_{|\al|=0,2}\int_{\T^3}\|\pa^\al_xf\|_{L^2_{14}}\|\<v\>^{k-4|\al|}f\|_{H^s_{-1+\f\ga 2}}\\
&&\times\|\<v\>^{k-4|\al|}\pa^\al_xf\|_{H^s_{\ga/2}}dx+C_k\sum_{|\al|=0,2}\int_{\T^3}\|f\|_{L^2_{14}}\|\<v\>^{k-4|\al|}\pa^\al_xf\|_{H^s_{-1+\ga/2}}\|\<v\>^{k-4|\al|}\pa^\al_xf\|_{H^s_{\ga/2}}dx.
\eeno
For $|\al|=0$, we have $\sup\limits_{x\in\T^3}\|f\|_{L^2_{14}}\leq \|f\|_{H^2_xL^2_{14}}$. While for $|\al|=2$, we have $\sup\limits_{x\in\T^3}\|\<v\>^{k-4|\al|}f\|_{H^s_{\ga/2}}\leq \|f\|_{Y_k}$, which implies that
$\mathcal{Q}_1\leq -c_0\|f\|^2_{Y_k}-(k^s-\|f\|_{H^2_xL^2_{14}}-\|f\|^4_{H^2_xL^2_{14}})\|f\|^2_{X_{k+\ga/2}}+C_k\|f\|_{H^2_xL^2_{14}}\|f\|_{Y_{k-1}}\|f\|_{Y_k}+C_k\|f\|^2_{H^2_xL^2_v}.$
On the other hand, due to Lemma \ref{Qspatial}, we have $\mathcal{Q}_2\leq C_k\|f\|_{H^2_xL^2_{14}}\|f\|_{Y_{k+s-1}}\|f\|_{Y_{k}}+k^s\|f\|_{H^2_xL^2_{14}}\|f\|^2_{X_{k+\ga/2}}.$ Then we get the desired results.
\end{proof}

\begin{lem}\label{L31} It holds that
$(Q(f, g), h)_{X_k} \leq C_k(  \|f\|_{H^2_xL^2_{14}} \Vert g \Vert_{\bar{Y}_k} \Vert h \Vert_{Y_k} + \|g\|_{H^2_xL^2_{14}}\Vert f \Vert_{Y_k}  \Vert h \Vert_{Y_k}).$
In particular,
\beno
\Vert Q(f, g) \Vert_{Z_k} \leq C_k( \|f\|_{H^2_xL^2_{14}} \Vert g \Vert_{\bar{Y}_k} + \|g\|_{H^2_xL^2_{14}}\Vert f \Vert_{Y_k} ),
\eeno
where $Z_k$ is the dual of $Y_k$ with respect to $X_k$ defined in \eqref{X_k} and \eqref{Y_k}.
\end{lem}
\begin{proof}
By definition of $X_k$ (see (\ref{X_k})), we have that
\beno
(Q(f,g),h)_{X_k}&=&\sum_{|\al|=0,2}(Q(f,\pa^\al_xg),\pa^\al_xh\<v\>^{2k-8|\al|})+\sum_{|\al_1|\neq0}(Q(\pa^{\al_1}_xf,\pa^{\al_2}_xg),\pa^\al_xh\<v\>^{2k-8|\al|}).
\eeno
Thanks to Lemma \ref{Qspatial}, the second term on righthand side can be bounded by
\beno
C_k(\|f\|_{H^2_xL^2_{14}}\|g\|_{Y_k}\|h\|_{Y_k}+\|g\|_{H^2_xL^2_{14}}\|f\|_{Y_k}\|h\|_{Y_k}).
\eeno
We only need to handle the first term. Observe that
\beno
&&\sum_{|\al|=0,2}(Q(f,\pa^\al_xg),\pa^\al_xh\<v\>^{2k-8|\al|})_{L^2_{x,v}}=\sum_{|\al|=0,2}(\<v\>^{k-4|\al|}Q(f,\<v\>^{k-4|\al|}\pa^\al_xg)\\
&&\qquad-Q(f,\<v\>^{k-4|\al|}\pa^\al_xg),\<v\>^{k-4|\al|}\pa^\al_xh)_{L^2_{x,v}}+\sum_{|\al|=0,2}(Q(f,\<v\>^{k-4|\al|}\pa^\al_xg),\<v\>^{k-4|\al|}\pa^\al_xg)_{L^2_{x,v}}:=\mathcal{S}_1+\mathcal{S}_2.
\eeno
By Lemma \ref{L12}, Lemma \ref{L23} and the proof of Lemma \ref{Qspatial}, we have
$|\mathcal{S}_1|\leq C_k(\|f\|_{Y_k}\|g\|_{H^2_xL^2_{14}}\|h\|_{Y_k}+\|g\|_{Y_k}\|f\|_{H^2_xL^2_{14}}\|h\|_{Y_k})$ and
 $|\mathcal{S}_2|\leq \|f\|_{H^2_xL^2_{14}}\|g\|_{\bar{Y}_k}\|h\|_{Y_k}.$ We complete the proof of this lemma by combining the above estimates.
\end{proof}

 To prove the global results, we introduce conjugate operator of $L$ in $L^2$. For any constant $l \ge 0$ we define
\beno
L_l f := \<v\>^l L(\<v\>^{-l} f  ) = -\<v\>^l( Q( {\<v\>^{-l}}{f}, \mu ) -Q (\mu,  {\<v\>^{-l}}{f})),
\eeno
and its dual $L^*_l$
\beno
(L_l f, g)_{L^2_v} = (f, L^{*}_l g)_{L^2_v}.
\eeno
We have that
\beno
(L^*_l f, g )_{L^2_k} &=& (L^*_l f,  \langle v \rangle^{2k} g )_{L^2_v} =  (L_l( \langle v \rangle^{2k} g), f ) =(L({\<v\>^{-l}} \langle v \rangle^{2k} g) , \<v\>^l f)_{L^2_v} \\
&=&-(Q (\mu, g \langle v \rangle^{2k-l}), \langle v \rangle^{l} f))_{L^2_v}  - (Q( g \langle v \rangle^{2k-l}  , \mu), \langle v \rangle^{l} f)_{L^2_v}.
\eeno
Similar to Corollary \ref{L} and \ref{c2.2}, we have

\begin{col}\label{c4.1}
 Let $\ga\in(-3,0]$, $k, k_0 \ge 0$ such that $l-k_0\geq l-k\geq 22$ and $f \in H^s_{k+\gamma/2}$. $-L^*_l=A_l+B^*_l$ with $A_l=M_l\chi_{R_l}$ and $B^*_l=-L^*_l-M_l\chi_{R_l}$ for some $M_l, R_l>0$ large, $\mathcal{S}_{L^*_l}$ and $\mathcal{S}_{B^*_l}$ are semigroups generated by $-L^*_l$ and $B^*_l$ respectively. There exist constants $c_0, C_
{k, l} >0$ such that
\beno
( L^*_l f, f \langle v \rangle^{2k} )_{L^2_xL^2_v}\ge c_0 \Vert f \Vert_{L^2_xH^s_{k+\gamma/2}}^2- C_{k, l}  \Vert f \Vert_{L^2_{x,v}}^2.
\eeno
Moreover, we have
$\|\mathcal{S}_{B^*_l}(t)\|_{H^2_x L^2_k\rightarrow H^2_xL^2_{k_0}}\leq \vartheta(t;k_0,k)$
with $\vartheta(t;k_0,k)$ defined in $\mathrm{Cor.}$ \ref{c2.2}.
\end{col}
\begin{proof}
Observe that
$(Q(\mu,f\<v\>^{2k-l}),\<v\>^lf)_{L^2_v}=(Q(\mu,g),\<v\>^{2(l-k)}g)_{L^2_v}$ and  $(Q(f\<v\>^{2k-l},\mu),\<v\>^lf)_{L^2_v}=(Q(g,\mu),\<v\>^{2(l-k)}g)_{L^2_v}$
with $g:=\<v\>^{2k-l}f$. We may easily get the desired results by Lemma \ref{L21}, Theorem \ref{T24} and the same argument used in $\mathrm{Cor.}$ \ref{c2.2}.
\end{proof}
% By Lemma \ref{L31}, for the bilinear kernel $Q(f, g)$ we have
%\beno
%\Vert Q(f, g) \Vert_{Z_k} \lesssim \Vert f \Vert_{X_k} \Vert g \Vert_{\bar{Y}_k} + \Vert f \Vert_{Y_k} \Vert g \Vert_{X_k},
%\eeno
%where $Z_k$ is the dual of $Y_k$ with respect to $X_k$.

The following lemma is the commutator between $L^*$ and $\<D\>^s$:
\begin{lem}\label{com}
For any $k_0\geq0$, we have
\beno
 |\mathcal{A}|:=|([-L_{k_0}^* , \langle D \rangle^s] f, \langle D \rangle^s f )_{L^2_v}| \le \epsilon \Vert f \Vert_{H^{2s}_{\gamma/2}}^2  + C_{\epsilon,k_0} \Vert f \Vert^2_{H^s_{10+\gamma/2}},
\eeno
 where $\epsilon$ can be sufficiently small. It  also holds true if we replace $L^*_{k_0}$ by $B^*_{k_0}$.
\end{lem}
\begin{proof}   We assume $2s\ge1$ since the cases $2s<1$ can be derived similarly. By definition,  one has
\beno
 \mathcal{A}=
\big(-\<v\>^{k_0}\<D\>^sf,L(\<v\>^{-k_0}\<D\>^sf))+(\<v\>^{k_0}f,L(\<v\>^{-k_0}\<D\>^{2s}f)\big):=\mathcal{A}_1+\mathcal{A}_2,
\eeno where  $\mathcal{A}_1=(Q(\<v\>^{-k_0}\<D\>^sf,\mu),\<v\>^{k_0}\<D\>^sf) -(Q(\<v\>^{-k_0}\<D\>^{2s}f,\mu),\<v\>^{k_0}f)$ and
$\mathcal{A}_2=(Q(\mu,\<v\>^{-k_0}\<D\>^sf),\\ \<v\>^{k_0}\<D\>^sf)-(Q(\mu,\<v\>^{-k_0}\<D\>^{2s}f),\<v\>^{k_0}f)$.

\underline{\it Estimate of $\mathcal{A}_1$.} By Lemma \ref{L12}, we deduce that
\beno&&
|\mathcal{A}_{1}|\ls\|\<\cdot\>^{-k_0}\<D\>^sf\|_{L^2_{k_0+\gamma/2+2s+2}}\|\mu\|_{H^{2s}_{k_0+2s+\gamma/2}}\|\<\cdot\>^{k_0}\<D\>^sf\|_{L^2_{-k_0+\gamma/2}} +\|\<\cdot\>^{-k_0}\<D\>^{2s}f\|_{L^2_{k_0+\gamma/2}}\|\mu\|_{H^{2s}_{k_0-2+\gamma/2}}\\&&\times\|\<\cdot\>^{k_0} f\|_{L^2_{-k_0+\gamma/2+2s+2}}\ls C_{k_0}\|f\|_{H^s_{\gamma/2+2s+2}}\|f\|_{H^s_{\gamma/2}}+C_{k_0}\|f\|_{H^{2s}_{\gamma/2}}\|f\|_{L^2_{2+2s+\gamma/2}}.
\eeno
By interpolation, we obtain that $|\mathcal{A}_{1}|\le\epsilon\|f\|^2_{H^{2s}_{\gamma/2}}+C_{\epsilon,k_0}\|f\|^2_{H^s_{10+\gamma/2}}.$

\medskip
\underline{\it Estimate of $\mathcal{A}_2$.} We first have that
\beno
|\mathcal{A}_2|
&\le&|(Q(\mu,[\<v\>^{-k_0},\<D\>^s]\<D\>^sf),\<v\>^{k_0}f)|+|(Q(\mu,\<v\>^{-k_0}\<D\>^s)f),[\<v\>^{k_0},\<D\>^s]f)|\\
&&+|(Q(\mu,\<D\>^s\<v\>^{-k_0}\<D\>^sf),\<v\>^{k_0}f)-(\<D\>^sQ(\mu,\<v\>^{-k_0}\<D\>^sf),\<v\>^{k_0}f)|\\
&:=&\mathcal{A}_{21}+\mathcal{A}_{22}+\mathcal{A}_{23}.
\eeno
 For $\mathcal{A}_{21}$, thanks to Lemma \ref{L12} and Lemma \ref{le1.2}, we have
\beno
\mathcal{A}_{21}&\lesssim&\|\mu\|_{L^2_{\gamma/2+2s+2+k_0}}\|[\<v\>^{-k_0},\<D\>^s]\<D\>^sf\|_{L^2_{k_0+2s+\gamma/2}}\|\<v\>^{k_0}f\|_{H^{2s}_{-k_0+\gamma/2}}\\
&\lesssim& C_{k_0}\|f\|_{H^{2s-1}_{2s-1+\gamma/2}}\|f\|_{H^{2s}_{\gamma/2}}\lesssim \epsilon\|f\|^2_{H^{2s}_{\gamma/2}}+C_{\epsilon,k_0}\|f\|^2_{H^s_{10+\gamma/2}}.
\eeno
For $\mathcal{A}_{22}$, similar to the argument of $\mathcal{A}_{21}$, we get that
\beno
\mathcal{A}_{22}&\lesssim& \|\mu\|_{L^2_{k_0+2+\gamma/2}}\|\<v\>^{-k_0}\<D\>^sf\|_{H^s_{k_0+\gamma/2}}\|[\<v\>^{k_0},\<D\>^s]f\|_{H^s_{-k_0+2s+\gamma/2}}\\
&\lesssim& C_{k_0}\|f\|_{H^{2s}_{\gamma/2}}\|f\|_{H^{2s-1}_{2s-1+\gamma/2}}\lesssim \epsilon\|f\|_{H^{2s}_{\gamma/2}}+C_{\epsilon,k_0}\|f\|_{H^s_{10+\gamma/2}}.
\eeno
\medskip
 To estimate  $\mathcal{A}_{23}$, we first observe that $\mathcal{A}_{23}$ is a commutation between $\<D\>^s$ and $Q$. From Theorem \ref{le1.24} with functions $\mu,\<v\>^{-k_0}\<D\>^sf$ and $\<D\>^{-s}\<v\>^{k_0}f$,  we have
\beno
\mathcal{A}_{23}&\ls&C_{N}(\|\mu\|_{L^1_{(\ga+2s)^++(-\omega_1)^++(-\omega_2)^++\de}}\|f\|_{H^{2s+a}_{\omega_1-k_0}}\|f\|_{H^{b}_{\omega_2+k_0}}+\|\mu\|_{L^2_{2+(-\omega_3)^++(-\omega_4)^+}}\|f\|_{H^{2s+c_1}_{\omega_3-k_0}}\|f\|_{H^{d_1}_{\omega_4+k_0}}\\
&&+\|f\|_{H^s_{(-\omega_5)^++(-\omega_6)^+-k_0}}\|\mu\|_{H^{s+c_2}_{\omega_5}}\|f\|_{H^{d_2}_{\omega_6+k_0}}+\|f\|_{H^{2s}_{(-\om_7)^++(-\om_8)^+-k_0}}\|\mu\|_{H^{c_3}_{\om_7}}\|f\|_{H^{d_3}_{\om_8+k_0}}\\
&&+(\|\mu\|_{L^1_{(\ga+2s)^++(-\omega_1)^++(-\omega_2)^++\de}}+\|\mu\|_{L^2_{2+(-\omega_3)^++(-\omega_4)^+}})\|f\|^2_{H_{-N}^{-N}}\\
&&+(\|f\|_{H^s_{(-\omega_5)^++(-\omega_6)^+-k_0}}+\|f\|_{H^{2s}_{(-\omega_7)^++(-\omega_8)^+-k_0}})\|\mu\|_{H_{-N}^{-N}}\|f\|_{H_{-N}^{-N}}).
\eeno
By choosing $N$ large enough, $\om_1=\om_3=k_0+\f{\ga}2,a=c_1=0$, $\om_6=-k_0-\f{\ga}2,d_2=0$ and $\om_8=-k_0-\f{\ga}2,d_3=0$, we get
\beno
\mathcal{A}_{23}&\lesssim& C_{k_0}(\|f\|_{H^{2s}_{\gamma/2}}\|f\|_{H^{2s-1}_{2s+\gamma/2-1}}+\|f\|_{H^{2s}_{\gamma/2}}\|f\|_{H^{2s-1/2}_{2s+\gamma/2-1}}+\|f\|_{H^s_{\ga/2}}\|f\|_{L^2_{-\ga/2}}+\|f\|_{H^{2s}_{\ga/2}}\|f\|_{L^2_{-\ga/2}})\\
&\lesssim& \epsilon\|f\|^2_{H^{2s}_{\gamma/2}}+C_{\epsilon,k_0}\|f\|^2_{H^s_{10+\gamma/2}},
\eeno
where we use interpolation inequalities.
The lemma follows by combining all above estimates.
\end{proof}
\smallskip

\begin{lem}\label{le4.2}
Let $\ga\in(-3,0]$, $\tilde{k}_0=13+22$, $\tilde{X}_0=H^2_xL^2_{\tilde{k}_0},\tilde{Y}_0=H^2_xH^s_{\tilde{k}_0+\ga/2},$ and $\tilde{Z}_0=H^2_xH^{-s}_{\tilde{k}_0-\ga/2}$ be the dual space of $\tilde{Y}_0$ with respect to $\tilde{X}_0$ . Then we have
$\Vert \mathcal{S}_B(t) f\Vert_{H^2_xL^2_v} \le\vartheta_1(t) \Vert f \Vert_{\tilde{Z}_0}$,
where $\vartheta_1(t)$ satisfies
\beno
\int_0^\infty \vartheta_1(t) \vartheta(t) dt < +\infty
\eeno
with $\vartheta(t)=\vartheta(t;10,13)$.  Moreover, by Duhamel's formula, it holds that
 $\Vert \mathcal{S}_L(t)\Pi^{\perp} f\Vert_{H^2_xL^2_v} \le\vartheta_1(t) \Vert f \Vert_{\tilde{Z}_0}$.

\end{lem}
\begin{proof}
We only prove the lemma for homogeneous case by duality. Let $\tilde{X}_0=L^2_{\tilde{k}_0},X_0=L^2_{13}, X_1=L^2_{10}$ and $Y_1=H^s_{10+\ga/2}$. Suppose that $f_* $ is a solution of equation
$\frac d  {d t} f_* = B^*_{\tilde{k}_0} f_*$.
Set $H(f_*, f_*) = (f_*, f_*)_{X_1} +  at(f_*, f_*)_{H^s}$
with some $a>0$ to be fixed. Then we have
\beno
\frac{d}{dt}H(f_*,f_*)&=&\frac{d}{dt}((f_*, f_*)_{X_1} +  at(f_*, f_*)_{H^s})\\
&=&( f_* ,B_{\tilde{k}_0}^*f_*)_{X_1}+at(B^*_{\tilde{k}_0}(\<D\>^sf_*),\<D\>^sf_*)+at ( \langle D \rangle^s f_* , [B_{\tilde{k}_0}^* , \langle D \rangle^s] f_*))+a(f_*, f_*)_{H^s}.
\eeno

Observing that $\tilde{k}_0=13+22> 10+22$, thanks to Corollary \ref{c4.1}, we have
$( f_* ,B_{\tilde{k}_0}^*f_*)_{X_1} \le -C_1\Vert f_* \Vert_{Y_1}^2+C_2\Vert f_* \Vert_{L^2}^2$.
Moreover,
 $(B_{k_0}^* ( \langle D \rangle^s f),\langle D \rangle^s f_* )_{L^2}\le -C_3 \Vert\langle D \rangle^s f \Vert_{H^s_{\gamma/2}}^2 + C_4\Vert \langle D \rangle^s f_*   \Vert_{L^2}^2.$
By Lemma \ref{com}, we deduce that
\beno
\frac{d}{dt}H(f_*,f_*)&\leq&-C_1\|f_*\|^2_{Y_1}+C_2\|f_*\|^2_{L^2}-atC_3\|\<D\>^sf_*\|^2_{H^s_{\ga/2}}+atC_4\|\<D\>^sf_*\|^2_{L^2}\\
&&+at\vep \|\<D\>^sf_*\|^2_{H^s_{\ga/2}}+atC_{\vep,k_0}\|f_*\|^2_{Y_1}+a\|f_*\|^2_{H^s}.
\eeno
Taking  suitable $\vep, a>0$ small, we easily get that
$\frac{d}{dt}H(f_*,f_*)\ls-\|f_*\|^2_{Y_1}+\|f_*\|^2_{L^2}, \quad \forall t\in(0,1),$
which implies that
$H(f_*,f_*)(t)\leq C\|f_0\|^2_{X_1}+C\int_0^t\|\mathcal{S}_{B^*_{\tilde{k}_0}}(\tau)f_0\|_{L^2}^2d\tau.$
By Corollary \ref{c4.1}, the above can be bounded by
$H(f_*, f_*) \le C\Vert f_0\Vert_{X_1}^2, ~t\in(0,1),$
which implies
$at(f_*, f_*)_{H^s} \le C \Vert f_0 \Vert_{X_1}^2, ~t\in(0,1).$

For large values of time $t\geq1$, let $\tilde{H}(f_*,f_*):=(f_*, f_*)_{X_1} + \vep(f_*, f_*)_{H^s}$ with suitably small $\vep$, then by the similar argument, we can obtain that
$(f_*, f_*)_{H^s} \le C \Vert f_0 \Vert_{X_1}^2, ~t\geq1.$
So combining the fact that
$\Vert \mathcal{S}_{B^*_{\tilde{k}_0}}(t)  f\Vert_{X_1} \lesssim \vartheta(t) \Vert f \Vert_{X_0}$
with $\vartheta(t)=\vartheta(t;10,13)$, we deduce that
\beno
\Vert \mathcal{S}_{B^*_{\tilde{k}_0}}(t)  f\Vert_{H^s} \lesssim \frac{\vartheta(t)}{t^{-1/2}\wedge1} \Vert f \Vert_{X_0}\leq\frac{\vartheta(t)}{t^{-1/2}\wedge1} \Vert f \Vert_{\tilde{X}_0}.
\eeno

Denote $m = \langle v \rangle^{\tilde{k}_0}$, $ g=  m f$, it is easy to see that
 $m B f = m B(m^{-1} g) = B_{\tilde{k}_0} g$. Then
 $ m \mathcal{S}_B(t) f = \mathcal{S}_{B_{\tilde{k}_0}} (t) g, \quad (\mathcal{S}_{B_{\tilde{k}_0}}(t)g, \phi )_{L^2}  = (g, \mathcal{S}_{B^*_{\tilde{k}_0}}(t)\phi )_{L^2}, \quad \forall t \ge 0$. Therefore, we have
\begin{equation*}
\begin{aligned}
\Vert  \mathcal{S}_B(t) f \Vert_{L^2} &= \Vert  m^{-1}\mathcal{S}_{B_{\tilde{k}_0}} (t) g \Vert = \sup_{\Vert \varphi \Vert_{L^2} \le 1} ( \mathcal{S}_{B_{\tilde{k}_0}}(t)g, m^{-1} \varphi )_{L^2}
 = \sup_{\Vert \phi \Vert_{L^2(m)} \le 1} (g,  \mathcal{S}_{B^*_{\tilde{k}_0}}(t)\phi )_{L^2}\\
 &\le \sup_{\Vert \phi \Vert_{L^2(m)} \le 1} \Vert g \Vert_{H^{-s}}  \Vert \mathcal{S}_{B^*_{\tilde{k}_0}}(t)\phi \Vert_{H^s}
\lesssim \frac{\vartheta(t)}{t^{-1/2}\wedge1}  \sup_{\Vert \phi \Vert_{L^2(m)} \le 1} \Vert g \Vert_{H^{-s}}  \Vert \phi \Vert_{L^2(m)}
\\
&=  \frac{\vartheta(t)}{t^{-1/2}\wedge1}  \Vert g \Vert_{H^{-s}}
  =  \frac{\vartheta(t)}{t^{-1/2}\wedge1} \Vert   f \Vert_{H^{-s}_{\tilde{k}_0}}\leq  \frac{\vartheta(t)}{t^{-1/2}\wedge1} \Vert   f \Vert_{\tilde{Z}_0}.
\end{aligned}
\end{equation*}
Take
$\vartheta_1(t) =\frac{\vartheta(t)}{t^{-1/2}\wedge1},$
we conclude the desired results.
\end{proof}
\bigskip
%\begin{rmk}
%In the paper \cite{AMSY} Section 5 actually the authors prove a weaker version, there exist a $T>0$ such that
%\beno
%\int_{0}^T \theta_1(t) \le C,\quad \theta_1(t) \le C\theta(t), \quad \forall t \ge T.
%\eeno
%\end{rmk}

Let $\tilde{k}_0,\tilde{X}_0$ be defined in Lemma \ref{le4.2} and $\bar{Y}_0:=H^2_xH^s_{\tilde{k}_0+\ga/2+2s}$, $k\geq \tilde{k}_0+13$, $X_k,Y_k,\bar{Y}_k$ be defined in (\ref{X_k}) and (\ref{Y_k}),  It is easy to see that $Y_{k}\hookrightarrow X_{k-2}\hookrightarrow H^2_xL^2_{k-10}$, then thanks to Corollary \ref{c2.2}, we have
\ben\label{slt}
&&\Vert \mathcal{S}_L(t)\Pi^{\perp} f\Vert_{\tilde{X}_0} \le\vartheta(t;\tilde{k}_0,k-10)\|f\|_{H^2_xL^2_{k-10}}\leq \vartheta(t;\tilde{k}_0,k-10) \Vert f\Vert_{Y_k},\\
 \notag&&\lim_{t \to \infty}\vartheta(t;\tilde{k}_0,k-10) =\lim_{t \to \infty}\<t\>^{-\frac{k-\tilde{k}_0-10}{|\gamma|}}=0, \quad \vartheta \in L^2(0, \infty).
\een

We now state a key stability theorem:

\begin{thm}\label{th4.1}
Let $f$ be a solution to the following equation
\ben\label{417}
\partial_t f +v\cdot\na_x f= L f +Q(f, f), \quad f|_{t=0} =f_0(x,v), \quad \Pi f_0 = 0
\een
such that $F(t,x,v)= \mu +f(t,x,v)$ satisfies
\ben\label{uni}
F \ge 0,\quad  \Vert F(x,\cdot) \Vert_{L^1} \ge 1/2, \quad \Vert F(x,\cdot) \Vert_{L^1_2} +\Vert F(x,\cdot) \Vert_{L \log L} \le 4,\quad \forall x\in\T^3.
\een
Suppose that $k\geq k_1 := \tilde{k}_0+13=26+22$, $\tilde{X}_0=H^2_xL^2_{\tilde{k}_0}$, and introduce a norm $||| f |||_{X_k}$ on $\Pi X$ by
\beno
||| f |||_{X_k}^2  = \eta \Vert f \Vert_{X_k}^2 + \int_0^\infty \Vert \mathcal{S}_L(\tau ) f \Vert_{H^2_xL^2_v}^2 d \tau,
\eeno
where $\eta>0$.  The associate scalar product can be defined by
\beno
\<\<f, g \>\>_{X_k}  = \eta (f, g )_{X_k} + \int_0^\infty (\mathcal{S}_L(\tau ) f,  \mathcal{S}_L(\tau ) g)_{L^2_vH^2_x} d \tau.
\eeno
Then there exists some $\eta_k>0$ such that $|||\cdot|||_{X_k}$ is equivalent to $\Vert \cdot \Vert_X$ on $\Pi X$. Moreover,
 there exist universal constants $C, K>0$,  not depending on $k$, such that
\beno
\frac d {dt} ||| f |||^2_{X_k} &\le& \eta_k(\| f \|_{\tilde{X}_0} - K ) \Vert f \Vert^2_{Y_k}-\eta_kk^s(C-\|f\|_{\tilde{X}_0}-\|f\|^4_{\tilde{X}_0})\|f\|^2_{X_{k+\ga/2}}+C_k\|f\|_{\tilde{X}_0}\|f\|_{Y_{48}}^2.
\eeno
\end{thm}
\begin{proof}
Thanks to (\ref{slt}),  $|||\cdot|||_{X_k}$ is equivalent to $\Vert \cdot \Vert_X$ on $\Pi X$ since
\ben\label{aa}
\int_0^\infty \Vert S_L(\tau ) f \Vert_{H^2_xL^2_v}^2 d \tau \lesssim \Vert f \Vert_{X_k}^2 \int_0^\infty \vartheta^2(\tau) d\tau\leq C_k\|f\|^2_{X_k}.
\een
 To get the desired result, we first have
\beno
\frac {d} {dt} ||| f(t) |||^2_{X_k}  =\eta (Q(\mu + f, \mu+f) ,f )_{X_k} + \int_0^\infty ( \mathcal{S}_L(\tau)(-L)f, \mathcal{S}_L(\tau) f )_{H^2_xL^2_v} d\tau
 + \int_0^\infty ( \mathcal{S}_L(\tau)Q(f, f ) , \mathcal{S}_L(\tau) f )_{H^2_xL^2_v} d\tau.
\eeno
 By Lemma \ref{T26} and the interpolation inequality $\|f\|_{Y_{k-1+s}}\le \vep\|f\|_{Y_{k}}+C_\vep\|f\|_{Y_{48}}$, we first have
\beno
(Q(\mu + f, \mu+f), f)_{X_k} &\le &-c_0 \Vert f \Vert_{Y_k}^2-k^s(1-\|f\|_{H^2_xL^2_{14}} -\|f\|^4_{H^2_xL^2_{14}} )\|f\|^2_{X_{k+\ga/2}}  
\\
&&+ C_k\| f \|_{H^2_xL^2_{14}} \|f\|_{Y_{k-1+s}}\Vert f \Vert_{Y_k}+ C_{k}  \Vert f \Vert_{H^2_xL^2_v}^2\\
&\leq&-(\f{c_0}2-\|f\|_{H^2_xL^2_{14}})\|Y\|^2_{Y_k}-k^s(1-\|f\|_{H^2_xL^2_{14}}-\|f\|^4_{H^2_xL^2_{14}})\|f\|^2_{X_{k+\ga/2}}
\\
&&+C_k\|f\|^2_{H^2_xL^2_v}+C_k\|f\|_{H^2_xL^2_{14}}\|f\|^2_{Y_{48}}.
\eeno
Recalling that
$\Vert \mathcal{S}_L(\tau) f(t) \Vert_{H^2_xL^2_v} \le \vartheta(\tau+t) \Vert f_0\Vert_{X_k} ,  \quad \lim_{\tau \to \infty}\vartheta(\tau+t) =0, \quad \forall  t \ge 0$, we deduce that
\beno
\int_0^\infty ( \mathcal{S}_L(\tau)(-L)f, \mathcal{S}_L(\tau) f )_{H^2_xL^2_v} d\tau =\int_0^\infty \frac {d} {d\tau} \Vert \mathcal{S}_L(\tau) f(t)\Vert_{H^2_xL^2_v}^2 d\tau =  -\Vert f(t) \Vert_{H^2_xL^2_v}^2.
\eeno
Since
  $\int_0^\infty ( \mathcal{S}_L(\tau)Q(f, f ) , \mathcal{S}_L(\tau) f )_{H^2_xL^2_v} d\tau
  \le\int_0^\infty \Vert \mathcal{S}_L(\tau)Q(f, f )\Vert_{H^2_xL^2_v}   \Vert \mathcal{S}_L(\tau) f \Vert_{H^2_xL^2_v} d\tau$,
by Lemma \ref{le4.2}, Lemma \ref{L31} and (\ref{slt}),  one has
\beno &&|\int_0^\infty ( \mathcal{S}_L(\tau)(-L)f, \mathcal{S}_L(\tau) f )_{H^2_xL^2_v} d\tau| \le
\Vert Q(f, f )\Vert_{\tilde{Z}_0}    \Vert  f \Vert_{Y_k} \int_0^\infty \vartheta_1(t) \vartheta(t) dt
\lesssim \Vert f \Vert_{\tilde{X}_0}\Vert f\Vert_{\bar{Y_0}} \Vert  f \Vert_{Y_k}\\
&&\le  \vep\| f \|_{\tilde{X}_0} \Vert  f \Vert_{Y_k}^2+C_\vep \| f \|_{\tilde{X}_0}\|f\|^2_{Y_{48}},
\eeno
where we use the fact that $\vartheta_1(t) =\frac{\vartheta(t)}{t^{-1/2}\wedge1}$ and $\vartheta(t)\in L^2(0,\infty)$.
We end the proof by taking $\eta_k=1/C_k$ and combining all the estimates.
\end{proof}

\subsection{Proof of  Theorem \ref{globaldecay}: Polynomial moment}
We are now ready to prove the global well-posedness and the decay rate when the initial data have polynomial moment.

\begin{proof}[Proof of Theorem \ref{globaldecay}: Polynomial moment]  Let $k_1=\tilde{k}_0+13=26+22$ and $f$ be a solution of (\ref{417}). By local well-posedness, we may assume that $T^*:=\sup\{t>0| ||| f(t) |||^2_{X_{k_1}}<\vep_0\}$. Thus if $F(t,x,v)=\mu+f(t,x,v)$, then it satisfies (\ref{uni}). Thanks to Theorem \ref{th4.1}, if $k\ge k_1$  we have
 \ben\label{L348}
 \frac d {dt} ||| f |||^2_{X_{k}} \le \eta_{k}(||| f |||_{X_{k_1}} - K) \| f \|^2_{Y_{k}}+\eta_{k}k_1^s(|||f|||^4_{X_{k_1}}-|||f|||_{X_{k_1}}-C) \|f\|^2_{X_{k+\ga/2}}  +C_k\|f\|_{X_{k_1}}\|f\|_{Y_{k_1}}^2.
 \een
 By the standard continuity argument, we deduce that $T^*=\infty$ which implies that $f\in L^\infty([0,\infty); X_{k_1})\cap L^2([0,\infty),Y_{k_1})$ if $\vep_0\ll1$.
 Moreover, $|||f|||_{L^\infty([0,\infty); X_{k_1})}\leq \epsilon$.
Thus for any $k\geq k_1$, if $\|f_0\|_{X_k}\leq \infty$, from \eqref{L348}, we have $f\in L^{\infty}([0,\infty],X_k)\cap L^2([0,\infty),Y_{k})$. Moreover, together with  (\ref{L348}), we can derive that $\|f(t)\|_{X_{k_1}}\leq C_{k_1,k}\<t\>^{-\f{k-k_1}{|\ga|}}\|f_0\|_{X_{k}}$.

This ends the  proof of the first part of the main theorem.
\end{proof}

\subsection{Proof of  Theorem \ref{globaldecay}: Upper bound of convergent rate}
 In this subsection, we shall prove the upper bound of the convergence rate for the longtime behavior of solutions in $L^1$ space. We prove it by construction of the sequence of the solutions to verify \eqref{upperbound}. We restrict solutions to the homogeneous case. The main idea is to use the localized techniques to focus on the propagation of the $L^1$ moment when initially it concentrates the region which is far away from the original point. Recall that $\cP_j(v)\<v\>^k=\vphi(2^{-j}v)\<v\>^k$ with $j\gg1$, then
we have
\beno
\frac{d}{dt}\int_{\R^3}\cP_j(v)F(t,v)\<v\>^kdv=\int_{\R^3}\cP_j(v)Q(F,F)\<v\>^kdv.
\eeno

We begin with a localization lemma:
\begin{lem}\label{le6.1}
	Let $\chi(v)$ be a smooth function satisfying $\mathrm{supp}~\chi\subset\{v|1/2<|v|<3\}$ and $\chi(v)=1$ if $v\in \mathrm{supp}~\vphi$, where $\vphi$ is defined in (\ref{7.1}), $\chi_{j}(v)=\chi(2^{-j}v)$. Suppose that $F(t,v)\geq0$ is a homogeneous solution to Eq.\eqref{1} with $\|F_0-\mu\|_{L^2_{48}}=\|f_0\|_{L^2_{48}}=\left(\int_{\R^3}f_0^2(v)\<v\>^{96}dv\right)^{1/2}<\vep_0$.% and $\|f_0\|_{L^2_{k+2}}<\infty$ for some $k\gg1$.
	Then we have
	\beno
	\frac{d}{dt}\|\cP_jF\|_{L^1_k}\geq -C_k2^{j\ga}(1+\|F\|_{H^s_{3}})\|F\|_{L^1_{k}}.
	\eeno
\end{lem}
\begin{proof}
The global existence of solution $f=F-\mu\in L^{\infty}([0,\infty],L^2_{48})\cap L^2([0,\infty),H^s_{48+\ga/2})$ is guaranteed by the proof of polynomial moment of Theorem \ref{globaldecay}. We only provide the detailed proof for $2s\geq1$.  It is easy to check that
	\beno
	\frac{d}{dt}\|\cP_jF\|_{L^1_k}&=&\int_{|v|\sim 2^j}Q(F,\cP_jF)\<v\>^kdv+\int_{|v|\sim 2^j}[\cP_jQ(F,F)-Q(F,\cP_jF)]\<v\>^kdv
	:= I+II.
	\eeno

	\noindent\underline{\it Step 1: Estimate of $II$.} One has
	\beno
	II&=&\int_{\R^3}[\cP_jQ(F,F)-Q(F,\cP_jF)]\<v\>^k\chi_{j}(v)dv\\
	&=&\int_{\R^6\times\S^2}b(\cos\th)|v-v_*|^\ga F(v_*)F(v)\chi_{j}(v')\<v'\>^k(\vphi(2^{-j}v')-\vphi(2^{-j}v))dvdv_*d\sigma.
	\eeno
	Since   $|v'|\sim 2^j$, we split the integration domain into three parts: $|v|\sim 2^j,|v|\gg2^j$ and $|v|\ll2^j$. Correspondingly  $II$ can be decomposed into three parts: $II_1,II_2$ and $II_3$ respectively.
	\smallskip
	
	\noindent\underline{\it Estimate of $II_1$}. By Taylor expansion, one has
	\beno
	II_1&=&\int_{|v|\sim 2^j}b(\cos\th)|v-v_*|^\ga F(v_*)F(v)\chi_{j}(v')\<v'\>^k(\vphi(2^{-j}v')-\vphi(2^{-j}v))dvdv_*d\sigma\\
	&=&2^{-j}\int_{\R^6\times \S^2}b(\cos\th)|v-v_*|^\ga F(v_*)F(v)\big(\chi_{j}(v')\<v'\>^k-\chi_{j}(v)\<v\>^k\big)\na\vphi(2^{-j}v)\cdot(v'-v)dvdv_*d\si\\
	&&+2^{-j}\int_{\R^6\times \S^2}b(\cos\th)|v-v_*|^\ga F(v_*)F(v)\chi_{j}(v)\<v\>^k\na\vphi(2^{-j}v)\cdot(v'-v)dvdv_*d\si+\int_0^1\int_{|v|\sim 2^j}b(\cos\th)\\
	&&\times2^{-2j-1}|v-v_*|^\ga F(v_*)F(v)\chi_{j}(v')\<v'\>^k(1-t) \na^2\vphi(2^{-j}(v+t(v-v')))(v-v')\otimes(v-v')dtdvdv_*d\si\\
	&:=& II_{1,1}+II_{1,2}+II_{1,3}.
	\eeno
	
	By symmetry with respect to $\sigma$, we have
	\beno
	|II_{1,2}|&\ls & 2^{-j}\int_{\R^6}\int_0^{\pi/2}\int_0^\pi b(\cos\th)\sin^2(\th/2)|v-v_*|^{1+\ga}F(v_*)F(v)\chi_{j}(v)\<v\>^kdvdv_*d\phi d\th.
	\eeno
	For $\ga\geq-1$, we have
	\beno
	|II_{1,2}|
	&\ls&2^{-j}\int_{\R^6}\max\{|v|,|v_*|\}^{1+\ga}F(v_*)F(v)\chi_{j}(v)\<v\>^kdvdv_*\\
	&\ls&\|F\|_{L^1_1}\|\chi_{j}F\|_{L^1_{k-1}}+\|F\|_{L^1}\|\chi_{j}F\|_{L^1_{k+\ga}}\ls\|\chi_{j}F\|_{L^1_{k+\ga}}.
	\eeno
	For $\ga<-1$, we have
	\beno
	|II_{1,2}|&\ls &2^{-j}\int_{|v-v_*|\geq \frac{1}{2}|v|}|v-v_*|^{1+\ga}F(v_*)F(v)\chi_{j}(v)\<v\>^kdvdv_*\\
	&&+2^{-j}\int_{|v-v_*|< \frac{1}{2}|v|}|v-v_*|^{1+\ga}F(v_*)F(v)\chi_{j}(v)\<v\>^kdvdv_*.
	\eeno
	The first term can be bounded by $ \|F\|_{L^1}\|\chi_{j}F\|_{L^1_{k+\ga}}$. For the second term, noticing  that the condition $|v-v_*|<\frac{1}{2}|v|$ implies that $|v_*|\sim |v|\sim 2^j$. We derive that
	\beno
	&&2^{-j}\int_{|v-v_*|< \frac{1}{2}|v|}|v-v_*|^{1+\ga}F(v_*)F(v)\chi_{j}(v)\<v\>^kdvdv_*\\
	&\leq& 2^{(-1-\ga)j}(\int_{|v-v_*|<\f12|v|}|v-v_*|^{2(1+\ga+s)}dv_*)^{1/2}(\int_{|v_*|\sim 2^j}\f{F_*^2}{|v-v_*|^{2s}}dv_*)^{1/2}\int_{\R^3}\chi_{j}(v)F \<v\>^{k+\ga}dv\\
	&\ls&\|F\|_{H^s_{s+3/2}} \|\chi_{j}F\|_{L^1_{k+\gamma}},
	\eeno
where we use the fact $\ga+s+1>-1$ and Hardy inequality. Thus we obtain that $|II_{1,2}|\ls (1+\|F\|_{H^s_{s+3/2}})\|\chi_{j}F\|_{L^1_{k+\gamma}}$.
	
	Observe that  $|\chi_{j}(v')\<v'\>^k-\chi_{j}(v)\<v\>^k|\ls 2^{j(k-1)}|v-v_*|\sin\theta$. Then for  $II_{1,1}$, we have
	\beno
	|II_{1,1}|&\ls & 2^{j(k-2)}\int_{|v|\sim 2^{j}} |v-v_*|^{2+\ga}F_*F dvdv_*.
	\eeno
For $\ga\geq-2$, we have 
\beno
|II_{1,1}|&\ls &\|F\|_{L^1_2}\|\chi_j F\|_{L^1_{k-2}}+\|F\|_{L^1}\|\chi_jF\|_{L^1_{k+\ga}}\ls \|\chi_jF\|_{L^1_{k+\ga}}.
\eeno	
For $\ga<-2$, wo also split into two part: $|v-v_*|\geq\f12|v|$ and $|v-v_*|<\f12|v|$ and we can derive that
\beno
|II_{1,1}|&\ls & \|F\|_{L^1}\|\chi_jF\|_{L^1_{k+\ga}}+\|F\|_{L^2_{3/2}}\|\chi_jF\|_{L^1_{k+\ga}}\ls \|\chi_jF\|_{L^1_{k+\ga}}.
\eeno
Thus we obtain that $|II_{1,1}|\ls\|\chi_jF\|_{L^1_{k+\ga}} $.

By the same argument, we have 
	\beno
	|II_{1,3}|\ls  2^{-(k-2)j}\int_{\R^6\times\S^2}b(\cos\th)\sin^2(\th/2)|v-v_*|^{2+\ga}F(v_*)F(v)\chi_{j}(v) dvdv_*d\si\ls\|\chi_{j}F\|_{L^1_{k+\ga}}.
	\eeno
	
	Patching together all the estimates, we derive that $|II_1|\ls (1+\|F\|_{H^s_{3}})\|\chi_{j}F\|_{L^1_{k+\ga}}$.
	
	\noindent\underline{\it Estimate of $II_2$.} Since  $|v|\gg2^j$ and $|v'|\sim 2^j$, we have $|v_*|\gg 2^j$.
	\beno
	|II_2|&=&\int_{|v|,|v_*|\gg 2^j}b(\cos\th)|v-v_*|^\ga F(v_*)F(v)\chi_{j}(v')\<v'\>^k\vphi(2^{-j}v')dvdv_*d\si\\
	&\ls& \Big|2^{jk}\int_{|v|, |v_*|\gg 2^j}b(\cos\th)|v-v_*|^\ga F(v_*)F(v)(\vphi(2^{-j}v')-\vphi(2^{-j}v))dvdv_*d\si\Big|\\
	&\ls& 2^{jk}\big(2^{-j}\int_{|v|, |v_*|\gg 2^j} |v-v_*|^{\ga+1}FF_*dvdv_*+2^{-2j}\int_{|v|, |v_*|\gg 2^j} |v-v_*|^{\ga+2}FF_*dvdv_*\big)\\
	&:=&II_{2,1}+II_{2,2}.
	\eeno
For $II_{2,1}$, if $\ga\geq-1$, it easy to check that $II_{2,1}\ls 2^{j\ga}\|F\|_{L^1_k}$. If $\ga<-1,$
\beno
II_{2,1}&=&2^{(k-1)j}(\int_{|v-v_*|\leq1}|v-v_*|^{1+\ga} FF_*dvdv_*+\int_{|v-v_*|>1} \<v-v_*\>^{1+\ga}FF_*dvdv_*)\\
&\ls&  2^{(-1-\ga)j}(\int_{|v-v_*|<1}|v-v_*|^{2(1+\ga+s)}dv_*)^{1/2}(\int_{|v_*|\gg 2^j}\f{F_*^2}{|v-v_*|^{2s}}dv_*)^{1/2}\int_{|v|\gg2^j}F \<v\>^{k+\ga}dv\\
&&+2^{j\ga}\|F\|_{L^1_{2}}\|F\|_{L^1_k}\ls2^{j\ga}(1+\|F\|_{H^s_3})\|F\|_{L^1_k}.\eeno

Similarly, we have  $II_{2,2}\ls 2^{j\ga}\|F\|_{L^1_k}$ and $|II_2|\ls2^{j\ga}(1+\|F\|_{H^s_3})\|F\|_{L^1_k}$.

	\noindent\underline{\it Estimate of $II_3$.} Since $|v|\ll 2^j, |v'|\sim 2^j$, from the fact $2^j \sim|v-v'|=\sin(\th/2)|v-v_*|$, we get that $|v_*|\gs 2^j$ and $\th\gs 2^j/|v-v_*|$.  Then we have
	\beno
	|II_3|&=&\int_{ |v_*|\gs 2^j}b(\cos\th)|v-v_*|^\ga F(v_*)F(v)\chi_{j}(v')\<v'\>^k\vphi(2^{-j}v')dvdv_*d\si\\
	&\ls& 2^{j(k-2)}\int_{|v_*|\gs 2^j} |v-v_*|^{\ga+2}FF_*dvdv_*\ls 2^{j\ga}\|F\|_{L^1_{k}}.
	\eeno
	
	Summing up all the estimates, we conclude that
$|II|\ls(1+\|F\|_{H^s_3}) (\|\chi_jF\|_{L^1_{k+\ga}} +2^{j\ga}\|F\|_{L^1_{k}}).$

	\noindent\underline{\it Step 2: Estimate of $I$.}  We first notice that
	\beno
	I
	&=&\int_{\R^6\times \S^2}b(\cos\th)|v-v_*|^\ga F(v_*)(\cP_jF)(v)(\chi_{j}(v')\<v'\>^k-\chi_{j}(v)\<v\>^k)dvdv_*d\si.\eeno
	Applying Taylor expansion for $\chi_{j}(v)\<v\>^l$ and copying the argument for $II_{2}$, we obtain that
	\beno &&|I|\ls  2^{j(k-1)}\int_{\R^6}  |v-v_*|^{\ga+1} F_* |\cP_jF|dvdv_*+ 2^{j(k-2)}\int  |v-v_*|^{\ga+2} F_* |\cP_jF| dvdv_*\ls
	(1+\|F\|_{H^s_3})\|\cP_jF\|_{L^1_{l+\ga}},\eeno
	from which together with the estimate of $II$, we get the desired result.

	Finally we mention that the case of $2s<1$ can be proved in a similar way.   We complete the proof of this lemma.
\end{proof}

Now we are ready to prove \eqref{upperbound}.
\begin{proof}[Proof of Theorem \ref{globaldecay}:]  Given $k>48$ and $\eta>0$. To construct the sequence of the solutions,  
for $j\gg1$,  we introduce the non-negative radial function $G_0(v):=2^{-(k+3)j}\vphi(2^{-j}v)+g_j(v)$. Here $\vphi$ is defined in \eqref{7.1} and $g_j$ is   a non-negative function supported in $\{v\in\R^3|1<|v|<4\}$ satisfying that $\|g_j\|_{L^\infty}<4$ and
	\beno
	\int_{\R^3}G_0(v)dv=1,\quad \int_{\R^3}G_0(v)|v|^2=3.
	\eeno
If $F_0=\mu+\f{\vep_0}{100^{100}}(G_0-\mu)=\mu+f_0$, then we have
\beno
\|f_0\|_{L^2_{48}}&\leq& \f{\vep_0}{100^{100}}(\|G_0\|_{L^2_{48}}+\|\mu\|_{L^2_{48}})\leq \f{\vep_0}{100^{100}} (2^{-(k-48+3/2)j}\|\vphi\|_{L^2}+\|g_j\|_{L^2_{48}}+\|\mu\|_{L^2_{48}})<\vep_0;\\
\|f_0\|_{L^1_{k}}&\leq & \f{\vep_0}{100^{100}}(\|G_0\|_{L^1_{k}}+\|\mu\|_{L^1_{k}})\leq \f{\vep_0}{100^{100}}(\|\vphi\|_{L^1}+4^k\|g\|_{L^1}+\|\mu\|_{L^1_k})\leq C_k;\\
\|\cP_jF_0\|_{L^1_k}&\geq& \|\cP_jf_0\|_{L^1_k}-\|\cP_j\mu\|_{L^1_k}\geq\f{\vep_0}{100^{100}}(\|\vphi\|_{L^1}-\|\cP_j\mu\|_{L^1_k})-\|\cP_j\mu\|_{L^1_k}\gs \f{\vep_0}{100^{100}}-e^{-2^j}2^{-kj}\geq \f{\vep_0}{200^{100}}.
\eeno
Therefore $f(t,v)\in \mathfrak{A}$ is a solution of \eqref{14} with initial data $f_0$.

Moreover, for any $0<\de<|\gamma|$, we have $\|f_0\|_{L^2_{k+3/2+\de}} \leq\f{\vep_0}{100^{100}} (2^{\de j}\|\vphi\|_{L^2}+\|\mu\|_{L^2_{k+3/2+\de}})\ls C_k2^{\de j}$, from which and the propagation of  polynomial moment in Theorem \ref{globaldecay}, we get that $\|F(t)\|_{L^1_k}\leq\|f(t)\|_{L^1_k}+\|\mu\|_{L^1_k}\leq\|f(t)\|_{L^2_{k+3/2+\de}}+\|\mu\|_{L^1_k}\leq C_{\de,k} 2^{\de j}$.
Thus from Lemma \ref{le6.1} and the fact that $F\in L^2_tH^s_3$, we have
	\beno
	\|\cP_jF(t)\|_{L^1_k}\geq \|\cP_jF_0\|_{L^1_k}-C_{\de,k}2^{j(\ga+\de)}\int_0^t(1+\|F\|_{H^s_3})dt\geq \|\cP_jF_0\|_{L^1_k}-C_{\de,k}2^{j(\ga+\de)}(1+t) ,
	\eeno
	which yields that  for $t\leq \frac{1}{2C_{\de,k}}2^{-j(\ga+\de)}\|\cP_jF_0\|_{L^1_k}-1$,
	$\|\cP_jF(t)\|_{L^1_k}\geq\frac{1}{2}\|\cP_jF_0\|_{L^1_k}.$
	Observe that
	$\|f(t)\|_{L^1}=\|F(t)-\mu\|_{L^1}\ge \|\cP_j(F-\mu)\|_{L^1}\ge C_k2^{-jk}\big(\f12 \|\cP_jF_0\|_{L^1_k}-\|\cP_j\mu\|_{L^1_k}\big).$ Thus when $t+1\sim C_{\de,k}2^{-j(\ga+\de)}$, we have
	\beno
	\|f(t)\|_{L^1}\ge \f14C_k2^{-jk}\|\cP_jF_0\|_{L^1_k}\ge C_{\de,k}2^{-jk}\ge C_{\de,k}(1+t)^{-\f{k}{|\gamma+\de|}}. 
	\eeno
It yields that $\sup_{f\in \mathfrak{A}}\|f(t)(1+t)^{\f k{|\ga|}+\eta}\|_{L^\infty_tL^1_{x,v}}>C_{\de,k}\sup_{t>0}(1+t)^{k(\f1{|\ga|}-\f1{|\ga+\de|}+\eta)}=+\infty$ since $\de$ can be sufficiently small. It ends the second part of the main theorem.
\end{proof}

\subsection{Proof of Theorem \ref{globaldecay}: Exponential moment} Since the existence of solution has been established, we now focus on the propagation of the exponential moment. To keep the length of the paper, we only
 provide a detailed proof for high singularity $2s\geq1$. The case $2s<1$ can be handled similarly.
\begin{lem}
Suppose $\ga\in(-3,1],\ga+2s>-1,2s\geq1$ and $\cN_0:=[2s^{-1}]+1$. Then there exist constants $c_1,c_2>0$ depending on $\ga_1,\ga_2$ such that for $j\geq M_0$ with $M_0$ suitably large fixed, we have
\ben\label{ene}
&&\frac{d}{dt}\sum_{|\al|=0,2}\|\<\cdot\>^{2js-4\cN_0s|\al|}\pa^\al_xf\|^2_{L^2_{x,v}}+c_1\sum_{|\al|=0,2}j^s\|\<\cdot\>^{2js-4\cN_0s|\al|}\pa^\al_xf\|^2_{L^2_xL^2_{\ga/2}}\\&&+c_2\sum_{|\al|=0,2}\|\<\cdot\>^{2js-4\cN_0s|\al|}\pa^\al_xf\|^2_{L^2_xH^s_{\ga/2}}\ls \mathfrak{G}_j+\mathfrak{H}_j, 
\qquad\mbox{where}\notag\\
&&\mathfrak{G}_j=\|f\|_{H^2_xL^2_{17}}\Big(\sum_{|\al|=0,2}\|\<\cdot\>^{2js-4\cN_0s|\al|}\pa^\al_xf\|^2_{L^2_xH^s_{\ga/2}}+j^8\sum_{|\al|=0,2}\|\<\cdot\>^{2js-4\cN_0s|\al|-8}\pa^\al_xf\|^2_{L^2_xH^s_{\ga/2}}\notag\\
&&+j^s\sum_{|\al|=0,2}\|\<\cdot\>^{2js-4\cN_0s|\al|}\pa^\al_xf\|^2_{L^2_xL^2_{\ga/2}}+j^{6+s}\sum_{|\al|=0,2}\|\<\cdot\>^{2js-4\cN_0s|\al|-7}\pa^\al_xf\|^2_{L^2_xL^2_{\ga/2}}\Big)\notag\\
&&+\sum_{|\al|=0,2}\Big(j^2\|\<\cdot\>^{2js-4\cN_0s|\al|
-2}\pa^\al_xf\|^2_{L^2_xH^s_{\ga/2}}+j^{1+s}\|\<\cdot\>^{2js-4\cN_0s|\al|-1}\pa^\al_xf\|^2_{L^2_xL^2_{\ga/2}}\notag\\
&&+j^{8+s}\|\<\cdot\>^{2js-4\cN_0s|\al|-8}\pa^\al_xf\|^2_{L^2_xL^2_{\ga/2}}\Big)+j^{14}\|\mu\<\cdot\>^{2js+12}\|^2_{L^2_{v}}\|f\|^2_{H^2_xL^2_{17}},\notag
\een
and
\ben\label{HJ}
\mathfrak{H}_j&=&\sum_{|\al|=2}\int_{\T^3}\hat{\mathcal{H}}_{j-4\cN_0}(\pa^{\al_1}_xf,\pa^{\al_2}_xf,\pa^\al_xf)dx+\sum_{|\al|=0,2}\int_{\T^3}\hat{\mathcal{H}}_{j-2\cN_0|\al|}(\pa^\al_xf,\mu,\pa^\al_xf)dx\notag\\&&+\sum_{|\al|=0,2}\int_{\T^3}\hat{\mathcal{H}}_{j-2\cN_0|\al|}(F,\pa^\al_xf,\pa^\al_xf)dx.
\een
Here $\hat{\mathcal{H}}_j$ is defined in Lemma \ref{B_3}.
\end{lem}
\begin{proof}
Since $\cN_0:=[2s^{-1}]+1$, then $4\cN_0s\geq8$. In what follows, we shall focus on $\mathfrak{G}_j$ since
 $\mathfrak{H}_j$ will be handled in Lemma \ref{h_j}. By applying the standard energy method to the equation
$(\partial_t + v\cdot\nabla_x)\pa^\al_xf=Q(F,\pa^\al_xf)+Q(\pa^\al_xf,\mu)+\sum_{|\al_1|\neq0}Q(\pa^{\al_1}_x f,\pa^{\alpha -\alpha_1}_x f)$, we derive that
\ben\label{energyeq}
&&\frac{d}{dt}\sum_{|\al|=0,2}\|\<\cdot\>^{2js-4\cN_0s|\al|}\pa^\al_xf\|^2_{L^2_{x,v}}= \sum_{|\al|=0,2}\int_{\T^3}(Q(F,\pa^\al_xf),\pa^\al_xf\<v\>^{4js-8\cN_0s|\al|})dx\\
\notag&&+\sum_{|\al|=0,2}\int_{\T^3}(Q(\pa^\al_xf,\mu),\pa^\al_xf\<v\>^{4js-8\cN_0s|\al|})dx+\sum_{|\al_1|\neq0}\int_{\T^3}(Q(\pa^{\al_1}_x f,\pa^{\alpha -\alpha_1}_x f),\<v\>^{4js-8\cN_0s|\al|}\pa^\al_xf)dx\\
\notag&&:=\mathcal{I}_1+\mathcal{I}_2+\mathcal{I}_3.
\een

\underline{\it Step 1: Estimate of $\mathcal{I}_3$.} Since $|
\al|=2$. Thanks to Lemma \ref{B_3}, we have
\beno
&&\left|\int_{\T^3}(Q(\pa^{\al_1}_xf,\pa^{\al_2}_xf),\<v\>^{4js-16\cN_0s}\pa^\al_xf)dx\right|=\left|\int_{\T^3}(Q(\pa^{\al_1}_xf,\<v\>^{2js-8\cN_0s}\pa^{\al_2}_xf),\<v\>^{2js-8\cN_0s}\pa^\al_xf)dx\right|\\
&&+\left|\int_{\T^3}(\<v\>^{2js-8\cN_0s}Q(\pa^{\al_1}_xf,\pa^{\al_2}_xf)-Q(\pa^{\al_1}_xf,\<v\>^{2js-8\cN_0s}\pa^{\al_2}_xf),\<v\>^{2js-8\cN_0s}\pa^\al_xf)dx\right|:=\mathcal{I}_{3,1}+\mathcal{I}_{3,2}.
\eeno
Due to Lemma \ref{L12},   one has
$\mathcal{I}_{3,1}\leq\|f\|_{H^2_xL^2_{14}}\sum\limits_{|\al|=2}\|\<\cdot\>^{2js-8\cN_0s}\pa^\al_xf\|^2_{L^2_xH^s_{\ga/2}}.
 $
By applying the Fourier transform w.r.t. $x$ variable to $\mathcal{I}_{3,2}$(see the proof of Lemma \ref{Qspatial}), Lemma \ref{L23}  and \eqref{estiprodx} will yield that
 \beno
 \mathcal{I}_3&\ls&
 j^s\|f\|_{H^2_xL^2_{14}}\sum_{|\al|=0,2}\|\<\cdot\>^{2js-4\cN_0s|\al|}\pa^\al_xf\|^2_{L^2_xL^2_{\ga/2}}+ \|f\|_{H^2_xL^2_{14}}\sum_{|\al|=0,2}\|\<\cdot\>^{2js-4\cN_0s|\al|}\pa^\al_xf\|^2_{L^2_xH^s_{\ga/2}}\\
 &&+j^{4}\|f\|_{H^2_xL^2_{14}}(\sum_{|\al|=0,2}\|\<\cdot\>^{2js-4\cN_0s|\al|-8}\pa^\al_xf\|_{L^2_xH^s_{\ga/2}})(\sum_{|\al|=0,2}\|\<\cdot\>^{2js-4\cN_0s|\al|}\pa^\al_xf\|_{L^2_xH^s_{\ga/2}})\\
 &&+j^{\f12(1+s)} \|f\|_{H^2_xL^2_{14}}(\sum_{|\al|=0,2}\|\<\cdot\>^{2js-4\cN_0s|\al|-2s}\pa^\al_xf\|_{L^2_xH^s_{\ga/2}})(\sum_{|\al|=0,2}\|\<\cdot\>^{2js-4\cN_0s|\al|-2+s}\pa^\al_xf\|_{L^2_xH^s_{\ga/2}})\\
 &&+j^{\f12(1+s)} \|f\|_{H^2_xL^2_{14}}(\sum_{|\al|=0,2}\|\<\cdot\>^{2js-4\cN_0s|\al|-2s}\pa^\al_xf\|_{L^2_xH^s_{\ga/2}})(\sum_{|\al|=0,2}\|\<\cdot\>^{2js-4\cN_0s|\al|-1+s}\pa^\al_xf\|_{L^2_xL^2_{\ga/2}})\\
 &&j^{3+s}\|f\|_{H^2_xL^2_{14}}(\sum_{|\al|=0,2}\|\<\cdot\>^{2js-4\cN_0s|\al|-7}\pa^\al_xf\|_{L^2_xL^2_{\ga/2}})(\sum_{|\al|=0,2}\|\<\cdot\>^{2js-4\cN_0s|\al|}\pa^\al_xf\|_{L^2_xL^2_{\ga/2}})\\
 &&+\int_{\T^3}\mathcal{H}_{j-4\cN_0}(\pa^{\al_1}_xf,\pa^{\al_2}_xf,\pa^\al_xf)dx.
 \eeno

\underline{\it Step 2: Estimate of $\mathcal{I}_2$.} Again by Fourier transform w.r.t. $x$ variable, Lemma \ref{L21} and \eqref{estiprodx}, we have
\beno
&&\mathcal{I}_2-\|b(\cos\th)\sin^{2js-3/2-\ga/2}\|_{L^1_\th}\sum_{|\al|=0,2}\|\<\cdot\>^{2js-4\cN_0s|\al|}\pa^\al_xf\|^2_{L^2_xL^2_{\ga/2}}\\
&\ls& \sum_{|\al|=0,2}\Big(j\|\<\cdot\>^{2js-4\cN_0s|\al|-1}\pa^\al_xf\|^2_{L^2_xL^2_{\ga/2}}+j^{4}\|\<\cdot\>^{2js-4\cN_0s|\al|-4}\pa^\al_xf\|^2_{L^2_xL^2_{\ga/2}}+\|\<\cdot\>^{2js-4\cN_0s|\al|-1/2}\pa^\al_xf\|^2_{L^2_xL^2_{\ga/2}}\\
&&+j^7\|\mu\<\cdot\>^{2js+12}\|_{L^2_{v}}\|\<\cdot\>^{17}\pa^\al_xf\|_{L^2_xL^2_v}\|\<\cdot\>^{2js-4\cN_0s|\al|}\pa^\al_xf\|_{L^2_xL^2_{\ga/2}}\Big)+\sum_{|\al|=0,2}\int_{\T^3}\mathcal{H}_{j-2\cN_0|\al|}(\pa^\al_xf,\mu,\pa^\al_xf)dx.
\eeno

\underline{\it Step 3: Estimate of $\mathcal{I}_1$.} Following the proof of Theorem \ref{T24} (\ref{exp2}),  we may get that
\beno
&&\mathcal{I}_1+\sum_{|\al|=0,2}(\ga_1\|\<\cdot\>^{2js-4\cN_0s|\al|}\pa^\al_xf\|^2_{L^2_xH^s_{\ga/2}}+\f{\ga_2}8\|b(\cos\th)(1-\cos^{4js-3-\ga}(\th/2))\|_{L^1_\th}\|\<\cdot\>^{2js-6|\al|
}\pa^\al_xf\|^2_{L^2_xL^2_{\ga/2}})\\
&\ls&\sum_{|\al|=0,2}\Big(\f1j\|f\|_{H^2_xL^2_{14}}\|\<\cdot\>^{2js-4\cN_0s|\al|}\pa^\al_xF\|_{L^2_xL^2_{\ga/2}}\|\<\cdot\>^{2js-4\cN_0s|\al|}\pa^\al_xf\|_{L^2_xL^2_{\ga/2}}+j\|f\|_{H^2_xL^2_{14}}\\
&&\times\|\<\cdot\>^{2js-4\cN_0s|\al|-2}\pa^\al_xF\|_{L^2_xH^s_{\ga/2}}\|\<\cdot\>^{2js-4\cN_0s|\al|}\pa^\al_xf\|_{L^2_xH^s_{\ga/2}}+j^4\|f\|_{H^2_xL^2_{14}}\|\<\cdot\>^{2js-6|\al|
-8}\pa^\al_xF\|_{L^2_xH^s_{\ga/2}}\\
&&\times\|\<\cdot\>^{2js-4\cN_0s|\al|}\pa^\al_xf\|_{L^2_xH^s_{\ga/2}}+\|F\|^4_{H^2_xL^2_{14}}\|\<\cdot\>^{2js-4\cN_0s|\al|}\pa^\al_xf\|^2_{L^2_xL^2_{\ga/2}}+j\|F\|_{H^2_xL^2_{14}}\\
&&\times\|\<v\>
^{2js-4\cN_0s|\al|
-2}\pa^\al_xf\|_{L^2_xH^s_{\ga/2}}\|\<\cdot\>^{2js-4\cN_0s|\al|}\pa^\al_xf\|_{L^2_xH^s_{\ga/2}}+j^4\|F\|_{H^2_xL^2_{14}}\|\<\cdot\>^{2js-4\cN_0s|\al|-8}\pa^\al_xf\|_{L^2_xL^2_{\ga/2}}\\
&&\times\|\<\cdot\>^{2js-4\cN_0s|\al|}\pa^\al_xf\|_{L^2_xL^2_{\ga/2}}+j^{1+s}\|F\|_{H^2_xL^2_{14}}(\|\<\cdot\>^{2js-4\cN_0s|\al|-2}\pa^\al_xf\|^2_{L^2_xH^s_{\ga/2}}+\|\<\cdot\>^{2js-4\cN_0s|\al|
-1}\pa^\al_xf\|^2_{L^2_xL^2_{\f\ga 2}})\\
&&+j^{s}\|F\|_{H^2_xL^2_{14}}\|\<\cdot\>^{2js-4\cN_0s|\al|-1}\pa^\al_xf\|_{L^2_xL^2_{\ga/2}}\|\<\cdot\>^{2js-4\cN_0s|\al|}\pa^\al_xf\|_{L^2_xL^2_{\ga/2}}+j^{3+s}\|F\|_{H^2_xL^2_{14}}\\
&&\times\|\<\cdot\>^{2js-4\cN_0s|\al|-7}\pa^\al_xf\|_{L^2_xL^2_{\ga/2}}\|\<\cdot\>^{2js-4\cN_0s|\al|}\pa^\al_xf\|_{L^2_xL^2_{\ga/2}}\Big)+\sum_{|\al|=0,2}\int_{\T^3}\mathcal{H}_{j-2\cN_0|\al|}(F,\pa^\al_xf,\pa^\al_xf)dx.
\eeno
We address that the Fourier transform w.r.t. $x$ variable and \eqref{estiprodx} are frequently used in particular when we try to give the upper bounds in the proof Theorem \ref{T24}.

Recalling that $\mathcal{H}_j\le\bar{\mathcal{H}}_j+\hat{\mathcal{H}}_j$, one may easily see that $\int_{\T^3}\bar{\mathcal{H}}_jdx$ enjoys the same estimate as $\mathcal{I}_i,i=1,2,3$. Plugging these estimates   into (\ref{energyeq}) and using facts that $\|F\|^4_{H^2_xL^2_{14}}\ls\|F\|_{H^2_xL^2_{14}}$ (thanks to bound of solution in Theorem \ref{globaldecay}: polynomial case), $ab\leq \vep a^2+b^2/(4\vep) $ and   $F=\mu+f$, we obtain that
\beno
&&\frac{d}{dt}\sum_{|\al|=0,2}\|\<\cdot\>^{2js-4\cN_0s|\al|}\pa^\al_xf\|^2_{L^2_{x,v}}+(\f{\ga_2}8\|b(\cos\th)(1-\cos^{4js-3-\ga}(\th/2))\|_{L^1_\th}-\|b(\cos\th)\sin^{2js-3/2-\ga/2}\|_{L^1_\th})\\
&&\times\sum_{|\al|=0,2}\|\<\cdot\>^{2js-4\cN_0s|\al|}\pa^\al_xf\|^2_{L^2_xL^2_{\ga/2}}+\ga_1\sum_{|\al|=0,2}\|\<\cdot\>^{2js-4\cN_0s|\al|}\pa^\al_xf\|^2_{L^2_xH^s_{\ga/2}}\leq \mathrm{L.O.T},\\
&&\mathrm{L.O.T}=\|f\|_{H^2_xL^2_{17}}\Big(\sum_{|\al|=0,2}\|\<\cdot\>^{2js-4\cN_0s|\al|}\pa^\al_xf\|^2_{L^2_xH^s_{\ga/2}}+j^8\sum_{|\al|=0,2}\|\<\cdot\>^{2js-4\cN_0s|\al|-8}\pa^\al_xf\|^2_{L^2_xH^s_{\ga/2}}\\
&&+j^s\sum_{|\al|=0,2}\|\<\cdot\>^{2js-4\cN_0s|\al|}\pa^\al_xf\|^2_{L^2_xL^2_{\ga/2}}+j^{6+s}\sum_{|\al|=0,2}\|\<\cdot\>^{2js-4\cN_0s|\al|-7}\pa^\al_xf\|^2_{L^2_xL^2_{\ga/2}}\Big)\\
&&+\sum_{|\al|=0,2}\Big(\vep\|\<\cdot\>^{2js-4\cN_0s|\al|}\pa^\al_xf\|^2_{L^2_xH^s_{\ga/2}}+\f1{4\vep}j^2\|\<\cdot\>^{2js-4\cN_0s|\al|
-2}\pa^\al_xf\|^2_{L^2_xH^s_{\ga/2}}+\vep j^s\|\<\cdot\>^{2js-4\cN_0s|\al|}\pa^\al_xf\|^2_{L^2_xL^2_{\ga/2}}\\
&&+\f1{4\vep}j^{1+s}\|\<\cdot\>^{2js-4\cN_0s|\al|-1}\pa^\al_xf\|^2_{L^2_xL^2_{\ga/2}}+\f1{4\vep}j^{8+s}\|\<\cdot\>^{2js-4\cN_0s|\al|-8}\pa^\al_xf\|^2_{L^2_xL^2_{\ga/2}}\Big)\\
&&+\|\<\cdot\>^{2js-4\cN_0s|\al|-1/2}\pa^\al_xf\|^2_{L^2_xL^2_{\ga/2}}+\f1{4\vep}j^{14}\|\mu\<\cdot\>^{2js+12}\|^2_{L^2_{v}}\|f\|^2_{H^2_xL^2_{17}}+\sum_{|\al|=2}\int_{\T^3}\hat{\mathcal{H}}_{j-4\cN_0}(\pa^{\al_1}_xf,\pa^{\al_2}_xf,\pa^\al_xf)dx\\
&&+\sum_{|\al|=0,2}\Big(\int_{\T^3}\hat{\mathcal{H}}_{j-2\cN_0|\al|}(\pa^\al_xf,\mu,\pa^\al_xf)dx+\hat{\mathcal{H}}_{j-2\cN_0|\al|}(F,\pa^\al_xf,\pa^\al_xf)dx\Big).
\eeno
Recall that $(\f{\ga_2}8\|b(\cos\th)(1-\cos^{4js-3-\ga}(\th/2))\|_{L^1_\th}-\|b(\cos\th)\sin^{2js-3/2-\ga/2}\|_{L^1_\th})\gs j^s$ because of Lemma \ref{lemma2.2}(i). Choosing suitably small $\vep$ and sufficiently large $j\geq M_0$, we can conclude the desired results (\ref{ene}).
\end{proof}
\bigskip

Now, we concentrate on the terms $\mathfrak{H}_j$. Let $\be\in(0,2)$ and $a\in(0,1],c\in\R$. We define
\beno
\mathcal{M}(\f4\be j,c)&:=&\sum_{|\al|=0,2}\int_{\T^3\times\R^3}\frac{a^{\f4\be (j-7\cN_0)s}(\<v\>^\be)^{\f4\be js-8\cN_0s|\al|/\be+c}}{(\mathbb{G}(\f4\be j))^s}|\pa^\al_xf|^2(t,x,v)dvdx,\\
\mathcal{M}_F(\f4\be j,c)&:=&\sum_{|\al|=0,2}\int_{\T^3\times\R^3}\frac{a^{\f4\be (j-7\cN_0)s}(\<v\>^\be)^{\f4\be js-8\cN_0s|\al|/\be+c}}{(\mathbb{G}(\f4\be j))^s}|\pa^\al_xF|^2(t,x,v)dvdx,\\
\mathcal{M}_\mu(\f4\be j,c)&:=&\int_{\R^3}\frac{a^{\f4\be (j-7\cN_0)s}(\<v\>^\be)^{\f4\be js+c}}{(\mathbb{G}(\f4\be j))^s}\mu^2(v)dv,\\
\mathcal{M}_s(\f4\be j,c)&:=&\sum_{|\al|=0,2}\int_{\T^3\times\R^3}\frac{a^{\f4\be (j-7\cN_0)s}|\<D_v\>^s((\<v\>^\be)^{\f2\be js-4\cN_0s|\al|/\be+c/2}\pa^\al_xf)|^2}{(\mathbb{G}(\f4\be j))^s}(t,x,v)dvdx.
\eeno
We remark that  $\mathcal{M}_s(\f4\be j,c)$ is different from others since it involves operator $\<D_v\>^s$. Furthermore, due to Prop. \ref{expweight},  the summation of $\mathcal{M}_\mu(\f4\be j,c)$ over $j$ always converges for any $\be\in(0,2)$ and $a,c\in\R$.

\begin{lem}\label{M123} Let $l\in\R^+$, $\mathfrak{N}=[(\be l)/(s4)]$ and $N_0>\cN_0+\mathfrak{N}+1$, $M_0\gg\max\{\mathfrak{N}+1,N_0\}$. Then we have
\ben
&&\sum_{j=M_0}^\infty\cM(\f4\be j,0)\leq C_{s,\be}a^{-\f4\be s(1-\th_1)}\sum_{j=M_0}^\infty j^s\cM(\f4\be j,-s); \label{M1}\\
&&\sum_{j=N_0}^\infty j^{l+s}\cM(\f4\be j,(\ga-2l)/\be)\leq C_{l,s,\be} a^{\f4\be(\mathfrak{N}+\th_2)s}( \sum_{j=M_0}^\infty j^{s}\cM(\f4\be j,\ga/\be)+M_0 \|f\|^2_{H^2_xL^2_v}),\label{M2}\\
&&\sum_{j=M_0}^\infty j^{l}\cM_s(\f4\be j,\f{\ga-2l}\be)\leq a^{\f4\be(\mathfrak{N}+\th_2)s}( C_{l,s,\be}\sum_{j=M_0}^\infty\cM_s(\f4\be j,\f\ga\be)+C_{M_0,l,s,\be} \|f\|^2_{H^2_xH^s_v}),\label{M3}
\een
where $\th_1=(4/\be-1)/(4/\be)$,  and $\th_2=(\f ls-\f4\be \mathfrak{N})/\f4\be$.
\end{lem}
\begin{proof} Let $Z\in\R^+$ and $\mathfrak{N}_Z=[\be Z/4]$.
  If $j\gg\max\{7\cN_0,\mathfrak{N}+1\}$ and $m\in\N$, then
\ben
&&\f{j^s(\<v\>^\be)^{\f 4\be js}}{(\mathbb{G}(\f4\be j))^s}\sim_{s,\be}\left(\f{(j+1)^s(\<v\>^{\be})^{\f4\be (j+1)s}}{(\mathbb{G}(\f4\be (j+1)))^s}\right)^{\f j{j+1}}\sim_{s,\be}\left(\f{(j+m)^s(\<v\>^{\be})^{\f4\be (j+m)s}}{(\mathbb{G}(\f4\be (j+m)))^s}\right)^{\f j{j+m}},\label{356}\\
&&\frac{1}{(\mathbb{G}(\f4\be j))^s}\sim_{s,\be} \left(\frac{(\<v\>^\be)^{-s}}{(\mathbb{G}(\f4\be j-1))^s}\right)^{\th_1}\left(\frac{(\<v\>^\be)^{(\f4\be -1)s}}{(\mathbb{G}(\f4\be (j+1)-1))^s}\right)^{1-\th_1}\label{3.53};\\
&&\frac{(\<v\>^\be)^{(-Zs)}}{(\mathbb{G}(\f4\be j-Z))^s}\sim_{s,\be,Z} \left(\frac{(\<v\>^\be)^{\f4\be (-\mathfrak{N}_Z-1)s}}{(\mathbb{G}(\f4\be (j-\mathfrak{N}_Z-1)))^s}\right)^{\th_2}\left(\frac{(\<v\>^\be)^{\f4\be (-\mathfrak{N}_Z)s}}{(\mathbb{G}(\f4\be (j-\mathfrak{N}_Z)))^s}\right)^{1-\th_2}\label{3.54}.
\een
 To see these, we resort to \eqref{gammafun} which implies that $(\mathbb{G}(\f4\be j))^{s/j}\sim_{s,\be}(\mathbb{G}(\f4\be (j+1)))^{s/(j+1)}$, $(\mathbb{G}(\f4\be j-1))^s\sim_{s,\be} j^{-s}(\mathbb{G}(\f4\be j))^s$, $(\mathbb{G}(\f4\be (j+1)-1))^s\sim_{s,\be} j^{(\f4\be -1)s}(\mathbb{G}(\f4\be j))^s$, $\mathbb{G}(\f4\be (j-\mathfrak{N}_Z-1))^s\sim_{s,\be,Z} j^{(Z-\f4\be(\mathfrak{N}_Z+1))s}\mathbb{G}(\f4\be j-Z)^s$ and $\mathbb{G}(\f4\be (j-\mathfrak{N}_Z))^s\sim_{s,\be,Z} j^{(Z-\f4\be\mathfrak{N}_Z)s}\mathbb{G}(\f4\be j-Z)^s$.
From these together with facts $j\sim(j+1)^{j/(j+1)}$, $\th_1=(\f4\be-1)(1-\th_1)$  and $(\mathfrak{N}_Z+1)\th_2+\mathfrak{N}_Z(1-\th_2)=\be Z/4$, we end the proof of (\ref{356}-\ref{3.54}).

Now we can give the proof of (\ref{M1}). %In fact,
%\beno
%\sum_{j=M_0}^\infty j^s\cM(\f4\be j,\ga/\be)=\sum_{j=M_0}^\infty\sum_{|\al|=0,2}\int_{\T^3\times\R^3}j^s\frac{a^{\f4\be js}(\<v\>^\be)^{\f4\be js-8\cN_0s|\al|/\be+\ga/\be}}{(\mathbb{G}(\f4\be j))^s}|\pa^\al_xf|^2(t,x,v)dvdx.
%\eeno
By the definition of $\cM(\f4\be j,0)$ and (\ref{3.53}), we deduce that
\beno
\sum_{j=M_0}^\infty\cM(\f4\be j,0)
&\leq&C_{s,\be}a^{-\f4\be s(1-\th_1)}\left(\sum_{j=M_0}^\infty\sum_{|\al|=0,2}\int_{\T^3\times\R^3}\frac{a^{\f4\be (j-7\cN_0)s}(\<v\>^\be)^{(\f4\be j-1)s-8\cN_0s|\al|/\be}}{(\mathbb{G}(\f4\be j-1))^s}|\pa^\al_xf|^2(t,x,v)dvdx\right)^{\th_1}\\
&&\times\left(\sum_{j=M_0}^\infty\sum_{|\al|=0,2}\int_{\T^3\times\R^3}\frac{a^{\f4\be (j-7\cN_0+1)s}(\<v\>^\be)^{(\f4\be (j+1)-1)s-8\cN_0s|\al|/\be}}{(\mathbb{G}(\f4\be (j+1)-1))^s}|\pa^\al_xf|^2(t,x,v)dvdx\right)^{1-{\th_1}}.
\eeno
Since $(\mathbb{G}(\f4\be j-1))^s\sim_{s,\be} j^{-s}(\mathbb{G}(\f4\be j))^s$, then   (\ref{M1}) follows by the definition of $j^s\cM(\f4\be j,-s)$.

To prove (\ref{M2}), by \eqref{gammafun}, we get that $j^{l+s}/(\mathbb{G}(\f4\be j))^s\sim_{l,s,\be} j^s/(\mathbb{G}(\f4\be j-l/s))^s$, which implies that
\beno
\sum_{j=N_0}^\infty j^{l+s}\cM(\f4\be j, \f{\ga-2l}{\be})\sim_{l,s,\be}
 \sum_{j=N_0}^\infty\sum_{|\al|=0,2}\int_{\T^3\times\R^3}j^{s}\frac{a^{\f4\be (j-7\cN_0)s}(\<v\>^\be)^{\f4\be js-8\cN_0s|\al|/\be+(\ga-2l)/\be}}{(\mathbb{G}(\f4\be j-l/s))^s}|\pa^\al_xf|^2(t,x,v)dvdx.
\eeno
Using (\ref{3.54}) with $Z=l/s$, the above can be bounded by
\beno
&&C_{l,s,\be}\left(\sum_{j=N_0}^\infty\sum_{|\al|=0,2}\int_{\T^3\times\R^3}j^s\frac{a^{\f4\be (j-7\cN_0-\mathfrak{N}-1)s}(\<v\>^\be)^{\f4\be (j-\mathfrak{N}-1)s-8\cN_0s\f{|\al|}{\be}+\f{\ga}\be+(1-\f2{\be})l}}{(\mathbb{G}(\f4\be (j-\mathfrak{N}-1)))^s}|\pa^\al_xf|^2(t,x,v)dvdx\right)^{\th_2}\\
&&\times a^{\f4\be(\mathfrak{N}+\th_2)s}\left(\sum_{j=N_0}^\infty\sum_{|\al|=0,2}\int_{\T^3\times\R^3}j^{s}\frac{a^{\f4\be (j-7\cN_0-\mathfrak{N})s}(\<v\>^\be)^{\f4\be (j-\mathfrak{N})s-8\cN_0s\f{|\al|}\be+\ga/\be+(1-\f2\be)l}}{(\mathbb{G}(\f4\be (j-\mathfrak{N})))^s}|\pa^\al_xf|^2(t,x,v)dvdx\right)^{1-\th_2}\\
&&\leq C_{l,s,\be}a^{\f4\be(\mathfrak{N}+\th_2)s}\sum^{\infty}_{j=N_0-\mathfrak{N}-1} j^s\cM(\f4\be j,\ga/\be+(1-2/\be)l).
%&\leq& C_{a,l,s}\Big(\sum_{j=M_0}^\infty\sum_{|\al|=0,2}\int_{\T^3\times\R^3}j^{s}\frac{a^{\f4\be (j-\mathfrak{N})s}(\<v\>^\be)^{(\f4\be j-\mathfrak{N})s-8\cN_0s|\al|/\be+\ga/\be+(1-2/\be)l}}{(\mathbb{G}(\f4\be j-\mathfrak{N}))^s}|\pa^\al_xf|^2(t,x,v)dvdx\\
%&&+\sum_{j=M_0}^\infty\sum_{|\al|=0,2}\int_{\T^3\times\R^3}j^s\frac{a^{\f4\be js}(\<v\>^\be)^{((\f4\be j-\mathfrak{N}-1)s-8\cN_0s|\al|/\be+\ga/\be+(1-2/\be)l}}{(\mathbb{G}(\f4\be j-\mathfrak{N}-1))^s}|\pa^\al_xf|^2(t,x,v)dvdx\Big).
\eeno
Recalling that $(\be-2)l<0$ (this is the first reason we need $\be<2$) and $\<v\>\geq1$, we have
\beno
&&\sum_{j=N_0-\mathfrak{N}-1}^\infty j^s\cM(\f4\be j,\f{\ga}\be+(1-\f2{\be})l)\le\sum_{j=N_0-\mathfrak{N}-1}^{M_0-1} j^s\cM(\f4\be j,\ga/\be)+\sum_{j=M_0}^\infty j^s\cM(\f4\be j,\ga/\be).
\eeno
When $j\in[N_0-\mathfrak{N}-1, M_0-1]$, due to (\ref{356}) it holds that $j^s\cM(\f4\be j,\ga/\be)\ls (j+M_0-N_0+\mathfrak{N}+1)^s\cM(\f4\be (j+M_0-N_0+\mathfrak{N}+1),\ga/\be)+\|f\|^2_{H^2_xL^2_v}$. It implies that $\sum_{j=N_0-\mathfrak{N}-1}^{M_0-1} j^s\cM(\f4\be j,\ga/\be)\leq \sum_{j=M_0}^\infty j^s\cM(\f4\be j,\ga/\be)+M_0\|f\|^2_{H^2_xL^2_v}$. Thus we derive (\ref{M2}).

We finally prove (\ref{M3}). It can be handled similarly as the proof of (\ref{M2}). The main difference lies in the fact that (\ref{M3}) involves the derivative w.r.t. $v$ variable. By \eqref{3.54} and
$\|\<D\>^s\<v\>^\ell f\|_{L^2}\sim_{s,\ell}\|\<v\>^\ell\<D\>^sf\|_{L^2}$ for a finite number of parameters $\ell$(see Lemma \ref{le1.1}), we first obtain that
\beno
\sum_{j=M_0}^\infty j^{l}\cM_s(\f4\be j,(\ga-2l)/\be)\leq C_{l,s,\be}a^{\f4\be(\mathfrak{N}+\th_2)s}\sum_{j=M_0-\mathfrak{N}-1}^\infty\cM_s(\f4\be j,\ga/\be+(1-2/\be)l).
\eeno
Again by $(\be-2)l\leq 0$ and $\<v\>\geq1$, we have
\beno
&&\sum_{j=M_0-\mathfrak{N}-1}^\infty\cM_s(\f4\be j,\ga/\be+(1-2/\be)l)\sim_{l,s,\be}\sum_{j=M_0-\mathfrak{N}-1}^\infty\sum_{|\al|=0,2}\int_{\T^3\times\R^3}\<v\>^{(\be-2)l}\frac{a^{\f4\be (j-7\cN_0)s}}{(\mathbb{G}(\f4\be j))^s}\\
&&\times |\<D_v\>^s((\<v\>^\be)^{(\f2\be js-4\cN_0s|\al|/\be+\ga/(2\be)}\pa^\al_xf)|^2(t,x,v)dvdx \le C_{l,s,\be}  \sum_{j=M_0-\mathfrak{N}-1}^{\infty}\cM_s(\f4\be j,\ga/\be).
\eeno
If $j\in[M_0-\mathfrak{N}-1,M_0-1]$, by the fact $\|\<D\>^s\<v\>^\ell f\|_{L^2}\sim_{s,\ell}\|\<v\>^\ell\<D\>^sf\|_{L^2}$, we have
\beno
&&\cM_s(\f4\be j,\ga/\be)
\sim_{M_0,s,\be}\int_{\T^3\times\R^3}\frac{a^{\f4\be (j-7\cN_0)s}(\<v\>^\be)^{\f2\be js-8\cN_0s|\al|/\be+\ga/\be}|\<D_v\>^s\pa^\al_xf)|^2}{(\mathbb{G}(\f4\be j))^s}(t,x,v)dvdx\\
&&\leq \vep C_{M_0,s}\cM_s(\f4\be (j+M_0-N_0+\mathfrak{N}+1),\ga/\be)+C_{\vep,M_0}\|f\|^2_{H^2_xH^s_v},
\eeno
where  (\ref{356})  and Young inequality are used in the last inequality.
We emphasize that the constant  $C_{M_0,s}$ comes from the commutator between $\<v\>^{j},j\in[M_0-\mathfrak{N}-1,M_0]$ and $\<D_v\>^s$. By choosing $\vep=C^{-1}_{M_0,s}$, we obtain that $ \sum_{j=M_0-\mathfrak{N}-1}^{M_0-1}\cM_s(\f4\be j,\ga/\be)\leq \sum_{j=M_0}^\infty\cM_s(\f4\be j,\ga/\be)+C_{M_0}\|f\|^2_{H^2_xH^s_v}$. Thus we derive (\ref{M3}).
We end the proof of (\ref{M1}), (\ref{M2}) and (\ref{M3}).
\end{proof}

 %By choosing $R$ sufficiently large, the first term can be bounded by $\vep \sum_{j=M_0}^\infty\cM_s(\f4\be j,\ga/\be)$, and the last term can be bounded by $C_{s,\be,\vep}a^{\f\be 4 s}\f1{M_0^{2s}}\sum_{j=M_0}^\infty\cM_s(\f4\be j,\ga/\be)$. Similar to $\mathbb{M}_3$, the third term can be bounded by $C a^{\f\be 4 s}\|f\|^2_{H^2_xH^s_v}$. While for the second term, by interpolation, it can be controlled by $\vep \sum_{j=M_0}^\infty\cM_s(\f4\be j,\f\ga\be)+C_{M_0,s,\be,\vep} \|f\|^2_{H^2_xH^s_v}$.

Now we are in a position to give the estimate for $\mathfrak{H}_j$. Again we only provide a detailed proof for $1/2\leq s<1$.

\begin{lem}\label{h_j}
Let $\ga\in(-3,1],\ga+2s>-1,2s\geq1$ and $\cN_0:=[2s^{-1}]+1,M_0\gg 7\cN_0$. For sufficiently large $N$ depending on $\be$ and $s$,
 we have \beno \sum_{j=M_0}^\infty\frac{a^{\f4\be (j-7\cN_0)s}}{(\mathbb{G}(\f4\be j))^s}\mathfrak{H}_j\ls \mathfrak{K}_1+\mathfrak{K}_2,\eeno where  $\mathfrak{H}_j$ be defined in \eqref{HJ} and
\beno
&&\mathfrak{K}_1:=\vep\sum_{j=M_0}^\infty j^s\cM(\f4\be j,\ga/\be)+C_{\vep}a\|\<\cdot\>^{2Ns+6}f\|^2_{H^2_xL^2_v}+C_{\vep}(\|\<\cdot\>^{2Ns+6}f\|^2_{H^2_xL^2_v}+1)\\
&&\times\Big(\sum_{m=0}^N\sum_{j=M_0}^\infty j^{2ms+2}\cM(\f4\be j,(\ga-4ms+2s-4)/\be)+\sum_{m=0}^N\sum_{j=M_0}^\infty j^{2ms+s+2}\cM(\f4\be j,(\ga-4ms-4)/\be)\Big)\\
&&\mathfrak{K}_2:=C_{\vep}\sum_{j=M_0}^\infty\sum_{m=N}^{[\frac{j-\cN_0+1}{2}]}\cM(\f4\be (m+7\cN_0),-6/\be)\Big((j-m)^{2+s}\cM_\mu(\f4\be(j-m),(\ga-4)/\be)\\
&&+(j-m)^{2}\cM_\mu(\f4\be(j-m),(\ga-4+2s)/\be)\Big)+C_{\vep}\sum_{j=M_0}^\infty\sum_{m=N}^{[\frac{j-\cN_0+1}{2}]}\cM_\mu(\f4\be (m+7\cN_0),-6/\be)\\
&&\times\Big((j-m)^{2+s}\cM(\f4\be(j-m),(\ga-4)/\be)+(j-m)^{2}\cM(\f4\be(j-m),(\ga-4+2s)/\be)\Big)\\
&&+C_{\vep}\sum_{j=M_0}^\infty\sum_{m=N}^{[\frac{j-\cN_0+1}{2}]}\cM(\f4\be (m+7\cN_0),-6/\be)\Big((j-m)^{2+s}\cM(\f4\be(j-m),(\ga-4)/\be)\\
&&+(j-m)^{2}\cM(\f4\be(j-m),(\ga-4+2s)/\be)\Big).
\eeno
\end{lem}
\begin{proof} By \eqref{HJ}, we split $\mathfrak{H}_j$ into three parts. Set $\mathfrak{H}^1_j:=\sum_{|\al|=2}\int_{\T^3}\hat{\mathcal{H}}_{j-4\cN_0}(\pa^{\al_1}_xf,\pa^{\al_2}_xf,\pa^\al_xf)dx$, $\mathfrak{H}^{2}_j:=\sum_{|\al|=0,2}\int_{\T^3}\hat{\mathcal{H}}_{j-2\cN_0|\al|}(\pa^\al_xf,\mu,\pa^\al_xf)dx$ and $\mathfrak{H}^3_j:=\sum_{|\al|=0,2}\int_{\T^3}\hat{\mathcal{H}}_{j-2\cN_0|\al|}(F,\pa^\al_xf,\pa^\al_xf)dx$.

\underline{\it Step 1: Estimate of $\mathfrak{H}^{1}_j$.}  Thanks to Fourier transform w.r.t. $x$ variable, Lemma \ref{B_3} and \eqref{estiprodx}, we have
\beno
&& \sum_{j=M_0}^\infty \frac{a^{\f4\be (j-7\cN_0)s}}{(\mathbb{G}(\f4\be j))^s}\mathfrak{H}^{1}_j \ls \mathfrak{H}^{1,1}+\mathfrak{H}^{1,2}+\mathfrak{H}^{1,3}+\mathfrak{H}^{1,4}+\mathfrak{H}^{1,5}+\mathfrak{H}^{1,6},\eeno
where   \beno &&\mathfrak{H}^{1,1}:=\sum_{j=M_0}^\infty\frac{a^{\f4\be (j-7\cN_0)s}}{(\mathbb{G}(\f4\be j))^s} \sum_{m=1}^{[\frac{j-5\cN_0}{2}]}(C_{j-4\cN_0-1}^{j-4\cN_0-m-1})^sm^{-(3/2)s}j^{s+1}\| f\<\cdot\>^{2js-4\cN_0s|\al|+\ga/2}\|_{H_x^2L^2_v}\|f\<\cdot\>^{2ms+6}\|_{H_x^2L^2_v}\\&&\times\|f\<\cdot\>^{2(j-m)s-4\cN_0s|\al|+\ga/2-2}\|_{H_x^2L^2_v};\quad \mathfrak{H}^{1,2}:=\sum_{j=M_0}^\infty\frac{a^{\f4\be (j-7\cN_0)}}{(\mathbb{G}(\f4\be j))^s} \sum_{m=1}^{[\frac{j-5\cN_0+1}{2}]}(C_{j-4\cN_0-1}^{j-4\cN_0-m-1})^sm^{-s}j^{\f12s+1}\\
&& \times\|f\<\cdot\>^{2(m-1/2)s+6}\|_{H_x^2L^2_v}\|f\<\cdot\>^{2(j-m+1/2)s-4\cN_0s|\al|+\ga/2-2}\|_{H_x^2L^2_v}\| f\<\cdot\>^{2js-4\cN_0s|\al|+\ga/2}\|_{H_x^2L^2_v};\\
&& \mathfrak{H}^{1,3}:=\sum_{j=M_0}^\infty\frac{a^{\f4\be (j-7\cN_0)s}}{(\mathbb{G}(\f4\be j))^s} \sum_{m=[\frac{j-5\cN_0}{2}]}^{j-5\cN_0+1}(C_{j-4\cN_0-1}^{j-4\cN_0-m-1})^s(j-m)^{-2-(5/2)s+2\cN_0s}j^{2s+3-2\cN_0s}\| f\<\cdot\>^{2ms+\ga/2-8}\|_{H_x^2L^2_v}\\&&\times\|f\<\cdot\>^{2s(j-m)-4\cN_0s|\al|+14}\|_{H_x^2L^2_v}\| f\<\cdot\>^{2js-4\cN_0s|\al|+\ga/2}\|_{H_x^2L^2_v};\quad  \mathfrak{H}^{1,4}:=\sum_{j=M_0}^\infty\frac{a^{\f4\be (j-7\cN_0)s}}{(\mathbb{G}(\f4\be j))^s}\sum_{m=1}^{[\frac{j-5\cN_0+1}{2}]}(C_{j-4\cN_0-1}^{j-4\cN_0-m-1})^s \\
&&\times m^{-(3/2)s}j^{1+s}\|f\<\cdot\>^{2(m+\cN_0)s+6}\|_{H_x^2L^2_v}\|f\<\cdot\>^{2(j-m-\cN_0)s-4\cN_0s|\al|+\ga/2-2}\|_{H_x^2L^2_v}\|f\<\cdot\>^{2js-4\cN_0s|\al|+\ga/2}\|_{H_x^2L^2_v};\\
 &&\mathfrak{H}^{1,5}:=\sum_{j=M_0}^\infty\frac{a^{\f4\be (j-7\cN_0)s}}{(\mathbb{G}(\f4\be j))^s} \sum_{m=[\frac{j-5\cN_0}{2}]}^{j-5\cN_0}(C_{j-4\cN_0-1}^{j-4\cN_0-m-1})^s(j-m)^{-(5/2)s-2+2\cN_0s}j^{2s+3-2\cN_0s}\|f\<\cdot\>^{2(m+\cN_0)s+\ga/2-2}\|_{H_x^2L^2_v}\\
&&\times\|f\<\cdot\>^{2(j-m-\cN_0)s-4\cN_0s|\al|+4}\|_{H_x^2L^2_v}\|f\<\cdot\>^{2js-4\cN_0s|\al|+\ga/2}\|_{H_x^2L^2_v};  \\
&&\mathfrak{H}^{1,6}:=\sum_{j=M_0}^\infty\frac{a^{\f4\be (j-7\cN_0)s}}{(\mathbb{G}(\f4\be j))^s}\sum_{m=[\frac{j-5\cN_0+1}{2}]}^{j-5\cN_0+1}(C_{j-4\cN_0-1}^{j-4\cN_0-m-1})^s(j-m)^{-2s-2+2\cN_0s}j^{\f32s+3-2\cN_0s}\| f\<\cdot\>^{2(m+\cN_0-1/2)s+\ga/2-2}\|_{H_x^2L^2_v}\\
&&\times\| f\<\cdot\>^{2(j-m-\cN_0+1/2)s-4\cN_0s|\al|+4}\|_{H_x^2L^2_v}\| f\<\cdot\>^{2js-4\cN_0s|\al|+\ga/2}\|_{H_x^2L^2_v}.
\eeno
We provide a detailed proof for $\mathfrak{H}^{1,1}$ and $\mathfrak{H}^{1,5}$. The other cases can be treated similarly.

\noindent\underline{\it Estimate of $\mathfrak{H}^{1,1}$.} Thanks to the fact $|\alpha|=2$ and Young's inequality, for any $\vep>0$, we   have
\beno
\mathfrak{H}^{1,1}&\ls&\underbrace{C_\vep\sum_{j=M_0}^\infty\frac{a^{\f4\be (j-7\cN_0)s}}{(\mathbb{G}(\f4\be j))^s}\sum_{m=1}^{[\frac{j-5\cN_0}{2}]}(C_{j-4\cN_0-1}^{j-4\cN_0-m-1})^{2s}m^{-3s}j^{s+2}\|\<\cdot\>^{2ms+6}f\|^2_{H^2_xL^2_v}\|\<\cdot\>^{2(j-m)s-8\cN_0s-2}f\|^2_{H^2_xL^2_{\ga/2}}}_{:=\Delta}\\
&& +\vep\sum_{j=M_0}^\infty j^s\cM(\f4\be j,\ga/\be).
\eeno
  Split $\Delta$ into two parts $\Delta_1$ and $\Delta_2$ which correspond to  the cases $m\in[1,N]$ and the case $m\in(N,[\frac{j-5\cN_0}{2}]]$, where $N>\frac{\f4\be(6\cN_0)-3}{4/\be-2}$,  respectively. We first have
\beno
\Delta_1
&\ls&C_{\vep}\|\<\cdot\>^{2Ns+6}f\|^2_{H^2_xL^2_v}\sum_{m=1}^N\sum_{j=M_0}^\infty j^{2ms+s+2}\cM(\f4\be j,(\ga-4ms-4)/\be).
\eeno
For $\Delta_2$, noticing that $j-m\sim j$ when $m\in(N,[\frac{j-5\cN_0}{2}]]$,  we have
\beno
&&\Delta_2\ls C_\vep\sum_{j=M_0}^\infty\frac{a^{\f4\be (j-7\cN_0)s}}{(\mathbb{G}(\f4\be j))^s}\sum_{m=N}^{[\frac{j-5\cN_0}{2}]}(C_{j-4\cN_0-1}^{j-4\cN_0-m-1})^{2s}m^{-3s}j^{s+2}\|\<\cdot\>^{2ms+6}f\|^2_{H^2_xL^2_v}\|\<\cdot\>^{2(j-m)s-8\cN_0s-2}f\|^2_{H^2_xL^2_{\ga/2}}\\
&&\leq C_\vep\sum_{j=M_0}^\infty\frac{a^{\f4\be (j-7\cN_0)s}}{(\mathbb{G}(\f4\be j))^s}\sum_{m=N}^{[\frac{j-5\cN_0}{2}]}(C_{j-4\cN_0-1}^{j-4\cN_0-m-1})^{2s}m^{-3s}\|\<\cdot\>^{(2m+14\cN_0)s-8\cN_0s-3}f\|^2_{H^2_xL^2_v}(j-m)^{2+s}\\
&&\times\|\<\cdot\>^{2(j-m)s-8\cN_0s-2}f\|_{H^2_xL^2_{\ga/2}}\ls C_{\vep}\sum_{j=M_0}^\infty\sum_{m=N}^{[\frac{j-5\cN_0}{2}]}(C_{j-4\cN_0-1}^{j-4\cN_0-m-1})^{2s}\frac{(\mathbb{G}(\f4\be m))^s(\mathbb{G}(\f4\be (j-m)))^s}{(\mathbb{G}(\f4\be j))^s}m^{-3s}\\
&&\times\frac{\mathbb{G}(\f4\be (m+7\cN_0)))^s}{\mathbb{G}((\f4\be m))^s}\cM(\f4\be (m+7\cN_0),-6/\be)(j-m)^{2+s}\cM(\f4\be(j-m),(\ga-4)/\be).
\eeno
Let us compute  the coefficient
$\Xi:=(C_{j-4\cN_0-1}^{j-4\cN_0-m-1})^{2s}\frac{(\mathbb{G}(\f4\be m))^s(\mathbb{G}(\f4\be (j-m)))^s}{(\mathbb{G}(\f4\be j))^s}\frac{(\f4\be (m+7\cN_0))!)^s}{((\f4\be m)!)^s}m^{-3s}$.
Using facts $j-m\sim j$ and \eqref{gammafun}, we  derive that
$C_{j-4\cN_0-1}^{j-4\cN_0-m-1}=\f{(j-4\cN_0-1)!}{(j-4\cN_0-m-1)!(m)!}\sim \f{(j)!}{(j-m)!(m)!}$ and $\frac{(\mathbb{G}(\f4\be (m+7\cN_0)))^s}{(\mathbb{G}(\f4\be m))^s}\ls j^{\f 4\be (7\cN_0)s}$, which yields that
\beno
&&\Xi\ls\frac{(j!)^{2s}}{(\mathbb{G}(\f4\be j))^s}\frac{(\mathbb{G}(\f4\be m))^s}{(m!)^{2s}}\frac{(\mathbb{G}(\f4\be(j-m)))^s}{((j-m)!)^s}m^{\f4\be(7\cN_0)s-3s}.
\eeno
Since $\mathbb{G}(\f 4\be j)\sim j^{\f12}(\f{\f 4\be j}{e})^{\f 4\be j}\sim j^{\f 2\be}(\f j e)^{\f 4\be j}(\f 4\be)^{\f 4\be j}j^{\f12-\f 2\be}= (j!)^{\f 4\be}(\f 4\be)^{\f 4\be j}j^{\f12-\f 2\be}$ and $1-2/\beta\le0$ (this is the second reason we need $\be<2$), we have
\beno
\frac{(j!)^{2s}}{(\mathbb{G}(\f4\be j))^s}\frac{(\mathbb{G}(\f4\be m))^s}{(m!)^{2s}}\frac{(\mathbb{G}(\f4\be(j-m)))^s}{((j-m)!)^s}\ls\f{(j!)^{(2-\f 4\be)s}}{((j-m)!)^{(2-\f 4\be)s}(m!)^{(2-\f 4\be)s}}=(C_j^m)^{(2-4/\be)s},
\eeno
which  implies that
 $\Xi\ls (C_j^m)^{(2-4/\be)s}j^{\f4\be(7\cN_0)s-3s}\ls j^{(\f4\be(7\cN_0)-3-N(4/\be-2))s}$.
Noting that $N>\frac{\f4\be(7\cN_0)-3}{4/\be-2}$, we  obtain that
$\Delta_2\ls C_{\vep}\sum_{j=M_0}^\infty\sum_{m=N}^{[\frac{j-5\cN_0}{2}]}\cM(\f4\be (m+7\cN_0),-6/\be)(j-m)^{2+s}\cM(\f4\be(j-m),(\ga-4)/\be).$

Then we conclude that
\beno
\mathfrak{H}^{1,1}&\ls&\vep\sum_{j=M_0}^\infty j^s\cM(\f4\be j,\ga/\be)+C_{\vep}\|\<\cdot\>^{2Ns+6}f\|^2_{H^2_xL^2_v}\sum_{m=1}^N\sum_{j=M_0}^\infty j^{2ms+s+2}\cM(\f4\be j,(\ga-4ms-4)/\be)\\
&&+C_{\vep}\sum_{j=M_0}^\infty\sum_{m=N}^{[\frac{j-5\cN_0}{2}]}\cM(\f4\be (m+7\cN_0),-6/\be)(j-m)^{2+s}\cM(\f4\be(j-m),(\ga-4)/\be).
\eeno

\noindent\underline{\it Estimate of $\mathfrak{H}^{1,5}$.} By change of variable from $m$ to $j-m-5\cN_0$, we can derive that
\beno
\mathfrak{H}^{1,5}&\ls&C_\vep\sum_{j=M_0}^\infty\frac{a^{\f4\be (j-7\cN_0)s}}{(\mathbb{G}(\f4\be j))^s}\sum_{m=0}^{[\frac{j-5\cN_0}{2}]}(C_{j-4\cN_0-1}^{m+\cN_0-1})^{2s}(m+5\cN_0)^{-5s-4+4\cN_0s}j^{4s+6-4\cN_0s}\\
&&\times\|f\<\cdot\>^{2(j-m)s-8\cN_0s-2}\|^2_{H^2_xL^2_{\ga/2}}\|f\<\cdot\>^{2ms + 6}\|^2_{H^2_{x,v}}+\vep\sum_{j=M_0}^\infty j^s\cM(\f4\be j,\ga/\be).
\eeno
Observe that $(C_{j-4\cN_0-1}^{m+\cN_0-1})^{2s}\sim (C_{j-4\cN_0-1}^{m})^{2s}(j/m)^{(\cN_0-1)s}$ which implies
\[(C_{j-4\cN_0-1}^{m+\cN_0-1})^{2s}(m+5\cN_0)^{-5s-4+4\cN_0s}j^{4s+6-4\cN_0s}\sim (C_{j-4\cN_0-1}^{m})^{2s}(m+5\cN_0)^{-6s-4+3\cN_0s}j^{3s+6-3\cN_0s}.\]
From this together with the fact $3s+6-4\cN_0s\leq s+2$ since $\cN_0 s>2$, we may copy the argument used for $\mathfrak{H}^{1,1}$ to derive the similar estimate. Finally, for suitably large $N$, we are led to that
\beno
&&\sum_{j=M_0}^\infty\frac{a^{\f4\be (j-7\cN_0)s}}{(\mathbb{G}(\f4\be j))^s}\mathfrak{H}^1_j\ls \vep\sum_{j=M_0}^\infty j^s\cM(\f4\be j,\ga/\be)+C_{\vep}\|\<\cdot\>^{2Ns+6}f\|^2_{H^2_xL^2_v}\Big(\sum_{m=0}^N\sum_{j=M_0}^\infty j^{2ms+s+2}\\
&&\times\cM(\f4\be j,(\ga-4ms-4)/\be)+\sum_{m=0}^N\sum_{j=M_0}^\infty  j^{2ms+2}\cM(\f4\be j,(\ga-4ms+2s-4)/\be)\Big)+C_{\vep}\sum_{j=M_0}^\infty\sum_{m=N}^{[\frac{j-5\cN_0+1}{2}]}\\
&&\cM(\f4\be (m+7\cN_0),-6/\be)\Big((j-m)^{2+s}\cM(\f4\be(j-m),(\ga-4)/\be)+(j-m)^{2}\cM(\f4\be(j-m),(\ga-4+2s)/\be)\Big).
\eeno

\underline{\it Step 2: Estimate of $\mathfrak{H}^{3}_j$.}
Again by Lemma \ref{B_3},  we can bound $\sum_{j=M_0}^\infty \frac{a^{\f4\be (j-7\cN_0)s}}{(\mathbb{G}(\f4\be j))^s}\mathfrak{H}^{3}_j$ by six terms denoted by $\mathfrak{H}^{3,i},i=1,\cdots,6$ respectively which are similar to $\mathfrak{H}^{1,i},i=1,\cdots,6$. In the next, we  only give a detailed estimate for $\mathfrak{H}^{3,1}$. We recall that
\beno
\mathfrak{H}^{3,1}&:=&\sum_{|\al|=0,2}\sum_{j=M_0}^\infty\frac{a^{\f4\be (j-7\cN_0)s}}{(\mathbb{G}(\f4\be j))^s}\int_{\T^3}\sum_{m=1}^{[\frac{j-(2|\al|+1)\cN_0}{2}]}(C_{j-1}^{j-m-1})^sm^{-(3/2)s}j^{s+1}\|\<\cdot\>^{2ms+6}F\|_{H^2_{x,v}}\\
&&\times\|\pa^{\al}_xf\<\cdot\>^{2(j-m)s-4\cN_0s|\al|-2}\|_{L_x^2L^2_{\ga/2}}\|\<\cdot\>^{2js-4\cN_0s|\al|}\pa^\al_xf\|_{L_x^2L^2_{\ga/2}}.
\eeno
One may observe that  $\mathfrak{H}^{3,1}$ enjoys the similar structure as $\mathfrak{H}^{1,1}$. Therefore we can obtain that
\beno
\mathfrak{H}^{3,1}&\ls&\vep\sum_{j=M_0}^\infty j^s\cM(\f4\be j,\ga/\be)+C_{\vep}\|\<\cdot\>^{2Ns+6}F\|^2_{H^2_xL^2_v}\sum_{m=1}^Nj^{2ms+s+2}\cM(\f4\be j,(\ga-4ms-4)/\be)\\
&&+C_{\vep}\sum_{j=M_0}^\infty\sum_{m=N}^{[\frac{j-\cN_0}{2}]}\cM_F(\f4\be (m+7\cN_0),-6/\be)(j-m)^{2+s}\cM(\f4\be(j-m),(\ga-4)/\be).
\eeno
 We conclude that
\beno
&&\sum_{j=M_0}^\infty\frac{a^{\f4\be (j-7\cN_0)s}}{(\mathbb{G}(\f4\be j))^s}\mathfrak{H}^{3}_j\ls \vep\sum_{j=M_0}^\infty j^s\cM(\f4\be j,\ga/\be)+C_{\vep}\|\<\cdot\>^{2Ns+6}F\|^2_{H^2_xL^2_v}\Big(\sum_{m=0}^Nj^{2ms+s+2}\cM(\f4\be j,(\ga-4ms-4)/\be)\\
&&+\sum_{m=0}^Nj^{2ms+2}\cM(\f4\be j,(\ga-4ms-2s-4)/\be)\Big)+C_{\vep}\|\<\cdot\>^{2Ns+6}f\|^2_{H^2_xL^2_v}\Big(\sum_{m=0}^Nj^{2ms+s+2}\cM_F(\f4\be j,(\ga-4ms-4)/\be)\\
&&+\sum_{m=0}^Nj^{2ms+2}\cM_F(\f4\be j,(\ga-4ms+2s-4)/\be)\Big)+C_{\vep}\sum_{j=M_0}^\infty\sum_{m=N}^{[\frac{j-\cN_0+1}{2}]}\cM_F(\f4\be (m+7\cN_0),-6/\be)\Big((j-m)^{2+s}\\
&&\times\cM(\f4\be(j-m),(\ga-4)/\be)+(j-m)^{2}\cM(\f4\be(j-m),(\ga-4+2s)/\be)\Big)+C_{\vep}\sum_{j=M_0}^\infty\sum_{m=N}^{[\frac{j-\cN_0+1}{2}]}\\
&&\cM(\f4\be (m+7\cN_0),-6/\be)\Big((j-m)^{2+s}\cM_F(\f4\be(j-m),(\ga-4)/\be)+(j-m)^{2}\cM_F(\f4\be(j-m),(\ga-4+2s)/\be)\Big).
\eeno

\underline{\it Step 3: Estimate of $\mathfrak{H}^{2}_j$.}  By definition, we know that
 \beno\sum_{j=M_0}^\infty\mathcal{M}_\mu(\f4\be j,c)=\sum_{j=M_0}^\infty \int_{\R^3}\frac{a^{\f4\be (j-7\cN_0)s}(\<v\>^\be)^{\f4\be js+c}}{(\mathbb{G}(\f4\be j))^s}\mu^2(v)dv\ls a^{\f4\be(M_0-7\cN_0)}\leq a.\eeno  Thus we can obtain that
%\beno
%\mathfrak{H}^{2,1}_j&\ls&C_\vep\|\<v\>^{6}\<D_x\>^2f\|^2_{L^2_{x,v}}\sum_{j=M_0}^\infty j^{2+s}\cM_\mu(\f4\be j,(\ga-4)/\be)+\vep\sum_{j=M_0}^\infty j^s\cM(\f4\be j,\ga/\be),
%\eeno
%and
%\beno
%\mathfrak{H}^{2,2}_j&\ls&\|\<v_*\>^{2Ns+6}\<D_x\>^2f\|^2_{L^2_xL^2_v}\sum_{m=1}^Nj^{2ms+s+2}\cM_\mu(\f4\be j,(\ga-4ms-4)/\be)+\vep\sum_{j=M_0}^\infty j^s\cM(\f4\be j,\ga/\be)\\
%&&+\sum_{j=M_0}^\infty\sum_{m=N}^{[\frac{j-3|\al|/s-\cN_0}{2}]}m^{-(12|\al|+12)/\be}\cM_\mu(\f4\be m,(12|\al|+12)/\be)(j-m)^{2+s}\cM(\f4\be(j-m),(\ga-4)/\be)\\
%&&+\sum_{j=M_0}^\infty\sum_{m=N}^{[\frac{j-3|\al|/s-\cN_0}{2}]}m^{-(12|\al|+12)/\be}\cM(\f4\be m,(12|\al|+12)/\be)(j-m)^{2+s}\cM_\mu(\f4\be(j-m),(\ga-4)/\be).
%\eeno
\beno
&&\sum_{j=M_0}^\infty\frac{a^{\f4\be (j-7\cN_0)s}}{(\mathbb{G}(\f4\be j))^s}\mathfrak{H}^{2}_j\ls C_\vep a\|\<\cdot\>^{2Ns+6}f\|^2_{H^2_xL^2_v}+\vep\sum_{j=M_0}^\infty j^s\cM(\f4\be j,\ga/\be)+C_{\vep}\Big(\sum_{m=0}^Nj^{2ms+s+2}\cM(\f4\be j,(\ga-4ms\\
&&-4)/\be)+\sum_{m=0}^Nj^{2ms+2}\cM(\f4\be j,(\ga-4ms-2s-4)/\be)\Big)+C_{\vep}\sum_{j=M_0}^\infty\sum_{m=N}^{[\frac{j-\cN_0+1}{2}]}\cM_\mu(\f4\be (m+7\cN_0),\ga/\be)\\
&&\times\Big((j-m)^{2+s}\cM(\f4\be(j-m),(\ga-4)/\be)+(j-m)^{2}\cM(\f4\be(j-m),(\ga-4+2s)/\be)\Big)+C_{\vep}\sum_{j=M_0}^\infty\sum_{m=N}^{[\frac{j-\cN_0+1}{2}]}\\
&&\cM(\f4\be (m+7\cN_0),-6/\be)\Big((j-m)^{2+s}\cM_\mu(\f4\be(j-m),(\ga-4)/\be)+(j-m)^{2}\cM_\mu(\f4\be(j-m),(\ga-4+2s)/\be)\Big).
\eeno
Patch together all above estimates and recall that $F=\mu+f$, then we complete the proof of this lemma.
\end{proof}

Now we are ready to complete the proof of Theorem \ref{globaldecay}.
\begin{proof}[Proof of Theorem \ref{globaldecay}: Exponential moment] Let $M_0$ be sufficiently large and determined later. By
 (\ref{ene}), we derive that
\beno
\frac{d}{dt}\sum_{j=M_0}^\infty\cM(\f4\be j,0)+c_1\sum_{j=M_0}^\infty j^s\cM(\f4\be j,\ga/\be)+c_2\sum_{j=M_0}^\infty\cM_s(\f4\be j,\ga/\be)
&\ls&\mathfrak{L}+\mathfrak{K}_1+\mathfrak{K}_2,
\eeno
where $\mathfrak{K}_1$ and $\mathfrak{K}_2$ are defined in Lemma \ref{h_j} and
\beno
&&\mathfrak{L}:=\|f\|_{H^2_xL^2_{17}}\Big(\sum_{j=M_0}^\infty\big(\cM_s(\f4\be j,\ga/\be)+j^8\cM_s(\f4\be j,(\ga-16)/\be)+j^s\cM(\f4\be j,\ga/\be)+j^{6+s}\cM(\f4\be j,(\ga-14)/\be)\Big)\\
&&+\sum_{j=M_0}^\infty\big(j^{2}\cM_s(\f4\be j,(\ga-4)/\be)+j^{1+s}\cM(\f4\be j,(\ga-2)/\be)+j^{8+s}\cM(\f4\be j,(\ga-16)/\be)\big)+\|f\|^2_{H^2_xL^2_{17}}.
\eeno
Here $\mathfrak{L}$ corresponds to the summation w.r.t. $\mathfrak{G}_j$ in (\ref{ene}).

\underline{Estimate of $\mathfrak{L}+\mathfrak{K}_1$.}  Thanks to Lemma \ref{M123}(\ref{M2})(\ref{M3}) with  $l\ge1$, we have
\beno
&&\mathfrak{L}+\mathfrak{K}_1\ls  aC_{\be,s} (1 + \|f\|_{H^2_xL^2_{17}})\sum_{j=M_0}^\infty\cM_s(\f4\be j,\ga/\be)+C_{\be,s,\vep}a(\|\<\cdot\>^{2Ns+17}f\|_{H^2_xL^2_{v}}+1)\Big(\sum_{j=M_0}^\infty j^s\cM(\f4\be j,\ga/\be)+\\
&&\|f\|^2_{H^2_xL^2_v}\Big)+\vep\sum_{j=M_0}^\infty j^s\cM(\f4\be j,\ga/\be)+
aC_{M_0,s,\be,\vep} \|f\|^2_{H^2_xL^2_v}+aC_{M_0,s,\be,\vep} \|f\|^2_{H^2_xH^s_v}+C_\vep a\|\<\cdot\>^{2Ns+6}f\|^2_{H^2_xL^2_v}.
\eeno

\underline{Estimate of $\mathfrak{K}_2$.} We split $\mathfrak{K}_2$ into three parts which are as follows:
\beno
&&\mathfrak{K}_{2,1}:=C_{\vep}\sum_{j=M_0}^\infty\sum_{m=N}^{[\frac{j-\cN_0+1}{2}]}\cM(\f4\be (m+7\cN_0),-6/\be)\Big((j-m)^{2+s}\cM_\mu(\f4\be(j-m),(\ga-4)/\be)+(j-m)^{2}\\
&&\times\cM_\mu(\f4\be(j-m),(\ga-4+2s)/\be)\Big);\quad \mathfrak{K}_{2,2}:=C_{\vep}\sum_{j=M_0}^\infty\sum_{m=N}^{[\frac{j-\cN_0+1}{2}]}\cM_\mu(\f4\be (m+7\cN_0),-6/\be)\Big((j-m)^{2+s}\\
&&\times\cM(\f4\be(j-m),(\ga-4)/\be)+(j-m)^{2}\cM(\f4\be(j-m),(\ga-4+2s)/\be)\Big);\mathfrak{K}_{2,3}:=C_{\vep}\sum_{j=M_0}^\infty\sum_{m=N}^{[\frac{j-\cN_0+1}{2}]}\\
&&\cM(\f4\be (m+7\cN_0),-6/\be)\Big((j-m)^{2+s}\cM(\f4\be(j-m),(\ga-4)/\be)+(j-m)^{2}\cM(\f4\be(j-m),(\ga-4+2s)/\be)\Big).
\eeno
Thanks to Lemma \ref{M123}(\ref{M2}) and the fact that summation of $\cM_\mu$ converges, we have
\beno
&&\mathfrak{K}_{2,1}\le C_{\vep}\Big(\sum_{j=M_0-N}^\infty j^{2+s}\cM_\mu(\f4\be j,(\ga-4)/\be)+\sum_{j=M_0-N}^\infty j^{2}\cM_\mu(\f4\be j,(\ga-4+2s)/\be)\Big)\\
&&\times \Big(\sum_{m=N+7\cN_0}^\infty\cM(\f4\be m,-6/\be)\Big)\le aC_\vep\sum_{j=M_0}^\infty\cM(\f4\be j,\ga/\be)+aM_0C_{\vep}\|f\|^2_{H^2_xL^2_v},
\eeno
 For $\mathfrak{K}_{2,2}$, again by Lemma \ref{M123}(\ref{M2}), we have
\beno
&&\mathfrak{K}_{2,2}\leq C_{\vep}\Big(\sum_{j=M_0-N}^\infty j^{2+s}\cM(\f4\be j,(\ga-4)/\be)+\sum_{j=M_0-N}^\infty j^{2}\cM(\f4\be j,(\ga-4+2s)/\be)\Big)\\ &&\times\Big(\sum_{m=N+7\cN_0}^\infty\cM_\mu(\f4\be m,-6/\be)\Big)
\leq aC_\vep\sum_{j=M_0}^\infty j^s\cM(\f4\be j,\ga/\be)+aM_0C_{\vep}\|f\|^2_{H^2_xL^2_v}.
\eeno
  Similarly, for $\mathfrak{K}_{2,3}$, we have
\beno
&&\mathfrak{K}_{2,3}\leq C_{\vep}\Big(\sum_{j=M_0-N}^\infty j^{2+s}\cM(\f4\be j,(\ga-4)/\be)+\sum_{j=M_0-N}^\infty j^{2}\cM(\f4\be j,(\ga-4+2s)/\be)\Big)\\
&&\times \Big(\sum_{m=N+7\cN_0}^\infty\cM(\f4\be m,-6/\be)\Big)\le aC_{\vep}\left(\sum_{j=M_0}^\infty\cM(\f4\be j,\ga/\be)\right)\left(\sum_{j=M_0}^\infty j^s\cM(\f4\be j,\ga/\be)\right)+aM_0C_\vep\|f\|^2_{H^2_xL^2_v}\\
&&\times \sum_{j=M_0}^\infty j^s\cM(\f4\be j,\ga/\be)+aM_0^2C_{\vep}\|f\|^4_{H^2_xL^2_v}.
\eeno

Putting together all the estimate and choosing small $\vep$, the energy inequality  becomes
\beno
&&\frac{d}{dt}\sum_{j=M_0}^\infty\cM(\f4\be j,0)+\left(\frac{c_1}{2}-aC_{\be,s}(M_0\|\<\cdot\>^{2Ns+17}f\|_{H^2_xL^2_v}+1)\right)\sum_{j=M_0}^\infty j^s\cM(\f4\be j,\ga/\be)\\
&&+(c_2-aC_{\be,s} (1 +\|f\|_{H^2_xL^2_{17}}) )\sum_{j=M_0}^\infty\cM_s(\f4\be j,\ga/\be)\leq a\left(\sum_{j=M_0}^\infty\cM(\f4\be j,0)\right)\left(\sum_{j=M_0}^\infty j^s\cM(\f4\be j,\ga/\be)\right)\\
&&+aC_{M_0,s,\be}(\|\<\cdot\>^{2Ns+17}f\|^2_{H^2_xL^2_v}+\|\<\cdot\>^{2Ns+17}f\|^4_{H^2_xL^2_v})+aC_{M_0,s,\be}\|f\|^2_{H^2_xH^s_v}.
\eeno
Thanks to Theorem \ref{globaldecay} for polynomial moment case,  we have $\|f\|^2_{L^\infty([0,\infty);X_{48})\cap L^2([0,\infty);Y_{48})}\ls \vep_0$ and $f\in L^\infty([0,\infty];X_{2Ns+27})  \cap L^2([0,\infty];Y_{2Ns+27})$. Hence the above inequality becomes
\ben\label{357}
&&\notag\frac{d}{dt}\sum_{j=M_0}^\infty\cM(\f4\be j,0)+\left(\frac{c_1}{4}-aCM_0-aC\sum_{j=M_0}^\infty\cM(\f4\be j,0)\right)\sum_{j=M_0}^\infty j^s\cM(\f4\be j,\ga/\be)\\
&&+(\frac{c_2}{2}-aC)\sum_{j=M_0}^\infty\cM_s(\f4\be j,\ga/\be)\leq  aC_{M_0,s,\be}\|f\|^2_{Y_{2Ns+27}}+a\|f\|^2_{Y_{2Ns+27}}.
\een
where the constant $C$ only depend on $\be, s,\vep_0$ and the bound of $f$ in $L^\infty([0,\infty];X_{2Ns+27})  \cap L^2([0,\infty];Y_{2Ns+27})$.
Recall the definition of $\mathcal{X}^\be_{a,M_0}$ (see (\ref{exp space})). Suppose that for some $b\in \R^+$,
\ben\label{fini}
\|f_0\|_{\mathcal{X}^\be_{b,M_0}}<\infty.
\een
Then for any $0<a<b$, we have that
\beno
aC\sum_{j=M_0}^\infty\cM(\f4\be j,0)|_{t=0}\leq Ca^{1-\f 4\be 7\cN_0 s}\|f_0\|^2_{_{\mathcal{X}^\be_{a,M_0}}}\leq C a^{1-\f 4\be 7\cN_0 s}(a/b)^{\f4\be M_0s}\|f_0\|^2_{_{\mathcal{X}^\be_{b,M_0}}}.
\eeno
Thus for fixed large $M_0$, we can derive that $aC\sum_{j=M_0}^\infty\cM(\f4\be j,0)|_{t=0}<\vep$ if $a$ is sufficiently small. Moreover, we have
\beno
\frac{d}{dt}\sum_{j=M_0}^\infty\cM(\f4\be j,0)+\left(\frac{c_1}{8}-aC\sum_{j=M_0}^\infty\cM(\f4\be j,0)\right)\sum_{j=M_0}^\infty j^s\cM(\f4\be j,\ga/\be)\leq aC_{M_0,s,\be}\|f\|^2_{Y_{2Ns+27}}+a\|f\|^2_{Y_{2Ns+27}}.
\eeno
Thus by the standard continuous argument, we obtain that $\|f(t)\|_{_{\mathcal{X}^\be_{a,M_0}}}<\vep,t\geq0$. In particular, we have
\beno
\frac{c_1}{8}-aC\sum_{j=M_0}^\infty\cM(\f4\be j,0)(t)\geq \frac{c_1}{9}, t\geq0.
\eeno
Then for fixed $a,M_0$, we obtain that
\beno
\frac{d}{dt}\sum_{j=M_0}^\infty\cM(\f4\be j,0)+\frac{c_1}{9}\sum_{j=M_0}^\infty j^s\cM(\f4\be j,\ga/\be)\leq C_{a,M_0,s,\be}\|f\|^2_{H^2_xL^2_v}+C_{a,M_0,s,\be}\|f\|^2_{H^2_xH^s_v}.
\eeno
Let $\|f\|^2_{\hat{\mathcal{X}}_{a,M_0}}:=\eta_1\|f\|^2_{\mathcal{X}^\be_{a,M_0}}+\eta\|f\|^2_{X_{48}}+\int_0^\infty\|\mathcal{S}_L(\tau)f\|^2_{H^2_xL^2_v}d\tau$ with $\eta_1$ to be fixed and $\eta,k$ defined in Theorem \ref{th4.1}.  Due to Theorem \ref{th4.1}, there exists $C$ and $K$ such that
\beno
\frac{d}{dt}\|f\|^2_{\hat{\mathcal{X}}_{a,M_0}}+\eta_1\f{c_1}9\sum_{j=M_0}^\infty j^s\cM(\f4\be j,\ga/\be)\leq(-1+\eta_1C_{a,M_0,s,\be})\|f\|^2_{H^2_xL^2_v}+(C\|f\|^2_{X_{48}}-K)\|f\|^2_{Y_{48}}+\eta_1aC_{M_0,s,\be}\|f\|^2_{H^2_xH^s_v}.
\eeno
Choosing suitably small $\eta_1$ and noticing that for small initial value of $f_0$,
 \ben\label{small2}
\|f(t)\|^2_{X_{48}}\ls \vep_0,\quad \forall t\geq0.\,
 \een
the terms on the right hand sides can be cancelled. Thus for fixed $\eta,\eta_1$ and $M_0$, $a$ we choose, we obtain that
\beno
\frac{d}{dt}\|f\|^2_{\hat{\mathcal{X}}_{a,M_0}}+C\sum_{j=M_0}^\infty j^s\cM(\f4\be j,\ga/\be)\leq0.
\eeno
Moreover, we also have
\ben\label{Xaeq}
\|f\|^2_{\hat{\mathcal{X}}_{a,M_0}}\sim_{\eta,\eta_1,\M_0}\|f\|^2_{\mathcal{X}_{a,M_0}}\sim_{\eta,\eta_1,\M_0}\|f\|^2_{\mathcal{X}_a}.
\een
We remark that the initial condition we need are only (\ref{fini}) and (\ref{small2}), which means the smallness assumption only need to imposed on the initial data with polynomial moment.

In what follows, we consider two cases: $\ga+\be s\geq0$ and $\ga+\be s<0$.

\noindent$\bullet$\,\underline{$\ga+\be s\geq0.$} Thanks to Lemma \ref{M123}(\ref{M1}), we have
 $\|f\|^2_{\hat{\mathcal{X}}_{a,M_0}}\leq C_{\eta,\eta_1,a,s,\be}\sum_{j=M_0}^\infty j^s\cM(\f4\be j,\ga/\be).$
Then we have that
\beno
\frac{d}{dt}\|f\|^2_{\hat{\mathcal{X}}_{a,M_0}}+C\|f\|^2_{\hat{\mathcal{X}}_{a,M_0}}\leq0.
\eeno
Thus we obtain that $\|f\|^2_{\hat{\mathcal{X}}_{a,M_0}}\leq Ce^{-Ct}\|f_0\|^2_{\hat{\mathcal{X}}_{a,M_0}},$ which equivalent to $\|f\|^2_{\mathcal{X}_a}\leq Ce^{-Ct}\|f_0\|^2_{\mathcal{X}_a}$ by (\ref{Xaeq}).
\smallskip

\noindent$\bullet$\,\underline{$\ga+\be s<0$.} Again by  Lemma \ref{M123}(\ref{M1}), we can derive that
\ben\label{Xa0}
\frac{d}{dt}\|f\|^2_{\hat{\mathcal{X}}_{a,M_0}}+C\|f\<v\>^{\f\ga2+\f\be2 s}\|^2_{\hat{\mathcal{X}}_{a,M_0}}\leq0.
\een
Moreover, thanks to Prop.\ref{expweight}(\ref{pro1.2}), we can choose some fixed $a_1<a$ satisfies $\mathcal{G}_\be^{a_1}(v)/\mathcal{G}_\be^{a}(v)\leq e^{-C_{a_1,a,s,\be}\<v\>^\be}$ and also derive that
\ben\label{Xb}
\frac{d}{dt}\|f\|^2_{\hat{\mathcal{X}}_{a_1,M_0}}+C\|f\<v\>^{\f\ga2+\f\be2 s}\|^2_{\hat{\mathcal{X}}_{a_1,M_0}}\leq0.
\een
Similar to (\ref{Xaeq}), we have
\ben\label{eqnorm}
\|f\|^2_{\hat{\mathcal{X}}_{a_1,M_0}}\sim \|f\|^2_{\mathcal{X}_{a_1,M_0}}\sim \|f\|^2_{\mathcal{X}_{a_1}}.
\een
Noticing that
\beno
\int_{\R^3}\mathcal{G}_\be^{a_1}(v)f^2dv&=&\int_{|v|\leq R}\mathcal{G}_\be^{a_1}(v)\<v\>^{\ga+\be s}f^2\<v\>^{-(\ga+\be s)}dv+\int_{|v|>R}\mathcal{G}_\be^{a}(v)f^2\mathcal{G}_\be^{a_1}(v)/\mathcal{G}_\be^{a}(v)dv\\
&\leq&\<R\>^{-(\ga+\be s)}\|f\<v\>^{\f\ga 2+\f\be 2 s}\|^2_{L^2_{\mathcal{G}_\be^{a_1}(v)^{1/2}}}+e^{-C_{a,a_1,s,\be}\<R\>^\be}\|f\|^2_{L^2_{\mathcal{G}_\be^{b}(v)^{1/2}}}
\eeno
for any $R>0$.  Together with (\ref{eqnorm}) and the definition
\beno
\|f\|^2_{\mathcal{X}_{a_1}}:=\sum_{|\al|=0,2}\int_{\T^3\times\R^3}\mathcal{G}_\be^{a_1}(v)\<v\>^{-8\cN_0s|\al|}|\pa^\al_xf|^2dvdx,
\eeno
we deduce that
%\beno
%\|f\<v\>^{\f\ga 2+\f\be 2 s}\|^2_{L^2_{\mathcal{G}_\be^{b}(v)^{1/2}}}\geq \<R\>^{(\ga+\be s)}\|f\|^2_{L^2_{\mathcal{G}_\be^{b}(v)^{1/2}}}-\<R\>^{\ga+\be s}e^{-C_{a,b,s,\be}\<R\>^\be}\|f\|^2_{L^2_{\mathcal{G}_\be^{a}(v)^{1/2}}}
%\eeno
\beno
\|f\<v\>^{\f\ga 2+\f\be 2 s}\|^2_{\hat{\mathcal{X}}_{a_1,M_0}}\geq C\<R\>^{(\ga+\be s)}\|f\|^2_{\hat{\mathcal{X}}_{a_1,M_0}}-C\<R\>^{\ga+\be s}e^{-C\<R\>^\be}\|f\|^2_{\hat{\mathcal{X}}_{a,M_0}}.
\eeno
Put it into (\ref{Xb}), we obtain that
\beno
\frac{d}{dt}\|f\|^2_{\hat{\mathcal{X}}_{a_1,M_0}}+C\<R\>^{\ga+\be s}\|f\|^2_{\hat{\mathcal{X}}_{a_1,M_0}}\leq C\<R\>^{\ga+\be s}e^{-C\<R\>^\be}\|f\|^2_{\hat{\mathcal{X}}_{a,M_0}}\leq C\<R\>^{\ga+\be s}e^{-C\<R\>^\be}\|f_0\|^2_{\hat{\mathcal{X}}_{a,M_0}},
\eeno
where we use (\ref{Xa0}) in the last inequality. It yields that
\beno
\|f\|^2_{\hat{\mathcal{X}}_{a_1,M_0}}\leq C(e^{-C\<R\>^{\ga+\be s}t}+e^{-C\<R\>^\be})\|f_0\|^2_{\hat{\mathcal{X}}_{a,M_0}}.
\eeno
By choosing $\<R\>=Ct^{\f1{\be-(\ga+\be s)}}$, we finally obtain that
\beno
\|f\|^2_{\hat{\mathcal{X}}_{a_1,M_0}}\leq Ce^{-Ct^\varrho}\|f_0\|^2_{\hat{\mathcal{X}}_{a,M_0}}
\eeno
with $\varrho=\frac{\be}{\be-(\ga+\be s)}$. Moreover, by the equivalence between those norms, we have
\ben\label{ab}
\|f\|^2_{\mathcal{X}_{a_1}}\leq Ce^{-Ct^\varrho}\|f_0\|^2_{\mathcal{X}_{a}}.
\een

In the above, we need $a_1<a$ small enough, but it dose not matters, since for any $a_2\in(a_1,a)$, we have the interpolation inequality
\beno
\mathcal{G}_\be^{a_2}\leq (\mathcal{G}_\be^{a})^\th(\mathcal{G}_\be^{a_1})^{1-\th}
\eeno
with $a_2=a_1\th+(1-\th)b$, which implies that
\beno
\|f\|^2_{\mathcal{X}_{a_2}}\leq \|f\|^{2\th}_{\mathcal{X}_{a}}\|f\|^{2(1-\th)}_{\mathcal{X}_{b_1}}\leq Ce^{-C\th t^{\varrho}}\|f_0\|^{2}_{\mathcal{X}_{a}}\leq Ce^{-C\th t^{\varrho}}\|f_0\|^2_{\mathcal{X}_a}.
\eeno
Then we can also obtain (\ref{ab}) for any $a_1<a$. It ends the proof of Theorem \ref{globaldecay}.
\end{proof}

\setcounter{equation}{0}
\section{Appendix: Toolbox and Proof of Theorem \ref{le1.24}}
The appendix is devoted to the proof of Theorem \ref{le1.24}. To do that, we need to introduce well-known results on the collision operator, the dyadic decomposition and the application to the commutators.

\subsection{Toolbox}
First we list the well-known results on the collision operators.

\begin{lem}\label{chv} (\cite{ADVW}) If $b\in L^1([0,1])$, then we have  \\
(1) (Regular change of variables)
\[
\int_{\R^3} \int_{\mathbb{S}^2} b(\cos \theta) |v-v_*|^\gamma f(v') d \sigma dv= \int_{\R^3} \int_{\mathbb{S}^2} b(\cos \theta)\frac 1 {\cos^{3+\gamma} (\theta/2)} |v-v_*|^\gamma f(v) d\sigma dv.
\]
(2) (Singular change of variables)
\[
\int_{\R^3} \int_{\mathbb{S}^2} b(\cos \theta) |v-v_*|^\gamma f(v') d \sigma dv_* = \int_{\R^3} \int_{\mathbb{S}^2} b(\cos \theta)\frac 1 {\sin^{3+\gamma} (\theta/2)} |v-v_*|^\gamma f(v_*) d\sigma dv_*.
\]
\end{lem}

\begin{lem}\label{canlemma} (Cancellation Lemma, \cite{ADVW} Lemma 1) For any smooth function $f$, it holds that
$\int_{\R^3 \times \mathbb{S}^{2}} B(v-v_*, \sigma) (f'-f) dv d\sigma= (f*S )(v_*),$
where
\[
S(z) = |\mathbb{S}^1| \int_{0}^{\frac \pi 2} \sin \theta \left[ \frac 1 {\cos^3(\theta/2)} B\left(\frac {|z|}  {\cos (\theta/2)} , \cos \theta \right) - B(|z|, cos \theta) \right].
\]
\end{lem}

\begin{lem}\label{L12}(\cite{H})
Let $w_1, w_2 \in \R$, $a, b \in[0, 2s]$ with $w_1+w_2 =\gamma+2s$  and $a+b =2s$. Then for any smooth functions $g, h, f$ we have \\
(1) if $\gamma + 2s > 0$, then
\[
|(Q(g, h), f)_{L^2_v}| \lesssim(\Vert g \Vert_{L^1_{\gamma+2s +(-w_1)^+(-w_2)^+}}  +\Vert g \Vert_{L^2} ) \Vert h \Vert_{H^a_{w_1}}   \Vert f \Vert_{H^b_{w_2}}.
\]
(2) if $\gamma + 2s = 0$, $w_3 = \max\{\delta,(-w_1)^+ +(-w_2)^+ \}$ with $\delta>0$ sufficiently small, then
\[
|(Q(g, h), f)_{L^2_v}|\lesssim (\Vert g \Vert_{L^1_{w_3}} + \Vert g \Vert_{L^2}) \Vert h \Vert_{H^a_{w_1}}   \Vert f \Vert_{H^b_{w_2}}.
\]
(3) if $\gamma + 2s < 0$, $w_3 = \max\{-(\gamma+2s), \gamma+2s +(-w_1)^+ +(-w_2)^+ \}$, then
\[
|(Q(g, h), f)_{L^2_v}|\lesssim  (\Vert g \Vert_{L^1_{w_4}} + \Vert g \Vert_{L^2_{-(\gamma+2s)}}) \Vert h \Vert_{H^a_{w_1}}   \Vert f \Vert_{H^b_{w_2}}.
\]
\end{lem}

\begin{lem}(\cite{H})\label{L13}
Suppose that $g$ is a non-negative   verifying that
$\Vert g \Vert_{L^1} \ge \delta, \quad \Vert g \Vert_{L^1_2} + \Vert g \Vert_{L \log L} < \lambda$.
Then there exist $C_1(\lambda, \delta)$ and $C_2(\lambda, \delta)$ such that\\
(1) If $\gamma +2s \ge 0$, then
\[
( -Q(g, f) , f)_{L^2_v} \ge C_1\Vert f \Vert_{H^s_{\gamma/2}}^2 - C_2 \Vert f \Vert_{L^2_{\gamma/2}}^2
\]
(2) If $-1-2s < \gamma < -2s$, then
\[
( -Q(g, f) , f)_{L^2_v}  \ge C_1\Vert f \Vert_{H^s_{\gamma/2}}^2 - C_2(1+\|g\|^{4}_{L^2_{|\ga|}})\|f\|^2_{L^2_{\ga/2}}.
\]
\end{lem}

\begin{lem}\label{L116} If $\gamma \in(-3, 1], s \in (0, 1), \gamma+2s>-1$, then for  $g\ge0$ and $f$,
\[
\mathcal{R}  := \int_{\R^3} \int_{\R^3} |v-v_*|^\gamma g_* f^2 dv_* dv \lesssim \Vert g \Vert_{L^2_{|\gamma|+2 }} \Vert f \Vert_{H^s_{ \gamma/2}}^2.
\]
\end{lem}
\begin{proof}
 For $\gamma \ge 0$, it is obvious that
$\mathcal{R} \lesssim \Vert g\Vert_{L^1_{\gamma}} \Vert f \Vert_{L^2_{ \gamma/2}}^2 \lesssim  \Vert g \Vert_{L^2_{|\gamma|+2 }} \Vert f \Vert_{H^s_{ \gamma/2}}^2$. Thus we only need to focus on the case that $-3< \gamma  <0$. Using the fact that
$\langle v \rangle^\gamma \langle v-v_* \rangle^\gamma \lesssim \langle v \rangle^{\gamma}$,
we derive that
\beno
\mathcal{R}  := \int_{\R^3} \int_{\R^3} \frac { \langle v-v_*\rangle^{|\gamma|} } { | v-v_*  |^{|\gamma|  }}  ( g_* \langle v \rangle^{|\gamma|}) (f \langle v \rangle^{\gamma/2} )^2 dv_* dv\lesssim \int_{\R^3} \int_{\R^3} (1+ | v-v_*  |^{\gamma } )  ( g_* \langle v \rangle^{|\gamma|}) (f \langle v \rangle^{\gamma/2} )^2 dv_* dv :=\mathcal{R_1} +\mathcal{R_2}.
\eeno

We first   have
$\mathcal{R_1} \lesssim \Vert g\Vert_{L^1_{|\gamma|}} \Vert f \Vert_{L^2_{ \gamma/2}}^2 \lesssim  \Vert g \Vert_{L^2_{|\gamma|+2 }} \Vert f \Vert_{H^s_{ \gamma/2}}^2$.
For $\mathcal{R}_2$,   by Hardy-Littlewood-Sobolev inequality, we have
$\mathcal{R}_2 \lesssim \Vert g\Vert_{L^p_{|\gamma|}} \Vert f \Vert_{L^{2q}_{ \gamma/2}}^2$,  where $p, q \in (1, +\infty)$ and $1/p+1/q =2+\gamma/3$.
By Sobolev embedding theorem,  it holds that
$\Vert f \Vert_{L^{2q}}  \le \Vert f \Vert_{H^s}$, with  $q \in [1, 3/(3-2s)]$.
 Thus if $\gamma<0$ and $\gamma+2s \ge 0$, taking $1/p=1/q=1+\gamma/6$,
we easily get that $ p \in (1, 2)$ and $1/q \ge 1+ \gamma/3 \ge (3-2s)/3$. In this case, we have
\[
\mathcal{R}_2 \lesssim \Vert g\Vert_{L^p_{|\gamma|}} \Vert f \Vert_{L^{2q}_{ \gamma/2}}^2 \lesssim  \Vert g \Vert_{L^2_{|\gamma|+2 }} \Vert f \Vert_{H^s_{ \gamma/2}}^2.
\]
In the case of $-1 < \gamma+2s <0$, taking $1/p = 1 +(\gamma+2s)/3$ and $1/q = 1-2s/3$, we get that $p \in (1, 2)$ and
\[
\mathcal{R_2} \lesssim \Vert g\Vert_{L^p_{|\gamma|}} \Vert f \Vert_{L^{2q}_{ \gamma/2}}^2 \lesssim  \Vert g \Vert_{L^2_{|\gamma|+2 }} \Vert f \Vert_{H^s_{ \gamma/2}}^2.
\]
The theorem follows by putting together all the estimates.
\end{proof}

\begin{lem}\label{L110}For  smooth functions $f, g, h$,  if $\gamma +2s>-1$ and $b\in L^1([0,1])$, then
\[
\int_{\R^6\times\mathbb{S}^2}b(\cos \theta) |v-v_*|^\gamma f_* g h' dv dv_* d\sigma\le \int_{\mathbb{S}^2} b(\cos \theta) \sin^{-3/ 2-\gamma/2} \frac \theta 2   d\sigma \|g\|_{L^2_{|\gamma|+2}}\|f\|_{H^s_{\gamma/2}}\|h\|_{H^s_{\gamma/2}},
\]
\[
\int_{\R^6\times\mathbb{S}^2}b(\cos \theta) |v-v_*|^\gamma f_* g h' dv dv_* d\sigma
\le \int_{\mathbb{S}^2} b( \cos \theta) \cos^{-3/ 2-\gamma/2} \frac \theta 2   d\sigma
\|f\|_{L^2_{|\gamma|+2}}\|g\|_{H^s_{\gamma/2}}\|h\|_{H^s_{\gamma/2}}.
\]
\end{lem}
\begin{proof} We only give the detailed proof for  the first inequality since the second one can be derived in the same manner.
By singular change of variable in Lemma \ref{chv} and Lemma \ref{L116}, we have
\beno
&& \int_{\R^6\times\mathbb{S}^2}b(\cos \theta) |v-v_*|^\gamma f_* g h' dv dv_* d\sigma  \le\big(\int_{\R^6\times\mathbb{S}^2} b(\cos \theta) \sin^{-\f32-\f\gamma2} \frac \theta 2  |v-v_*|^\gamma |f_*|^2 |g| dv dv_*  d\sigma\big)^\f12\big(
\int_{\R^6\times\mathbb{S}^2}   |g|
\\
&&\times b(\cos \theta)\sin^{ \f32+\f\gamma2} \frac \theta 2 |v-v_*|^\gamma ||h'|^2 dv dv_*  d\sigma\big)^{\f12}
\le \big(\int_{\mathbb{S}^2} b(\cos \theta) \sin^{-\f32-\f\gamma2} \frac \theta 2 d\sigma\big)\|g\|_{L^2_{|\gamma|+2}}\|f\|_{H^s_{\gamma/2}}\|h\|_{H^s_{\gamma/2}}.
\eeno This ends the proof of the lemma.
\end{proof}

\begin{lem}\label{L112}(\cite{V}, Section 1.4)(Pre-post collisional change of variable)
For any $f$ smooth enough, we have
\[
\int_{\R^3}\int_{\R^3} \int_{\mathbb{S}^2 }F(v, v_*, v', v_*') B(|v-v_*|, \cos \theta) dv dv_* d \sigma =\int_{\R^3}\int_{\R^3} \int_{\mathbb{S}^2 }F(v', v_*', v, v_*) B(|v-v_*|, \cos \theta) dv dv_* d \sigma
\]
As a consequence, we have
$\int_{\R^3} Q(f, g) h = \int_{\R^3}\int_{\R^3} \int_{\mathbb{S}^2 } B(v-v_*, \sigma) f_* g (h'-h) dv dv_*d\sigma$.
\end{lem}

\begin{lem}\label{L28} For any smooth function $g$ and $f$, if $\gamma > -3$ and $l := \max \{\gamma+2, 2 \}$, then we have
\[
\int_{\R^3} \int_{\R^3} |v-v_*|^\gamma g_* f dv_* dv \le C\Vert g \Vert_{L^2_{ l}} \Vert f \Vert_{L^2_{ l}}.
\]
\end{lem}
\begin{proof}
In the case of $\gamma\ge 0$, it is clear that
$\int_{\R^3} \int_{\R^3} |v-v_*|^\gamma g_* f dv_* dv\lesssim \Vert f \Vert_{L^{1}_{\gamma}}\Vert g \Vert_{L^{1}_{\gamma}}\lesssim \Vert f \Vert_{L^{2}_{\gamma+2}}\Vert g \Vert_{L^{2}_{\gamma+2}}$.
For the case $-3< \gamma < 0$, by Hardy-Littlewood-Sobolev inequality, i.e. $1/p =1/2(2 +\gamma/3) \in (\frac 1 2, 1)$, we have
\[
\int_{\R^3} \int_{\R^3} |v-v_*|^\gamma g_* f dv_* dv  \lesssim \Vert f \Vert_{L^{p}} \Vert g \Vert_{L^{p}}\lesssim \Vert f \Vert_{L^{2}_2}\Vert g \Vert_{L^{2}_{2}},
\]
 where we use the fact that $p \in (1, 2)$. It ends the proof of the lemma.
\end{proof}

We  recall the estimates for the linearized Boltzmann operator with weight $\mu^{-1/2}$.

\begin{lem}(see \cite{MS}, Theorem 1.1, \cite{DLSS}, Lemma A.3, \cite{DLYZ2}, Lemma 2.6)\label{624}
Let $L_{\mu}$ be given by (\ref{Lmu}), for any smooth function $f$ we have
\[
(L_{\mu}f, f)_{L^2}\leq -C(\|f\|^2_{H^s_{\ga/2}}+\|f\|^2_{L^2_{s+\ga/2}}), \quad P_\mu f =0
\]
where
$P_\mu g = \sum_{\varphi \in \{1,v_1,v_2,v_3,|v|^2\}} \left(\int_{\R^3 } g \varphi \sqrt{\mu} dv  \right)\varphi \sqrt{\mu}$.
Moreover, for $-1<\gamma+2s<0$ and $q>0$, for any smooth function $f$, we have
\begin{equation}
\label{L mu estimate soft}
(L_{\mu}f,e^{q\<v\>}f)_{L^2}\leq -C_q\|e^{\frac{q}{2}\<v\>}f\|^2_{L^2_{s+\ga/2}}+C\|f\|_{L^2_{B_R}}^2,
\end{equation}
where $C_q,C>0$ and $B_R$ denotes the closed ball in $\R^3_v$ with center zero and radius  $R>0$.
\end{lem}

From this, we deduce that

\begin{lem}\label{SL}
Let $L=-Q(\mu,\cdot)-Q(\cdot,\mu)$ and $\mathcal{S}_L$ is the semigroup of $-L$, then we have
\[
\|\Pi^\perp  \mathcal{S}_L(t)\|_{H^2_x L^2(\mu^{-(1/2+\varepsilon)})\rightarrow H^2_x L^2(\mu^{-1/2})}\leq C e^{-\lambda t^b}, \quad \Pi^\perp = I - \Pi
\]
for some $\lambda >0$, where  $b=1, \epsilon =0$ if $\gamma+2s \ge 0$ and $b\in(0,1), \epsilon > 0$ if $-1 < \gamma+2s < 0$.
\end{lem}
\begin{proof} For the case $\gamma+2s \ge 0$ it can be proved similarly as Theorem 8.4 in \cite{GS}. For the case $-1<\gamma+2s <0$, it can be proved by using \eqref{L mu estimate soft} and  the interpolation as \eqref{ab}.
\end{proof}

\subsection{Dyadic decompositions} We will introduce two
 types of the dyadic decomposition in phase and frequency spaces. Let us first list some basic knowledge on the Littlewood-Paley decomposition. Let $B_{\frac{4}{3}}:=\{\xi\in\R^3||\xi|\leq\frac{4}{3}\}$ and $C:=\{\xi\in\R^3||\frac{3}{4}\leq|\xi|\leq\frac{8}{3}\}$. Then one way may introduce two radial functions $\psi\in C_0^\infty(B_{\frac{4}{3}})$ and $\vphi\in C_0^\infty(C)$ which satisfy
\ben
   \label{7.1}\psi,\vphi\geq0,~~&and&~~\psi(\xi)+\sum_{j\geq0}\vphi(2^{-j}\xi)=1,~\xi\in\R^3,\\
   \notag|j-k|\geq2&\Rightarrow& \rm{Supp}~\vphi(2^{-j}\cdot)\cap \rm{Supp}~\vphi(2^{-k}\cdot)=\emptyset,\\
   \notag j\geq1&\Rightarrow& \rm{Supp}~\psi\cap \rm{Supp}~\vphi(2^{-j}\cdot)=\emptyset.
\een

We first introduce the dyadic decomposition in the phase space. The dyadic operator in the phase space $\cP_j$ can be defined as
\ben
  \label{7.2}\cP_{-1}f(x):=\psi(x)f(x),~\cP_jf(x):=\vphi(2^{-j}x)f(x),~j\geq0.
\een
Let $\tP_lf(x)=\sum\limits_{|k-l|\leq N_0}\cP_kf(x)$ and $\U_jf(x)=\sum\limits_{k\leq j}\cP_kf(x)$ where $N_0\geq2$ will be chosen in the later which verifies $\cP_j\cP_k=0$ if $|j-k|\geq N_0$. For any smooth function $f$, we have $f=\cP_{-1}f+\sum\limits_{j\geq0}\cP_jf.$

Next we introduce the dyadic decomposition in the frequency space. We denote $\tm\eqdefa\cF^{-1}\psi$ and $\tphi\eqdefa\mathcal{F}^{-1}\vphi$ where they are the inverse Fourier transform of $\vphi$ and $\psi$. If we set $\tphi_j(x)=2^{3j}\tphi(2^jx)$, then the dyadic operator in the frequency space $\F_j$ can be defined as follows
\begin{align*}
  \F_{-1}f(x):=\int_{\R^3}\tm(x-y)f(y)dy,~\F_jf(x):=\int_{\R^3}\tphi_j(x-y)f(y)dy,~j\geq0.
\end{align*}
Let $\tF_jf(x)=\sum_{|k-j|\leq 3N_0}\F_kf(x)$ and $S_jf(x)=\sum_{k\leq j}\F_kf$. Then for any $f\in \mathscr{S}'$, it holds $f=\F_{-1}f+\sum_{j\geq0}\F_jf.$

\bigskip

 To simplify the above notations, we recall the definition of symbol $S^m_{1,0}$ and pseudo-differential operator:
\begin{defi}\label{de2.1}
 A smooth function $a(v,\xi)$ is said to be a symbol of type  $S^m_{1,0}$ if $a(v,\xi)$ verifies for any multi-indices $\alpha$ and $\beta$,
\beno
 |(\partial_\xi^\alpha\pa_v^\beta a)(v,\xi)|\leq C_{\alpha,\beta}\<\xi\>^{m-|\al|},
\eeno
where $C_{\alpha,\beta}$ is a constant depending only on $\alpha$ and $\beta$. $a(x,D)$ is called a pseudo-differential operator with the symbol $a(x,\xi)$ and it is defined as
\beno
(a(x,D)f)(x):=\frac{1}{(2\pi)^3}\int_{\R^3}\int_{\R^3}e^{i(x-y)\xi}a(x,\xi)f(y)dyd\xi.
\eeno
\end{defi}

\begin{defi}\label{Fj} Let  $\alpha=(\alpha_1,\alpha_2,\alpha_3,),|\alpha|:=\alpha_1+\al_2+\al_3$ and $\vphi_\alpha:=(\frac{1}{i}\pa_{x_1})^{\alpha_1}(\frac{1}{i}\pa_{x_2})^{\alpha_2}(\frac{1}{i}\pa_{x_3})^{\alpha_3}\vphi$. Then we define  $\cP_{j,\alpha}$, $\F_{j,\alpha}$ and $\hat{\F}_{j,\al}$ as
\beno
&&\cP_{-1,\alpha}f:=\psi_{\alpha}f,\quad\quad\quad\quad \cP_{j,\alpha}f:=\vphi_\alpha(2^{-j}\cdot)f,\quad j\geq0;\\
&&\F_{-1,\alpha}f:=\psi_{\alpha}(D)f,\quad\quad\quad \F_{j,\alpha}:=\vphi_\alpha(2^{-j}D)f,\quad j\geq0;\\
&&\hat{\F}_{-1,\alpha}f:=(\psi^2)_{\alpha}(D)f,\quad\quad \hat{\F}_{j,\alpha}:=(\vphi^2)_\alpha(2^{-j}D)f,\quad j\geq0.
\eeno
Similar to $\tP_j$   and $\tF_j$, we can also define
  $$\tP_{l,\alpha}=\sum_{|k-l|<N_0}\cP_{k,\alpha},\quad \U_{j,\alpha}=\sum_{k\leq j}\cP_{k,\alpha}, \quad \tF_{j,\alpha}=\sum_{|k-j|<3N_0}\F_{k,\alpha}.$$

  Finally we define the special localized operators $\mF_j$   and $\mP_j$ which fulfill the following conditions:
  \begin{enumerate} \item The support of the Fourier transform of $\mF_jf$ and the support of $\mP_jf$  will be localized in the annulus $\{|\cdot|\sim 2^j\}$;
  \item It hold that for any fixed $N\in\N$, $\|\tF_j f\|_{L^2}+\sum\limits_{|\alpha|\le N} (\|\tF_{j,\alpha}f\|_{L^2}+\|\hat{\F}_{j,\alpha}f\|_{L^2})\le C_N\|\mF_jf\|_{L^2}$ and $\|\tP_j f\|_{L^2}+\sum\limits_{|\alpha|\le N} \|\tP_{j,\alpha}f\|_{L^2}\le C_N\|\mP_jf\|_{L^2}$.
  \end{enumerate}
\end{defi}
\begin{rmk}  The definitions of the localized operators stem from the estimate for the commutator $[\cP_k, \F_j]$ thanks to   Lemma \ref{le1.2}. While the introduction of $\mF_j$ is just to simply  the presentation of $\tF_j, \tF_{j,\alpha}$ and $\hat{\F}_{j,\alpha}$.
\end{rmk}

Before going further, we recall some basic results.

\begin{lem}[see \cite{H} ]\label{le1.1}
Let $s, r\in\R$ and $a(v),b(\xi)\in C^\infty$ satisfy for any $\alpha\in\Z^3_+$,
\ben\label{abconstants}
 |D_v^\al a(v)|\leq C_{1,\al}\<v\>^{r-|\al|},~|D_\xi^\al b(\xi)|\leq C_{2,\al}\<\xi\>^{s-|\al|}
\een
for constants $C_{1,\al},C_{2,\al}$. Then there exists a constant $C$ depending only on $s,r$ and finite numbers of $C_{1,\al},C_{2,\al}$ such that for any $f\in \mathscr{S}(\R^3)$,
\beno
 \|a(v)b(D)f\|_{L^2}\leq C\|\<D\>^s\<v\>^rf\|_{L^2},~\|b(D)a(v)f\|_{L^2}\leq C\|\<v\>^r\<D\>^sf\|_{L^2}.
\eeno
As a direct consequence, we get that $\|\<D\>^m\<v\>^lf\|_{L^2}\sim\|\<v\>^l\<D\>^mf\|_{L^2}\sim|f|_{H^m_l}$.
\end{lem}

\begin{lem}[see \cite{H}]\label{le1.2}
Let $l,s,r\in\R,M(\xi)\in S_{1,0}^r$ and $\Phi(v)\in S_{1,0}^l$. Then there exists a constant $C$ such that $|[M(D_v),\Phi(v)]f|_{H^s}\leq C|f|_{H_{l-1}^{r+s-1}}$. Moreover, for any $N\in\N,$
\ben\label{MPHICOMMU}
M(D_v)\Phi=\Phi M(D_v)+\sum_{1\leq|\al|\leq N}\frac{1}{\al!}\Phi_\al M^\al(D_v)+r_N(v,D_v),
\een
where $\Phi_\al(v)=\pa_v^\al\Phi,~M^\al(\xi)=\pa_\xi^\al(\xi)$ and $\<v\>^{N-l}r_N(v,\xi)\in S^{r-N}_{1,0}$. Moreover, for any $\beta,\beta'\in \in\Z^3_+$, we have
\ben\label{rN}
\pa^\beta_v\pa^{\beta'}_\xi r_N(v,\xi)\leq C_{\beta,\beta'}\<\xi\>^{r-N-|\beta|}\<v\>^{l-N-|\beta'|}.
\een
Furthermore, use (\ref{MPHICOMMU}) repeatedly, we can also obtain that
\ben\label{MPHICOMMU2}
M(D_v)\Phi=\Phi M(D_v)+\sum_{1\leq|\al|\leq N}C_{\al} M^\al(D_v)\Phi_\al+C_Nr_N(v,D_v).
\een
\end{lem}
\begin{proof}
The proof of (\ref{MPHICOMMU}) and (\ref{rN}) can be found in  (\cite{H}). Here, we only prove (\ref{MPHICOMMU2}). Indeed, we need to change the order of operators $\Phi_\al$ and $M^\al(D_v)$. For $|\al|=1$, due to (\ref{MPHICOMMU}), we have
\beno
\Phi_\al M^\al(D_v)=M^\al(D_v)\Phi_\al+C_\be\sum_{1\leq|\be|\leq N-|\al|}\Phi_{\al+\be}M^{\al+\be}(D_v)+r_N(v,D_v),
\eeno
where $r_{N}(v,\xi)$ also satisfies (\ref{rN}). Then we can derive that
\beno
M(D_v)\Phi=\Phi M(D_v)+\sum_{|\al|=1}C_{\al} M^\al(D_v)\Phi_\al+\sum_{2\leq|\al|\leq N}C_{\al} \Phi_\al M^\al(D_v)+C_Nr_N(v,D_v).
\eeno
Then we can obtain the desired results (\ref{MPHICOMMU2}) by induction.
\end{proof}

\begin{rmk}\label{CONSTS} We emphasize that in the statement of Lemma \ref{le1.2}, the constant $C$ appearing in the inequality depends only on $C_{1,\al},C_{2,\al}$ in \eqref{abconstants} with $a=\Phi$ and $b=M$ and also the constants $ C_{\beta,\beta'}$ for $r_N(v,\xi)$. This fact is crucial for the estimates of commutators and the profiles of weighted Sobolev spaces. For instance,  if $M(D_v)$ and $\Phi(v)$ are chosen to be the localized operators $\F_j$ and $\cP_k$, the constant $C$ in Lemma \ref{le1.2} does not depend on $j$ and $k$. Indeed,
 for any $k\geq0,N\in\N$, $2^{Nk}\vphi(2^{-k}v)$ satisfies that for any $\alpha\in\Z^3_+$,
   \ben\label{Ncon}
   |D_v^\al 2^{Nk}\vphi(2^{-k}v) |\leq C_{N,\al}\<v\>^{N-|\al|}|\vphi_\al(2^{-k}v)|\leq C_{N,\al}\<v\>^{N-|\al|}.
   \een
Moreover, in this case, the term $r_N(v,\xi)$ appearing in \eqref{MPHICOMMU} satisfies
\[\pa_v^\be\pa_\xi^{\be'}r_{2N+1}(v,\xi)\leq C_{\be,\be'}\<\xi\>^{-N-1-|\beta'|}\<v\>^{-N-1-|\be|}.\]
It is clear that  $C_{N,\al}$ and $C_{\be,\be'}$ do not depend on $j$ and $k$.
\end{rmk}

\begin{lem}\label{7.8}
(Bernstein inequality). There exists a constant $C$ independent of $j$ and $f$ such that

(1) For any $s\in\R$ and $j\geq 0$,
\beno
  C^{-1}2^{js}\|\F_jf\|_{L^2(\R^3)}\leq\|\F_jf\|_{H^s(\R^3)}\leq C 2^{js}\|\F_jf\|_{L^2(\R^3)}.
\eeno

(2) For integers $j,k\geq0$ and $p,q\in[1,\infty],q\geq p$, the Bernstein inequality are shown as
\beno
 \sup_{|\al|=k}\|\pa^\al\F_jf\|_{L^q(\R^3)}\ls2^{jk}2^{3j(\frac{1}{p}-\frac{1}{q})}\|\F_jf\|_{L^p(\R^3)},\\
  \sup_{|\al|=k}\|\pa^\al S_jf\|_{L^q(\R^3)}\ls2^{jk}2^{3j(\frac{1}{p}-\frac{1}{q})}\|S_jf\|_{L^p(\R^3)},\\
 2^{jk}\|\F_jf\|_{L^p(\R^3)}\ls\sup_{|\al|=k}\|\pa^\al\F_jf\|_{L^p(\R^3)}\ls2^{jk}\|\F_jf\|_{L^p(\R^3)}.
\eeno
\end{lem}

\begin{lem}\label{lemma1.3}
If $\cP_k,\U_k$ and $\F_j$ are defined in in Definition \ref{de2.1} and $n\in\R^+$, then

 (i) For any $N\in \N$, there exists a constant $C_{N}$ such that
 \beno
 \|[\cP_k,\F_j\<D\>^\ell ]f\|_{L^2}&=&\|(\cP_k\F_j\<D\>^\ell -\F_j\<D\>^\ell \cP_k)f\|_{L^2}\leq C_{N,\ell } \left( 2^{(\ell -1)j}2^{-k}\sum_{|\al|=1}^{2N}\|\cP_{k,\al}\F_{j,\al}f\|_{L^2}+2^{-jN}2^{-kN}\|f\|_{H_{-N}^{\ell -N}}\right),\\
  \|[\U_k,\F_j\<D\>^\ell ]f\|_{L^2}&=&\|(\U_k\F_j\<D\>^\ell -\F_j\<D\>^\ell \U_k)f\|_{L^2}\leq C_{N,\ell } \left( 2^{(\ell -1)j}\sum_{|\al|=1}^{2N}\|\U_{k,\al}\F_{j,\al}f\|_{L^2}+2^{-jN}\|f\|_{H_{-N}^{\ell -N}}\right),\\
  %|\U_{k+N_0}f|_{H^a}&\ls& |f|_{H^a},
 \eeno
 where $\cP_{k,\al},\F_{j,\al}$ and $\U_{k,\al}$ are defined in Definition \ref{de2.1}. Moreover, replace $\cP_{k,\al}$ and $\F_{j,\al}$ by $\tP_{k,\al}$ and $\tF_{j,\al}$ respectively, the above results still hold.

(ii) For $|m-p|>N_0$ and $\forall N\in \N$, there exists a constant $C_N$ such that
 \beno
 \|\F_m\cP_k\F_pg\|_{L^2}\leq C_N2^{-(p+m+k)N}\|\F_pg\|_{L^2_{-N}},\quad \|\F_m\U_{k}\F_pg\|_{L^2}\leq C_N2^{-(p+m)N}\|\F_pg\|_{L^2_{-N}}.
 \eeno
If $m>p+N_0$, we have
\beno
 \|\F_m\cP_{k}S_pg\|_{L^2}\leq C_N2^{-(m+k)N}\|S_pg\|_{L^2_{-N}},\quad
  \|\F_m\U_{k}S_pg\|_{L^2}\leq C_N2^{-mN}\|S_pg\|_{L^2_{-N}}.
\eeno

 (iii) For any $a,w\in \R$, we have
 $\|\U_{k+N_0}h\|_{H^a}\leq C_{a,w}2^{k(-w)^+}\|h\|_{H^a_w}$ and
  $\|S_{p+N_0}h\|_{L^2_l}\leq C_{l}\|h\|_{L^2_l}.$

 (iv) For any $j\geq -1$ and $l\in\R$, we have
$ \|\F_jf\|_{L^1_l}+\|S_jf\|_{L^1_l}\leq C_l\|f\|_{L^1_l}.$
 \end{lem}
\begin{proof}We first address that  all the constants derived in the below are universal thanks to {\bf Remark} \ref{CONSTS}. 

(i) If $k\geq 0,j\geq 0$,
   due to Lemma \ref{le1.2}(\ref{MPHICOMMU2}), we have
   \beno
   &&\|(\cP_k\F_j\<D\>^\ell -\F_j\<D\>^\ell \cP_k)f\|_{L^2}=2^{-jN}2^{-kN}\|(2^{Nk}\cP_k2^{Nj}\F_j\<D\>^\ell -2^{Nj}\F_j\<D\>^\ell 2^{Nk}\cP_k)f\|_{L^2}\\
   &\leq&C_{N,\ell }2^{-Nj}2^{-Nk}\sum_{1\leq|\al|\leq 2N}2^{(N-|\al|)k}2^{(N-|\al|)j}2^{lj}\|\vphi_\al(2^{-j}D)\vphi_\al(2^{-k}v)f\|_{L^2}+C_{N,\ell }2^{-Nj}2^{-Nk}\|r_{2N+1}(v,D_v)f\|_{L^2}\\
   &\leq& C_{N,\ell }2^{(l-1)j}2^{-k}\sum_{1\leq|\al|\leq 2N}\|\vphi_\al(2^{-j}D)\vphi_\al(2^{-k}v)f\|_{L^2}+C_{N,\ell }2^{-Nj}2^{-Nk}\|r_{2N+1}(v,D_v)f\|_{L^2}\\
   &\leq& C_{N,n}2^{(\ell -1)j}2^{-k}\sum_{1\leq|\al|\leq 2N}\|\vphi_\al(2^{-k}v)\vphi_\al(2^{-j}D)f\|_{L^2}+C_{N,\ell }2^{-Nj}2^{-Nk}\|r_{2N+1}(v,D_v)f\|_{L^2},
   \eeno
    where we commute the operator $\vphi_\al(2^{-k}v)$ and $\vphi_\al(2^{-j}D)$ again by Lemma  \ref{le1.2}(\ref{MPHICOMMU}). Since
    \ben\label{rNN}
    \pa_v^\be\pa_\xi^{\be'}r_{2N+1}(v,\xi)\leq C_{N,\be,\be'}\<\xi\>^{\ell -N-1-|\beta'|}\<v\>^{-N-1-|\be|},
    \een
  it remains
     to prove that $\|r_{2N+1}(v,D_v)f\|_{L^2}\leq C\|f\|_{H^{\ell -N}_{-N}}$, which is equivalent to $\|r_{2N+1}(v,D_v)\<D\>^{N-\ell }\<v\>^{N}f\|_{L^2}\leq C\|f\|_{L^2}$. By the fundamental theorem for the algebra of pseudo-differential operators(see \cite{HKgo}),  the symbol of operator $r_{2N+1}(v,D_v)\<D\>^{N-\ell }\<v\>^{N}$ is
     \beno
     \mathcal{r}(v,\xi)=\mathrm{Os}-\f1{(2\pi)^3}\int_{\R^6}e^{-iu\cdot\eta}r(v,\xi+\eta)\<v+u\>^Ndud\eta,
     \eeno
where $r(v,\xi)=r_{2N+1}(v,\xi)\<\xi\>^{N-\ell }$ and $``\mathrm{Os-}"$ means the oscillating integral. By the boundness of pseudo-differential operator in $L^2(\R^3)$, we need to prove $\mathcal{r}(v,\xi)\in S^0_{1,0}$, that is $|\pa_v^\al\pa^\be_\xi \mathcal{r}(v,\xi)|\leq C_{\al,\be}\<\xi\>^{-|\be|}$. Using the identities
\beno
e^{-iu\cdot\eta}=\<\eta\>^{-2l}(1-\Delta_u)^le^{-iu\cdot\eta},\quad e^{-iu\cdot\eta}=\<u\>^{-2k}(1-\Delta_\eta)^ke^{-iu\cdot\eta},
\eeno
we have for $l,k\in\N$ with $l>|\be|/2+3/2$ and $k>N+3/2$,
\beno
\pa_v^\al\pa^\be_\xi\mathcal{r}(v,\xi)=\f1{(2\pi)^3}\sum_{\al_1+\al_2=\al}C^{\al_1}_{\al}\int\Big(\int e^{-iu\cdot \eta}\<u\>^{-2k}(1-\Delta_\eta)^k\{\<\eta\>^{-2l}(1-\Delta_u)^l\pa^{\al_1}_v\pa^\be_\xi r(v,\xi+\eta)\pa^{\al_2}_v(\<\cdot\>^N)(v+u)\}d\eta\Big)du.
\eeno
Let take one of these terms, we have
\beno
&&\int\{(1-\Delta_u)^l\pa^{\al_2}_v(\<\cdot\>^N)(v+u)\}\Big(\int e^{-iu\cdot \eta}(1-\Delta_\eta)^k\{\<\eta\>^{-2l}\pa^{\al_1}_v\pa^\be_\xi r(v,\xi+\eta)\}d\eta\Big)\f{du}{\<u\>^{2k}}\\
&=&\int\{(1-\Delta_u)^l\pa^{\al_2}_v(\<\cdot\>^N)(v+u)\}\Big(\int_{|\eta|\leq\f{|\xi|}2}+\int_{|\eta|>\f{|\xi|}2}\Big)\f{du}{\<u\>^{2k}}\\
&:=&\int\{(1-\Delta_u)^l\pa^{\al_2}_v(\<\cdot\>^N)(v+u)\}\Big(I_1(v,\xi;u)+I_2(v,\xi;u)\Big)\f{du}{\<u\>^{2k}}.
\eeno
Since $\<\xi\>$ and $\<\xi+\eta\>$ are equivalent in $I_1$, it follows from (\ref{rNN}) that
\beno
|I_1|\leq C\<\xi\>^{-|\be|}\<v\>^{-N},
\eeno
and moreover the same bound for $|I_2|$ holds because $2l>|\be|+3$. Furthermore, since $2k>N+3$, we can also obtain that
%\beno
%\int e^{-iu\cdot \eta}(1-\Delta_\eta)^k\{\<\eta\>^{-2l}\pa^{\al_1}_v\pa^\be_\xi r(v,\xi+\eta)\}d\eta\leq C_{k}\<v\>^{-N}.
%\eeno
%Then we have that
\beno
\int\{(1-\Delta_u)^l\pa^{\al_2}_v(\<\cdot\>^N)(v+u)\<v\>^{-N}\<u\>^{-2k}du\leq C,
\eeno
which leads us to $|\pa_v^\al\pa^\be_\xi \mathcal{r}(v,\xi)|\leq C_{\al,\be}\<\xi\>^{-|\be|}$.

    The second inequality can be proved similarly and we omit the details. We also remark that the cases that $k\geq 0,j=-1$,
    $k=-1,j\geq 0$ and
    $k=-1,j=-1$ can be proved in the same manner.

(ii) Since $|m-p|>N_0$, we have $\F_m\F_p=0$, then
 \beno
  \|\F_m\cP_k\F_pg\|_{L^2}&=&\|[\cP_k,\F_m]\F_pg\|_{L^2}
 =2^{-kN}2^{-mN}\|(2^{kN}\cP_k2^{mN}\F_m-2^{mN}\F_m2^{kN}\cP_k)\F_pg\|_{L^2}\\
  &=&2^{-kN}2^{-mN}\|r_{2N+1}(v,D_v)\F_pg\|_{L^2},
 \eeno
 where we use the fact that $\F_{m,\alpha}\F_p=0$. Since $r_{2N+1}(v,\xi)\in S^{-N-1}_{1,0}$, we have
  \beno
  \|\F_m\cP_k\F_pg\|_{L^2}&=2^{-kN}2^{-mN}\|r_{2N+1}(v,D_v)\F_pg\|_{L^2}\leq C_N2^{-(m+k)N}\|\F_pg\|_{H^{-N}_{-N}}\sim2^{-(m+p+k)N}\|\F_pg\|_{L^2_{-N}}.
 \eeno
Other inequalities can be proved similarly. 

 (iii) We begin with the first inequality. By the definition of $\U_{k+N_0}$, we have
$(\U_{k+N_0}h)(v)=[\psi(v)+\sum_{j\leq k+N_0}\varphi_k(v)]h(v):=\tilde{\psi}_{k+N_0}(v)h(v)$.
Thanks to Lemma \ref{le1.2} and the facts that if $w\geq0$,
$\partial^{\al}_v(\tilde{\psi}_{k+N_0}\<v\>^{-w})\ls\<v\>^{-w-|\al|}\ls\<v\>^{-|\al|},$
and if $w<0$
$\pa^\al_v(2^{kw}\tilde{\psi}_{k+N_0}\<v\>^{-w})\ls\<v\>^{-|\al|},$
we deduce that
\beno
\|\U_{k+N_0}h\|_{H^a}\ls\|\tilde{\psi}_{k+N_0}\<v\>^{-w}(\<v\>^wh)\|_{H^a}\ls2^{k(-w)^+}\|h\|_{H^a_w}.
\eeno
The second inequality can be proved by Lemma \ref{le1.2} and the fact that $S_{p+N_0}\in S^{0}_{1,0}$ and $\<\cdot\>^l\in S^{l}_{1,0}$. Indeed, noting that  $(S_{p+N_0}f)(v)=\cF^{-1}((\psi+\sum_{k\leq p+N_0}\vphi(2^{-k}\cdot))\cF f)(v)$, we have that $\pa^\al_{\xi}((\psi+\sum_{k\leq p+N_0}\vphi(2^{-k}\xi)))\leq C \<\xi\>^{-|\al|}$ with the constant $C$ independent of $p$. Then we complete the proof.

(iv) We first prove  first inequality with $j\geq0$. We have
\beno
\|\F_jf\|_{L^1_l}&\le& \int_{\R^3}\int_{\R^3}\<v\>^l|\tphi_j(v-u)f(u)|dudv\ls C_l\int_{\R^3\times\R^3}\<v-u\>^l|\tphi_j(v-u)f(u)|dudv\\&&+\int_{\R^3\times\R^3}|\tphi_j(v-u)f(u)|\<u\>^ldudv\ls
C_l\||\<\cdot\>^l\tphi_j|*|f|\|_{L^1}+C_l\||\tphi_j|*|f\<\cdot\>^l|\|_{L^1}.
\eeno
By Young inequality and the facts $\|\tphi_j\|_{L^1}=\|\tphi\|_{L^1}$ and $\|\<\cdot\>^l\tphi_j\|_{L^1}\ls \|\tphi\|_{L^1_l}$, we deduce that
$\|\F_jf\|_{L^1_l}\leq C_l\|\tphi\|_{L^1_l}\|g\|_{L^1_l}.$
We remark that the case of $j=-1$ can be handled similarly.

Noting that  $(S_jf)(v)=\cF^{-1}((\psi+\sum_{k\leq j}\vphi(2^{-k}\cdot))\cF f)(v)$. By setting $V_j=\cF^{-1} (\psi+\sum_{k\leq j}\vphi(2^{-k}\cdot))$,   we may copy the   argument for the first inequality  to obtain the desired result. Then we ends the proof of this lemma.
\end{proof}
 \begin{lem}\label{lemma1.4}
(i) Let $m, l\in \R.$ Then for $f\in H_l^m$,
\ben\label{Ber}
\sum_{k=-1}^\infty2^{2kl}\|\tP_k f\|^2_{H^m}\sim\sum_{k=-1}^\infty2^{2kl}\|\cP_k f\|^2_{H^m}\sim\|f\|^2_{H^m_l}\sim\sum_{j=-1}^\infty2^{2 j m}\|\F_jf\|^2_{L^2_l}\sim\sum_{j=-1}^\infty2^{2 j m}\|\tF_jf\|^2_{L^2_l}.
\een
  Moreover, we   have
\ben\label{7.70}
 \sum_{k=-1}^\infty2^{2kl}\|\mP_{k}f\|^2_{H^m}\le C_{m,l}\sum_{k=-1}^\infty2^{2kl}\|\cP_kf\|^2_{H^m},\quad
  \sum_{j=-1}^\infty2^{2jm}\|\mF_{j}f\|^2_{L^2_l}\le C_{m,l}\sum_{j=-1}^\infty2^{2jm}\|\F_jf\|^2_{L^2_l}.
\een

(ii) If $m, n, l\in\R$ and $\de>0$, then we have
\ben\label{7.77}
\sum_{j=-1}^\infty2^{2 j n}  \|\F_jf\|^2_{H^m_l}  \lesssim C_{m, n, l}\|f\|^2_{H^{m+n}_l},\quad \|f\|_{H^{-\frac{3}{2}-\delta}_l}   \lesssim C_l\|f\|_{L^1_l}.
\een
\end{lem}
\begin{proof} We address that  all the constants derived in the below are universal thanks to {\bf Remark} \ref{CONSTS}.
	The  equivalences in $(i)$ are proved in \cite{H, HJZ}. \eqref{7.70}   can be easily checked by the definitions of $\mP_k$ and $\mF_j$.   Thanks to the results in $(i)$, we get that   \beno
	\sum_{j=-1}^\infty2^{2jn}\|\F_jf\|^2_{H^m_l}&\sim&\sum_{j=-1}^\infty2^{2jn}\sum_{k=-1}^\infty2^{2km}\|\F_k\F_jf\|^2_{L^2_l}
	\ls\sum_{j=-1}^\infty\sum_{|k-j|<N_0}2^{2j(n+m)}2^{2(k-j)m}\|\F_k\F_jf\|^2_{L^2_l},
	\eeno
	from which together with Lemma \ref{lemma1.3}(iii), we deduce that
	\beno
	\sum_{j=-1}^\infty2^{2jn}\|\F_jf\|^2_{H^m_l}
	\ls\sum_{j=-1}^\infty2^{2j(n+m)}\|\F_jf\|^2_{L^2_l}
	\sim\|f\|_{H^{m+n}_l},
	\eeno
	which yields the first part of (\ref{7.77}). To prove the second part of (\ref{7.77}), we notice that
	\beno
	\|f\|^2_{H^{-\frac{3}{2}-\delta}_l}&\sim&\sum_{k=-1}^\infty2^{2kl}\|\cP_kf\|^2_{H^{-\frac{3}{2}-\delta}}\ls\sum_{k=-1}^\infty2^{2kl}\|\cP_kf\|^2_{L^1}\ls C_l\|f\|^2_{L^1_l}.
	\eeno
	This ends the proof of the lemma.\end{proof}

\subsection{Proof of Theorem \ref{le1.24}} We first apply the dyadic decomposition to the collision operator.

\subsubsection{Dyadic decomposition of the operator in the  phase space} We first use the dyadic decompositions to reduce the commutator to the annulus in the phase space.
 We
set \ben\label{DefPhi} \Phi_k^\gamma(v):=\left\{\begin{aligned} & |v|^\gamma \varphi(2^{-k}|v|), \quad\mbox{if}\quad k\ge0;\\
& |v|^\gamma \psi( |v|),\quad\mbox{if}\quad k=-1.\end{aligned}\right.\een
 Then we derive that
$\langle Q(g, h), f \rangle_v=\sum_{k=-1}^\infty \langle Q_k(g, h), f \rangle_v=\sum_{k=-1}^\infty\sum_{j=-1}^\infty \langle Q_k(\mathcal{P}_jg, h), f \rangle_v$,
where
 $$ Q_{k}(g, h):=\iint_{\sigma\in \SS^2,v_*\in \R^3} \Phi_k^\gamma(|v-v_*|)b(\cos\theta) (g'_*h'-g_*h)d\sigma dv_*.$$

It is not difficult to check that there exists a integer $N_0\in \N$ such that(see also  $(2.1)$ in \cite{H})
\ben\label{ubdecom} \langle Q(g,h), f\rangle_v &=&\sum_{k\ge N_0-1}\langle Q_k(\U_{k-N_0} g, \tilde{\mathcal{P}}_kh), \tilde{\mathcal{P}}_kf \rangle_v +
\sum_{j\ge k+N_0}\langle Q_k(\mathcal{P}_{j} g, \tilde{\mathcal{P}}_jh), \tilde{\mathcal{P}}_jf \rangle_v\notag\\&&\quad+\sum_{|j-k|\le N_0}\langle Q_k( \mathcal{P}_{j} g, \U_{k+N_0}h), \U_{k+N_0}f \rangle_v.  \een

\subsubsection{Dyadic decomposition of the operator in the  frequency space} By Bobylev's equality we have
 \ben\label{bobylev}&& \qquad\langle\mathcal{F}\big( Q_k(g, h)\big), \mathcal{F}f \rangle\\&&=\iint_{\sigma\in \SS^2, \eta,\xi\in \R^3} b(\f{\xi}{|\xi|}\cdot \sigma)\big[ \mathcal{F}(\Phi_k^\gamma ) (\eta-\xi^{-})-\mathcal{F}(\Phi_k^\gamma)(\eta)\big](\mathcal{F}g)(\eta)(\mathcal{F}h)(\xi-\eta)\overline{(\mathcal{F}f)}(\xi)d\sigma d\eta d\xi,\nonumber \een
where $\mathcal{F}f$ denotes the Fourier transform of $f$ and $\xi^{\pm}:= \frac{\xi\pm|\xi|\si}{2}$. Then one may derive that
\ben
\label{2.2} \<\F_jQ_k(g,h),\F_jf\>
\notag&=&\sum_{|p-j|<2N_0}\sum_{p'\leq p+3N_0}\<\F_jQ_k(\F_{p'}g,\F_ph),\F_jf\>
 +\sum_{p>j+2N_0}\sum_{|p-p'|\leq N_0}\<\F_jQ_k(\F_{p'}g,\F_ph),\F_jf\>\\&&
+\sum_{p<j-2N_0}\sum_{|m-j|\leq 2N_0}\<\F_jQ_k(\F_mg,\F_ph),\F_jf\>.
\een

\subsubsection{Dyadic decomposition of the commutator} Now we go back to the commutator. Observe that
 \beno
 &&(\<D\>^\ell Q(g,h)-Q(g,\<D\>^\ell h),\<D\>^\ell f)=\sum_{j=-1}^\infty(\F_j^{\f12}\<D\>^\ell Q(g,h)-\F_j^{\f12}Q(g,\<D\>^\ell h),\F_j^{\f12}\<D\>^\ell f)\\
 &=&\sum_{j=-1}^\infty(\F_j^{\f12}\<D\>^\ell Q(g,h)-Q(g,\F_j^{\f12}\<D\>^\ell h),\F_j^{\f12}\<D\>^\ell f)-\sum_{j=-1}^\infty(\F_j^{\f12}Q(g,\<D\>^\ell h)-Q(g,\F_j^{\f12}\<D\>^\ell h),\F_j^{\f12}\<D\>^\ell f).
 \eeno
The second term in the right-hand side can be regarded as a special case of the first term. Thus we only need to estimate the term $(\F_j\<D\>^\ell Q(g,h)-Q(g,\F_j\<D\>^\ell h),\F_j\<D\>^\ell f)$, here we replace $\F_j^{\f12}$ by $\F_j$ which will not change the results.
Thanks to \eqref{ubdecom}, we further have
\ben\label{2.1}&&\<\F_j\<D\>^\ell Q(g,h)-Q(g,\F_j\<D\>^\ell h),\F_j\<D\>^\ell f\>
=\<Q(g,h),\F^2_j\<D\>^{2\ell }f\>-\<Q(g,\F_j\<D\>^\ell h),\F_j\<D\>^\ell f\>\notag\\
 &&=\sum_{k\geq N_0-1}\<Q_k(\U_{k-N_0}g,\tP_kh),\tP_k\F^2_j\<D\>^{2\ell }f\>+\sum_{l\geq k+N_0}\<Q_k(\cP_lg,\tP_lh),\tP_l\F^2_j\<D\>^{2\ell }f\> \notag\\
 &&+\sum_{|l-k|\leq N_0}\<Q_k(\cP_lg,\U_{k+N_0}h),\U_{k+N_0}\F^2_jf\<D\>^{2\ell }\>-\sum_{k\geq N_0-1}\<Q_k(\U_{k-N_0}g,\tP_k\F_j\<D\>^\ell h),\tP_k\F_j\<D\>^\ell f\>\\
 &&-\sum_{l\geq k+N_0}\<Q_k(\cP_lg,\tP_l\F_j\<D\>^\ell h),\tP_l\F_j\<D\>^\ell f\>-\sum_{|l-k|\leq N_0}\<Q_k(\cP_lg,\U_{k+N_0}\F_j\<D\>^\ell h),\U_{k+N_0}\F_j\<D\>^\ell f\>.\notag
\een
From this, we finally derive that
\beno
\<\F_j\<D\>^\ell Q(g,h)-Q(g,\F_j\<D\>^\ell h),\F_j\<D\>^\ell f\>&=&\mathfrak{D}_1+\mathfrak{D}_2+\mathfrak{D}_3,
\eeno
where
\beno
 \mathfrak{D}_1&:=&\sum_{k\geq N_0-1}\<\F_j\<D\>^\ell Q_k(\U_{k-N_0}g, \tP_kh)-Q_k(\U_{k-N_0}g,\F_j\<D\>^\ell \tP_kh),\F_j\<D\>^\ell \tP_kf\>\\
 &&+\sum_{l\geq k+N_0,k\geq0}\<\F_j\<D\>^\ell Q_k(\cP_lg,\tP_lh)-Q_k(\cP_lg,\F_j\<D\>^\ell \tP_lh),\F_j\<D\>^n\tP_lf\>\eeno\beno
 &&+\sum_{|l-k|\leq N_0,k\geq0}\<\F_j\<D\>^\ell Q_k(\cP_lg,\U_{k+N_0}h)-Q_k(\cP_lg,\F_j\<D\>^\ell \U_{k+N_0}h),\F_j\<D\>^\ell \U_{k+N_0}f\>;\\
  &&\mathfrak{D}_2:=\sum_{k\geq N_0-1}\bigg(\<Q_k(\U_{k-N_0}g,\tP_kh),(\tP_k\F^2_j\<D\>^{2\ell }-\F^2_j\<D\>^{2\ell }\tP_k)f\>+\<Q_k(\U_{k-N_0}g,(\F_j\<D\>^\ell \tP_k-\tP_k\F_j\<D\>^\ell )h)\\
  &&,\F_j\<D\>^\ell \tP_kf\>+\<Q_k(\U_{k-N_0}g,\tP_k\F_j\<D\>^\ell h),(\F_j\<D\>^\ell \tP_k-\tP_k\F_j\<D\>^\ell )f\>\bigg) +\sum_{l\geq k+N_0,k\geq0}\bigg(\<Q_k(\cP_lg,\tP_lh),\\
  &&(\tP_l\F^2_j\<D\>^{2\ell }-\F^2_j\<D\>^{2\ell }\tP_l)f\>+\<Q_k(\cP_lg,(\F_j\<D\>^\ell \tP_l-\tP_l\F_j\<D\>^\ell )h),\F_j\<D\>^\ell \tP_lf\>
 +\<Q_k(\cP_lg,\tP_l\F_j\<D\>^\ell h),\\
 &&(\F_j\<D\>^\ell \tP_l-\tP_l\F_j\<D\>^\ell )f\>\bigg)+\sum_{|l-k|\leq N_0,k\geq0}\bigg(\<Q_k(\cP_lg,\U_{k+N_0}h),(\U_{k+N_0}\F^2_j\<D\>^{2\ell }-\F^2_j\<D\>^{2\ell }\U_{k+N_0})f\>\\
 &&+\<Q_k(\cP_lg,(\F_j\<D\>^\ell \U_{k+N_0}-\U_{k+N_0}\F_j\<D\>^\ell )h,\F_j\<D\>^\ell \U_{k+N_0}f\> +\<Q_k(\cP_lg,\U_{k+N_0}\F_j\<D\>^\ell h),(\F_j\<D\>^\ell \U_{k+N_0}\\
&&-\U_{k+N_0}\F_j\<D\>^\ell )f\>\bigg); \qquad \mathfrak{D}_3 :=\<\F_j\<D\>^\ell Q_{-1}(g,h)-Q_{-1}(g,\F_j\<D\>^\ell h),\F_j\<D\>^\ell f\>.
\eeno

Roughly speaking, $\mathfrak{D}_1$ contains the commutator between localized operator $\F_j$ and the localized collision operator $Q_k$; $\mathfrak{D}_2$ focuses on the commutators between localized operators $\F_j$ and $\cP_k$; while $\mathfrak{D}_3$ concentrates on the commutator for the singular part of the collision operator.   The rest of the section is devoted to the  upper bounds of $\mathfrak{D}_1, \mathfrak{D}_2$ and $\mathfrak{D}_3$.

\begin{lem}\label{lemma1.5}
Recall  that $\<Q(g,h),f\>=\sum_{k=-1}^\infty\<Q_k(g, h),f\>$, where  $Q_k(g, h)=\iint_{\si\in\mathbb{S}^2,v_*\in\R^3}\Phi_k^\ga(|v-v_*|)  b(\cos\th)(g_*'h'-g_*h)d\si dv_*$ with $\ga\in(-3,2]$ and
\ben
 \label{1.2}\Phi^\ga_k(v):=\begin{cases}
|v|^\ga\varphi(2^{-k}|v|),~\mbox{if}~~k\ge  0; \\
|v|^\ga\psi(|v|),~\mbox{if}~~k=-1.
\end{cases}
\een
  Then we have
\begin{enumerate}
\item For $k\ge 0$, $i\in \N$
\ben\label{PhiK1}
&\int_{\R^3} |\cF(\Phi_k^\ga)(y)||y|dy\ls 2^{k(\ga-1)},|\nabla^i\cF(\Phi_k^\ga)(\eta)|\ls C_{N,i}2^{k(\ga+3+i)}\<2^k\eta\>^{-N},
\een
where $N$ can be arbitrarily large.
\item For $k=-1$,
\ben\label{PhiK2}
 |\nabla^i\cF(\Phi_{-1}^\ga)(\eta)|\ls \<\eta\>^{-(\ga+3+i)}.
\een \end{enumerate}
\end{lem}
\begin{proof}
(1)	For $k\geq0,$ by definition, it is easy to see that
	\ben\label{1.3}
	\cF(\Phi_k^\ga)(y)=2^{(\ga+3)k}\cF(\Phi_0^\ga)(2^ky),~\nabla^i\cF(\Phi_k^\ga)(\eta)=2^{(\ga+3+i)k}\nabla^i(\cF(\Phi_0^\ga)(2^k\eta)).
	\een
	Then
	\beno
	\int_{\R^3}|\cF(\Phi_k^\ga)(y)||y|dy=2^{k(\ga+3)}\int_{\R^3}|\cF(\Phi_0^\ga)(2^ky)||y|dy=2^{k(\ga-1)}\int_{\R^3}|\cF(\Phi_0^\ga)(y)||y|dy\ls2^{k(\ga-1)}.
	\eeno
	Since   $\na^i\cF(\Phi_0^\ga)$ is a Schwartz function, for any $N\in \N$, we have
	$\na^i \cF(\Phi_0^\ga)(\eta)\leq C_{N,i}\<\eta\>^{-N}$. From this together with
	(\ref{1.3}), we conclude the result.
	
(2)	For $k=-1$ and $\ga<0$, we prove it for $i=2$ and $i=0,1$ can be handled similarly. We only need to consider large $\eta$. By direct calculation, one has
	\beno
	&\pa_{ij}\cF(\Phi_{-1}^\ga)(\eta)=C\pa_{ij}(|\cdot|^{-3-\ga}*(\cF\psi))(\eta)=C\int_{\R^3}\pa_{ij}\cF(\psi)(\eta-\xi)|\xi|^{-3-\ga}d\xi.
	\eeno
	We split the integration domain into  two parts $|\xi|>1$ and $|\xi|\leq1$. Since $\cF(\psi)\in\mathscr{S}$ and $\<\eta\>\sim\<\eta-\xi\>$ for $|\xi|\leq1$, it is easy to obtain that
	$\big|\int_{|\xi|\leq1}\pa_{ij}\cF(\psi)(\eta-\xi)|\xi|^{-3-\ga}d\xi\big|\ls\<\eta\>^{-(\ga+5)}$.
	On the other hand, integrating by parts, we have
	\beno
	\int_{|\xi|>1}\pa_{ij}\cF(\psi)(\eta-\xi)|\xi|^{-3-\ga}d\xi=\int_{|\xi|=1}\pa_{j}\cF(\psi)(\eta-\xi)\xi_i-\cF(\psi)(\eta-\xi)\xi_i\xi_jdS+\int_{|\xi|>1}\cF(\psi)(\eta-\xi)\pa_{ij}|\xi|^{-3-\ga}d\xi.
	\eeno
	The first term in the right-hand side can be bounded by $\<\eta\>^{-(5+\ga)}$, and for the second term, we have
	\beno
	&&\left|\int_{|\xi|>1}\cF(\psi)(\eta-\xi)\pa_{ij}|\xi|^{-3-\ga}d\xi\right|
	\ls\int|\cF(\psi)(\eta-\xi)|\<
	\eta-\xi\>^{(5+\ga)}d\xi\<\eta\>^{-(5+\ga)}\ls\<\eta\>^{-(5+\ga)}.
	\eeno
	
For $k=-1$ and $0\le   \gamma \le  2$, we only prove the case $0<\gamma<2$ since it is trivial for $\gamma =0,2$. We first give the proof of $i=0$. Noticing that $\Phi_{-1}^\ga(v)=|v|^\ga \psi(|v|)$ has compact support and belongs to space $L^1$, so it holds for for $|\eta|\leq1$. For $|\eta|>1$, we aims to prove that $|\eta|^2\cF(|v|^\gamma\psi(|v|))(\eta)\ls\<\eta\>^{-(1+\ga)}$, which is equivalent to $\cF(\Delta(|v|^\gamma\psi(|v|)))=C\cF(|v|^{\gamma -2}\psi(|v|)))+C\cF(\na|v|^\gamma\cdot\na\psi(|v|))+C\cF(|v|^\gamma  \Delta \psi(|v|) )\lesssim \<\eta\>^{-(1+\gamma)}$. Noticing that $\na\Phi$ vanishes near $0$, thus $\na|v|^\gamma\cdot\na\psi(|v|)$ and $|v|^\gamma  \Delta \psi(|v|)$ are smoothing function with compact support. Then we only need to consider the first term. Since $\gamma - 2\in(-3,0)$, then by the previous results, we have $\cF(|v|^{\gamma-2}\psi(|v|))(\eta)\leq\<\eta\>^{-(1+\ga)}$, which ends the proof of case $i=0$.

For case $i=1$ and $i=2$, by the similar argument, one may check that
\beno
\cF(\Delta(v_i|v|^\ga\psi(|v|)))(\eta)\leq \<\eta\>^{-(2+\ga)}\quad\mbox{and}\quad\cF(\Delta(v_iv_j|v|^\ga\psi(|v|)))(\eta)\leq \<\eta\>^{-(3+\ga)}.
\eeno
It ends the proof of this lemma.
\end{proof}

\begin{lem}\label{lemma1.7}
(see \cite{H}(2.3), Lemma 2.1, 2.2 and 2.3)Recall that $\Phi_k^\ga(v)$ is defined by (\ref{1.2}). Then we have the following decomposition
\beno
&&\<Q_k(g,h),f\>_v=\sum_{l\leq p-N_0}\fM^1_{k,p,l}(g,h,f)+\sum_{l\geq-1}\fM^2_{k,l}(g,h,f)+\sum_{p\geq-1}\fM^3_{k,p}(g,h,f)+\sum_{m<p-N_0}\fM^4_{k,p,m}(g,h,f),
\eeno
where
\beno
\fM^1_{k,p,l}(g,h,f)&:=&\iint_{\si\in\S^2,v_*,v\in\R^3}(\tF_p\Phi_k^\ga)(|v-v_*|)b(\cos\th)(\F_pg)_*(\F_lh)[(\tF_pf)'-\tF_pf]d\si dv_*dv,\\
\fM^2_{k,l}(g,h,f)&:=&\iint_{\si\in\S^2,v_*,v\in\R^3}\Phi_k^\ga(|v-v_*|)b(\cos\th)(S_{l-N_0} g)_*(\F_lh)[(\tF_lf)'-\tF_lh]d\si dv_*dv,\\
\fM^3_{k,p}(g,h,f)&:=&\iint_{\si\in\S^2,v_*,v\in\R^3}\Phi_k^\ga(|v-v_*|)b(\cos\th)(\F_pg)_*(\tF_ph)[(\tF_pf)'-\tF_pf]d\si dv_*dv,\\
\fM^4_{k,p,m}(g,h,f)&:=&\iint_{\si\in\S^2,v_*,v\in\R^3}(\tF_p\Phi_k^\ga)(|v-v_*|)b(\cos\th)(\F_pg)_*(\tF_ph)[(\F_mf)'-\F_mf]d\si dv_*dv.
\eeno
(i) If $l\leq p-N_0$, then for $k\ge 0$, for any $N \ge 0$,
\beno
|\fM^1_{k,p,l}|&\ls& 2^{k(\ga+\frac{5}{2}-N)}(2^{-p(N-2s)}2^{2s(l-p)}+2^{-(N-\frac{5}{2})p}2^{\frac{3}{2}(l-p)})\|\Phi_0^\ga\|_{H^{N+2}}\|\vphi\|_{W_N^{2,\infty}}\|\F_pg\|_{L^1}\|\F_lh\|_{L^2}\|\tF_pf\|_{L^2}.
\eeno
(ii) For $k\geq0$,
\beno
|\fM_{k,l}^2|\ls2^{(\ga+2s)k}2^{2sl}\|S_{l-N_0}g\|_{L^1}\|\F_lh\|_{L^2}\|\tF_lf\|_{L^2},\,
|\fM^3_{k,p}|\ls2^{(\ga+2s)k}2^{2s p}\|\F_pg\|_{L^1}\|\tF_ph\|_{L^2}\|\tF_pf\|_{L^2}.
\eeno
(iii) If $m<p-N_0$, then for $k\geq0,$
\beno
|\fM^4_{k,p,m}|&\ls&2^{2s(m-p)}2^{(\ga+\frac{3}{2}-N)k}2^{-p(N-\frac{5}{2})}\|\Phi_0^\ga\|_{H^{N+2}}\|\vphi\|_{W_N^{2,\infty}}\|\F_pg\|_{L^1}\|\tF_ph\|_{L^2}\|\F_mf\|_{L^2}.
%|\fM^4_{-1,p,m}|&\ls&2^{2sm}2^{-p}|\F_pg|_{L^2}|\tF_ph|_{L^2}|\F_mf|_{L^2}.
\eeno
(iv) Let $a, b\in[0,2s]$ with $a+b=2s$, then
$|\<Q_{-1}(g, h),f\>|\ls (\|g\|_{L^1}+\|g\|_{L^2})\|h\|_{H^a}\|f\|_{H^b}.
 $
\end{lem}

 \subsection{Estimate of $\mathfrak{D}_1$} For $\mathfrak{D}_1$, we first consider the estimate for $|\<\F_j\<D\>^\ell Q_k(g,h)-Q_k(g,\F_j\<D\>^\ell h),\F_j\<D\>^\ell f\>|$.
  We have the following lemma:
  \begin{lem}\label{leD1}
 For smooth function $g,h$ and $f$, we have
 \begin{itemize}
  \item[(i)] If $2s<1$,
\ben\label{le2.1form1}
  |\<\F_j\<D\>^\ell  Q_k(g,h)-Q_k(g,\F_j\<D\>^\ell h),\F_j\<D\>^\ell f\>|\ls 2^{k(\ga+2s-1)}2^{2j\ell }\|g\|_{L^1}\|h\|_{L^2}\|\F_j f\|_{L^2},~k\geq0.
\een
 \item[(ii)]  If $2s>1$, \ben\label{le2.1form2}
  \quad\quad|\<\F_j\<D\>^\ell  Q_k(g,h)-Q_k(g,\F_j\<D\>^\ell h),\F_j\<D\>^\ell f\>|\ls2^{k(\ga+2s-1)}2^{(2\ell +2s-1)j}\|g\|_{L^1}\|h\|_{L^2}\|\F_jf\|_{L^2},~k\geq0.
\een
\item[(iii)]If $2s=1,$
\ben\label{le2.1form2.1}
  &&\\
  \notag&&|\<\F_j\<D\>^\ell  Q_k(g,h)-Q_k(g,\F_j\<D\>^\ell h),\F_j\<D\>^\ell f\>|\ls2^{k(\ga+2s-1-)}2^{(2\ell +2s-1-)j}\|g\|_{L^1}\|h\|_{L^2}\|\F_jf\|_{L^2},~k\geq0.
\een
\end{itemize}
 \end{lem}
\begin{proof}
 We only give the detail proof for $j\geq0$, and $j=-1$ can be handle similarly. Recalling the definition $\xi^{\pm}:=\frac{\xi\pm|\xi|\si}{2}$ and the  Bobylev's formula \eqref{bobylev},
  we observe that
 \beno
   &&\<\F_j\<D\>^\ell Q_k(g,h)-Q_k(g,\F_j\<D\>^\ell h),\F_j\<D\>^\ell f\>=\int_{\si\in \mathbb{S}^2,\eta,\xi\in \R^3}b(\frac{\xi}{|\xi|}\cdot\si)[\cF(\Phi_k^\ga)(\eta-\xi^-)-\cF(\Phi_k^\ga)(\eta)]\\
   &&\times(\cF g)(\eta)(\cF h)(\xi-\eta)\<\xi\>^\ell \vphi(2^{-j}\xi)\overline{(\cF f)}(\xi)(\<\xi\>^\ell \vphi(2^{-j}\xi)-\<\xi-\eta\>^\ell \vphi(2^{-j}(\xi-\eta))d\si d\eta d\xi:=\A.
 \eeno
Next we split the integration domain of $\A$ into two parts: $2|\xi^-|\leq2^{-k}\<\eta\>$ and $2|\xi^-|>2^{-k}\<\eta\>$. Correspondingly  $\A$ can be decomposed into two parts: $\A_1$ and $\A_2$. The proof will be decomposed into three steps. The first two steps will focus on the proof of \eqref{le2.1form2} and  \eqref{le2.1form2.1}, that is, in the case of $2s\geq1$. The we will explain   how to extend the proof to the case \eqref{le2.1form1} in the last step.
\smallskip

\noindent\underline{\it Step 1: Estimate of $\A_1$.} In the region $2|\xi^-|\leq2^{-k}\<\eta\>$, for $t\in[0,1]$, we have
\ben\label{restrictoftheta1}
 \sin(\th/2)\le \frac 1 2 2^{-k}\<\eta\>/|\xi|,  \quad \<\eta-t\xi^-\>\thicksim\<\eta\>.
 \een Then
\beno
 &&|\A_1|\ls\Big|\int_{2|\xi^-|\leq2^{-k}\<\eta\>}b(\frac{\xi}{|\xi|}\cdot\si)(\nabla\cF(\Phi_k^\ga)(\eta))\xi^-(\cF g)(\eta)(\cF h)(\xi-\eta)\<\xi\>^\ell
 \vphi(2^{-j}\xi)(\overline{\cF f})(\xi)(\<\xi\>^\ell \vphi(2^{-j}\xi)\\
 &&-\<\xi-\eta\>^\ell \vphi(2^{-j}(\xi-\eta))d\si d\eta d\xi\Big|+\int_0^1\int_{2|\xi^-|\leq2^{-k}\<\eta\>}b(\frac{\xi}{|\xi|}\cdot\si)|(\nabla^2\cF(\Phi_k^\ga)(\eta-t\xi^-))||\xi^-|^2| (\cF g)(\eta)||(\cF h)(\eta-\xi)|\\
 &&\times\<\xi\>^\ell |\vphi(2^{-j}\xi)(\overline{\cF f})(\xi)(\<\xi\>^\ell \vphi(2^{-j}\xi)-\<\xi-\eta\>^\ell \vphi(2^{-j}(\xi-\eta))|d\si d\eta d\xi dt:= \A_{1,1}+\A_{1,2}.
\eeno

$\bullet$ For $\A_{1,1}$, thanks to the symmetric structure, for any   function $\Psi$,  we have
\beno \int_{\sigma\in\S^2} \Psi(\frac{\xi}{|\xi|}\cdot\si) b(\frac{\xi}{|\xi|}\cdot\si)\xi^-d\si= \int_{\sigma\in\S^2} \Psi(\frac{\xi}{|\xi|}\cdot\si) b(\frac{\xi}{|\xi|}\cdot\si)(\xi^-\cdot \frac{\xi}{|\xi|})\frac{\xi}{|\xi|} d\si. \eeno
Recall that $\xi^-=\f{\xi-|\xi|\sigma}2$ which implies that $|(\xi^-\cdot \frac{\xi}{|\xi|})|=|\xi|\sin^2(\th/2)$.
From this together with \eqref{restrictoftheta1} and  Assumption $\mathbf{(A2)}$ i.e. $b(\cos\th)\sin\th\sim\th^{-1-2s}$, we are led to that
\begin{align*}
\A_{1,1}&\ls\int2^{(2\ell-1)j}|\nabla\cF(\Phi_k^\ga)(\eta)|\min\{1,(2^{-k}\<\eta\>/|\xi|)\}^{2-2s}|\xi||\eta||(\cF g)(\eta)||(\cF h)(\eta-\xi)||\vphi(2^{-j}\xi)(\cF f)(\xi)|d\xi d\eta\\
&\ls 2^{-k(2-2s)}2^{2j\ell}\|g\|_{L^1}\|h\|_{L^2}\|\F_j f\|_{H^{2s-1}}\int |\nabla\cF(\Phi_k^\ga)(\eta)||\eta|\<\eta\>^{2-2s}d\eta.
\end{align*}
Here we use the fact $|\<\xi\>^\ell \vphi(2^{-j}\xi)-\<\xi-\eta\>^\ell \vphi(2^{-j}(\xi-\eta))|\ls 2^{(\ell -1)j}|\eta|$.
Due to (\ref{PhiK1}), one has
\beno
&&2^{-k(2-2s)}\int  |\nabla\cF(\Phi_k^\ga)(\eta)||\eta|\<\eta\>^{2-2s}d\eta
 \ls C_N2^{-k(2-2s)}\int_{|\eta|\geq1}\<\eta\>^{2-2s}|\eta|\<2^k\eta\>^{-N}d\eta2^{k(\ga+3+1)}\\&&\qquad\qquad+2^{-k(2-2s)}\int_{|\eta|<1}\<2^k\eta\>^{-N}|\eta|\<\eta\>^{2-2s}d\eta2^{k(\ga+3+1)}
\ls C_N2^{k(-N+\ga+2s+12)}\\
&&+2^{k(\ga+2s+2)}\int_{|\eta|<2^k}\<\eta\>^{-N}2^{-k}|\eta|2^{-3k}d\eta\leq C_N2^{k(-N+\ga+2s+12)}+2^{k(\ga+2s-2)},
\eeno
which implies that
$\A_{1,1}
\ls2^{k(\ga+2s-1)}2^{2j\ell }\|g\|_{L^1}\|h\|_{L^2}\|\F_j f\|_{H^{2s-1}}.$

$\bullet$ For $\A_{1,2}$,
 also by \eqref{restrictoftheta1} and (\ref{PhiK1}), we have
\beno
 \A_{1,2}&\ls&\int2^{(2\ell-1)j}|\eta|(2^{-k}\<\eta\>/|\xi|)^{2-2s}|\xi|^2 2^{k(\ga+5)}\<2^k\eta\>^{-N}|(\cF g)(\eta)||(\cF h)(\eta-\xi)||\vphi(2^{-j}\xi)(\cF f)(\xi)|d\xi d\eta\\
 &\ls&2^{-k(2-2s)}2^{2j\ell}\|g\|_{L^1}\|h\|_{L^2}\|\F_j f\|_{H^{2s-1}}\int\<\eta\>^{2-2s}|\eta|\<2^k\eta\>^{-N}d\eta 2^{k(\ga+5)}\\
 &\ls&(C_N2^{k(-N+\gamma+2s+3)}+2^{k(\ga+2s-1)})2^{2j\ell }\|g\|_{L^1}\|h\|_{L^2}\|\F_j f\|_{H^{2s-1}}.
 \eeno
Thanks to the above estimates, we conclude that for $2s\ge 1$,
 $|\A_1|\ls2^{k(\ga+2s-1)}2^{2j\ell }|g|_{L^1}|h|_{L^2}|\F_j f|_{H^{2s-1}}.$

\noindent\underline{\it Step 2: Estimate of $\A_2$.} In the region $2|\xi^-|>2^{-k}\<\eta\>$, since $\sin(\th/2)=|\xi^-|/|\xi|$, it is easy to check that
\ben\label{Restrictoftheta2} \sin(\th/2)\gs2^{-k}\<\eta-\xi^-\>/(3|\xi|)\geq2^{-k}/(3|\xi|), \quad \sin(\th/2)\geq 2^{-k} \<\eta\>/(2|\xi|)\geq2^{-k}/(2|\xi|).\een
Using the fact $|\<\xi\>^\ell (\vphi(2^{-j}\xi)-\<\xi-\eta\>^\ell \vphi(2^{-j}(\xi-\eta))|\ls 2^{(\ell -1)j}|\eta|$, we have
\begin{align*}
|\A_2|
&\ls \int_{2|\xi^-|>2^{-k}\<\eta\>} 2^{(2\ell-1)j}b(\frac{\xi}{|\xi|}\cdot \si)\Big[|\cF(\Phi_k^\ga)(\eta-\xi^-)||\eta-\xi^-|+|\cF(\Phi_k^\ga)(\eta)||\eta|\Big]|(\cF g)(\eta)|\\
&~~~~~~~~\times|(\cF h)(\xi-\eta)| |\vphi(2^{-j}\xi)(\F f)(\xi)|d\si d\eta d\xi+\int_{2|\xi^-|>2^{-k}\<\eta\>} 2^{(2\ell-1)j}b(\frac{\xi}{|\xi|}\cdot \si)|\xi|\sin\frac{\th}{2}|\cF(\Phi_k^\ga)(\eta-\xi^-)|\\
&~~~~~~~~\times|(\cF g)(\eta)||(\cF f)(\xi-\eta)||\vphi(2^{-j}\xi)(\cF f)(\xi)|d\si d\eta d\xi:=\A_{2,1}+\A_{2,2}+\A_{2,3}.
\end{align*}

$\bullet$ \underline{Estimate of $\A_{2,1}$.}
By Cauchy-Schwartz inequality and  the change of variables $\eta-\xi^-\rightarrow \tilde{\eta}$ (which implies $\xi-\eta=\xi^+- \tilde{\eta}$), we get that
\begin{align*}
\A_{2,1}&\ls 2^{2\ell j} |g|_{L^1}\left(\int_{2|\xi^-|>2^{-k}\<\eta\>} b(\frac{\xi}{|\xi|}\cdot \si)|\cF(\Phi_k^\ga)(\eta-\xi^-)||\eta-\xi^-||(\cF h)(\xi-\eta)|^2|\xi|^{-2s}d\si d\eta d\xi\right)^{\frac{1}{2}}\\
&~~~~~~~~\times\left(\int_{2|\xi^-|>2^{-k}\<\eta\>} b(\frac{\xi}{|\xi|}\cdot \si)|\cF(\Phi_k^\ga)(\eta-\xi^-)||\eta-\xi^-||\vphi(2^{-j}\xi)(\cF f)(\xi)|^2|\xi|^{2s}2^{-2j}d\si d\eta d\xi\right)^{\frac{1}{2}}\\
&\ls  2^{2\ell j}|g|_{L^1}\left(\int_{\sin(\theta/2)\ge2^{-k}/(3|\xi|)} b(\frac{\xi^+}{|\xi^+|}\cdot\si)|\cF(\Phi_k^\ga)( \tilde{\eta})|| \tilde{\eta}||(\cF h)(\xi^+- \tilde{\eta})|^2|\xi^+|^{-2s}d\si d \tilde{\eta} d\xi^+\right)^{\frac{1}{2}}\\
&~~~~~~~~\times\left(\int_{\sin(\theta/2)\ge2^{-k}/(3|\xi|)} b(\frac{\xi}{|\xi|}\cdot \si)|\cF(\Phi_k^\ga)( \tilde{\eta})|| \tilde{\eta}||\vphi(2^{-j}\xi)(\cF f)(\xi)|^2|\xi|^{2s-2}d\si d \tilde{\eta} d\xi\right)^{\frac{1}{2}}\\
&\ls 2^{2sk }2^{2\ell j}\|g\|_{L^1}\|h\|_{L^2}\|\F_jf\|_{H^{2s-1}}\int_{\R^3}|\cF(\Phi_k^\ga)( y)|| y|dy,
\end{align*}
where we use the facts $\frac{\xi}{|\xi|}\cdot \si=\cos\theta ,\frac{\xi^+}{|\xi^+|}\cdot\si=\cos(\th/2)$ and   $|\xi^+| \sim |\xi|$. Thanks to \eqref{PhiK1}, we obtain that
\begin{equation*}
 \A_{2,1}\ls2^{k(\ga+2s-1)}2^{2\ell j}\|g\|_{L^1}\|h\|_{L^2}\|\F_jf\|_{H^{2s-1}}.
\end{equation*}

$\bullet$ \underline{Estimate of $\A_{2,2}$ and $\A_{2,3}$.}
The similar argument can be applied to $\A_{2,2}$  to get
\beno
 \A_{2,2}
 &\ls&2^{k(\ga+2s-1)}2^{2\ell j}\|g\|_{L^1}\|h\|_{L^2}\|\F_jf\|_{H^{2s-1}}.
 \eeno
For $\A_{2,3}$, by change of variables and the condition \eqref{Restrictoftheta2}, we have
 \beno
 \A_{2,3}&\ls& 2^{(2\ell -1)j}\int_{2|\xi^-|>2^{-k}\<\eta\>}b(\frac{\xi}{|\xi|}\cdot \si)|\xi|\sin\frac{\th}{2}|\cF(\Phi_k^\ga)(\eta-\xi^-)|
|(\cF g)(\eta)||(\cF f)(\xi-\eta)||\vphi(2^{-j}\xi)(\cF f)(\xi)|d\si d\eta d\xi\\
&\ls&2^{2\ell j}|g|_{L^1}\left(\int_{\sin(\th/2)\gs2^{-k}\<\tilde{\eta}\>/(3|\xi|)} b(\frac{\xi^+}{|\xi^+|}\cdot \si)\sin\frac{\th}{2} |\cF(\Phi_k^\ga)(\tilde{\eta})||(\cF h)(\xi^++\tilde{\eta})|^2  |\xi|^{-2s+1}    d\si d\tilde{\eta} d\xi\right)^\frac{1}{2}\\
&&~~~~~~~~\times\left(\int_{\sin(\th/2)\gs 2^{-k}\<\tilde{\eta}\>/(3|\xi|)} b(\frac{\xi}{|\xi|}\cdot \si)\sin\frac{\th}{2}|\cF(\Phi_k^\ga)(\tilde{\eta})||\vphi(2^{-j}\xi)(\cF f)(\xi)|^2  |\xi|^{2s-1}   d\si d\tilde{\eta} d\xi\right)^\frac{1}{2}\\
&\ls&1_{2s>1}2^{k(\ga+2s-1)}2^{2\ell j}\|g\|_{L^1}\|h\|_{L^2}\|\F_jf\|_{H^{2s-1}}+1_{2s=1}kj2^{k\ga}2^{2\ell j}\|g\|_{L^1}\|h\|_{L^2}\|\F_jf\|_{L^2}.
\eeno

Now we may conclude that
for $2s\ge 1$,
\beno
 |\A_2|\ls1_{2s>1}2^{k(\ga+2s-1)}2^{2\ell j}\|g\|_{L^1}\|h\|_{L^2}\|\F_jf\|_{H^{2s-1}}+1_{2s=1}kj2^{k\ga}2^{2\ell j}\|g\|_{L^1}\|h\|_{L^2}\|\F_jf\|_{L^2}.
\eeno

Since we have $k\ls 2^{k\epsilon}$ and $j\ls 2^{j\epsilon}$ for any $\epsilon >0$, then from the estimates of $\A_1$ and $\A_2$, we complete the proof for \eqref{le2.1form2} and \eqref{le2.1form2.1}.
\smallskip

\noindent\underline{\it Step 3: The proof of case \eqref{le2.1form1}.} The difference only lies in the decomposition of the integration domain. To prove \eqref{le2.1form1}, we separate the domain into regions: $2|\xi^-|\leq\<\eta\>$ and $2|\xi^-|\ge\<\eta\>$. One may easily get desired results by following the same argument used in the previous steps.
\end{proof}

Next we will use the dyadic decompositions in frequency space to improve the above results.
\begin{lem}\label{le2.2}
For $k\geq0$ and sufficiently large $N\in\N$, we have
 \begin{itemize}
  \item[(i)] If $2s<1$,
\ben\label{le2.2form1}
 &&|\<\F_j\<D\>^\ell  Q_k(g,h)-Q_k(g,\F_j\<D\>^\ell h),\F_j\<D\>^\ell f\>|\ls 2^{k(\ga+2s-1)}2^{2\ell j}\|g\|_{L^1}\|\tF_jh\|_{L^2}\|\F_jf\|_{L^2}\\
  &&+C_{N,\ell }\sum_{p>j+2N_0}2^{-kN}2^{-pN}\|g\|_{L^1}\|\F_ph\|_{L^2}\|\F_jf\|_{L^2}+C_{N,\ell }\sum_{m\leq j-2N_0}2^{-kN}2^{-jN}\|g\|_{L^1}\|\F_mh\|_{L^2}\|\F_jf\|_{L^2}; \notag
\een
 \item[(ii)]  If $2s>1$, \ben\label{le2.2form2}
 &&  |\<\F_j\<D\>^\ell  Q_k(g,h)-Q_k(g,\F_j\<D\>^\ell h),\F_j\<D\>^\ell f\>|\ls 2^{k(\ga+2s-1)}2^{(2\ell +2s-1)j}\|g\|_{L^1}\|\tF_jh\|_{L^2}\|\F_jf\|_{L^2}\\
&&+C_{N,\ell }\sum_{p>j+2N_0}2^{-kN}2^{-pN}\|g\|_{L^1}\notag\|\F_ph\|_{L^2}\|\F_jf\|_{L^2}+C_{N,\ell }\sum_{m\leq j-2N_0}2^{-kN}2^{-jN}\|g\|_{L^1}\|\F_mh\|_{L^2}\|\F_jf\|_{L^2};
\een
\item[(iii)] If $2s=1$,
\ben\label{le2.2form2.1}
  &&|\<\F_j\<D\>^\ell  Q_k(g,h)-Q_k(g,\F_j\<D\>^\ell h),\F_j\<D\>^\ell f\>|\ls 2^{k(\ga+2s-1-)}2^{(2\ell +2s-1-)j}\|g\|_{L^1}\|\tF_jh\|_{L^2}\|\F_jf\|_{L^2}\\
&&+C_{N,\ell }\sum_{p>j+2N_0}2^{-kN}2^{-pN}\notag\|g\|_{L^1}\|\F_ph\|_{L^2}\|\F_jf\|_{L^2}+C_{N,\ell }\sum_{m\leq j-2N_0}2^{-kN}2^{-jN}\|g\|_{L^1}\|\F_mh\|_{L^2}\|\F_jf\|_{L^2}.
\een
\end{itemize}
 \end{lem}
\begin{proof}
By Bobylev's equality \eqref{bobylev} and \eqref{2.2},
we have
\beno
 &&\<\F_j\<D\>^\ell Q_k(g,h)-Q_k(g,\F_j\<D\>^\ell h),\F_j\<D\>^\ell f\>\\
 &=&\sum_{|p-j|<2N_0}\sum_{|p-p'|\leq3N_0}\<\F_j\<D\>^\ell Q_k(\F_{p'}g,\F_ph)-Q_k(\F_{p'}g,\F_j\<D\>^\ell \F_ph),\F_j\<D\>^\ell f\>\\
\notag&&+\sum_{p>j+2N_0}\sum_{|p-p'|\leq N_0}\<\F_j\<D\>^\ell Q_k(\F_{p'}g,\F_ph),\F_j\<D\>^\ell f\>
+\sum_{|p-j|\leq N_0}\sum_{m<j-2N_0}\<\F_j\<D\>^\ell Q_k(\F_pg,\F_mh),\F_j\<D\>^\ell f\>\\
 &:=&\cB_1+\cB_2+\cB_3.
\eeno
We first give the proof of (\ref{le2.2form2}). Thanks to Lemma \ref{leD1}(\ref{le2.1form2}), we obtain that
\begin{align*}
|\cB_1|\ls2^{k(\ga+2s-1)}2^{(2\ell +2s-1)j}\|g\|_{L^1}\|\tF_jh\|_{L^2}\|\F_jf\|_{L^2}.
\end{align*}
For $\cB_2$ and $\cB_3$,   by Bobylev's equality, it is easy to see that $\<\F_j\<D\>^\ell Q_k(\F_{p'}g,\F_ph),\F_j\<D\>^\ell f\>$ \\and $\<\F_j\<D\>^\ell Q_k(\F_pg,\F_mh),\F_j\<D\>^\ell f\>$ enjoy the same structure as  $\fM^4_{k,p,m}$ and $\fM^1_{k,p,l}$ defined in Lemma \ref{lemma1.7} respectively, then from Lemma \ref{lemma1.7} (i) and (iii), we have
\beno
|\cB_2|+|\cB_3|&\leq& C_{N,\ell }\sum_{p>j+2N_0}2^{-kN}2^{-pN}\|g\|_{L^1}\|\F_ph\|_{L^2}\|\F_jf\|_{L^2}+C_{N,n}\sum_{m\leq j-2N_0}2^{-kN}2^{-jN}\|g\|_{L^1}\|\F_mh\|_{L^2}\|\F_jf\|_{L^2},
\eeno
where $N$ can be large enough. Thus for $k\geq 0$, we derive that
\beno
  &&|\<\F_j\<D\>^\ell  Q_k(g,h)-Q_k(g,\F_j\<D\>^\ell h),\F_j\<D\>^\ell f\>|
  \ls 2^{k(\ga+2s-1)}2^{(2\ell +2s-1)j}\|g\|_{L^1}\|\tF_jh\|_{L^2}\|\F_jf\|_{L^2}\\
  &&+C_{N,\ell }\sum_{p>j+2N_0}2^{-kN}2^{-pN}\|g\|_{L^1}\|\F_ph\|_{L^2}\|\F_jf\|_{L^2}
+C_{N,\ell }\sum_{m\leq j-2N_0}2^{-kN}2^{-jN}\|g\|_{L^1}\|\F_mh\|_{L^2}\|\F_jf\|_{L^2},~k\geq0.
\eeno

We finally remark that the other cases can be obtained in a similar way.  We skip the details here and end the proof.
\end{proof}

Now we can give the estimate of $\mathfrak{D}_1$:
\begin{lem}\label{le2.3}

For smooth function $g,h$ and $f$ and sufficiently large $N\in\N$, we have
 \begin{itemize}
  \item[(i)] If $2s<1$,
\ben\label{le2.3form1}
  \mathfrak{D}_1\ls C_{N,\ell}\|g\|_{L^1_{(\ga+2s-1)^+(-\omega_1)^++(-\omega_2)^++\de}}(\|\mF_jh\|_{H^a_{\omega_1}}+2^{-jN}\|h\|_{H_{-N}^{-N}})(\|\mF_jf\|_{H^{b}_{\omega_2}}+2^{-jN}\|f\|_{H_{-N}^{-N}}).
\een
 \item[(ii)]  If $2s>1$, \ben\label{le2.3form2}
  \mathfrak{D}_1\ls
C_{N,\ell}\|g\|_{L^1_{(\ga+2s-1)^++(-\omega_1)^++(-\omega_2)^++\de}}(\|\mF_jh\|_{H^a_{\omega_1}}+2^{-jN}\|h\|_{H_{-N}^{-N}})(\|\mF_jf\|_{H^{b}_{\omega_2}}+2^{-jN}\|f\|_{H_{-N}^{-N}}).
\een
\item[(iii)] If $2s=1$ ,
\ben\label{le2.3form2.1}
  \quad\mathfrak{D}_1\ls C_{N,\ell}\|g\|_{L^1_{(\ga+2s-1+\de)^++(-\omega_1)^++(-\omega_2)^++\de}}(\|\mF_jh\|_{H^a_{\omega_1}}+2^{-jN}\|h\|_{H_{-N}^{-N}})(\|\mF_jf\|_{H^{b}_{\omega_2}}+2^{-jN}\|f\|_{H_{-N}^{-N}}),
\een
\end{itemize}
where  $a,b\geq 0,~a+b=(2\ell )1_{2s<1}+(2\ell+ 2s-1)1_{2s>1}+(2\ell +2s-1+\de)1_{2s=1}$, $\omega_1+\omega_2=(\ga+2s-1)1_{2s\neq1}+(\ga+2s-1+\de)1_{2s=1},\de\ll1$ and $\mF_j$ is defined in Definition \ref{Fj}.
 \end{lem}
\begin{proof}
We first give the proof of (\ref{le2.3form2}). We set
$\mathfrak{D}_1=F_1+F_2+F_3$, where $F_1=\sum\limits_{k\geq N_0-1}\<\F_j\<D\>^\ell Q_k(\U_{k-N_0}g,\tP_kh)-Q_k(\U_{k-N_0}g,\F_j\<D\>^\ell \tP_kh),\F_j\<D\>^\ell \tP_kf\>$,
$F_2=\sum\limits_{l\geq k+N_0,k\geq0}\<\F_j\<D\>^\ell Q_k(\cP_lg,\tP_lh)-Q_k(\cP_lg,\F_j\<D\>^\ell \tP_lh),\F_j\<D\>^\ell \tP_lf\>$ and $F_3=\sum\limits_{|l-k|\leq N_0,k\geq0}\<\F_j\<D\>^\ell Q_k(\cP_lg,\U_{k+N_0}h)-Q_k(\cP_lg,\F_j\<D\>^\ell \U_{k+N_0}h),\F_j\<D\>^\ell \U_{k+N_0}f\>$.
  \smallskip

\underline{\it Estimate of $F_1$.} Thanks to Lemma \ref{le2.2}(\ref{le2.2form2}), we obtain that
 \beno
  |F_1|&\ls&\sum_{k\geq N_0-1,k\geq0}\Big(2^{k(\ga+2s-1)}\|\U_{k-N_0}g\|_{L^1}\|\tF_j\tP_kh\|_{H^a}\|\F_j\tP_kf\|_{H^b}+C_N\sum_{p>j+2N_0}2^{-3kN}2^{-3pN}\\
  &&\times\|\U_{k-N_0}g\|_{L^1}\|\F_p\tP_kh\|_{L^2}\|\F_j\tP_kf\|_{L^2}
+C_N\sum_{m<j-2N_0}2^{-3kN}2^{-3jN}\|\U_{k-N_0}g\|_{L^1}\|\F_m\tP_kh\|_{L^2}\|\F_j\tP_kf\|_{L^2}\Big)
 \eeno
with $a+b=2\ell +2s-1$. By Lemma \ref{lemma1.3} and Lemma \ref{lemma1.4}, we derive that
  \beno
&&\sum_{k\geq N_0-1,k\geq0}2^{k(\ga+2s-1)}\|\U_{k-N_0}g\|_{L^1}\|\tF_j\tP_kh\|_{H^a}\|\F_j\tP_kf\|_{H^b}\\
&\ls&C_{N,\ell }\|g\|_{L^1}\sum_{k\geq N_0-1}(\sum_{|\alpha|=0}^{2N+2\ell +1}2^{\omega_1k}2^{aj}\|\tP_{k,\al}\tF_{j,\al}h\|_{L^2}+2^{-jN}2^{-kN}2^{\omega_1k}\|h\|_{H_{-N}^{-N}})(\sum_{|\alpha|=0}^{2N+2\ell+1}2^{\omega_2k}2^{bj}\\
&&\times\|\tP_{k,\al}\tF_{j,\al}f\|_{L^2}+2^{-jN}2^{-kN}2^{\omega_2k}\|f\|_{H^{-N}_{-N}})
\\&\ls& C_{N,\ell }\|g\|_{L^1}(\sum_{|\al|=0}^{2N+2\ell+1}\|\tF_{j,\al}h\|_{H^a_{\omega_1}}+2^{-jN}\|h\|_{H_{-N}^{-N}})(\sum_{|\al|=0}^{2N+2\ell+1}\|\tF_{j,\al}f\|_{H^{b}_{\omega_2}}+2^{-jN}\|f\|_{H_{-N}^{-N}}) \\
&\ls& C_{N,\ell }\|g\|_{L^1}(\|\mF_jh\|_{H^a_{\omega_1}}+2^{-jN}\|h\|_{H_{-N}^{-N}})(\|\mF_jf\|_{H^{b}_{\omega_2}}+2^{-jN}\|f\|_{H_{-N}^{-N}}),
 \eeno
 where $\omega_1+\omega_2=\ga+2s-1$. From this together with Lemma \ref{lemma1.4}, we conclude that
 \beno
|F_1|\ls C_{N,\ell }\|g\|_{L^1}(\|\mF_jh\|_{H^a_{\omega_1}}+2^{-jN}\|h\|_{H_{-N}^{-N}})(\|\mF_jf\|_{H^{b}_{\omega_2}}+2^{-jN}\|f\|_{H_{-N}^{-N}}).
 \eeno

\underline{\it Estimate of $F_2$.} We have
 \beno
  |F_2|&\ls&\sum_{l\geq k+N_0,k\geq0}\Big(2^{k(\ga+2s-1)}\|\cP_lg\|_{L^1}\|\tF_j\tP_lh\|_{H^a}\|\F_j\tP_lf\|_{H^b}+C_{N,\ell }\sum_{p>j+2N_0}2^{-3kN}2^{-3pN}\\
  &&\times\|\cP_lg\|_{L^1}\|\F_p\tP_lh\|_{L^2}\|\F_j\tP_lf\|_{L^2}
+C_{N,\ell }\sum_{m<j-2N_0}2^{-3kN}2^{-3jN}\|\cP_lg\|_{L^1}\|\F_m\tP_lh\|_{L^2}\|\F_j\tP_lf\|_{L^2}\Big)
 \eeno
By Lemma \ref{lemma1.3} and Lemma \ref{lemma1.4}, we derive that
\beno
&&\sum_{l\geq k+N_0,k\geq0}2^{k(\ga+2s-1)}\|\cP_lg\|_{L^1}\|\tF_j\tP_lh\|_{H^a}\|\F_j\tP_lf\|_{H^b}\ls\sum_{l\geq0} 2^{l(\de+(\ga+2s-1)^+)}\|\cP_lg\|_{L^1}\|\tF_j\tP_lh\|_{H^a}\|\F_j\tP_lf\|_{H^b}\\
&\ls& C_{N,\ell }\|g\|_{L^1_{(\ga+2s-1)^++(-\omega_1)^++(-\omega_2)^++\de}}(\|\mF_jh\|_{H^a_{\omega_1}}+2^{-jN}\|h\|_{H_{-N}^{-N}})(\|\mF_jf\|_{H^{b}_{\omega_2}}+2^{-jN}\|f\|_{H_{-N}^{-N}}).
\eeno
Then we conclude that
\beno
|F_2|\ls C_{N,\ell }\|g\|_{L^1_{(\ga+2s-1)^++(-\omega_1)^++(-\omega_2)^++\de}}(\|\mF_jh\|_{H^a_{\omega_1}}+2^{-jN}\|h\|_{H_{-N}^{-N}})(\|\mF_jf\|_{H^{b}_{\omega_2}}+2^{-jN}\|f\|_{H_{-N}^{-N}}).
\eeno
\underline{\it Estimate of $F_3$.} Similarly, we have
 \beno
 && |F_3| \ls \sum_{|l-k|\leq N_0,k\geq0}\Big(2^{k(\ga+2s-1)}\|\cP_lg\|_{L^1}\|\tF_j\U_{k+N_0}h\|_{H^a}\|\F_j\U_{k+N_0}f\|_{H^b}+C_{N,\ell }\sum_{p>j+2N_0}2^{-3kN}2^{-3pN}\\
  &&\times\|\cP_lg\|_{L^1}\|\F_p\U_{k+N_0}h\|_{L^2}\|\F_j\U_{k+N_0}f\|_{L^2}
+C_{N,\ell }\sum_{m<j-2N_0}2^{-3kN}2^{-3jN}\|\cP_lg\|_{L^1}\|\F_m\U_{k+N_0}h\|_{L^2}\|\F_j\U_{k+N_0}f\|_{L^2}\Big)
 \eeno
 Thanks to $\|\U_{k+N_0}h\|_{H^a}\leq C_{a,w}2^{k(-w)^+}\|h\|_{H^a_w}$(see Lemma \ref{lemma1.3}(iii)). We can obtain that
 \beno
 &&\sum_{|l-k|\leq N_0,k\geq0}2^{k(\ga+2s-1)}\|\cP_lg\|_{L^1}\|\tF_j\U_{k+N_0}h\|_{H^a}\|\F_j\U_{k+N_0}f\|_{H^b}\\
 &\ls& C_{N,\ell}\|g\|_{L^1_{(\ga+2s-1)^++(-\omega_1)^++(-\omega_2)^++\de}}(\|\mF_jh\|_{H^a_{\omega_1}}+2^{-jN}\|h\|_{H_{-N}^{-N}})(\|\mF_jf\|_{H^{b}_{\omega_2}}+2^{-jN}\|f\|_{H_{-N}^{-N}}).
 \eeno
 Then we conclude that
\beno
|F_3|\ls C_{N,\ell}\|g\|_{L^1_{(\ga+2s-1)^++(-\omega_1)^++(-\omega_2)^++\de}}(\|\mF_jh\|_{H^a_{\omega_1}}+2^{-jN}\|h\|_{H_{-N}^{-N}})(\|\mF_jf\|_{H^{b}_{\omega_2}}+2^{-jN}\|f\|_{H_{-N}^{-N}}).
\eeno

 We complete the proof of (\ref{le2.3form2}). We finally remark that
 the other cases can be handled in a similar way. This ends the proof of the lemma.
\end{proof}

\subsection{Estimate of $\mathfrak{D}_2$} Now we give the estimate of  $\mathfrak{D}_2$.
 \begin{lem}\label{le2.4}
 For smooth function $g,h$ and $f$, we have
\beno
  \mathfrak{D}_2\ls C_{N,\ell}\|g\|_{L^1_{(\ga+2s)^++(-\omega_1)^++(-\omega_2)^++\de}}(\|\mF_jh\|_{H^a_{\omega_1}}+2^{-jN}\|h\|_{H_{-N}^{-N}})(\|\mF_jf\|_{H^{b}_{\omega_2}}+2^{-jN}\|f\|_{H_{-N}^{-N}}).
\eeno
where $a,b\geq0$ and $\om_1,\om_2\in\R$ satisfying $a+b=2\ell+2s-1$, $\omega_1+\omega_2=\ga+2s-1,\de\ll1$, $N\in\N$ can be large enough  and $\mF_j$ is defined in Definition \ref{Fj}.
 \end{lem}
\begin{proof}
 We first focus on case: $2s>1$. We set $\mathfrak{D}_2=H_1+H_2+H_3$, where
 \beno
  H_1&=&\sum_{k\geq N_0-1}\bigg(\<Q_k(\U_{k-N_0}g,\tP_kh),(\tP_k\F^2_j\<D\>^{2\ell} -\F^2_j\<D\>^{2\ell} \tP_k)f\>+\<Q_k(\U_{k-N_0}g,(\F_j\<D\>^\ell \tP_k-\tP_k\F_j\<D\>^\ell )h,\\ &&\F_j\<D\>^\ell \tP_kf\>
  +\<Q_k(\U_{k-N_0}g,\tP_k\F_j\<D\>^\ell h),(\F_j\<D\>^\ell \tP_k-\tP_k\F_j\<D\>^\ell )f\>\bigg),\\
  H_2&=& \sum_{l\geq k+N_0,k\geq0}\bigg(\<Q_k(\cP_lg,\tP_lh),(\tP_l\F^2_j\<D\>^{2\ell} -\F^2_j\<D\>^{2\ell} \tP_l)f\>+\<Q_k(\cP_lg,(\F_j\<D\>^\ell \tP_l-\tP_l\F_j\<D\>^\ell )h), \\&&\F_j\<D\>^\ell \tP_lf\>
+\<Q_k(\cP_lg,\tP_l\F_j\<D\>^\ell h),(\F_j\<D\>^\ell \tP_l-\tP_l\F_j\<D\>^\ell )f\>\bigg),
 \\H_3&=&\sum_{|l-k|\leq N_0,k\geq0}\bigg(\<Q_k(\cP_lg,\U_{k+N_0}h),(\U_{k+N_0}\F^2_j\<D\>^{2\ell} -\F^2_j\<D\>^{2\ell} \U_{k+N_0})f\>+\<Q_k(\cP_lg,(\F_j\<D\>^\ell \U_{k+N_0}\\
 &&-\U_{k+N_0}\F_j\<D\>^\ell )h,\F_j\<D\>^\ell \U_{k+N_0}f\>
  +\<Q_k(\cP_lg,\U_{k+N_0}\F_j\<D\>^\ell h),(\F_j\<D\>^\ell \U_{k+N_0}-\U_{k+N_0}\F_j\<D\>^\ell )f\>\bigg).
\eeno

\noindent\underline{\it Step 1: Estimate of $H_{1}$.} We first handle
\beno&&
H_{1,1}:=\sum_{k\geq N_0-1,k\geq0}\<Q_k(\U_{k-N_0}g,\tP_kh),(\tP_k\F^2_j\<D\>^{2\ell} -\F^2_j\<D\>^{2\ell} \tP_k)f\>
=\sum_{k\geq N_0-1}\Big(\sum_{l\leq p-N_0}\fM^1_{k,p,l}(\U_{k-N_0}g,\tP_kh,\\&&(\tP_k\F^2_j\<D\>^{2\ell} -\F^2_j\<D\>^{2\ell} \tP_k)f)+\sum_{l\geq-1}\fM^2_{k,l}(\U_{k-N_0}g,\tP_kh,(\tP_k\F^2_j\<D\>^{2\ell} -\F^2_j\<D\>^{2\ell} \tP_k)f)+\sum_{p\geq-1}\fM^3_{k,p}(\U_{k-N_0}g,\\
&&\tP_kh,(\tP_k\F^2_j\<D\>^{2\ell} -\F^2_j\<D\>^{2\ell} \tP_k)f)+\sum_{m<p-N_0}\fM^4_{k,p,m}(\U_{k-N_0}g,\tP_kh,(\tP_k\F^2_j\<D\>^{2\ell} -\F^2_j\<D\>^{2\ell} \tP_k)f)\Big):=\sum_{i=1}^4H_{1,1}^i,
\eeno
where $\fM^1_{k,p,l},\fM^2_{k,l},\fM^3_{k,p}$ and $\fM^4_{k,p,m}$ are defined in Lemma \ref{lemma1.7}.

$\bullet$ For $H_{1,1}^1$ and $H_{1,1}^4$, thanks to Lemma \ref{lemma1.7}, we have
\beno
 &&\big| H_{1,1}^1+H_{1,1}^4\big|
 \ls C_N\sum_{k\geq N_0-1}\Big(\sum_{l\leq p-N_0}2^{k(\ga+\frac{5}{2}-4N)}(2^{-p(4N-2s)}2^{2s(l-p)}+2^{-(4N-\frac{5}{2})p}2^{\frac{3}{2}(l-p)})\|\F_p\U_{k-N_0}g\|_{L^1}\\
&&\qquad\times\|\F_l\tP_kh\|_{L^2}\|\tF_p(\tP_k\F^2_j\<D\>^{2\ell} -\F^2_j\<D\>^{2\ell} \tP_k)f\|_{L^2}+\sum_{m<p-N_0}2^{2s(m-p)}2^{(\ga+\frac{3}{2}-4N)k}2^{-p(4N-\frac{5}{2})}\\
&&\qquad\times\|\F_p\U_{k-N_0}g\|_{L^1}\|\tF_p\tP_kh\|_{L^2}\|\F_m(\tP_k\F^2_j\<D\>^{2\ell} -\F^2_j\<D\>^{2\ell} \tP_k)f\|_{L^2}\Big).\eeno
Due to Lemma \ref{lemma1.3}, we get that
\ben\label{fpPkfj} &&\|\tF_p(\tP_k\F^2_j\<D\>^{2\ell} -\F^2_j\<D\>^{2\ell} \tP_k)f\|_{L^2}\ls\mathrm{1}_{|p-j|\le 3N_0} \|[\tP_k,\F^2_j\<D\>^{2\ell} ]f\|_{L^2}+\mathrm{1}_{|p-j|> 3N_0}\|\tF_p\tP_k\F^2_j\<D\>^{2\ell} f\|_{L^2}\notag\\
&&\ls \mathrm{1}_{|p-j|\le 3N_0}C_{N,\ell}(2^{(2\ell-1)j-k}\sum_{|\alpha|=1}^{2N+2\ell }\|\tP_{k,\alpha}\hat{\F}_{j,\alpha}f\|_{L^2}+2^{-jN-kN}\|f\|_{H_{-N}^{-N}})+ \mathrm{1}_{|p-j|> 3N_0}C_{N,\ell}2^{-(p+k+j)3N} \|\F_jf\|_{L_{-N}^2}\notag\\&&\hspace{1cm}\ls
C_{N,\ell}(\mathrm{1}_{|p-j|\le 3N_0}2^{(2\ell-1)j-k}\sum_{|\alpha|=1}^{2N+2\ell }\|\tP_{k,\alpha}\hat{\F}_{j,\alpha}f\|_{L^2})+\mathrm{1}_{|p-j|\le 3N_0}2^{-jN}\|f\|_{H_{-N}^{-N}}+2^{-2jN}\|f\|_{H_{-N}^{-N}}),\een
which implies that
\beno
&&\big| H_{1,1}^1+H_{1,1}^4\big|
 \ls C_{N,\ell}\|g\|_{L^1}\sum_{k\geq N_0-1}\Big(\sum_{l\leq p-N_0}2^{k(\ga+\frac{5}{2}-4N)}(2^{-p(N-2s)}2^{(2s+N)(l-p)}+2^{-(N-\frac{5}{2})p}2^{(\frac{3}{2}+N)(l-p)}) \\
&&\quad\times(2^{-lN}\|\F_l\tP_kh\|_{L^2})C_{N,\ell}(\mathrm{1}_{|p-j|\le 3N_0}2^{-2jN}\sum_{|\alpha|=1}^{2N+2\ell }\|\tP_{k,\alpha}\hat{\F}_{j,\alpha}f\|_{L^2}+2^{-2jN}\|f\|_{H_{-N}^{-N}})+\sum_{m<p-N_0}2^{2s(m-p)}\\
&&\quad\times2^{(\ga+\frac{3}{2}-4N)k}2^{-p(N-\frac{5}{2})}(2^{-pN}\|\tF_p\tP_kh\|_{L^2})
C_{N,\ell}(\mathrm{1}_{|m-j|\le 3N_0}2^{-2jN}\sum_{|\alpha\|=1}^{2N+2\ell }\|\tP_{k,\alpha},\hat{\F}_{j,\alpha}f\|_{L^2}+2^{-2jN}\|f\|_{H_{-N}^{-N}})\Big)\\&&\ls C_{N,\ell}2^{-2jN}\|g\|_{L^1}\|h\|_{H_{-N}^{-N}}\|f\|_{H_{-N}^{-N}}.
\eeno

 $\bullet$  For $H_{1,1}^2$ and $H_{1,1}^3$,  due to Lemma \ref{lemma1.7} and Lemma \ref{lemma1.3}(iv), we have
\beno
&&\big| H_{1,1}^2+H_{1,1}^3\big|
 \ls C_N\sum_{k\geq N_0-1}\Big(\sum_{l\geq-1} 2^{(\ga+2s)k}2^{2sl}\|S_{l-N_0}\U_{k-N_0}g\|_{L^1}\|\F_l\tP_kh\|_{L^2}\|\tF_l(\tP_k\F^2_j\<D\>^{2\ell} -\F^2_j\<D\>^{2\ell} \tP_k)f\|_{L^2}\\
&&+\sum_{p\geq-1}
2^{(\ga+2s)k}2^{2s p}\|\F_p\U_{k-N_0}g\|_{L^1}\|\tF_p\tP_kh\|_{L^2}\|\tF_p(\tP_k\F^2_j\<D\>^{2\ell} -\F^2_j\<D\>^{2\ell} \tP_k)f\|_{L^2}\Big)\\
&\ls&\sum_{k\geq N_0-1}\sum_{l\geq-1}2^{(\ga+2s)k}\|\U_{k-N_0}g\|_{L^1}\|\F_l\tP_kh\|_{H^s}\|\tF_l(\tP_k\F^2_j\<D\>^{2\ell} -\F^2_j\<D\>^{2\ell} \tP_k)f)\|_{H^{s}}\\
&\ls&\sum_{k\geq N_0-1}\sum_{|l-j|\le 3N_0}2^{(\ga+2s)k}\|\U_{k-N_0}g\|_{L^1}\|\tF_j\tP_kh\|_{L^2}2^{2sj}\|(\tP_k\F^2_j\<D\>^{2\ell} -\F^2_j\<D\>^{2\ell} \tP_k)f)\|_{L^2}\\
&&+\sum_{k\geq N_0-1}\sum_{|l-j|\geq3N_0}2^{(\ga+2s)k}\|\U_{k-N_0}g\|_{L^1}\|\F_l\tP_kh\|_{H^s}2^{sl}\|\tF_l\tP_k\F^2_j\<D\>^{2\ell} f\|_{L^2}.
\eeno
Recalling the definition of  $\hat{\F}_{j,\al}$(see Definition \ref{Fj}),  Lemma \ref{lemma1.3}(i)(ii), the above can be bounded by
\beno
&&C_{N,\ell}\sum_{k\geq N_0-1}2^{(\ga+2s)k}\|\U_{k-N_0}g\|_{L^1}2^{2sj}(\|\tP_k\tF_jh\|_{L^2}+ 2^{-j}2^{-k}\sum_{|\al|=1}^{2N+2\ell }\|\tP_{k,\al}\tF_{j,\al}h\|_{L^2}+2^{-jN}2^{-kN}\|h\|_{H_{-N}^{-N}})(2^{(2n-1)j}2^{-k}\\
&&\times\sum_{|\al|=1}^{2N+2\ell }\|\tP_{k,\al}\hat{\F}_{j,\al}f\|_{L^2}+2^{-jN}2^{-kN}\|f\|_{H_{-N}^{-N}})+C_{N,\ell}\sum_{k\geq N_0-1}\sum_{|l-j|\geq3N_0}2^{(\ga+2s)k}\|\U_{k-N_0}g\|_{L^1}\\
&&\times \|\F_l\tP_kh\|_{L^2}2^{-l(3N+2)+2ls}2^{-k(3N+2)}2^{-j(3N+2)}\|\F^2_j\<D\>^{2\ell} f\|_{L^2}.
\eeno
Thanks to Lemma \ref{lemma1.4}, we finally  derive that
\beno
&&\big| H_{1,1}^2+H_{1,1}^3\big|
 \ls C_{N,\ell}\sum_{k\ge N_0-1}2^{(\ga+2s-1)k}\|g\|_{L^1}2^{(2n+2s-1)j}\big(\|\tP_k\tF_jh\|_{L^2}+ 2^{-j}2^{-k}\sum_{|\al|=1}^{2N+2\ell }\|\tP_{k,\al}\tF_{j,\al}h\|_{L^2}\\
&&+2^{-jN}2^{-kN}\|h\|_{H_{-N}^{-N}}\big) \big(\sum_{|\al|=1}^{2N+2\ell }\|\tP_{k,\al}\hat{\F}_{j,\al}f\|_{L^2}+2^{-jN}2^{-kN}\|f\|_{H_{-N}^{-N}}\big)+C_{N,\ell}2^{-2jN}\|g\|_{L^1}\|h\|_{H_{-N}^{-N}}\|f\|_{H_{-N}^{-N}}\\
&\ls&C_{N,\ell}\|g\|_{L^1}(\sum_{|\al|=0}^{2N+2\ell }\|\tF_{j,\alpha}h\|_{H^a_{\om_1}}+2^{-jN}\|h\|_{H_{-N}^{-N}})(\sum_{|\al|=0}^{2N+2\ell }\|\hat{\F}_{j,\al}f\|_{H^{b}_{\om_2}}+2^{-jN}\|f\|_{H_{-N}^{-N}}).
\eeno
By the definition of $\mF_j$ (see Definition \ref{Fj}), we conclude that
\beno
|H_{1,1}|\ls C_{N,\ell}\|g\|_{L^1}(\|\mF_jh\|_{H^a_{\omega_1}}+2^{-jN}\|h\|_{H_{-N}^{-N}})(\|\mF_jf\|_{H^{b}_{\omega_2}}+2^{-jN}\|f\|_{H_{-N}^{-N}}).
\eeno
 The other two terms in  $H_{1}$ can be treated by the same manner. We conclude the estimate for $H_1$.

\smallskip
\underline{\it Step 2: Estimate of $H_{2}$}. By the definition of $H_2$, we first focus on the term
\beno&&
H_{2,2}:=\sum_{l\geq k+N_0,k\geq0}\<Q_k(\cP_lg,(\F_j\<D\>^\ell \tP_l-\tP_l\F_j\<D\>^\ell )h),\F_j\<D\>^\ell \tP_lf\>\\
&&=\sum_{l\geq k+N_0,k\geq0}\Big(\sum_{a\leq p-N_0}\fM^1_{k,p,a}(\cP_lg,(\F_j\<D\>^\ell \tP_l-\tP_l\F_j\<D\>^\ell )h),\F_j\<D\>^\ell \tP_lf)\\
&&+\sum_{a\geq-1}\fM^2_{k,a}(\cP_lg,(\F_j\<D\>^\ell \tP_l-\tP_l\F_j\<D\>^\ell )h),\F_j\<D\>^\ell \tP_lf)+\sum_{p\geq-1}\fM^3_{k,p}(\cP_lg,(\F_j\<D\>^\ell \tP_l-\tP_l\F_j\<D\>^\ell )h),\\&&
\F_j\<D\>^\ell \tP_lf)+\sum_{m<p-N_0}\fM^4_{k,p,m}(\cP_lg,(\F_j\<D\>^\ell \tP_l-\tP_l\F_j\<D\>^\ell )h),\F_j\<D\>^\ell \tP_lf)\Big) :=\sum_{i=1}^4H_{2,2}^{i}.
\eeno

 $\bullet$ For $H_{2,2}^{1}$ and $H_{2,2}^{4}$,  by   Lemma \ref{lemma1.7} and \eqref{fpPkfj}, we have
\beno &&|H_{2,2}^{1}+H_{2,2}^{4}|\ls \sum_{l\geq k+N_0,k\geq0}\Big(\sum_{a\leq p-N_0}2^{k(\ga+\frac{5}{2}-4N)}(2^{-p(4N-2s)}2^{2s(a-p)}+2^{-(4N-\frac{5}{2})p}2^{\frac{3}{2}(l-p)}) \|\F_p\cP_lg\|_{L^1}\\
&&\times\|\F_a(\F_j\<D\>^\ell \tP_l-\tP_l\F_j\<D\>^\ell )h\|_{L^2}\|\tF_p\F_j\<D\>^\ell \tP_lf\|_{L^2}+\sum_{m<p-N_0}2^{2s(m-p)}2^{(\ga+\frac{3}{2}-4N)k}2^{-p(4N-\frac{5}{2})} \\
&&\times\|\F_p\cP_lg\|_{L^1}\|\tF_p(\F_j\<D\>^\ell \tP_l-\tP_l\F_j\<D\>^\ell )h\|_{L^2}\|\F_m\F_j\<D\>^\ell \tP_lf\|_{L^2}\Big)\ls C_{N,\ell}2^{-2jN}\|g\|_{L^1}\|h\|_{H_{-N}^{-N}}\|f\|_{H_{-N}^{-N}}.
\eeno

$\bullet$ For $H_{2,2}^{2}$ and $H_{2,2}^{3}$,
\beno
  &&|H_{2,2}^{2}+H_{2,2}^{3}|\ls  \sum_{l\geq k+N_0,k\geq0}\Big(\sum_{a\geq-1} 2^{(\ga+2s)k}2^{2sa}\|S_{a-N_0}\cP_l g\|_{L^1}\|\F_a(\F_j\<D\>^\ell \tP_l-\tP_l\F_j\<D\>^\ell )h\|_{L^2}\|\tF_a\F_j\<D\>^\ell \tP_lf\|_{L^2}\\
&&+\sum_{p\geq-1}
2^{(\ga+2s)k}2^{2s p}\|\F_p\cP_l g\|_{L^1}\|\tF_p(\F_j\<D\>^\ell \tP_l-\tP_l\F_j\<D\>^\ell )h\|_{L^2}\|\tF_p\F_j\<D\>^\ell \tP_lf\|_{L^2}\Big)\\
&\ls&\sum_{l\geq N_0}\Big(\sum_{p\geq-1}2^{(\ga+2s)l}2^{2ps}\|\cP_lg\|_{L^1}\|\tF_p(\F_j\<D\>^\ell \tP_l-\tP_l\F_j\<D\>^\ell )h\|_{L^2}\|\tF_p\F_j\<D\>^\ell \tP_lf\|_{L^2}\Big),
\eeno
where we use Lemma \ref{lemma1.3}(iv). Thanks to \eqref{fpPkfj}, Lemma \ref{lemma1.3}(i)  and Lemma \ref{lemma1.4}, it can be bounded by
\beno
C_{N,\ell}\|g\|_{L^1_{(\ga+2s-1)^++(-\omega_1)^++(-\omega_2)^++\de}}(\|\mF_jh\|_{H^a_{\omega_1}}+2^{-jN}\|h\|_{H_{-N}^{-N}})(\|\mF_jf\|_{H^{b}_{\omega_2}}+2^{-jN}\|f\|_{H_{-N}^{-N}}),
\eeno
which yields that
\beno
H_{2,2}\ls C_{N,\ell}\|g\|_{L^1_{(\ga+2s-1)^++(-\omega_1)^++(-\omega_2)^++\de}}(\|\mF_jh\|_{H^a_{\omega_1}}+2^{-jN}\|h\|_{H_{-N}^{-N}})(\|\mF_jf\|_{H^{b}_{\omega_2}}+2^{-jN}\|f\|_{H_{-N}^{-N}}).
\eeno

The other terms in $H_2$ can be treated by the same manner and we skip the details here to conclude the estimate for $H_2$.

\underline{\it Step 3: Estimate of $H_{3}$}. Let us give the estimate for the typical term in $H_3$:
\beno
&&H_{3,3}:=\sum_{|l-k|\leq N_0,k\geq0}\<Q_k(\cP_lg,\U_{k+N_0}\F_j\<D\>^\ell h),(\F_j\<D\>^\ell \U_{k+N_0}-\U_{k+N_0}\F_j\<D\>^\ell )f\>\\
&=&\sum_{|l-k|\leq N_0,k\geq0}\Big(\sum_{a\leq p-N_0}\fM^1_{k,p,a}(\cP_l g,\U_{k+N_0}\F_j\<D\>^\ell h,(\F_j\<D\>^\ell \U_{k+N_0}-\U_{k+N_0}\F_j\<D\>^\ell )f)+\sum_{a\geq-1}\fM^2_{k,a}(\cP_l g,\\
&&\U_{k+N_0}\F_j\<D\>^\ell h,(\F_j\<D\>^\ell \U_{k+N_0}-\U_{k+N_0}\F_j\<D\>^\ell )f)+\sum_{p\geq-1}\fM^3_{k,p}(\cP_l g,\U_{k+N_0}\F_j\<D\>^\ell h,(\F_j\<D\>^\ell \U_{k+N_0}\\
&&-\U_{k+N_0}\F_j\<D\>^\ell )f)+\sum_{m<p-N_0}\fM^4_{k,p,m}(\cP_l g,\U_{k+N_0}\F_j\<D\>^\ell h,(\F_j\<D\>^\ell \U_{k+N_0}-\U_{k+N_0}\F_j\<D\>^\ell )f)\Big):=\sum_{i=1}^4H_{3,3}^i.
\eeno

$\bullet$ For $H_{3,3}^1$ and $H_{3,3}^4$,    by Lemma \ref{lemma1.7}, we have
\beno
&& |H_{3,3}^1+H_{3,3}^4|
\ls C_N\sum_{|l-k|\leq N_0,k\geq0}\Big(\sum_{a\leq p-N_0}2^{k(\ga+\frac{5}{2}-4N)}(2^{-p(4N-2s)}2^{2s(a-p)}+2^{-(4N-\frac{5}{2})p}2^{\frac{3}{2}(l-p)}) \|\F_p\cP_lg\|_{L^1}\\
&&\times\|\F_a\U_{k+N_0}\F_j\<D\>^\ell h\|_{L^2}\|\tF_p(\F_j\<D\>^\ell \U_{k+N_0}-\U_{k+N_0}\F_j\<D\>^\ell )f\|_{L^2}+\sum_{m<p-N_0}2^{2s(m-p)}2^{(\ga+\frac{3}{2}-4N)k}2^{-p(4N-\frac{5}{2})} \\
&&\times \|\F_p\cP_lg\|_{L^1}\|\tF_p\U_{k+N_0}\F_j\<D\>^\ell h\|_{L^2}\|\F_m(\F_j\<D\>^\ell \U_{k+N_0}-\U_{k+N_0}\F_j\<D\>^\ell )f\|_{L^2}\Big).
\eeno

Similar to \eqref{fpPkfj}, we have
\ben\label{fpukfj}&&\hspace{1cm}\|\tF_p(\F_j\<D\>^\ell \U_{k+N_0}-\U_{k+N_0}\F_j\<D\>^\ell )f\|_{L^2}\\
&&\notag\ls
C_{N,\ell}(\mathrm{1}_{|p-j|\le 3N_0}2^{(\ell-1)j}\sum_{|\alpha|=1}^{2N+2\ell}\|\U_{k+N_0,\alpha}\F_{j,\alpha}f\|_{L^2})+\mathrm{1}_{|p-j|\le 3N_0}2^{-jN}\|f\|_{H_{-N}^{-N}}+\mathrm{1}_{|p-j|> 3N_0}2^{-2jN}\|f\|_{H_{-N}^{-N}}),\een
which implies that
  $|H_{3,3}^1+H_{3,3}^4|\ls C_{N,\ell}2^{-2jN}\|g\|_{L^1}\|h\|_{H_{-N}^{-N}}\|f\|_{H_{-N}^{-N}}.$

$\bullet$ For $H_{3,3}^2$ and $H_{3,3}^3$,    by Lemma \ref{lemma1.7}, we first have
\beno
&& |H_{3,3}^2+H_{3,3}^3|
\ls\sum_{|l-k|\leq N_0,k\geq0}\Big(\sum_{a\geq-1} 2^{(\ga+2s)k}2^{2sa}\|S_{a-N_0}\cP_l g\|_{L^1}\|\F_a\U_{k+N_0}\F_j\<D\>^\ell h\|_{L^2}\|\tF_a(\F_j\<D\>^\ell \U_{k+N_0}\\
&&-\U_{k+N_0}\F_j\<D\>^\ell )f\|_{L^2}+\sum_{p\geq-1}
2^{(\ga+2s)k}2^{2s p}\|\F_p\cP_l g\|_{L^1}\|\tF_p\U_{k+N_0}\F_j\<D\>^\ell h\|_{L^2}\|\tF_p(\F_j\<D\>^\ell \U_{k+N_0}-\U_{k+N_0}\F_j\<D\>^\ell )f\|_{L^2}\Big)\\
&\ls&\sum_{|l-k|\leq N_0,k\geq0}\Big(\sum_{p\geq-1}2^{(\ga+2s)k}2^{2sp}\|\cP_lg\|_{L^1}\|\tF_p\U_{k+N_0}\F_j\<D\>^\ell h\|_{L^2}\|\tF_p(\F_j\<D\>^\ell \U_{k+N_0}-\U_{k+N_0}\F_j\<D\>^\ell )f\|_{L^2}\Big).
\eeno

Due to \eqref{fpukfj}, Lemma \ref{lemma1.3}(ii)(iii) and Lemma \ref{lemma1.4}, the above can be bounded by
\beno
&&C_{N,\ell}\sum_{|l-k|\leq N_0,k\geq0}\sum_{|p-j|>3N_0}2^{(\ga+2s)k}\|\cP_lg\|_{L^1} 2^{(s+\ell)p-2pN-2jN}\|\F_jh\|_{L^2}2^{(s+\ell)p-2Np-2Nj}\|\F_jf\|_{L^2}
\\
&&+C_{N,\ell}\sum_{|l-k|\leq N_0,k\geq0}2^{(\ga+2s)k}2^{(2\ell+2s)j}\|\cP_lg\|_{L^1} \|\U_{k+N_0}\F_jh\|_{L^2} \Big(2^{-j}\sum_{|\al|=1}^{2N+2\ell}\|\U_{k+N_0,\al}\F_{j,\alpha}f\|_{L^2}+2^{-jN}\|f\|_{H_{-N}^{-N}}\Big)
\\
&\ls&C_{N,\ell}\sum_{l\geq-1}2^{(\ga+2s)l}2^{-2jN}\|\cP_lg\|_{L^1}\|h\|_{H_{-N}^{-N}}\|f\|_{H_{-N}^{-N}}+C_N\sum_{k\ge-1} 2^{(\gamma+2s)k}\|\cP_k g\|_{L^1}2^{(-\omega_1)^+k}
\\
&&\times \|\F_jh\|_{H_{\omega_1}^a} 2^{(-\omega_2)^+k}\Big(\sum_{|\al|=1}^{2N+2\ell}\|\F_{j,\alpha}f\|_{H^{b}_{\omega_2}}+2^{-jN}\|f\|_{H_{-N}^{-N}}\Big)
\ls C_{N,\ell}\sum_{l\geq-1}2^{-2jN}\|g\|_{L^1_{(\ga+2s)^+}}\|h\|_{H_{-N}^{-N}}\|f\|_{H_{-N}^{-N}}\\
&&+C_N\|g\|_{L_{(\ga+2s)^++(\omega_1)^++(\omega_2)^+}^1}\|\F_jh\|_{H_{\omega_1}^a}
 (\sum_{|\al|=0}^{2N+2\ell}\|\F_{j,\alpha}f\|_{H^{b}_{\omega_2}}+2^{-jN}\|f\|_{H_{-N}^{-N}}),
\eeno
which implies that
\beno
|H_{3,3}|\ls C_{N,\ell}\|g\|_{L^1_{(\ga+2s)^++(\omega_1)^++(\omega_2)^+}}(\|\mF_jh\|_{H^a_{\omega_1}}+2^{-jN}\|h\|_{H_{-N}^{-N}})(\|\mF_jf\|_{H^{b}_{\omega_2}}+2^{-jN}\|f\|_{H_{-N}^{-N}}).
\eeno
It is not difficult to check that the other terms in $H_3$ can be estimated by the same way. We conclude that
 \beno
|H_{3}| &\ls C_{N,\ell}\|g\|_{L^1_{(\ga+2s)^++(\omega_1)^++(\omega_2)^++\de}}(\|\mF_jh\|_{H^a_{\omega_1}}+2^{-jN}\|h\|_{H_{-N}^{-N}})(\|\mF_jf\|_{H^{b}_{\omega_2}}+2^{-jN}\|f\|_{H_{-N}^{-N}}).
\eeno
We complete the proof for case (ii) by patching together all above estimates.
\smallskip

\noindent\underline{\it Step 4: Proof of the other cases.} We may copy the idea in the above to treat the other cases. We skip the details   and conclude the lemma.
\end{proof}

\subsection{Estimate of $\mathfrak{D}_3$} Similar to \eqref{2.2}, we have
\ben\label{D3}
 \mathfrak{D}_3
 &=&\sum_{|p-j|\leq3N_0}\<\F_j\<D\>^\ell Q_{-1}(S_{p+4N_0}g,\F_ph)-Q_{-1}(S_{p+4N_0}g,\F_j\<D\>^\ell \F_ph),\F_j\<D\>^\ell f\>\\ \notag&&+\sum_{p>j+3N_0}\<\F_j\<D\>^\ell Q_{-1}(\tF_{p}g,\F_ph),\F_j\<D\>^\ell f\>
 +\<\F_j\<D\>^\ell Q_{-1}(\tF_jg,S_{j-3N_0}h),\F_j\<D\>^\ell f\>.
\een
The rest is to estimate the righthand side of \eqref{D3} term by term.

\begin{lem}\label{31}
For smooth functions $g,h,f$, it holds that
\begin{itemize}
  \item[(i)] If $p>j+3N_0$,
\ben\label{2.30}\qquad
 &&\|\<\F_j\<D\>^\ell Q_{-1}(\tF_pg,\F_ph),\F_j\<D\>^\ell f\>\|
 \ls2^{-(3-2s)(p-j)}2^{(2\ell +(2s-1/2)^+)j}\|g\|_{L^2}\|\|\F_ph\|_{L^2}\|\F_jf\|_{L^2}
 \een
  \item[(ii)]
\ben\label{2.31}
|\<\F_j\<D\>^\ell Q_{-1}(\tF_jg,S_{j-3N_0}h),\F_j\<D\>^\ell f\>|
\ls 2^{(2\ell +(2s-1/2)^+)j}\|h\|_{L^2}\|\tF_jg\|_{L^2}\|\F_jf\|_{L^2}1_{2s\geq1}.
\een
\item[(iii)] If $|p-j|\leq3N_0$,
\ben\label{2.32}
&&|\<\F_j\<D\>^\ell Q_{-1}(S_{p+4N_0}g,\F_ph)-Q_{-1}(S_{p+4N_0}g,\F_j\<D\>^\ell \F_ph),\F_j\<D\>^\ell f\>|\\
&\ls& \notag2^{(2\ell +(2s-1/2)^+)j}\|g\|_{L^2}\|\F_ph\|_{L^2}\|\F_jf\|_{L^2}1_{s>1/2}+2^{(2\ell +1/2+\de)j}\|g\|_{L^2}\|\F_ph\|_{L^2}\|\F_jf\|_{L^2}1_{s=1/2}.
\een
\end{itemize}
\end{lem}
\begin{proof}  We first give the detailed proof for the case $2s\geq1$ and $j\geq0$. The case $2s\geq1$ and $j=-1$ can be handled   similarly.

 \underline{\it Step 1: Proof of $(i)$.} Denote $T:=|\<\F_j\<D\>^\ell Q_{-1}(\tF_pg,\F_ph),\F_j\<D\>^\ell f\>|$. Then by  (\ref{bobylev}), we have
\beno
T
&=&\Big|\int_{\si\in \mathbb{S}^2,\eta,\xi\in \R^3}b(\frac{\xi}{|\xi|}\cdot\si)[\cF(\Phi_{-1}^\ga)(\eta-\xi^-)-\cF(\Phi_{-1}^\ga)(\eta)](\cF \tF_pg)(\eta)(\cF\F_p h)(\xi-\eta)\<\xi\>^{2\ell} \vphi(2^{-j}\xi)\overline{(\cF f)}(\xi)\\
&&\times\vphi(2^{-j}\xi)d\si d\eta d\xi\Big|
  \ls\Big|\int_{\si\in \mathbb{S}^2,\eta,\xi\in \R^3}b(\frac{\xi}{|\xi|}\cdot\si)\xi^-\cdot\nabla\cF(\Phi^\ga_{-1})(\eta)(\cF \tF_pg)(\eta)(\cF\F_p h)(\xi-\eta)\<\xi\>^{2\ell} \vphi(2^{-j}\xi)\overline{(\cF f)}(\xi)\\
  &&\times\vphi(2^{-j}\xi)d\si d\eta d\xi\Big|+\Big|\int_0^1\int_{\si\in \mathbb{S}^2,\eta,\xi\in \R^3}b(\frac{\xi}{|\xi|}\cdot\si)\nabla^2\cF(\Phi^{\ga}_{-1})(\eta-t\xi^-):\xi^-\otimes\xi^-|(\cF \tF_pg)(\eta)(\cF\F_p h)(\xi-\eta)\\
  &&\times\vphi(2^{-j}\xi)\overline{(\cF f)}(\xi)\<\xi\>^{2\ell} \vphi(2^{-j}\xi)d\si d\eta d\xi dt\Big|.
\eeno
Since  $p>j+3N_0$, we have $\<\eta\>\sim\<\xi-\eta\>\sim\<\eta-t\xi^-\>\sim 2^p$ for $t\in[0,1]$. Thanks to  \eqref{PhiK2}, we derive that
 \beno
T&\ls&2^{2j\ell }\int (\<\eta\>^{-(\ga+4)}|\xi|+\<\eta\>^{-(\ga+5)}|\xi|^2)|(\cF \tF_pg)(\eta)(\cF\F_p h)(\xi-\eta)\vphi(2^{-j}\xi)(\cF f)(\xi)|d\xi d\eta\\
 &\ls&2^{2j\ell }2^{-(\ga+4)p}\int_{|\xi|\sim 2^j}|\xi||\vphi(2^{-j}\xi)(\cF f)(\xi)|\int_{\eta\in\R^3}|(\cF \tF_pg)(\eta)(\cF\F_p h)(\xi-\eta)|d\eta d\xi.
 \eeno
 By Cauchy-Schwartz inequality, we observe that
  \beno
 T&\ls&2^{-(\ga+4)p}2^{(\frac{5}{2}+2\ell )j}\|\tF_pg\|_{L^2}\|\F_ph\|_{L^2}\|\F_jf\|_{L^2}\ls2^{-(\ga+2s+1)p}2^{-(3-2s)(p-j)}2^{(2\ell+2s-1/2)j}\|\tF_pg\|_{L^2}\|\F_ph\|_{L^2}\|\F_jf\|_{L^2}.
 \eeno

\underline{\it Step 2: Proof of (iii).} Set $R:=|\<\F_j\<D\>^\ell Q_{-1}(S_{p+4N_0}g,\F_ph)-Q_{-1}(S_{p+4N_0}g,\F_j\<D\>^\ell \F_ph),\F_jf\>|$. By \eqref{bobylev}, one has
\beno
R&=&|\int_{\si\in \mathbb{S}^2,\eta,\xi\in \R^3}b(\frac{\xi}{|\xi|}\cdot\si)[\cF(\Phi_{-1}^\ga)(\eta-\xi^-)-\cF(\Phi_{-1}^\ga)(\eta)](\cF S_{p+4N_0} g)(\eta)(\cF \F_ph)(\xi-\eta)\\
  \notag &&\times\<\xi\>^\ell \vphi(2^{-j}\xi)\overline{(\cF f)}(\xi)(\<\xi\>^\ell \vphi(2^{-j}\xi)-\<\xi-\eta\>^\ell \vphi(2^{-j}(\xi-\eta))d\si d\eta d\xi|.
\eeno
Similar to the proof of Lemma \ref{leD1}, we split   $R$  into two parts: $R_{1}$ and $R_{2}$, which correspond to the integration domain of $R$: $2|\xi^-|\leq \<\eta\>$ and $2|\xi^-|>\<\eta\>$ respectively. The proof falls in several steps.
\smallskip

\noindent\underline{\it Step 2.1: Estimate of $R_1$.} In the region $2|\xi^-|\leq \<\eta\>$, we have $\sin(\th/2)\leq\<\eta\>/|\xi|$ and $\<\eta-t\xi^-\>\sim\<\eta\>$ for $t\in[0,1]$. By Taylor expansion and Lemma \ref{lemma1.5}((\ref{PhiK2})), by $|\<\xi\>^\ell \vphi(2^{-j}\xi)-\<\xi-\eta\>^\ell \vphi(2^{-j}(\xi-\eta))|\ls 2^{(\ell -1)j}|\eta|$, we have
\beno
|R_{1}|&\leq&\bigg|\int_{2|\xi^-|\leq \<\eta\>}b(\frac{\xi}{|\xi|}\cdot\si)(\na \cF(\Phi^\ga_{-1}))(\eta)\cdot\xi^-(\cF S_{p+4N_0}g)(\eta)(\cF\F_ph)(\xi-\eta)\<\xi\>^\ell \vphi(2^{-j}\xi)\overline{(\cF f)}(\xi)(\<\xi\>^\ell \vphi(2^{-j}\xi)\\
&&-\<\xi-\eta\>^\ell \vphi(2^{-j}(\xi-\eta))d\si d\eta d\xi\bigg|+\bigg|\int_0^1\int_{{2|\xi^-|\leq \<\eta\>}}b(\frac{\xi}{|\xi|}\cdot\si)(\na^2\cF(\Phi^\ga_{-1}))(\eta-t\xi^-):\xi^-\otimes\xi^-)\\
&&\times(\cF S_{p+4N_0}g)(\eta)(\cF\F_ph)(\xi-\eta)\<\xi\>^\ell \vphi(2^{-j}\xi)\overline{(\cF f)}(\xi)(\<\xi\>^\ell \vphi(2^{-j}\xi)-\<\xi-\eta\>^\ell \vphi(2^{-j}(\xi-\eta))d\si d\eta d\xi dt\bigg|\\
&\ls& R_{1,1}+R_{1,2},
\eeno
where 
\beno
R_{1,1}&=& 2^{(2\ell -1)j}\int_{}|\eta|\<\eta\>^{-(\ga+4)}|\xi|\min\{1,(\<\eta\>/|\xi|)^{2-2s}\}|(\cF S_{p+4N_0}g)(\eta)(\cF\F_ph)(\xi-\eta)\vphi(2^{-j}\xi)\overline{(\cF f)}(\xi)|d\eta d\xi,\\
R_{1,2}&=& 2^{(2\ell -1)j}\int_{}|\eta|\<\eta\>^{-(\ga+5)}|\xi|^2(\<\eta\>/|\xi|)^{2-2s}|(\cF S_{p+4N_0}g)(\eta)(\cF\F_ph)(\xi-\eta)\vphi(2^{-j}\xi)\overline{(\cF f)}(\xi)|d\eta d\xi.
\eeno

\noindent$\bullet$ \underline{Estimate of $R_{1,1}$.}  We split the integration domain of $R_{1,1}$ into two parts: $\<\eta\>\geq |\xi|$ and $\<\eta\>< |\xi|$, which are due to $\min\{1,(\<\eta\>/|\xi|)^{2-2s}\}$. Recalling that $\<\eta\>\ls 2^j,\<\xi\>\sim 2^j$ and $p\sim j$, we have
\beno
R_{1,1}
&\ls&2^{(2\ell -1)j}\int_{\<\eta\>\geq|\xi|}\<\eta\>^{-(\ga+3)}|\xi||(\cF S_{p+4N_0}g)(\eta)(\cF\F_ph)(\xi-\eta)\vphi(2^{-j}\xi)\overline{(\cF f)}(\xi)|d\eta d\xi\\
&&+2^{(2\ell -1)j}\int_{{\<\eta\><|\xi|}}\<\eta\>^{-(\ga+2s+1)}|\xi|^{2s-1}|(\cF S_{p+4N_0}g)(\eta)(\cF\F_ph)(\xi-\eta)\vphi(2^{-j}\xi)\overline{(\cF f)}(\xi)|d\eta d\xi\\&\ls&2^{-(\ga+2s+1)j}2^{(2\ell +2s-1/2)j}\|S_{p+4N_0}g\|_{L^2}\|\F_ph\|_{L^2}\|\F_jf\|_{L^2}.
\eeno

\noindent$\bullet$ \underline{Estimate of $R_{1,2}$.} We have
\beno
R_{1,2}
&\ls&2^{(2\ell -1)j}\int\,\<\eta\>^{-(\ga+2s+2)}|\xi|^{2s}|(\cF S_{p+4N_0}g)(\eta)(\cF\F_ph)(\xi-\eta)\vphi(2^{-j}\xi)\overline{(\cF f)}(\xi)|d\eta d\xi\\
&\ls&2^{-(\ga+2s+1)j}2^{(2\ell +2s-1/2)j}\|S_{p+4N_0}g\|_{L^2}\|\F_ph\|_{L^2}\|\F_jf\|_{L^2}.
\eeno

Patching together the estimates of $R_{1,1}$ and $R_{1,2}$, we have
\beno
R_{1}\ls2^{-(\ga+2s+1)j}2^{(2\ell+2s-1/2)j}\|S_{p+4N_0}g\|_{L^2}\|\F_ph\|_{L^2}\|\F_jf\|_{L^2}.
\eeno

\noindent\underline{\it Step 2.2: Estimate of $R_2$.} We first note that
\beno
|R_{2}|
&\ls&2^{(2\ell -1)j}\int_{2|\xi^-|\geq\<\eta\>}|\eta|b(\frac{\xi}{|\xi|}\cdot\si)|\cF(\Phi^\ga_{-1}(\eta)||(\cF S_{p+4N_0}g)(\eta)(\cF\F_ph)(\xi-\eta)\vphi(2^{-j}\xi)\overline{(\cF f)}(\xi)|d\si d\eta d\xi\\
&+&2^{(2\ell -1)j}\int_{2|\xi^-|\geq\<\eta\>}|\eta-\xi^-|b(\frac{\xi}{|\xi|}\cdot\si)|\cF(\Phi^\ga_{-1})(\eta-\xi^-)|(\cF S_{p+4N_0}g)(\eta)(\cF\F_ph)(\xi-\eta)\vphi(2^{-j}\xi)\overline{(\cF f)}(\xi)|d\si d\eta d\xi\\
&+&2^{(2\ell -1)j}\int_{2|\xi^-|\geq\<\eta\>}|\xi^-|b(\frac{\xi}{|\xi|}\cdot\si)|\cF(\Phi^\ga_{-1})(\eta-\xi^-)|(\cF S_{p+4N_0}g)(\eta)(\cF\F_ph)(\xi-\eta)\vphi(2^{-j}\xi)\overline{(\cF f)}(\xi)|d\si d\eta d\xi\\
&:=& R_{2,1}+R_{2,2}+R_{2,3}.
\eeno
Since  $2|\xi^-|>\<\eta\>$, in what follows, we will frequently use the facts that \ben\label{thetaR} \sin(\th/2)\gs\<\eta-\xi^-\>/(3|\xi|), \quad \sin(\th/2)\geq\<\eta\>/(2|\xi|). \een
\noindent$\bullet$ \underline{Estimate of $R_{2,1}$.} We have
\beno
R_{2,1}
&\ls&2^{(2\ell -1)j}\int\,\<\eta\>^{-(\ga+2s+2)}|\xi|^{2s}|(\cF S_{p+4N_0}g)(\eta)(\cF\F_ph)(\xi-\eta)\vphi(2^{-j}\xi)\overline{(\cF f)}(\xi)|d\eta d\xi.
\eeno
Similar to $R_{1,2}$, we easily have
$R_{2,1}
 \ls 2^{-(\ga+2s+1)j}2^{(2\ell +2s-1/2)j}\|S_{p+4N_0}g\|_{L^2}\|\F_ph\|_{L^2}\|\F_jf\|_{L^2}.$

\noindent$\bullet$ \underline{Estimate of $R_{2,2}$.}
We may split the integration domain of $R_{2,2}$ into two parts:  $|\eta-\xi^-|\geq\<\eta\>$ and $|\eta-\xi^-|<\<\eta\>$. In the region $|\eta-\xi^-|\geq\<\eta\>$, thanks to Lemma \ref{lemma1.5}, one may get $|\eta-\xi^-||\cF(\Phi^\ga_{-1})(\eta-\xi^-)|\ls \<\eta\>^{-(\ga+2)}$. While in the region $|\eta-\xi^-|\leq\<\eta\>$, we use the change of variables: $\eta-\xi^-\rightarrow \tilde{\eta}$(see the estimate of $\A_{2,1}$ in Lemma \ref{leD1}) and \eqref{thetaR}. Then  we have
\beno
R_{2,2}
&\ls&2^{(2\ell -1)j}\int_{\eta,\xi}\<\eta\>^{-(\ga+2s+2)}|\xi|^{2s}|(\cF S_{p+4N_0}g)(\eta)(\cF\F_ph)(\xi-\eta)\vphi(2^{-j}\xi)\overline{(\cF f)}(\xi)|d\eta d\xi\\&&+2^{(2\ell -1)j}\left(\int_{\sin(\theta/2)\ge\<\eta\>/(2|\xi|)}b(\frac{\xi}{|\xi|}\cdot\si)|\cF S_{p+4N_0}g(\eta)|^2|\cF\F_jf(\xi)|^2\<\xi\>^{2s}\<\eta\>^{2s}d\si d\eta d\xi\right)^{1/2}\\
&&\times\left(\int_{\sin(\theta/2)\gs\<\tilde{\eta}\>/(3|\xi|)}b(\frac{\xi^+}{|\xi^+|}\cdot\si)|\tilde{\eta}|^2|\cF\Phi^\ga_{-1}(\tilde{\eta})|^2\<\xi^+\>^{-2s}\<\tilde{\eta}\>^{-2s}|\F_ph(\xi^+-\tilde{\eta})|^2d\si d\tilde{\eta} d\xi^+\right)^{1/2}.\eeno
From this, we get that
$R_{2,2}\ls2^{(2\ell +2s-1/2)j}\|S_{p+4N_0}g\|_{L^2}\|\F_ph\|_{L^2}\|\F_jf\|_{L^2}.$

\noindent$\bullet$ \underline{Estimate of $R_{2,3}$.}
Similar to $R_{1,2}$, we shall split the integration domain of $R_{2,3}$ into two parts: $|\eta-\xi^-|\geq4\<\eta\>$ and $|\eta-\xi^-|<4\<\eta\>$.  In the region $|\eta-\xi^-|\ge 4\<\eta\>$,  one get that $|\xi^-||\cF(\Phi^\ga_{-1})(\eta-\xi^-)|\ls \<\eta\>^{-(\ga+2)}$.   \\
$\bullet$  If $2s>1$, by Cauchy-Schwartz inequality, we derive that
\beno
R_{2,3}
&\ls&2^{(2\ell -1)j}   \int\<\eta\>^{-(\ga+2s+2)}|\xi|^{2s}|(\cF S_{p+4N_0}g)(\eta)(\cF\F_ph)(\xi-\eta)\vphi(2^{-j}\xi)\overline{(\cF f)}(\xi)|d\eta d\xi
\\
&&+2^{(2\ell -1)j} \left(\int_{\sin(\theta/2)\ge\<\eta\>/(2|\xi|)}  |\xi|^2 \sin(\theta/2)b(\frac{\xi}{|\xi|}\cdot\si)|\cF S_{p+4N_0}g(\eta)|^2|\cF\F_jf(\xi)|^2\<\xi\>^{2s-1}\<\eta\>^{2s-1}d\si d\eta d\xi\right)^{1/2}\\
&&\times\left(\int_{\sin(\theta/2)\gs\<\tilde{\eta}\>/(3|\xi|)}\sin(\theta /2)b(\frac{\xi^+}{|\xi^+|}\cdot\si)|\cF\Phi^\ga_{-1}(\tilde{\eta})|^2\<\xi^+\>^{1-2s}\<\tilde{\eta}\>^{1-2s}|\F_ph(\xi^+-\tilde{\eta})|^2d\si d\tilde{\eta} d\xi^+\right)^{1/2}\\
&\ls&2^{(2\ell +2s-1/2)j}\|S_{p+4N_0}g\|_{L^2}\|\F_ph\|_{L^2}\|\F_jf\|_{L^2}.\eeno
$\bullet$ If $2s=1$, for any $0<\delta\ll 1$,   we have
\beno
R_{2,3}&\ls&2^{(2\ell -1)j}\int\<\eta\>^{-(\ga+2s+2)}|\xi|^{2s}|(\cF S_{p+4N_0}g)(\eta)(\cF\F_ph)(\xi-\eta)\vphi(2^{-j}\xi)\overline{(\cF f)}(\xi)|d\eta d\xi\\
&&+\left(\int_{\sin(\theta/2)\ge\<\eta\>/(2|\xi|)}\sin^{1-2\de}(\th/2)b(\frac{\xi}{|\xi|}\cdot\si)|\cF S_{p+4N_0}g(\eta)|^2|\cF\F_jf(\xi)|^2d\si d\eta d\xi\right)^{1/2}\\
&&\times\left(\int_{\sin(\theta/2)\gs\<\tilde{\eta}\>/(3|\xi|)}\sin^{1+2\de}(\th/2)b(\frac{\xi^+}{|\xi^+|}\cdot\si)|\cF\Phi^\ga_{-1}(\tilde{\eta})|^2|\F_ph(\xi^+-\tilde{\eta})|^2d\si d\tilde{\eta} d\xi^+\right)^{1/2}\\
&\ls&2^{(2\ell +1/2+\de)j}\|S_{p+4N_0}g\|_{L^2}\|\F_ph\|_{L^2}\|\F_jf\|_{L^2}..
\eeno
We derive  that
 \beno R_{2,3}\ls2^{(2\ell +2s-1/2)j}\|g\|_{L^2}\|\F_ph\|_{L^2}\|\F_jf\|_{L^2}1_{s>1/2}+2^{(2\ell +1/2+\de)j}\|g\|_{L^2}\|\F_ph\|_{L^2}\|\F_jf\|_{L^2}1_{s=1/2}\},\eeno
from which together with the estimates of $R_{2,1}$ and  $R_{2,2}$, we get that
\beno
R_{2}\ls2^{(2\ell+2s-1/2)j}\|g\|_{L^2}\|\F_ph\|_{L^2}\|\F_jf\|_{L^2}1_{s>1/2}+2^{(2\ell +1/2+\de)j}\|g\|_{L^2}\|\F_ph\|_{L^2}\|\F_jf\|_{L^2}1_{s=1/2}\}.
\eeno

Finally patching together all the estimates, we conclude that
\beno
R\ls2^{(2\ell +2s-1/2)j}\|g\|_{L^2}\|\F_ph\|_{L^2}\|\F_jf\|_{L^2}1_{s>1/2}+2^{(2\ell +1/2+\de)j}\|g\|_{L^2}\|\F_ph\|_{L^2}\|\F_jf\|_{L^2}1_{s=1/2}\}.
\eeno
It ends the proof of (\ref{2.32}).
\smallskip

\underline{\it Step 3: Proof of $(ii)$.} We denote $P:=|\<\F_j\<D\>^\ell Q_{-1}(\tF_jg,S_{j-3N_0}h),\F_j\<D\>^\ell f\>|$. We have
\beno
P&=& \Big|\int_{\si\in \mathbb{S}^2,\eta,\xi\in \R^3}b(\frac{\xi}{|\xi|}\cdot\si)[\cF(\Phi_{-1}^\ga)(\eta-\xi^-)-\cF(\Phi_{-1}^\ga)(\eta)](\cF \tF_jg)(\eta)(\cF S_{j-3N_0} h)(\xi-\eta)\\
  \notag &&\times\<\xi\>^{2\ell} \vphi(2^{-j}\xi)\overline{(\cF f)}(\xi)\vphi(2^{-j}\xi)d\si d\eta d\xi\Big|.
\eeno
Notice that $|\xi|\sim 2^j,|\eta|\sim 2^j$ and $|\eta-\xi|\ls 2^{j-3N_0}$, then $|\eta-\xi^-|=|\eta-\xi+\xi^+|\sim 2^j$. We  split the integration domain of $P$ into two parts: $2|\xi^-|\leq \<\eta\>$ and $2|\xi^-|>\<\eta\>$ and denote them by $P_{1}$ and $P_{2}$ respectively.
\smallskip

\noindent$\bullet$ \underline{Estimate of $P_{1}$.} We may copy the argument for $R_{1}$ to $P_{1}$ to get that
\beno
|P_{1}|&\ls&2^{-(\ga+2s+1)j}2^{(2\ell+2s-1/2)j}\|\tF_jg\|_{L^2}\|S_{j-3N_0}h\|_{L^2}\|\F_jf\|_{L^2}.
\eeno

\noindent$\bullet$ \underline{Estimate of $P_{2}$.} We have
\beno
|P_{2}|
&\ls&2^{2\ell j}\int_{2|\xi^-|\geq\<\eta\>}b(\frac{\xi}{|\xi|}\cdot\si)|\cF(\Phi^\ga_{-1}(\eta)||(\cF \tF_jg)(\eta)(\cF\F_ph)(\xi-\eta)\vphi(2^{-j}\xi)\overline{(\cF f)}(\xi)|d\si d\eta d\xi\\
&&+2^{2\ell j}\int_{2|\xi^-|\geq\<\eta\>}b(\frac{\xi}{|\xi|}\cdot\si)|\cF(\Phi^\ga_{-1})(\eta-\xi^-)|(\cF \tF_jg)(\eta)(\cF S_{j-3N_0}h)(\xi-\eta)\vphi(2^{-j}\xi)\overline{(\cF f)}(\xi)|d\si d\eta d\xi.
\eeno
Since $\<\eta-\xi^-\>,\<\eta\> \sim2^j$, we deduce that
\beno
|P_{2}|&\ls&
2^{-(\ga+2s+1)j}2^{(2\ell +2s-1/2)j}\|\tF_jg\|_{L^2}\|S_{j-3N_0}h\|_{L^2}\|\F_jf\|_{L^2}.
\eeno
Patching together the estimate of $P_1$ and $P_2$, we derive that
\beno
P&\ls&2^{-(\ga+2s+1)j}2^{(2\ell +2s-1/2)j}\|\tF_jg\|_{L^2}\|S_{j-3N_0}h\|_{L^2}\|\F_jf\|_{L^2}.
\eeno
Then we conclude the estimate (\ref{2.31}).

We emphasize that it is even easier to get the estimate for the case $2s<1$ since we only need the first order Taylor expansion. Thus we omit the details and end the proof of this lemma.
 \end{proof}

Next we use the dyadic decomposition in phase space to improve the above estimates. Indeed, we have the following lemma.
\begin{lem}\label{32}
For smooth functions $g,h,f$ and $M>0$, we have that
\begin{itemize}
  \item[(i)] If $p>j+3N_0$,
\ben\label{reg1}
&&\<\F_j\<D\>^\ell Q_{-1}(\tF_pg,\F_ph),\F_j\<D\>^\ell f\>| \\
 &\ls&\notag C_{N,\ell}2^{-(3-2s)(p-j)}\|g\|_{L^2_{(-\omega_3)^++(-\omega_4)^+}}(\|\mF_p h\|_{H^{c}_{\omega_3}}+2^{-pN}\|h\|_{H_{-N}^{-N}})(\|\mF_j f\|_{H^{d}_{\omega_4}}+2^{-jN}\|f\|_{H_{-N}^{-N}})
\\ \quad {\rm (ii) }\label{reg2}
&&|\<\F_j\<D\>^\ell Q_{-1}(\tF_jg,S_{j-3N_0}h),\F_j\<D\>^\ell f\>|\\
&\ls&\notag C_{N,\ell}\|h\|_{L^2_{(-\omega_3)^++(-\omega_4)^+}}(\|\mF_jg\|_{H^{c}_{\omega_3}}+2^{-jN}\|g\|_{H_{-N}^{-N}})(\|\mF_jf\|_{H^{d}_{\omega_4}}+2^{-jN}\|f\|_{H_{-N}^{-N}})\\
&&\notag+C_N2^{-2jN}\|g\|_{L^1}\|h\|_{H_{-N}^{-N}}\|f\|_{H_{-N}^{-N}}.
 \een
 \item[(iii)] If $|p-j|\leq3N_0$,
\ben\label{reg3}
\notag&&|\<\F_j\<D\>^\ell Q_{-1}(S_{p+4N_0}g,\F_ph)-Q_{-1}(S_{p+4N_0}g,\F_j\<D\>^\ell \F_ph),\F_j\<D\>^\ell f\>|
\\&&\ls C_{N,\ell}\|g\|_{L^2_{2+(-\omega_3)^++(-\omega_4)^+}}(\|\mF_ph\|_{H^{c}_{\omega_3}}+2^{-pN}\|h\|_{H_{-N}^{-N}})(\|\mF_jf\|_{H^{d}_{\omega_4}}+2^{-jN}\|f\|_{H_{-N}^{-N}}).
\een
\end{itemize}
where $\om_3,\om_4\in\R$ satisfying $\om_3+\om_4=\ga+2s-1$ and $c,d\geq0$ satisfying $c+d=(2\ell +2s-1/2)_{2s\geq1/2,2s\neq1}+(2\ell +1/2+\de)1_{2s=1}+2\ell 1_{2s<1/2}$ with $0<\de\ll1$, $N\in\N$ can be large enough  and $\mF_j$ is defined in Definition \ref{Fj}. We remark that $\om_3,\om_4$ and $c,d$ can be different in different lines.
\end{lem}
\begin{proof} We only provide the proof for the case $2s>1$. Then case $2s\leq1$ can be handled similarly.

 \underline{\it Step 1: Proof of (i).}  Similar to (\ref{2.1}), we have $\<\F_j\<D\>^\ell Q_{-1}(\tF_pg,\F_ph),\F_j\<D\>^\ell f\>
 := G_1+G_2$,
where
\beno
G_1&=&\sum_{l\geq N_0}\<\F_j\<D\>^\ell Q_{-1}(\cP_l\tF_pg,\F_p\tP_lh),\F_j\<D\>^\ell \tP_lf\>+\sum_{l<N_0}\<\F_j\<D\>^\ell Q_{-1}(\cP_l\tF_pg,\F_p\U_{N_0}h),\F_j\<D\>^\ell \U_{N_0}f\>\\
G_2&=&\sum_{l\geq N_0}(\<Q_{-1}(\cP_l\tF_pg,\tP_l\F_ph),(\F_j^2\<D\>^{2\ell} \tP_l-\tP_l\F_j^2\<D\>^{2\ell} )f\>+\<Q_{-1}(\cP_l\tF_pg,(\tP_l\F_p-\F_p\tP_l)h),\F_j^2\<D\>^{2\ell} \tP_lf\>)\\
&&+\sum_{l<N_0}(\<Q_{-1}(\cP_l\tF_pg,\F_p\U_{N_0}h),(\F_j^2\<D\>^{2\ell} \U_{N_0}-\U_{N_0}\F_j^2\<D\>^{2\ell} )f\>+\<Q_{-1}(\cP_l\tF_pg,(\F_p\U_{N_0}-\U_{N_0}\F_p)h),\\
&&\F_j^2\<D\>^{2\ell} \U_{N_0}f\>):=G_{2,1}+G_{2,2}+G_{2,3}+G_{2,4}.
\eeno

\noindent \underline{\it Step 1.1 Estimate of $G_1$.} Since $\xi\sim 2^j,\<\eta-\xi\>\sim 2^p$ and $p>j+3N_0$, we have that $\<\eta\>\sim 2^p$, which implies
\beno
G_1&=&\sum_{l\geq N_0}\<\F_j\<D\>^\ell Q_{-1}(\tF_p\cP_l\tF_pg,\F_p\tP_lh),\F_j\<D\>^\ell \tP_lf\>+\sum_{l< N_0}\<\F_j\<D\>^\ell Q_{-1}(\tF_p\cP_l\tF_pg,\F_p\U_{N_0}h),\F_j\<D\>^\ell \U_{N_0}f\>.
\eeno
From  Lemma \ref{31}(\ref{2.30}), we have
\beno
 |G_{1}|&\ls&\sum_{l\geq N_0}2^{-(3-2s)(p-j)}2^{(2\ell +2s-1/2)j}\|\cP_l\tF_pg\|_{L^2}\|\F_p\tP_lh\|_{L^2}\|\F_j\tP_lf\|_{L^2}\\
 &&+\sum_{l<N_0}2^{-(3-2s)(p-j)}2^{(2\ell +2s-1/2)j} \|\cP_l\tF_pg\|_{L^2}\|\F_p\U_{N_0}h\|_{L^2}\|\F_j\U_{N_0}f\|_{L^2}.
\eeno
Thanks to Lemma \ref{lemma1.3}(i),
\ben\label{FPPLC}\qquad\qquad &&\|\F_p\tP_lh\|_{L^2}\ls C_{N,\ell}(\|\tP_l\F_ph\|_{L^2}+\sum_{|\al|=1}^{2N+2\ell }\|\tP_{l,\al}\tF_{p,\al}h\|_{L^2}+2^{-lN}2^{-pN}\|h\|_{H_{-N}^{-N}}), \een
together with  Lemma \ref{lemma1.4}(\ref{7.70}), we have
\beno
&&\sum_{l\geq N_0}2^{(2\ell +2s-1/2)j}\|\cP_l\tF_pg\|_{L^2}\|\F_p\tP_lh\|_{L^2}\|\F_j\tP_lf\|_{L^2}\ls C_{N,\ell}\sum_{l\geq N_0}2^{(2\ell +2s-1/2)j}\|\cP_l\tF_pg\|_{L^2_{(-\omega_3)^++(-\omega_2)^+}}2^{(\omega_3+\omega_4)l}\big(\|\tP_l\F_ph\|_{L^2}\\&&+\sum_{|\al|=1}^{2N+2\ell }\|\tP_{l,\al}\tF_{p,\al}h\|_{L^2}+2^{-lN}2^{-pN}\|h\|_{H_{-N}^{-N}}\big) \big(\|\tP_l\F_jf\|_{L^2}+\sum_{|\al|=1}^{2N+2\ell }\|\tP_{l,\al}\tF_{j,\al}f\|_{L^2}+2^{-lN}2^{-jN}\|f\|_{H_{-N}^{-N}}\big)
\eeno\beno &&\ls C_{N,\ell} \|\tF_pg\|_{L^2_{(-\omega_3)^++(-\omega_2)^+}}(\|\mF_{p}h\|_{H^{c}_{\omega_3}}+2^{-pN}\|h\|_{H_{-N}^{-N}})(\|\mF_{j}f\|_{H^{d}_{\omega_4}}+2^{-jN}\|f\|_{H_{-N}^{-N}}). \eeno
We can also copy the above argument to $\|\cP_l\tF_pg\|_{L^2}\|\F_p\U_{N_0}h\|_{L^2}\|\F_j\U_{N_0}f\|_{L^2}$. Thanks to
 facts $\|\tF_pg\|_{L^2_{l}}\ls\|g\|_{L^2_{l}}$(Lemma \ref{lemma1.4}), we conclude that
\beno
|G_{1}|&\ls&C_{N,\ell}2^{-(3-2s)(p-j)}\|\tF_pg\|_{L^2_{(-\omega_3)^++(-\omega_2)^+}}(\|\mF_{p}h\|_{H^{c}_{\omega_3}}+2^{-pN}\|h\|_{H_{-N}^{-N}})(\|\mF_{j}f\|_{H^{d}_{\omega_4}}+2^{-jN}\|f\|_{H_{-N}^{-N}}).
\eeno

\noindent  \underline{\it Step 1.2 Estimate of $G_2$.} We shall give the estimates term by term.

$\bullet$ \underline{\it Estimate of $G_{2,1}$.} We introduce the following decomposition: $G_{2,1}=\sum_{i=1}^5G^{(i)}_{2,1}$ where\\
$G^{(1)}_{2,1}=\sum\limits_{l\geq N_0}\sum\limits_{|a-p|>N_0}\<Q_{-1}(\cP_l\tF_pg,\F_a\tP_l\F_ph),(\F_j^2\<D\>^{2\ell} \tP_l-\tP_l\F_j^2\<D\>^{2\ell} )f\>$, $G^{(2)}_{2,1}=\sum\limits_{l\geq N_0}\sum\limits_{|a-p|\leq N_0}\sum\limits_{m<j-N_0}\<Q_{-1}(\cP_l\tF_pg,\\
\F_a\tP_l\F_ph),\F_m(\F_j^2\<D\>^{2\ell} \tP_l-\tP_l\F_j^2\<D\>^{2\ell} )f\>$,  $G^{(3)}_{2,1}=\sum\limits_{l\geq N_0}\sum\limits_{|a-p|\leq N_0}\sum\limits_{j-N_0\leq m<j+N_0}\<Q_{-1}(\cP_l\tF_pg,\F_a\tP_l\F_ph),\F_m(\F_j^2\<D\>^{2\ell} \tP_l\\-\tP_l\F_j^2\<D\>^{2\ell} )f\>$,  $G^{(4)}_{2,1}=\sum\limits_{l\geq N_0}\sum\limits_{|a-p|\leq N_0}\sum\limits_{j+N_0\leq m<p-2N_0}\<Q_{-1}(\cP_l\tF_pg,\F_a\tP_l\F_ph),\F_m(\F_j^2\<D\>^{2\ell} \tP_l-\tP_l\F_j^2\<D\>^{2\ell} )f\>$ and\\ $G^{(5)}_{2,1}=\sum\limits_{l\geq N_0}\sum\limits_{|a-p|\leq N_0}\sum\limits_{m\geq p-2N_0}\<Q_{-1}(\cP_l\tF_pg,\F_a\tP_l\F_ph),\F_m(\F_j^2\<D\>^{2\ell} \tP_l-\tP_l\F_j^2\<D\>^{2\ell} )f\>$.

\underline{\it Estimate of $G^{(1)}_{2,1}$ and $G^{(5)}_{2,1}$.} We begin with the estimate of $G^{(1)}_{2,1}$. From Lemma \ref{lemma1.7}(iv), we have
\beno
|G^{(1)}_{2,1}|&=&\sum_{l\geq N_0}\sum_{|a-p|>N_0}\<Q_{-1}(\cP_l\tF_pg,\F_a\tP_l\F_ph),(\F_j^2\<D\>^{2\ell} \tP_l-\tP_l\F_j^2\<D\>^{2\ell} )f\>\\
 &\ls& \sum_{l\geq N_0}\sum_{|a-p|>N_0}(\|\cP_l\tF_pg\|_{L^1}+\|\cP_l\tF_pg\|_{L^2})\|\F_a\tP_l\F_ph\|_{H^{2s}}\|(\F_j^2\<D\>^{2\ell} \tP_l-\tP_l\F_j^2\<D\>^{2\ell} )f\|_{L^2}.
\eeno By Bernstein inequality(see Lemma \ref{7.8}) that $\|\F_pf\|_{L^2}\ls2^{\frac{3}{2}p}\|\F_pf\|_{L^1}$ and Lemma \ref{lemma1.3}(ii), we derive that
\beno
|G^{(1)}_{2,1}|
 &\ls& C_{N,\ell}\sum_{l\geq N_0}\sum_{|a-p|>N_0}2^{-4(a+p+l)N}\|\tF_pg\|_{L^1}2^{2as}2^{2\ell j}\|\F_ph\|_{L^2}(\sum_{|\al|=1}^{2N+2\ell }\|\tP_{l,\al}\hat{\F}_{j,\al}f\|_{L^2}+2^{-jN}2^{-lN}\|f\|_{H_{-N}^{-N}})\\
 &\ls&C_{N,\ell}2^{-jN-pN}2^{-(p-j)N}\|g\|_{L^1}\|h\|_{H_{-N}^{-N}}\|f\|_{H_{-N}^{-N}}.
  \eeno
Similarly, since $p>j+3N_0$, we have $m\geq p-2N_0>j+N_0$. Then Lemma \ref{lemma1.3}(ii) implies that
 \beno
|G^{(5)}_{2,1}|
&=&\big|\sum_{l\geq N_0}\sum_{m\geq p-2N_0}\<Q_{-1}(\cP_l\tF_pg,\tF_p\tP_l\F_ph),\F_m\tP_l\F_j^2\<D\>^{2\ell} f\>\big|\\
&\ls&\sum_{l\geq N_0} \sum_{m\geq p-2N_0}(\|\cP_l\tF_pg\|_{L^1}+\|\cP_l\tF_pg\|_{L^2})2^{2sm}\|\tF_p\tP_l\F_ph\|_{L^2}\|\F_m\tP_l\F^2_j\<D\>^{2\ell} f\|_{L^2}\\
&\ls& C_{N,\ell}2^{-jN-pN}2^{-(p-j)N}\|g\|_{L^1}\|h\|_{H_{-N}^{-N}}\|f|
_{H_{-N}^{-N}}.\eeno

\underline{\it Estimate of $G^{(2)}_{2,1}$ and $G^{(4)}_{2,1}$.} It is not difficult to see that
 $G^{(2)}_{2,1}
 =\sum\limits_{l\geq N_0}\sum\limits_{|a-p|\leq N_0}\sum\limits_{m<j-N_0}\<Q_{-1}(\cP_l\tF_pg,\F_a\tP_l\F_ph),\\\F_m\tP_l\F_j^2\<D\>^{2\ell} f\>.$
 Notice that $m<j-N_0<p-3N_0<a-2N_0$.   Then by Lemma \ref{31}(\ref{2.30}), we have
\beno
|G^{(2)}_{2,1}|
 &\ls&\sum_{l\geq N_0}\sum_{m<j-N_0} 2^{-(3-2s)(p-m)}2^{(2\ell +2s-1/2)m}(\|\cP_l\tF_pg\|_{L^2}\|\tF_p\tP_l\F_ph\|_{L^2}\|\F_m\tP_l\F^2_jf\|_{L^2}.
 \eeno
Since $m<j-N_0$, we may apply Lemma \ref{lemma1.3}(ii) to $\F_m\tP_l\F^2_jf$ to get that
\beno
&&|G^{(2)}_{2,1}|
 \ls C_N\sum_{l\ge C_{N,\ell}q N_0}\sum_{m<j-N_0} 2^{-(m+j+l)N}2^{-(3-2s)(p-m)}\|\cP_l\tF_pg\|_{L^2}\|\tP_l\F_ph\|_{L^2}\|\F_jf\|_{H^{-N}}\\
 &\ls& C_{N,\ell}2^{-(3-2s)p}2^{-jN}\|\tF_pg\|_{L^2_{(-\omega_3)^+}}\|\F_ph\|_{L^2_{\omega_3}}\|\mF_jf\|_{H^{-N}}.
\eeno

For $G^{(4)}_{2,1}$, we first have $G^{(4)}_{2,1}=-\sum\limits_{l\geq N_0}\sum\limits_{|a-p|\leq N_0}\sum\limits_{j+N_0\leq m<p-2N_0}\<Q_{-1}(\cP_l\tF_pg,\F_a\tP_l\F_ph),\F_m\tP_l\F_j^2\<D\>^{2\ell} )f\>.$
Since    $m<p-2N_0\leq a-N_0$, we may copy the argument for $G^{(2)}_{2,1}$ to get that
\beno
 |G^{(4)}_{2,1}|
 &\ls&C_{N,\ell}2^{-(3-2s)p}2^{-jN}\|\tF_pg\|_{L^2_{(-\omega_3)^+}}\|\F_ph\|_{L^2_{\omega_3}}\|\mF_jf\|_{H^{-N}}.
\eeno

\underline{\it Estimate of $G^{(3)}_{2,1}$.} We first note that $m<j+N_0<p-2N_0\leq a-N_0$.  Lemma \ref{31}(\ref{2.30}) implies that
\beno
|G^{(3)}_{2,1}|
 &\ls& \sum_{l \geq N_0}2^{-(3-2s)(p-j)}2^{(2s-1/2)j}\|\cP_l\tF_pg\|_{L^2}\|\tF_p\tP_l\F_ph\|_{L^2}\|\tF_j(\F^2_j\<D\>^{2\ell} \tP_l-\tP_l\F_j^2\<D\>^{2\ell} )f\|_{L^2}.
\eeno
From this together with \eqref{fpPkfj} and Lemma \ref{lemma1.4}(\ref{7.70}), we deduce that
\beno
|G^{(3)}_{2,1}|
 &\ls&C_{N,\ell}2^{-(3-2s)(p-j)}\|\tF_pg\|_{L^2_{(\omega_3)^++(-\omega_4)^+}}\|\F_ph\|_{H^{c}_{\omega_3}}(\|\mF_jf\|_{H^{d}_{\omega_4}}+2^{-jN}\|f\|_{H_{-N}^{-N}}).
\eeno

Now putting together all these estimates, we obtain that
\beno
 |G_{2,1}|
 &\ls&C_{N,\ell}2^{-(3-2s)(p-j)}\|\tF_pg\|_{L^2_{(\omega_3)^++(-\omega_4)^+}}\|\F_ph\|_{H^{c}_{\omega_3}}(\|\mF_jf\|_{H^{d}_{\omega_4}}+2^{-jN}\|f\|_{H_{-N}^{-N}}).
\eeno

$\bullet$\underline{\it Estimate of $G_{2,4}$.} We set $G_{2,4}=G_{2,4}^{(1)}+G_{2,4}^{(2)}$, where
$G_{2,4}^{(1)}=\sum\limits_{l< N_0}\sum\limits_{|a-p|>N_0}\<Q_{-1}(\cP_l\tF_pg, \F_a(\F_p\U_{N_0}-\U_{N_0}\F_p)h),\F_j^2\<D\>^{2\ell} \U_{N_0}f\>$ and $G_{2,4}^{(2)}=\sum\limits_{l< N_0}\sum\limits_{|a-p|\leq N_0,|m-j|<N_0}\<Q_{-1}(\cP_l\tF_pg,\F_a(\F_p\U_{N_0}-\U_{N_0}\F_p)h),\F_m\F_j^2\<D\>^{2\ell} \U_{N_0}f\>$.

\underline{\it Estimate of $G_{2,4}^{(1)}$.} We first observe that $G_{2,4}^{(1)}=\sum\limits_{l< N_0}\sum\limits_{|a-p|>N_0}\<Q_{-1}(\cP_l\tF_pg,\F_a\U_{N_0}\F_ph),\F_j^2\<D\>^{2\ell} \U_{N_0}f\>$.
Then by Lemma \ref{lemma1.7}(iv) and \eqref{fpukfj}, we have
\beno
|G_{2,4}^{(1)}|&\ls&\sum_{l< N_0}\sum_{|a-p|>N_0}\|\cP_l\tF_pg\|_{L^1}\|\F_a\U_{N_0}\F_ph\|_{H^{2s}}\|\F_j^2\<D\>^{2\ell} \U_{N_0}f\|_{L^2}\ls C_N\sum_{l< N_0}\sum_{|a-p|>N_0}2^{-a(4N+1)}2^{-p(4N+1)}2^{2sa}\\&&\times\|\cP_l\tF_pg\|_{L^1}\|\F_ph\|_{L^2}\|\F_j^2\<D\>^{2\ell} \U_{N_0}f\|_{L^2}\ls C_N2^{-pN-jN}2^{-2(p-j)}\|g\|_{L^1}\|h\|_{H_{-N}^{-N}}\|f\|_{H_{-N}^{-N}}.
\eeno

\underline{\it Estimate of $G_{2,4}^{(2)}$.}  Since
 $G_{2,4}^{(2)}=\sum\limits_{l< N_0}\sum\limits_{|a-p|\leq N_0,|m-j|<N_0}\<Q_{-1}(\cP_l\tF_pg,\F_a(\F_p\U_{N_0}-\U_{N_0}\F_p)h),\F_m\F_j^2\<D\>^{2\ell} \U_{N_0}f\>$,  Lemma \ref{31}(\ref{2.30}) and Lemma \ref{lemma1.3} imply that
\beno
|G_{2,4}^{(2)}|
&\ls&\sum_{l< N_0}2^{-(3-2s)(p-j)}2^{(2\ell +2s-1)j}\|\cP_l\tF_pg\|_{L^2_1}\|(\F_p\U_{N_0}-\U_{N_0}\F_p)h\|_{L^2}\|\F_j^2\U_{N_0}f\|_{L^2}\\
&\ls&C_{N,\ell}2^{-(3-2s)(p-j)}\|\tF_pg\|_{L^2_{(\omega_3)^++(-\omega_4)^+}}(\|\mF_p h\|_{H^{c}_{\omega_3}}+2^{-pN}\|h\|_{H_{-N}^{-N}})(\|\mF_j f\|_{H^{d}_{\omega_4}}+2^{-jN}\|f\|_{H_{-N}^{-N}}).
\eeno

We conclude that
\beno
|G_{2,4}|&\ls&C_{N,\ell}2^{-(3-2s)(p-j)}\|\tF_pg\|_{L^2_{(\omega_3)^++(-\omega_4)^+}}(\|\mF_p h\|_{H^{c}_{\omega_3}}+2^{-pN}\|h\|_{H_{-N}^{-N}})(\|\mF_j f\|_{H^{d}_{\omega_4}}+2^{-jN}\|f\|_{H_{-N}^{-N}}).
\eeno

It is not difficult to check that structures of  $G_{2,2}$ and $G_{2,3}$ are similar to $G_{2,1}$ and $G_{2,4}$.
We have
\beno
|G_2|
&\ls &C_{N,\ell}2^{-(3-2s)(p-j)}\|\tF_pg\|_{L^2_{(\omega_3)^++(-\omega_4)^+}}(\|\mF_p h\|_{H^{c}_{\omega_3}}+2^{-pN}\|h\|_{H_{-N}^{-N}})(\|\mF_j f\|_{H^{d}_{\omega_4}}+2^{-jN}\|f\|_{H_{-N}^{-N}}).
\eeno

\underline{\it Step 1.3 Conclusion.} Finally  due to Lemma \ref{lemma1.3}(iii)(iv), we deduce that
\beno
&&|\<\F_j\<D\>^\ell Q_{-1}(\tF_pg,\F_ph),\F_j\<D\>^\ell f\>|\\
&\ls& C_{N,\ell}2^{-(3-2s)(p-j)}\|g\|_{L^2_{(\omega_3)^++(-\omega_4)^+}}(\|\mF_p h\|_{H^{c}_{\omega_3}}+2^{-pN}\|h\|_{H^{-N}})(\|\mF_j f\|_{H^{d}_{\omega_4}}+2^{-jN}\|f\|_{H^{-N}}),
\eeno
which implies  (\ref{reg1}).
\smallskip

\noindent\underline{\it Step 2: Proof of (ii).} Observe that $\<\F_j\<D\>^\ell Q_{-1}(\tF_jg,S_{j-3N_0}h),\F_j\<D\>^\ell f\>=X_1+X_2$,
where
\beno
X_1&=&\sum_{l\geq N_0}\<\F_j\<D\>^\ell Q_{-1}(\tF_j\cP_lg,\tP_lS_{j-3N_0}h),\F_j\<D\>^\ell \tP_lf\>+\sum_{l<N_0}\<\F_j\<D\>^\ell Q_{-1}(\tF_j\tP_lg,\U_{N_0}S_{j-3N_0}h),\F_j\<D\>^\ell \U_{N_0}f\>,\\
X_2
&=&\sum_{l\geq N_0}(\<Q_{-1}(\cP_l\tF_jg,\tP_lS_{j-3N_0}h),(\tP_l\F_j^2\<D\>^{2\ell} -\F_j^2\<D\>^{2\ell} \tP_l)f\>+\<Q_{-1}((\cP_l\tF_j-\tF_j\cP_l)g,\tP_lS_{j-3N_0}h),\F_j^2\<D\>^{2\ell} \tP_lf\>)\\
&&+\sum_{l< N_0}(\<Q_{-1}(\cP_l\tF_jg,\U_{N_0}S_{j-3N_0}h),(\U_{N_0}\F_j^2\<D\>^{2\ell} -\F_j^2\<D\>^{2\ell} \U_{N_0})f\>+\<Q_{-1}((\cP_l\tF_j-\tF_j\cP_l)g,\U_{N_0}S_{j-3N_0}h),\\&&\F_j^2\<D\>^{2\ell} \U_{N_0}f\>):=X_{2,1}+X_{2,2}+X_{2,3}+X_{2,4}.
\eeno

$\bullet$ \noindent\underline{\it Estimate of $X_1$.} We split $X_1$ into two parts: $X_1=X_{1,1}+X_{1,2}$ where\\
\beno
X_{1,1}&=&\sum\limits_{l\geq N_0}\<\F_j\<D\>^\ell Q_{-1}(\tF_j\cP_lg,S_{j-2N_0}\tP_lS_{j-3N_0}h),\F_j\<D\>^\ell \tP_lf\>\\
&&+\sum\limits_{l< N_0}\<\F_j\<D\>^\ell Q_{-1}(\tF_j\cP_lg,S_{j-2N_0}\U_{N_0}S_{j-3N_0}h),\F_j\<D\>^\ell \U_{N_0}f\>
\eeno
 and
 \beno
 X_{1,2}&=&\sum\limits_{l\geq N_0}\sum\limits_{a\geq j-2N_0}\<\F_j\<D\>^\ell Q_{-1}(\tF_j\cP_lg,\F_a\tP_lS_{j-3N_0}h),\F_j\<D\>^\ell \tP_lf\>\\
 &&+\sum\limits_{l<N_0}\sum\limits_{a\geq j-2N_0}\<\F_j\<D\>^\ell Q_{-1}(\tF_j\cP_lg,\F_a\U_{N_0}S_{j-3N_0}h),\F_j\<D\>^\ell \U_{N_0}f\>.
 \eeno

\smallskip
\noindent\underline{\it Estimate of $X_{1,1}$.}
From Lemma \ref{31}(\ref{2.31}), \eqref{FPPLC} and Lemma \ref{lemma1.4}(\ref{7.70}), we have
\beno
|X_{1,1}|&\ls& \sum_{l\geq N_0}2^{(2\ell +2s-1/2)j}\|\tP_lS_{j-3N_0}h\|_{L^2}\|\F_j\cP_lg\|_{L^2}\|\F_j\tP_lf\|_{L^2}+\sum_{l<N_0}2^{(2\ell +2s-1/2)j}\|\U_{N_0}S_{j-3N_0}h\|_{L^2}\|\F_j\cP_lg\|_{L^2}\\
&&\times\|\F_j\U_{N_0}f\|_{L^2}
\ls C_{N,\ell}\|h\|_{L^2_{(-\omega_3)^++(-\omega_4)^+}}(\|\mF_jg\|_{H^{c}_{\omega_3}}+2^{-jN}\|g\|_{H_{-N}^{-N}})(\|\mF_jf\|_{H^{d}_{\omega_4}}+2^{-jN}\|f\|_{H_{-N}^{-N}}).
\eeno

\noindent\underline{\it Estimate of $X_{1,2}$.} Due to Lemma \ref{lemma1.7}$(iv)$ and Bernstein's inequality(see Lemma \ref{7.8}), one has
\beno
 |X_{1,2}|&\ls&\sum_{l\geq N_0}\sum_{a\geq j-2N_0}2^{\frac{3}{2}j}\|\tF_j\cP_lg\|_{L^1}\|\F_a\tP_lS_{j-3N_0}h\|_{L^2}2^{(2\ell +2s)j}\|\F_j\tP_lf\|_{L^2}\\
 &&+\sum_{l< N_0}\sum_{a\geq j-2N_0}2^{\frac{3}{2}j}\|\tF_j\cP_lg\|_{L^1}\|\F_a\U_{N_0}S_{j-3N_0}h\|_{L^2}2^{(2\ell +2s)j}\|\F_j\U_{N_0}f\|_{L^2}).\eeno
Applying Lemma \ref{lemma1.3}$(ii)$ to $\F_a\tP_lS_{j-3N_0}h$ and $\F_a\U_{N_0}S_{j-3N_0}h$, we get that
 \beno  |X_{1,2}|
&\ls& C_{N,\ell}\sum_{l\geq N_0}2^{-4lN}2^{-4jN}\|\tF_j\cP_lg\|_{L^1}\|S_{j-3N_0}h\|_{L^2}\|\F_j\tP_lf\|_{L^2}\\
&&+C_{N,\ell}\sum_{l<N_0}2^{-4jN}\|\tF_j\cP_lg\|_{L^1}\|S_{j-3N_0}h\|_{L^2}\|\F_j\U_{N_0}f\|_{L^2}\ls C_{N,\ell}2^{-2jN}\|g\|_{L^1}\|h\|_{H_{-N}^{-N}}\|f\|_{H_{-N}^{-N}}.\eeno

We conclude that
\beno
|X_{1}|&\ls& C_{N,\ell}\|h\|_{L^2_{(-\omega_3)^++(-\omega_4)^+}}(\|\mF_jg\|_{H^{c}_{\omega_3}}+2^{-jN}\|g\|_{H_{-N}^{-N}})(\|\mF_jf\|_{H^{d}_{\omega_4}}+2^{-jN}\|f\|_{H_{-N}^{-N}})\\
&&+C_{N,\ell}2^{-2jN}\|g\|_{L^1}\|h\|_{H_{-N}^{-N}}\|f\|_{H_{-N}^{-N}}.
\eeno

$\bullet$ \noindent\underline{\it Estimate of $X_2$.} We  will give the estimates term by term.

\noindent\underline{\it Estimate of $X_{2,1}$.} We have  $X_{2,1}=X_{2,1}^{(1)}+X_{2,1}^{(2)}+X_{2,1}^{(3)}$, where  $X_{2,1}^{(1)}=\sum\limits_{l\geq N_0}\sum\limits_{|a-j|>N_0}\<Q_{-1}(\cP_l\tF_jg,\\\tP_lS_{j-3N_0}h),\F_a(\tP_l\F_j^2\<D\>^{2\ell} -\F_j^2\<D\>^{2\ell} \tP_l)f\>$, $X_{2,1}^{(2)}=\sum\limits_{l\geq N_0}\sum\limits_{|a-j|\leq N_0}\sum\limits_{b>j-2N_0}\<Q_{-1}(\cP_l\tF_jg,\F_b\tP_lS_{j-3N_0}h),\F_a(\tP_l\F_j^2\<D\>^{2\ell} \\-\F_j^2\<D\>^{2\ell} \tP_l)f\>$ and  $X_{2,1}^{(3)}=\sum\limits_{l\geq N_0}\sum\limits_{|a-j|\leq N_0}\<Q_{-1}(\cP_l\tF_jg,S_{j-2N_0}\tP_lS_{j-3N_0}h),\F_a(\tP_l\F_j^2\<D\>^{2\ell} -\F_j^2\<D\>^{2\ell} \tP_l)f\>$.

 Similar to the estimate of $G_{2,1}^{(1)}$, we have
\beno
|X_{2,1}^{(1)}|+|X_{2,1}^{(2)}|&\ls&C_{N,\ell}2^{-2jN}\|g\|_{L^1}\|h\|_{H_{-N}^{-N}}\|f\|_{H_{-N}^{-N}}.
\eeno

For $X_{2,1}^{(3)}$, we have   $|\xi|\sim 2^j$ and $|\xi-\eta|\ls 2^{j-2N_0}$, which implies $|\eta|\sim 2^j$. Similar to $X_{1,1}$, we get that
\beno
|X_{2,1}^{(3)}|&\ls& C_{N,\ell}\|h\|_{L^2_{(-\omega_3)^++(-\omega_4)^+}}(\|\mF_jg\|_{H^{c}_{\omega_3}}+2^{-jN}\|g\|_{H_{-N}^{-N}})(\|\mF_jf\|_{H^{d}_{\omega_4}}+2^{-jN}\|f\|_{H_{-N}^{-N}}).
\eeno
Then we conclude that
\beno &&
|X_{2,1}|\ls C_{N,\ell}\|h\|_{L^2_{(-\omega_3)^++(-\omega_4)^+}}(\|\mF_jg\|_{H^{c}_{\omega_3}}+2^{-jN}\|g\|_{H_{-N}^{-N}})(\|\mF_jf\|_{H^{d}_{\omega_4}}+2^{-jN}\|f\|_{H_{-N}^{-N}})\\
&&+C_{N,\ell}2^{-2jN}\|g\|_{L^1}\|h\|_{H_{-N}^{-N}}\|f\|_{H_{-N}^{-N}}.
\eeno

\noindent\underline{\it Estimate of $X_{2,4}$.} We introduce $X_{2,4}=X_{2,4}^{(1)}+X_{2,4}^{(2)}$, where $X_{2,4}^{(1)}=\sum\limits_{l<N_0}\sum\limits_{a>j-2N_0}\<Q_{-1}((\cP_l\tF_j-\tF_j\cP_l)g,\\\F_a\U_{N_0}S_{j-3N_0}h),\F_j^2\<D\>^{2\ell} \U_{N_0}f\>$ and $X_{2,4}^{(2)}=\sum\limits_{l<N_0}\<Q_{-1}((\cP_l\tF_j-\tF_j\cP_l)g,S_{j-2N_0}\U_{N_0}S_{j-3N_0}h),\F_j^2\<D\>^{2\ell} \U_{N_0}f\>$.

 Similar to $X_{1.2}$, we first have
\beno
|X_{2,4}^{(1)}|&\ls&\sum_{l<N_0}\sum_{a>j-2N_0} (\|\cP_l\tF_jg\|_{L^1}+\|\tF_j\cP_lg\|_{L^1}+\|\cP_l\tF_jg\|_{L^2}+\|\tF_j\cP_lg\|_{L^2})\|\F_a\U_{N_0}S_{j-3N_0}h\|_{L^2}\|\F_j^2\<D\>^{2\ell} \U_{N_0}f\|_{L^2}\\
&\ls& C_{N,\ell}2^{-2jN}\|g\|_{L^1}\|h\|_{H_{-N}^{-N}}\|f\|_{H_{-N}^{-N}}.
\eeno

Copying the argument used for $X_{2,1}^{(3)}$ to $X_{2,4}^{(2)}$, then we have
\beno
|X_{2,4}^{(2)}|&\ls& C_{N,\ell}\|h\|_{L^2_{(-\omega_3)^++(-\omega_4)^+}}(\|\mF_jg\|_{H^{c}_{\omega_3}}+2^{-jN}\|g\|_{H_{-N}^{-N}})(\|\mF_jf\|_{H^{d}_{\omega_4}}+2^{-jN}\|f\|_{H_{-N}^{-N}}).
\eeno
Then we conclude that
\beno
|X_{2,4}|&\ls& C_{N,\ell}\|h\|_{L^2_{(-\omega_3)^++(-\omega_4)^+}}(\|\mF_jg\|_{H^{c}_{\omega_3}}\\
&&+2^{-jN}\|g\|_{H_{-N}^{-N}})(\|\mF_jf\|_{H^{d}_{\omega_4}}+2^{-jN}\|f\|_{H_{-N}^{-N}})+C_{N,\ell}2^{-2jN}\|g\|_{L^1}\|h\|_{H_{-N}^{-N}}\|f\|_{H_{-N}^{-N}}.
\eeno

 The estimates of $X_{2,2}$ and $X_{2,3}$ could be handled in a similar manner as  $X_{2,1}$ and $X_{2,4}$. We skip the details here and then conclude our desired result.
\smallskip

\noindent\underline{\it Step 3: Proof of (iii).} For $(iii)$, we introduce the following decompositions:  $\<\F_j\<D\>^\ell Q_{-1}(S_{p+4N_0}g,\F_ph)\\-Q_{-1}(S_{p+4N_0}g,\F_j\<D\>^\ell \F_ph),\F_j\<D\>^\ell f\>=\sum\limits_{i=1}^4Y_i$, where
\beno &&Y_1=\sum\limits_{l\geq N_0}\<\F_j\<D\>^\ell Q_{-1}(\cP_lS_{p+4N_0}g,\F_p\tP_lh)-Q_{-1}(\cP_lS_{p+4N_0}g,\F_j\<D\>^\ell \F_p\tP_lh),\F_j\<D\>^\ell \tP_lf\>,\\&&
 Y_2=\sum\limits_{l<N_0}\<\F_j\<D\>^\ell Q_{-1}(\cP_lS_{p+4N_0}g,\F_p\U_{N_0}h)-Q_{-1}(\cP_lS_{p+4N_0}g,\F_j\<D\>^\ell \F_p\U_{N_0}h),\F_j\<D\>^\ell \U_{N_0}f\>,
  \\&& Y_3=\sum\limits_{l\geq N_0}\big(\<Q_{-1}(\cP_lS_{p+4N_0}g,\tP_l\F_ph),(\tP_l\F^2_j\<D\>^{2\ell} -\F^2_j\<D\>^{2\ell} \tP_l)f\>+\<Q_{-1}(\cP_lS_{p+4N_0}g,\tP_l\F_j\<D\>^\ell \F_ph),\\ &&(\tP_l\F_j\<D\>^\ell -\F_j\<D\>^\ell \tP_l)f\>
+\<Q_{-1}(\cP_lS_{p+4N_0}g,\F_j\<D\>^\ell (\F_p\tP_l-\tP_l\F_p)h),\F_j\<D\>^\ell \tP_lf)\>\\ &&+\<Q_{-1}(\cP_lS_{p+4N_0}g,(\F_j\<D\>^\ell \tP_l-\tP_l\F_j\<D\>^\ell )\F_jh),\F_j\<D\>^\ell
\tP_lf)\>\big), \\&&Y_4=\sum\limits_{l< N_0}\big(\<Q_{-1}(\cP_lS_{p+4N_0}g,\U_{N_0}\F_ph),(\U_{N_0}\F^2_j\<D\>^{2\ell} -\F^2_j\<D\>^{2\ell} \U_{N_0})f\>+\<Q_{-1}(\cP_lS_{p+4N_0}g,\U_{N_0}\F_j\<D\>^\ell \F_ph),\\ &&(\U_{N_0}\F_j\<D\>^\ell
-\F_j\<D\>^\ell \U_{N_0})f\>+\<Q_{-1}(\cP_lS_{p+4N_0}g,\F_j\<D\>^\ell (\F_p\U_{N_0}-\U_{N_0}\F_p)h),\F_j\<D\>^\ell \U_{N_0}f\>\\ &&+\<Q_{-1}(\cP_lS_{p+4N_0}g,(\F_j\<D\>^\ell \U_{N_0}-\U_{N_0}\F_j)\F_jh),\F_j\<D\>^\ell \U_{N_0}f\>\big).\eeno

  Since  $Y_1$ and $Y_2$ enjoy almost the same structure, we only need to give the detailed proof for $Y_1$. By \eqref{bobylev}, we first note that
$Y_1=\sum\limits_{l\geq N_0}\<\F_j\<D\>^\ell
Q_{-1}(S_{p+4N_0}\cP_lS_{p+4N_0}g,\F_p\tP_lh)-Q_{-1}(S_{p+4N_0}\cP_lS_{p+4N_0}g,\F_j\<D\>^\ell \F_p\tP_lh),\\\F_j\<D\>^\ell \tP_lf\>$. Then by (\ref{2.32}), \eqref{FPPLC} and Lemma \ref{lemma1.3}(iii), one has
\beno |Y_1|&\ls& \sum_{l\geq N_0}2^{(2\ell +2s-1/2)j}(\|\cP_lS_{p+4N_0}g\|_{L^2}\|\F_p\tP_lh\|_{L^2}\|\F_j\tP_lf\|_{L^2}+\|\cP_lS_{p+4N_0}g\|_{L^2}\|\F_p\U_{N_0}h\|_{L^2}\|\F_j\U_{N_0}f\|_{L^2})
\\&\ls& C_{N,\ell}\|g\|_{L^2_{(-\omega_3)^++(-\omega_4)^+}}(\|\mF_jh\|_{H^{c}_{\omega_3}}+2^{-pN}\|h\|_{H_{-N}^{-N}}) (\|\mF_jf\|_{H^{d}_{\omega_4}}+2^{-jN}\|f\|_{H_{-N}^{-N}}).
\eeno
As for $Y_3$ and $Y_4$, let us choose $Z:=\sum\limits_{l\geq N_0}\<Q_{-1}(\cP_lS_{p+4N_0}g,\tP_l\F_ph),(\tP_l\F^2_j\<D\>^{2\ell} -\F^2_j\<D\>^{2\ell} \tP_l)f\>$ as a typical term to give the estimate. It is easy to see that $Z:=Z_1+Z_2+Z_3$ where $Z_1=\sum\limits_{l\geq N_0}\sum\limits_{|a-p|>N_0}\<Q_{-1}(\cP_lS_{p+4N_0}g,\F_a\tP_l\F_ph),\\(\tP_l\F^2_j\<D\>^{2\ell} -\F^2_j\<D\>^{2\ell} \tP_l)f\>$, $Z_2=\sum\limits_{l\geq N_0}\sum\limits_{|a-p|\leq N_0}\sum\limits_{|b-j|>N_0}\<Q_{-1}(\cP_lS_{p+4N_0}g,\F_a\tP_l\F_ph), \F_b(\tP_l\F^2_j\<D\>^{2\ell} -\F^2_j\<D\>^{2\ell} \tP_l)f\>$ and\\
$Z_3=\sum\limits_{l\geq N_0}\sum\limits_{|a-p|\leq N_0}\sum\limits_{|b-j|\leq N_0}\<Q_{-1}(S_{p+6N_0}\cP_lS_{p+4N_0}g,\F_a\tP_l\F_ph),\F_b(\tP_l\F^2_j\<D\>^{2\ell} -\F^2_j\<D\>^{2\ell} \tP_l)f\>$.

Applying the argument used for $G^{(1)}_{2,1}$, one may obtain that
$|Z_1+Z_2|\ls C_{N,\ell}2^{-2jN}\|g\|_{L^1}\|h\|_{H_{-N}^{-N}}\|f\|_{H_{-N}^{-N}}$.  Moreover, due to Lemma \ref{lemma1.7}(iv), we have
\beno
|Z_3|&\ls&\sum_{l\geq N_0}\sum\limits_{|b-j|\leq N_0}(2^{2sj}\|S_{p+6N_0}\cP_lS_{p+4N_0}g\|_{L^2_2}\|\tP_l\F_ph\|_{L^2}\|\F_b(\tP_l\F^2_j\<D\>^{2\ell} -\F^2_j\<D\>^{2\ell} \tP_l)f\|_{L^2}.
\eeno
Then \eqref{fpPkfj}, Lemma \ref{lemma1.3}(iii) and Lemma \ref{lemma1.4} yield that
\beno
|Z_3|&\ls&C_{N,\ell}\|g\|_{L^2_{2+(-\omega_3)^++(-\omega_4)^+}}\|\F_ph\|_{H^{c}_{\omega_3}}(\|\mF_jf\|_{H^{d}_{\omega_4}}+2^{-jN}\|f\|_{H_{-N}^{-N}}).
\eeno
From these, we get that $|Z|\ls C_{N,\ell}\|g\|_{L^2_{2+(-\omega_3)^++(-\omega_4)^+}}(\|\mF_jh\|_{H^{c}_{\omega_3}}+2^{-pN}\|h\|_{H_{-N}^{-N}}) (\|\mF_jf\|_{H^{d}_{\omega_4}}+2^{-jN}\|f\|_{H_{-N}^{-N}}).$
We complete the proof of $(iii)$ by patching together the estimates of $Y_1$ and $Z$.
This ends the proof of the lemma.
\end{proof}

\begin{lem}\label{le2.5}
 For smooth function $g,h$ and $f$, we have
\ben\label{regu}
&& \notag\mathfrak{D}_3\ls \sum_{p>j+3N_0}C_{N,\ell}2^{-(3-2s)(p-j)}\|g\|_{L^2_{(-\omega_3)^++(-\omega_4)^+}}(\|\mF_p h\|_{H^{c}_{\omega_3}}+2^{-pN}\|h\|_{H_{-N}^{-N}})(\|\mF_j f\|_{H^{d}_{\omega_4}}+2^{-jN}\|f\|_{H_{-N}^{-N}})\\
&&+\notag C_{N,\ell}\|h\|_{L^2_{(-\omega_3)^++(-\omega_4)^+}}(\|\mF_jg\|_{H^{c}_{\omega_3}}+2^{-jN}\|g\|_{H_{-N}^{-N}})(\|\mF_jf\|_{H^{d}_{\omega_4}}+2^{-jN}\|f\|_{H_{-N}^{-N}})+C_{N,\ell}2^{-2jN}\|g\|_{L^1}\|h\|_{H_{-N}^{-N}}\\
&&\times\|f\|_{H_{-N}^{-N}}+C_{N,\ell}\|g\|_{L^2_{2+(-\omega_3)^++(-\omega_4)^+}}(\|\mF_ph\|_{H^{c}_{\omega_3}}+2^{-pN}\|h\|_{H_{-N}^{-N}})(\|\mF_jf\|_{H^{d}_{\omega_4}}+2^{-jN}\|f\|_{H_{-N}^{-N}}).
\een
where $\om_3,\om_4,c,d,\de$ are defined in Lemma \ref{32}. We remark that $\om_3,\om_4$ and $c,d$ can be different in different lines.
\end{lem}
\begin{proof} This can be easily derived by
 Lemma \ref{32} and  (\ref{D3}).
\end{proof}

Now we are ready to prove Theorem \ref{le1.24}.
\begin{proof}[Proof of Theorem \ref{le1.24}]
Patching together the estimates of $\mathfrak{D}_1,\mathfrak{D}_2$ and $\mathfrak{D}_3$, we obtain that
\beno
&&\left|\sum_{j=-1}^\infty(\F_j\<D\>^\ell Q(g,h)-Q(g,\F_j\<D\>^\ell h),\F_j\<D\>^\ell f)\right|\\
&\ls& \sum_{j=-1}^\infty C_{N,\ell}(\|g\|_{L^1_{(\ga+2s-1)^++(-\omega_1)^++(-\omega_2)^++\de}}(\|\mF_jh\|_{H^a_{\omega_1}}+2^{-jN}\|h\|_{H_{-N}^{-N}})(\|\mF_jf\|_{H^{b}_{\omega_2}}+2^{-jN}\|f\|_{H_{-N}^{-N}})\\
&&+\sum_{p>j+3N_0}C_{N,\ell}2^{-(3-2s)(p-j)}\|g\|_{L^2_{(-\omega_3)^++(-\omega_4)^+}}(\|\mF_p h\|_{H^{c_1}_{\omega_3}}+2^{-pN}\|h\|_{H_{-N}^{-N}})(\|\mF_j f\|_{H^{d_1}_{\omega_4}}+2^{-jN}\|f\|_{H_{-N}^{-N}})\\
&&+\notag C_{N,\ell}\|h\|_{L^2_{(-\omega_5)^++(-\omega_6)^+}}(\|\mF_jg\|_{H^{c_2}_{\omega_5}}+2^{-jN}\|g\|_{H_{-N}^{-N}})(\|\mF_jf\|_{H^{d_2}_{\omega_6}}+2^{-jN}\|f\|_{H_{-N}^{-N}})+C_N2^{-2jN}\|g\|_{L^1}\|h\|_{H_{-N}^{-N}}\|f\|_{H_{-N}^{-N}}\\
&&+\notag C_{N,\ell}\|g\|_{L^2_{2+(-\omega_3)^++(-\omega_4)^+}}(\|\mF_jh\|_{H^{c_1}_{\omega_3}}+2^{-jN}\|h\|_{H_{-N}^{-N}})(\|\mF_jf\|_{H^{d_1}_{\omega_4}}+2^{-jN}\|f\|_{H_{-N}^{-N}}))\\
&\ls& C_{N,\ell}(\|g\|_{L^1_{(\ga+2s-1)^++(-\omega_1)^++(-\omega_2)^++\de}}\|h\|_{H^a_{\omega_1}}\|f\|_{H^b_{\omega_2}}+\|g\|_{L^2_{2+(-\omega_3)^++(-\omega_4)^+}}\|h\|_{H^{c_1}_{\omega_3}}\|f\|_{H^{d_1}_{\omega_4}}\\
&&+\|h\|_{L^2_{(-\omega_5)^++(-\omega_6)^+}}\|g\|_{H^{c_2}_{\omega_5}}\|f\|_{H^{d_2}_{\omega_6}}+(\|g\|_{L^1_{(\ga+2s-1)^++(-\omega_1)^++(-\omega_2)^++\de}}+\|g\|_{L^2_{2+(-\omega_3)^++(-\omega_4)^+}})\\
&&\times\|h\|_{H_{-N}^{-N}}\|f\|_{H_{-N}^{-N}}+\|h\|_{L^2_{(-\omega_5)^++(-\omega_6)^+}}\|g\|_{H_{-N}^{-N}}\|f\|_{H_{-N}^{-N}})
\eeno
with $a,b\geq0,a+b=(2\ell +2s-1)1_{2s>1}+(2\ell+2s-1+\de)1_{2s=1}+2\ell 1_{2s<1}$ and $c_j,d_j\geq0,c_j+d_j=(2\ell +2s-1/2)_{2s\geq1/2,2s\neq1}+(2\ell +1/2+\de)1_{2s=1}+2\ell 1_{2s<1/2},j=1,2$. $\om_i\in\R,i=1,\cdots,6$ satisfying $\om_i+\om_{i+1}=\ga+2s-1,i=1,3,5$.
In particular, we have
\beno
&&\left|\sum_{j=-1}^\infty(\F_jQ(g,\<D\>^\ell h)-Q(g,\F_j\<D\>^\ell h),\F_j\<D\>^\ell f)\right|\\
&\ls& C_{N,\ell}(\|g\|_{L^1_{(\ga+2s-1)^++(-\omega_1)^++(-\omega_2)^++\de}}\|\<D\>^\ell h\|_{H^{a_1}_{\omega_1}}\|\<D\>^\ell f\|_{H^{b_1}_{\omega_2}}+\|g\|_{L^2_{2+(-\omega_3)^++(-\omega_4)^+}}\|\<D\>^\ell
h\|_{H^{c_1}_{\omega_3}}\|\<D\>^\ell f\|_{H^{d_1}_{\omega_4}}\\
&&+\|\<D\>^\ell h\|_{L^2_{(-\omega_7)^++(-\omega_8)^+}}\|g\|_{H^{c_3}_{\omega_7}}\|\<D\>^\ell f\|_{H^{d_3}_{\omega_8}}+(\|g\|_{L^1_{(\ga+2s-1)^++(-\omega_1)^++(-\omega_2)^++\de}}+\|g\|_{L^2_{2+(-\omega_3)^++(-\omega_4)^+}})\\
&&\times\|\<D\>^\ell h\|_{H_{-N}^{-N}}\|\<D\>^\ell
f\|_{H_{-N}^{-N}}+\|\<D\>^\ell h\|_{L^2_{(-\omega_7)^++(-\omega_8)^+}}\|g\|_{H_{-N}^{-N}}\|\<D\>^\ell f\|_{H_{-N}^{-N}}).
\eeno
where $a_1,b_1,c_1,d_1$ satisfy $a_1+b_1=(2s-1)1_{2s>1}+(2s-1+\de)1_{2s=1}$ and $a_1=b_1=0$ when $2s<1$. $c_j+d_j=(2s-1/2)1_{2s\geq1/2,2s\neq1}+(1/2+\de)1_{2s=1}$ and $c_j=d_j=0$ when $2s<1/2$,$j=1,3$. $\om_7,\om_8\in\R$ satisfying $\om_7+\om_8=\ga+2s-1.$
Then we conclude the desired results by combining above two estimates.
\end{proof}

 {\bf Acknowledgments.} Chuqi Cao  is supported by grants from Beijing Institute of Mathematical Sciences and Applications and Yau Mathematical Science Center, Tsinghua University.  Ling-Bing He  and Jie Ji are supported by NSF of China under  Grants 11771236 and 12141102.


\begin{thebibliography}{1000}
\bibitem{ADVW}
R. Alexandre, L. Desvillettes, C. Villani, B. Wennberg.
{\em Entropy Dissipation and Long-Range Interactions.}
Arch. Rational Mech. Anal. 152, 327-355 (2000).


\bibitem{AMUXY}
R. Alexandre, Y. Morimoto, S. Ukai, C.-J. Xu, and T. Yang.
{\em Regularizing effect and local existence for the non-cutoff Boltzmann equation.} Arch. Ration. Mech. Anal., 198(1): 39-123, 2010.


\bibitem{AMUXY2}
R. Alexandre, Y. Morimoto, S. Ukai, C.-J. Xu, and T. Yang,
{\em The Boltzmann equation without angular cutoff in the whole space: I, global existence for soft potential. }
 Journal of Functional Analysis 262 (2012), no. 3, 915-1010.


\bibitem{AMUXY3}
R. Alexandre, Y. Morimoto, S. Ukai, C.-J. Xu, and T. Yang.
{\em The Boltzmann equation without angular cutoff in the whole space: II, Global existence for hard potential.} Anal. Appl. (Singap.), 9(2):113-134, 2011.

\bibitem{AMUXY4}
R. Alexandre, Y. Morimoto, S. Ukai, C.-J. Xu, and T. Yang.
{\em Global existence and full regularity of the Boltzmann equation without angular cutoff.} Comm. Math. Phys., 304(2):513-581, 2011.

\bibitem{AV}
R. Alexandre and C. Villani.
{\em On the Boltzmann equation for long-range interactions.}  Comm. Pure Appl. Math., 55(1):30-70, 2002.


 \bibitem{AMSY}
 R. Alonso, Y. Morimoto, W. Sun and T. Yang.
 {\em Non-cutoff Boltzmann equation with polynomial decay perturbation}. Revista Matematica Iberoamericana,
37(2021), no. 1, 189-292.



\bibitem{AMSY2}
 R. Alonso, Y. Morimoto, W. Sun and T. Yang.
 {\em De Giorgi argument for weighted $L^2 \cap L^\infty$ solutions to the non-cutoff Boltzmann equation.}  arXiv:2010.10065, 2020.


\bibitem{Arkyard}
L. Arkeryd, {\em Stability in $L^1$ for the spatially homogeneous Boltzmann equation}, Arch. Rational Mech. Anal. 103 (1988), no. 2, pp. 151-167.



\bibitem{CTW}
K. Carrapatoso, I. Tristani, and K.-C. Wu.
{\em Cauchy problem and exponential stability for the inhomogeneous Landau equation. }
Arch. Ration. Mech. Anal. 221(2016), no.1, 363-418.


\bibitem{CM}
K. Carrapatoso and S. Mischler.
{\em Landau equation for very soft and Coulomb potentials near Maxwellians.} Annals of PDE (2017), no.~3, 1-65.







\bibitem{CH}
 Y. Chen and L. He.
 {\em Smoothing estimates for Boltzmann equation with full-range interactions: Spatially homogeneous case.} Archive for rational mechanics and analysis, 201(2):501-548, 2011.

\bibitem{CH2}
Y. Chen and L. He.
{\em Smoothing estimates for Boltzmann equation with full-range interactions: Spatially inhomogeneous case.} Archive for Rational Mechanics and Analysis, 203(2):343-377, 2012.




\bibitem{DL}
R. J. DiPerna and P.-L. Lions.
{\em On the Fokker-Planck-Boltzmann equation.}
Comm. Math. Phys., 120(1):1-23, 1988.

\bibitem{DL2}
 R. J. DiPerna and P.-L. Lions.
 {\em On the Cauchy problem for Boltzmann equations: global existence and weak stability.}  Ann. of Math. (2), 130(2):321-366, 1989.



\bibitem{DHYZ} R. Duan, L.-B. He, T. Yang and Y.-L. Zhou, {\em Solutions to the non-cutoff Boltzmann equation in the grazing limit}, arXiv:2105.13606.

\bibitem{DHWY}
R. Duan, F. Huang, Y. Wang, and T. Yang.
{\em Global well-posedness of the Boltzmann equation with large amplitude initial data.}
Arch. Ration. Mech. Anal., 225(1):375-424, 2017.





\bibitem{DLSS}
R. Duan, S. Liu, S. Sakamoto, and R. Strain
{\em Global mild solutions of the Landau and non-cutoff Boltzmann equations} Communications on Pure and Applied
Mathematics, 74 (2021), no. 5, 932-1020.




\bibitem{DLYZ2}
R. Duan, S. Liu, T. Yang, and H. Zhao.
{\em Stability of the nonrelativistic Vlasov-Maxwell-Boltzmann system for angular non-cutoff potentials.} Kinet. Relat. Models, 6(1): 159-204, 2013.


\bibitem{DV}
 L. Desvillettes and C. Villani.
 {\em On the trend to global equilibrium for spatially inhomogeneous kinetic systems: the Boltzmann equation.}
 Invent. Math., 159(2):245-316, 2005.

\bibitem{F}
N. Fournier. 
{\em On exponential moments of the homogeneous Boltzmann equation for hard potentials without cutoff.} Comm. Math. Phys. Vol. 387, 973-994, 2021.


\bibitem{GS}
P. T. Gressman and R. M. Strain.
{\em Global classical solutions of the Boltzmann equation without angular cut-off.} J. Amer. Math. Soc., 24(3):771-847, 2011.




\bibitem{GMM}
M. P. Gualdani, S. Mischler, and C. Mouhot.
{\em Factorization of non-symmetric operators and exponential H-theorem.} M\'em. Soc. Math. Fr. (N.S.), (153):137, 2017.


\bibitem{G2}
Y. Guo.
{\em Classical solutions to the Boltzmann equation for molecules with an angular cutoff.} Arch. Ration. Mech. Anal., 169(4):305-353, 2003.



\bibitem{G4}
Y. Guo.
{\em Decay and continuity of the Boltzmann equation in bounded domains.} Arch. Ration. Mech. Anal., 197(3), 713-809,
(2010).


\bibitem{GKTT}
Y. Guo, C. Kim, D. Tonon, A. Trescases,
{\em Regularity of the Boltzmann equation in convex domains.} Invent. Math., 207(1), 115-290, 2017.


\bibitem{H}
L.-B. He,
{\em Sharp bounds for Boltzmann and Landau collision operators.} Annales Scientifiques de l'\'Ecole Normale Sup\'erieure 51(5), 2018.

%\bibitem{HJ}
%L.-B. He. and J. Ji,
% {\em Uniqueness and global dynamics of spatially homogeneous non cutoff Boltzmann equation with moderate soft potentials},  arXiv:2106.01525.

\bibitem{HJ2}
L.-B. He and J.-C. Jiang.
{\em On the global dynamics of the inhomogeneous Boltzmann equations without angular cutoff: Hard potentials and Maxwellian molecules}, arXiv:1710.00315.


\bibitem{HJZ}
 L.-B. He, J.-C. Jiang, Y.-L. Zhou,
{ \em On the cutoff approximation for the Boltzmann equation with long-range interaction.}  J. Stat. Phys.  181 (2020), no. 5, 1817-1905.

\bibitem{HZ1} L.-B. He,Y.-L. Zhou, {\em Asymptotic analysis of the linearized Boltzmann collision operator from angular cutoff to non-cutoff},
 to appear in Ann. Inst. H. Poincare Anal. Non Lineaire.

 \bibitem{HZ2} L.-B. He,Y.-L. Zhou, {\em Boltzmann equation with cutoff Rutherford scattering cross section near Maxwellian},
  Arch. Ration. Mech. Anal.. 242 (2021), no. 3, 1631–1748.


\bibitem{H2}
F. H\'erau.
{\em Short and long time behavior of the Fokker-Planck equation in a confining potential and applications.} J. Funct. Anal. 244, 1 (2007), 95-118.



\bibitem{HTT}
F. H\'erau, and D. Tonon and I. Tristani.
{\em  Regularization estimates and Cauchy Theory for inhomogeneous Boltzmann equation for hard potentials without cut-off.}
 Commun. Math. Phys., 377, 697-771, (2020).


\bibitem{HST}
Henderson, Christopher; Snelson, Stanley; Tarfulea, Andrei, {\em Local well-posedness of the Boltzmann equation with polynomially decaying initial data},  Kinet. Relat. Models 13 (2020), no. 4, 837-867.




\bibitem{HMUY}
Z. H. Huo, Y. Morimoto, S. Ukai, and T. Yang,
{\em Regularity of solutions for spatially homogeneous Boltzmann equation  without angular cutoff.} Kinetic and Related Models 1 (2008), 453-489.

\bibitem{IMS}
C. Imbert, C. Mouhot, L. Silvestre.
{\em Decay estimates for large velocities in the Boltzmann equation without cutoff. }
Journal de l'\'Ecole Polytechnique Math\'ematiques 7, 143-184 (2020)

\bibitem{IMS2}
C. Imbert, C. Mouhot, L. Silvestre.
{\em Gaussian lower bounds for the Boltzmann equation without cut-off.} SIAM J. Math. Anal. 52, no. 3, 2930-2944.


\bibitem{IS}
C.  Imbert and L. Silvestre.
{\em The weak Harnack inequality for the Boltzmann equation without cut-off.} J. Eur. Math. Soc. 22(2):507-592, 2020.

\bibitem{IS2}
C. Imbert and L. Silvestre.
{\em Global regularity estimates for the Boltzmann equation without cut-off}, accepted by Journal of the American Mathematical Society.


\bibitem{IS3}
C. Imbert and L. Silvestre.
{\em The Schauder estimate for kinetic integral equations.} Analysis and PDE, 14(1), 171-204.






\bibitem{LL}
E. Lieb and M. Loss. {\em Analysis 2nd}, American Mathematical Society.


\bibitem{L}
P.-L. Lions.
{\em On Boltzmann and Landau equations.}  Philos. Trans. Roy. Soc. London Ser. A, 346(1679):191-204, 1994.



\bibitem{MS}
C. Mouhot and R. M. Strain.
{\em Spectral gap and coercivity estimates for linearized Boltzmann collision operators without angular cutoff.} J. Math. Pures Appl. (9), 87(5):515?35, 2007.




\bibitem{P}
Y. P. Pao,
{\em Boltzmann collision operator with inverse-power intermolecular potentials. i, ii.}, Comm. Pure Appl. Math. 27
(1974), 407-428, 559-581.

\bibitem{S}
L. Silvestre.
{\em A new regularization mechanism for the Boltzmann equation without cut-off.} Comm. Math. Phys., 348(1):69-100, 2016.


\bibitem{SS}
L. Silvestre, S. Snelson.
{\em Solutions to the non-cutoff Boltzmann equation uniformly near a Maxwellian}, arXiv:2106.03909.




\bibitem{V}
C. Villani.
{\em A review of mathematical topics in collisional kinetic theory.} In Handbook of mathematical
fluid dynamics, Vol. I. North-Holland, Amsterdam, 71-305,  2002.


\bibitem{V2}
C. Villani.
{\em On the Cauchy problem for Landau equation: sequential stability, global existence.} Adv. Differential Equations, 1(5):793-816, 1996.

\bibitem{V3}
C. Villani.
{\em On a new class of weak solutions to the spatially homogeneous Boltzmann and Landau equations.}
Arch. Rational Mech. Anal., 143(3): 273-307, 1998.

 \bibitem{Bernt}
B. Wennberg, {\em Stability and exponential convergence for the Boltzmann equation}, Arch. Rational Mech. Anal. 130 (1995), no. 2, pp. 103-144.
\bibitem{HKgo}
H. Kumano-go, {\em  Pseudo-differntial operators. } MIT Press, 1982.




\end{thebibliography}
\end{document}